\documentclass{article}

\usepackage{graphicx} 
\usepackage{import}
\usepackage{titling}
\usepackage{graphicx} 
\usepackage{mathrsfs}
\usepackage[utf8]{inputenc}
\usepackage[T2A]{fontenc}
\usepackage[english]{babel}
\usepackage{xcolor}
\usepackage{amsmath}
\usepackage{amsfonts, amssymb}
\usepackage{amsthm,epsfig,epstopdf,titling,array}
\usepackage{float}
\usepackage{soul}
\usepackage[initials]{amsrefs}
\usepackage{fullpage}
\usepackage{color}
\usepackage{eucal}
\usepackage{geometry}
\usepackage{authblk}

\usepackage[hidelinks]{hyperref}

\theoremstyle{plain}
\newtheorem{thm}{Theorem}[section]
\newtheorem{lemma}[thm]{Lemma}
\newtheorem{claim}[thm]{Claim}
\newtheorem*{claim*}{Claim}
\newtheorem{corollary}[thm]{Corollary}
\newtheorem*{prop*}{Proposition}
\newtheorem{prop}[thm]{Proposition}
\newtheorem*{thm*}{Theorem}
\newtheorem*{lemma*}{Lemma}

\newtheorem{exmp}{Example}[section]

\newtheorem*{theorem*}{Theorem}
\newtheorem*{proposition*}{Proposition}

\newcommand{\e}{e}

\newtheorem{rem}[thm]{Remark}

\newcommand{\pl}{\partial}
\newcommand{\wt}{\widetilde}

\newcommand{\T}{\mathbb{T}}
\newcommand{\R}{\mathbb{R}}

\newcommand{\ud}{\ensuremath{\mathrm{d}}}

\renewcommand{\div}{\textup{div}}

\newcommand{\comment}[1]{}
\newcommand{\textcalor}[2]{}
\newcommand{\pseudosection}[1]{
\vspace{5pt}

\textbf{#1}

\vspace{5pt}
}

\title{Eigenfunctions with double exponential rate of localization}

\author{S. Krymskii\thanks{Université de Genève, Département de Mathématiques, 1205, Geneva,
Switzerland/\\  Stanford University, Department of Mathematics, Stanford, California 94305, \url{krymskiy.stas@gmail.com}}
\and
A. Logunov\thanks{Department of Mathematics, Massachusetts Institute of Technology, Simons Building (Building 2), Room 2-478,
77 Massachusetts Avenue,
Cambridge, MA 02139-4307, \url{log239@yandex.ru}}
\and 
F. Pagano\thanks{Université de Genève, Département de Mathématiques, 1205, Geneva,
Switzerland, \url{francois.pagano@unige.ch}}
}

\date{}
\begin{document}

\maketitle

\begin{abstract}
    
    We construct  a real-valued solution to the eigenvalue problem $-\div(A\nabla u)=\lambda u$, $\lambda>0,$ in the cylinder $\T^2\times \R$ with a real, uniformly elliptic, and uniformly $C^1$  matrix $A$  such that $|u(x,y,t)|\leq C e^{-c\e^{c|t|}}$ for some $c,C>0$. We also construct a complex-valued solution to the heat equation $u_t=\Delta u + B \nabla u$ in a half-cylinder with continuous and uniformly bounded $B$, which also decays with double exponential speed. Related classical ideas, used in the construction of counterexamples to the unique continuation by Plis and Miller, are reviewed.
    
\end{abstract}

\tableofcontents

\comment{
\begin{center}
    \textbf{Abstract}
\end{center}
We construct a real, smooth, elliptic  matrix $A$ with bounded coefficients and a real solution $u$ to $\div(A\nabla u)=0$ on the cylinder $\T^2\times \R^+$ which decays like $e^{-c\e^{dt}}$ at infinity. We also construct a real solution to the eigenvalue problem $-\div(A\nabla u)=\lambda u$ with the same rate of decay on the whole cylinder $\T^2\times \R$. Finally, we construct a complex-valued solution with the same decay, to the heat equation $u_t=\Delta u + B \nabla u$ on a half-cylinder with bounded $B$. These constructions are new and optimal: non-trivial solutions cannot decay faster.
}

\section{Introduction}

\subsection{Landis conjecture}

In the late 60's, Landis asked the following question: how rapidly can a non-zero solution to  
\begin{align}\label{landis}
    -\Delta u + Vu=0, \hspace{0.5cm} V \in L^{\infty}(\R^d)
\end{align}
decay to zero at infinity? According to \cite{KSW15}, he conjectured that if $u$ decays faster than exponentially, i.e., if $|u(x)|\leq \e^{-|x|^{1+\epsilon}},$ for some $\epsilon>0,$ then $u \equiv 0.$ A stronger version \cite{KL88} of Landis conjecture suggests that if $|u(x)|\leq \e^{-C|x|}$ for sufficiently large $C$ depending on $\|V\|_\infty$, then $u\equiv 0$.

Landis' conjecture and its versions for linear elliptic PDE of second order were extensively studied (see \cite{D23}, \cite{EKPV06}, \cite{K06}, \cite{KSW15}, \cite{K98}, \cite{LMNN20}, \cite{M92}, \cite{Z16} and references therein). 
Our motivation for questions of quantitative unique continuation such as the quantitative side of Landis' conjecture arises from problems of spectral theory.
The article of Bourgain and Kenig \cite{BK05} brought a lot of attention to Landis' conjecture in connection to the Anderson-Bernoulli model. The artilce \cite{BK05} has a setting in $\R^d$, but the Anderson-Bernoulli model has a natural discrete setting in $\mathbb{Z}^d$. So far, Anderson localization in the discrete setting has only been proven near the edge of the spectrum and only in dimension $d=2$ and $d=3$  by Ding, Smart, Li and Zhang, see  \cite{DS20}, \cite{LZ22}, \cite{S22} and also \cite{BLMS22}.  The main difficulty in proving Anderson localization in $\mathbb{Z}^d$ is the lack of theorems in discrete quantitative unique continuation.

Another question related to Landis conjecture and popularised by Kuchment \cite{K12} concerns periodic operators. Let $Lu=\div(A\nabla u)+ V u$ be an elliptic differential operator of second order, whose coefficients are real, bounded, $C^1$-smooth (or even $C^{\infty}$ smooth) in $\mathbb{R}^d$ and periodic with respect to $\mathbb{Z}^d$ translations. Is it true that if $u$ is a solution to $Lu=0$ and 
 \begin{align}\label{kuchment}
     |u(x)| \leq \e^{-|x|^{\gamma}},
 \end{align}
where $\gamma>1$, then $u\equiv 0$? According to Kuchment (see \cite[Theorem 4.1.5]{K12}, \cite[Theorem 4.1.6]{K12} and \cite[Corollary 6.15]{K16}) a positive answer to this question would help towards the folklore conjecture that every such periodic operator only has absolutely continuous spectrum. 

\vspace{5pt}

\subsection{Main results}
The main results of this paper are the following theorems.
\begin{thm}
\label{theoremforhalfcylinder}
    In the 3-dimensional cylinder $\mathbb{T}^2\times\mathbb{R}^+$, there exists a $C^1$-uniformly smooth and uniformly elliptic real-valued matrix $A$ and a non-zero uniformly $C^2$  real function $u$ such that \begin{equation} \label{eigen} \div(A\nabla u)=0\end{equation}and  $u$ has double exponential decay: for any $T\gg 1,$
$$\sup_{\T^2 \times \{t\geq T\}}|u(x,y,t)|\leq Ce^{-ce^{c T}}$$
for some numericals $c, C>0$.
\end{thm}

Note that we cannot construct a non-trivial smooth $A$-harmonic solution on the whole cylinder $\T^2 \times \R $ with double exponential decay in both directions. Indeed, by periodicity, we would get a bounded $A$-harmonic solution on $\R^3$ which decays to zero in some directions. But then, Liouville's Theorem implies that this solution is necessarily trivial.

However, solutions to the eigenvalue equation do not satisfy the Liouville property and we can indeed construct an eigenfunction with double exponential decay in both directions. This is our second main result.

\begin{thm}
\label{EigenTheorem}
    In the 3-dimensional cylinder $\mathbb{T}^2\times\mathbb{R}$, for every $\mu>0$, there exists a $C^1$-uniformly smooth and uniformly elliptic real-valued matrix $A=A_{\mu}(x,y,t)$ and  a non-zero uniformly $C^2$ real function $u=u_\mu(x,y,t)$ such that \begin{equation} \label{eigen2} \div(A\nabla u)=-\mu u\end{equation}and  $u$ has double exponential decay: for any $T\gg 1,$
$$\sup_{\T^2 \times \{|t|\geq T\}}|u(x,y,t)|\leq Ce^{-ce^{cT}}$$
for some numericals $c(\mu), C(\mu)>0$.
\end{thm}

Surprisingly, we can also achieve a double exponential decay for a parabolic equation with a constant higher-order term and a continuous and uniformly bounded first-order coefficient. However, this construction is complex-valued and we don't know whether the example below can be made real-valued. In the next Theorem, $\dot{u}$ denotes the $t$-derivative and $ \nabla, \Delta$ only involve spatial derivatives.

\begin{thm}\label{maintheoremparabolicbeg}
    In the 3-dimensional cylinder $\T^2 \times \R^+$,  there exists a complex vector field $B \in \mathbb C^2$ which is continuous and uniformly bounded and there exists a non-zero, uniformly $C^2$ complex-valued function $u$ such that 
    \begin{align}
        \dot{u} = \Delta u + B\nabla u
    \end{align}
    and $u$ has double exponential decay: for any $T\gg 1,$
    \begin{align}
       \sup_{\T^2 \times \{t \geq T\} }|u(x,y,t)| \leq e^{-ce^{cT}}
    \end{align}
    for some numerical $c>0.$
\end{thm}

Theorems 1.1 and 1.2 are proved in sections \ref{toolboxnew}---\ref{proofofpropslabel}, and Theorem \ref{maintheoremparabolicbeg} is proved in section \ref{parab}.

\begin{rem}
  We emphasize that the solutions that we construct in our three main Theorems have the fastest possible decay. Non-trivial solutions to these equations cannot decay strictly faster than double-exponentially. We refer to \cite{L63} for a proof in the elliptic case and to  \cite{CM2022} for a proof in the parabolic case. 
\end{rem}

\subsubsection{Interpretation of our results}

The results that we prove have a negative meaning. They show that there is an obstacle to generalize (at least directly) the approach for the continuous Anderson-Bernoulli model by Bourgain and Kenig to operators in divergence form $Lu=\div(A\nabla u)+ V u$ because one of the ingredients of the approach fails: quantitative unique continuation properties for operators in divergence form (with variable coefficients of higher order) are substantially weaker than of the operators in the form $Lu=\Delta u + V u$. Our results also show that we cannot hope for a positive answer to Kuchment's question in the periodic setting without using periodicity in all directions: a cylinder is periodic in all but one direction, in which the solution can decay super-exponentially fast.

 In both questions however, there is additional information on the coefficients: the randomness of $V$ (and $A$) in the Anderson-Bernoulli model and the periodicity of $A$ and $V$ in the periodic setting. The current methods of quantitative unique continuation such as Carleman inequalities and monotonicity formulas do not take these informations into account, and it would be interesting to find a method to use randomness or periodicity of the coefficients to prove quantitative results ensuring slow decay on large scales. We mention a non-exhaustive list of results related to unique continuation and to the spectrum of random and periodic operators that uses these extra pieces of information: \cite{AKS21}, \cite{AKS23}, \cite{DG23}, \cite{KZZ22}.

\subsection{Few subtleties related to Landis' conjecture}
\label{subtle}

In dimension $d=1$, the strong Landis conjecture has been solved in \cite{R21} (see also \cite{LB20}) for a general second-order elliptic equation with real bounded coefficients. In higher dimensions, the situation is quite different and strongly depends on the type of potential $V$ that we consider.

\pseudosection{Complex vs real solutions} Landis' conjecture \eqref{landis}  has a real  and a complex setting. In the complex setting, the problem was solved by   Meshkov \cite{M92}, who constructed a bounded complex-valued potential $V$ and a  non-zero complex-valued solution $u$ to $\Delta u + Vu=0$ in $\R^2$ satisfying $|u(x)| \leq \e^{-c |x|^{4/3}}.$ Using the method of Carleman's inequalities, Meshkov also proved that if $V$ is a bounded complex potential in $\mathbb{R}^d$, $d\geq 2$, then any solution to $\Delta u + Vu=0$ in $\mathbb{R}^d$ with $$|u(x)| \leq \exp(-|x|^{4/3+\varepsilon}),\,\,\varepsilon>0$$ is identically zero. 

In the case of real bounded potential $V$ and dimension $d=2$, the weak Landis conjecture was recently proved in \cite{LMNN20}. Namely, it was shown that if $u$ solves $\Delta u+ Vu=0$ with real bounded $V$ and $u$ decays at infinity faster than exponentially:$|u(x)|\leq C \exp(-|x|^{1+\varepsilon}) $, then $u$ is identically zero. The Landis conjecture for real potentials is still open in dimension $3$ and higher, but a recent article \cite{FK23} raises a doubt that Landis conjecture for real potentials is true in dimension $4$ and higher as it essentially disproves a local quantitative version of Landis conjecture in dimension $4$.

\pseudosection{Local version}
A local version of Landis conjecture has been studied and used by Bourgain and Kenig in connection to Anderson localization \cite{BK05}. They proved the following fact. 

Let $u$ be a solution to $-\Delta u + Vu=0$ in $B(0, 2R)\subset \R^n$ for sufficiently large $R>10$ and let $V$ be a bounded  potential. Suppose that $|u(0)|=\sup_{B(0,2R)} |u(x)|=1$ and define $$M(R):=\inf_{|x_0|=R} \sup_{B(x_0,1)}|u(x)|.$$ Then
\begin{align}\label{locallandisbk}
    M(R) \geq  \e^{-C R^{4/3} \log(R)}.
\end{align}
This quantitative result of Bourgain and Kenig has been generalized to more general second-order elliptic PDEs with lower-order terms in \cite{D14} and \cite{LW14}.  All of these results are proved using Carleman inequalities, which do not care whether the potential is real or complex. In the case of real bounded potentials on the plane, \eqref{locallandisbk} has been recently improved in \cite{LMNN20}: with the same notation, when $d=2$ and $V$ is real and bounded, it holds that
\begin{align}\label{locallandislog}
    M(R) \geq \e^{-C R \log^{3/2}(R)}.
\end{align} 
 The approach in \cite{LMNN20} uses quasi-conformal mappings and a new method of holes. Unfortunately, the part of the approach involving quasi-conformal mappings doesn't work in higher dimensions. We also refer to \cite{KSW15} where quasi-conformal mappings were also used to study Landis conjecture in the plane for the equation $\Delta u + B\nabla u - Vu=0$ with $B,V$ bounded, $V \geq 0$ and to \cite{DKW17}  for the equation $\div(A\nabla u) - Vu=0$ with $V\geq 0$ and bounded. Finally, a recent work \cite{LB24}
 shows that real solutions to $\Delta u + \textup{div}(W_1 u) + W_2\nabla u + Vu=0$ on the plane with bounded $W_1,W_2,V$ cannot decay like $\exp(-|x|^{1+\delta})$, $\delta > 0$.

\pseudosection{Constant vs variable coefficients}
We want to stress that solutions to elliptic PDEs with constant coefficients for the higher order derivatives behave \textit{drastically different} from elliptic PDEs with variable coefficients. Indeed, Theorem \ref{theoremforhalfcylinder} and Theorem \ref{EigenTheorem} show that in the cylinder $\T^2 \times \R^+$, solutions to elliptic PDEs with variable coefficients can decay at double exponential speed in the non-periodic direction. On the other hand, Elton \cite{E20} proved that solutions to $-\Delta u + Vu=0$ in $\T^2 \times \R^+$ with bounded potential $V$, can decay at most exponentially fast in the non-periodic direction. This result is sharp by considering $u=\e^{-t}$ where $t \in \R^+$. When the torus $\T^d$ has dimension $d \geq 3,$ the speed of decay is governed by the local result of Bourgain and Kenig \eqref{locallandisbk}. Recently, this result is shown to be sharp by constructing a real-valued example \cite{FK23}.

\vspace{5pt}

\textbf{Regularity of the coefficients and non-unique continuation}
\vspace{5pt}

A differential operator $L$ is said to have the unique continuation property, 
if any solution to $Lu=0$ in a connected domain $\Omega$ vanishing in an open subset of $\Omega$ is zero in the whole $\Omega$.

Solutions to linear elliptic differential equations with sufficiently regular coefficients satisfy the unique continuation property, see \cites{GL86,GL87}, but Plis and Miller described in their articles \cites{P63, M74} 3-dimensional counterexamples to the unique continuation property for elliptic second-order PDEs with Hölder continuous coefficients. In particular, there is a matrix $A$ with Hölder coefficients and a non-zero solution $u$ to the elliptic equation  $\div(A\nabla u)=0$ such that $u$ vanishes on an open set. The work \cite{F01} constructed a compactly supported solution to $\div(A\nabla u)+\lambda u=0$ in $\mathbb{R}^3$ where the elliptic matrix $A$ can be chosen $\alpha$-Hölder for any $\alpha\in (0,1)$. For the parabolic equation, Miller \cite{M74} constructed a counterexample to unique continuation with coefficients in some Hölder space for the time variable.

While writing this article, we discovered a paper by Mandache \cite{M98} which does not seem very well known.  In his article, he constructed counterexamples to unique continuation for both elliptic $\div(A\nabla u)=0$ and parabolic $u_t = \div(A\nabla u)$ equations generalizing the constructions of Miller. His coefficients are in all Hölder spaces $C^{1-\epsilon}$ with $\epsilon>0.$ We refer to section \ref{PML1} for more details on counterexamples to unique continuation. 

\pseudosection{Coefficients and speed of decay}
At the same moment, if the coefficients of the matrix $A$ are Lipschitz and if $A$ is elliptic, then the unique continuation property holds, see \cites{GL86,GL87}. The proof usually relies on the frequency function approach or on Carleman inequalities. We would like to mention that Gusarov announced in \cite{G79} the construction of a real solution $u$ to a differential equation $$\div(A\nabla u) + B\nabla u + C u =0 $$ in $\mathbb{T}^2\times \mathbb{R}^+$ with the following properties: $A(x,y,t)$ is uniformly elliptic and Lipschitz, $B(x,y,t)$ and $C(x,y,t)$ are bounded and decay to zero exponentially fast as $t \to +\infty$, but $|u(x,y,t)|\leq e^{-ce^{ct}}$ as $t \to +\infty$. Our first main result, Theorem \ref{theoremforhalfcylinder}, allows one to have zero $B$ and $C$, while our second main result, Theorem \ref{EigenTheorem}, allows one to obtain a function decaying in both directions and have a constant $C.$

\subsection{Notation}
We present the notation that will be used in the remainder of this article.
\begin{itemize}        
        \item For 2-dimensional matrices we will use the Latin letter $A$ and denote its entries as $A_{xx},A_{xy},A_{yx},A_{yy},$ whereas  3-dimensional matrices will be denoted by the calligraphic letter $\mathcal{A}.$
        \item The 3-dimensional matrix $\mathcal{A}$ is obtained from the 2-dimensional matrix $A$ by \begin{align}\label{threedmatrixaugmented}
            \mathcal{A} = \begin{pmatrix}A_{xx}&A_{xy}&0\\A_{yx}&A_{yy}&0\\0&0&1\end{pmatrix}.
        \end{align} In other words, $A$ is the top $2\times 2$ minor of $\mathcal{A}.$
        \item We denote the $\alpha$-Hölder semi-norm by 
        \begin{align}\label{semi-norm}|f|_{C^{0, \alpha}} := \sup_{x\neq y}\frac{|f(x)-f(y)|}{|x-y|^\alpha}.
        \end{align}
        \item Unless specified otherwise, the functions are defined in the cylinder $\T^2\times \R^+$ or $\T^2\times \R$. The first two (periodic) coordinates in the cylinder will be denoted by $x,y$ and the third coordinate will be denoted by $t$. 
        \item The derivative of a function $u$ with respect to $t$ will be denoted by $\dot u$ while the spatial derivatives are denoted as $\nabla u.$ 
        \item Here and further, unless noted otherwise, we consider only equations of the type  $\ddot{u}+\div (A\nabla u) = -\mu u$ with a 2-dimensional matrix $A$ and $\mu \in \R.$ 
        \item By $c,C,C_1,C_2...$ we will usually denote numerical constants, whose value may be different from line to line, and sometimes these constants will depend on additional parameters and this dependence can be reflected as $c=c(\mu), C=C(\mu)$. 
        The constants $k,k'...$ in trigonometric polynomials will be assumed to be integers. By assuming  $A\lesssim B$ we mean that $A\leq C B$ for some fixed constant $C$, by assuming $A\ll B$ we mean that $A\leq C B$ for some sufficiently large constant $C$. By claiming  $A\lesssim B$ we mean that $A\leq C B$ for some sufficiently large constant $C$. 
    \end{itemize}

\subsection{Structure of the paper} 
In Section \ref{counterexamplesection}, we present an idea that plays an important role, both in the construction of fast-decaying solutions and in the construction of counterexamples to unique continuation: switching to faster-decaying solutions infinitely many times. We illustrate this idea by explicitly constructing a non-trivial solution to $\ddot{u}+\div(A\nabla u)=0$ which vanishes on an open set and with $A$ belonging to the Hölder class. Sections  \ref{toolboxnew} to \ref{proofofpropslabel} are devoted to the proof of  Theorem \ref{theoremforhalfcylinder}, about the construction of a solution to $\ddot{u}+\div(A\nabla u)=0$ in a half-cylinder with double-exponential decay and to the proof of Theorem \ref{EigenTheorem}, about a double-exponentially decaying eigenfunction to a divergence form elliptic operator on the full cylinder $\T^2 \times \R$. In Section \ref{thebuildingblocknewlabel} we explain the main ideas and constructive  steps in the proof of Theorem \ref{theoremforhalfcylinder}. In Section \ref{parab}, we present our last main result, Theorem \ref{maintheoremparabolicbeg}, which explains the construction of a (complex-valued) double-exponentially decaying solution to a parabolic equation with constant higher-order terms. Finally, to ease the reading of this article, we relegate some technical parts of the proofs in the appendix.

\subsection{Acknowledgments}
FP was funded by NCCR Swissmap and by the SNSF.\\
AL was supported by Packard Fellowship and by the SNSF.




\section{Transforming solutions}\label{counterexamplesection}

This introductory section is formally not used to prove the main results, but it is an attempt to explain an idea from the works
\cite{M74} of Miller and \cite{P63} of Plis. This idea was originally used to construct a solution to an elliptic PDE with Hölder continuous coefficients vanishing on an open set. A closely related idea will be used to construct a fast-decaying solution to an \textit{elliptic} and a \textit{parabolic} equation. We start with an abstract problem. 

\textbf{Problem of transforming solutions.} Suppose we are given two solutions $u_1,u_2$ to two different partial differential equations: 
$$ L_1 u_1 =0  \quad \& \quad L_2 u_2 =0$$
in the cylinder $\T^2\times \R$. Can we find a solution $u$ and  a partial differential equation $Lu=0$ in $\T^2\times \R$  within a certain class of equations such that
\begin{equation} \label{gluing}
u=u_1, \quad L=L_1  \textup{ on  }  \{(x,y,t): t \leq 0\} \hspace{0.5cm} \textup{ and } \hspace{0.5cm} u=u_2, \quad L=L_2 \textup{ on   } \{(x,y,t): t \geq T\}
\end{equation}
for some $T>0$?

\textbf{Example of transformation.} Let $k$ be a large integer number. The functions
$$f(x,y,t) = \cos(kx)e^{-k t} \hspace{0.5cm} \textup{ and } \hspace{0.5cm} g(x,y,t) = \cos((k+1)y)e^{-(k+1)t}$$
are both harmonic in $\T^2\times \R$: $\Delta f=\Delta g=0$.
Lemma \ref{PML1} given below shows that there are numerical constants $T,C>0$ and a function $u$ (transforming $f$ into $g$), which solves an elliptic equation $\ddot u + \div\left(A\nabla u\right )=0$ in the cylinder and
 $$u(x,y,t)=f(x,y,t) \textup{ for }  t \leq 0 \hspace{0.5cm} \textup{ and } \hspace{0.5cm} u(x,y,t)=g(x,y,t) \textup{ for } t\geq T,$$
where $A$ is the identity matrix for $t\notin (0,T)$, $A$ is elliptic in $\T^2\times \R$ with ellipticity constant smaller than $C$ and the derivatives of the coefficients of $A$ are bounded by $C$ in absolute value. 
\begin{figure}[H] 
\includegraphics[scale=0.22]{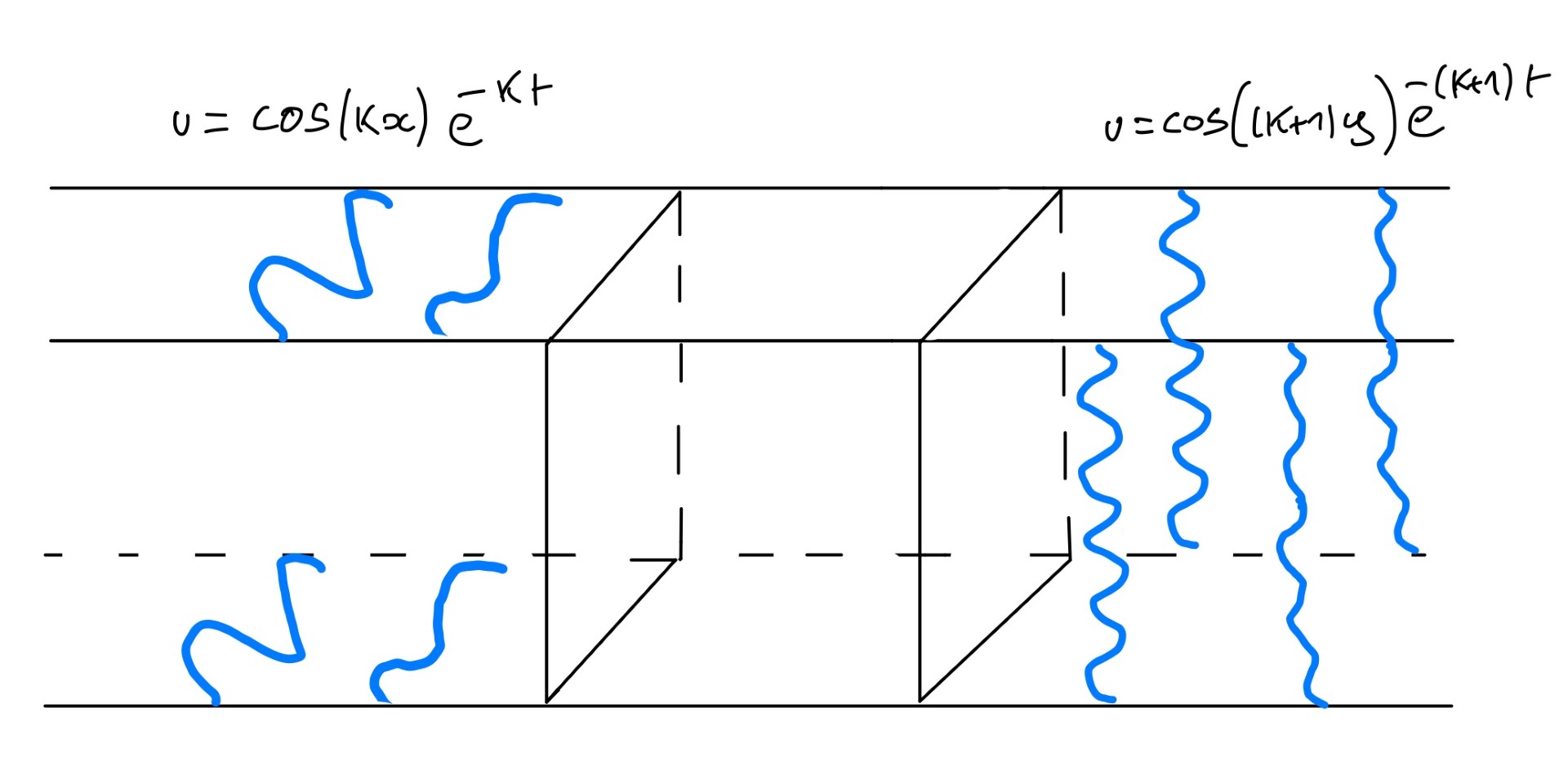}
	\centering
	\caption{Transforming $f$ into $g$}
\label{fig:theta}
\end{figure}

Let $\theta(t)$ be a $C^\infty$-smooth step function which is equal to 1 in  $\{t\leq 0\}$ and to zero in   $\{t\geq 1\}$ and which monotonically decreases for $t\in(0, 1)$, see Figure \ref{fig:theta} and footnote \footnote{We choose $\theta(t):= 1-G(\tan(\pi(t-1/2)))$ for $t\in[0,1]$ where $G(x)=\frac{1}{\sqrt{\pi}}\int_{-\infty}^x \e^{-\eta^2} \ud{\eta}$. \label{footnoteblablablaintro}}).

\begin{figure}[H] 
\includegraphics[scale=3]{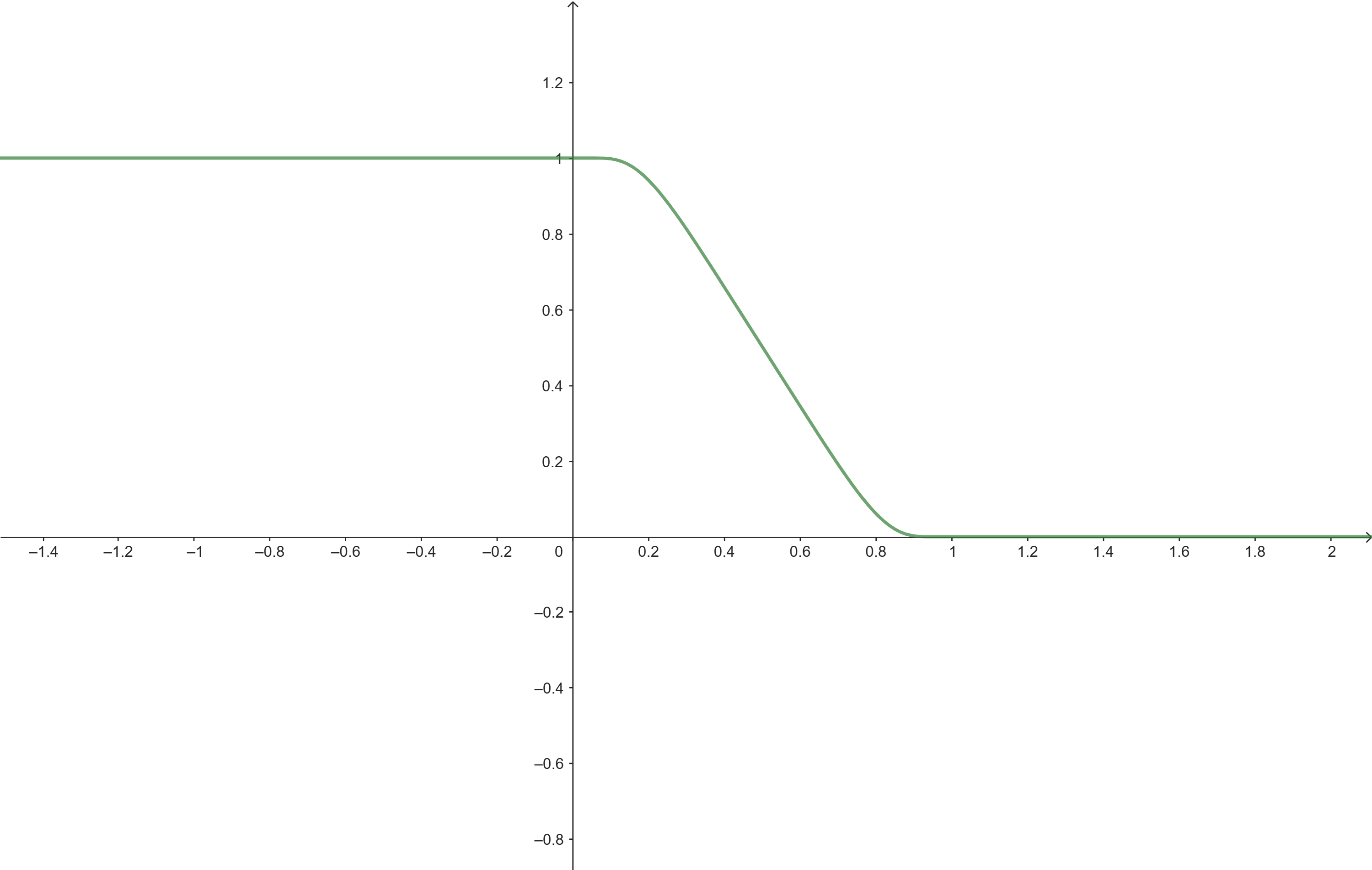}
	\centering
	\caption{A graph of the function $\theta$}
\label{fig:theta}
\end{figure}
\begin{lemma}

\label{PML1}
Consider two harmonic functions in $\T^2\times \mathbb{R}$ $$f:=\cos(k x)e^{-k t} \textup{\quad and \quad} g=\cos(k'y)e^{-k' t}$$
and assume that the positive numbers $k, k',w >0$ satisfy
\begin{equation} \label{assumptions}
0<k'-k\lesssim w^{-1}\lesssim k.
\end{equation}
Put $\alpha(t) := \theta(t/w)$ and
\begin{equation}
    \label{defn_u}
    u := \begin{cases} f+\big(1-\alpha(t)\big)g, & t\leq w\\
 \alpha(t-w)f+g,&t \geq w.
 \end{cases}
\end{equation}

Then the following holds:
\begin{enumerate}
    \item[a)] There exists a $C^1$ smooth matrix-valued function $A$ such that $u$ solves  $$\ddot u+ \div(A\nabla u)=0 \hspace{0.5cm} \textup{ on } \quad  \T^2 \times \R$$
    \item[b)] and
\begin{equation}
\label{props}
    \|A-Id\|\lesssim \frac{1}{wk}, \hspace{0.5cm} \|\nabla A\|\lesssim \frac{1}{w}, \hspace{0.5cm} \| \dot A \| \lesssim \frac{1}{w} \quad \textup{ on } \quad \T^2 \times \R. 
\end{equation}

\item[c)] The solution $u$ satisfies the following estimates on $\T^2\times [0, 2w]$: For any multi-index $\alpha \in \mathbb N^3$ of order $|\alpha|\leq 2$,
$$
|\partial^{\alpha} u(x,y,t)| \lesssim (k')^{|\alpha|} \sup_{\T^2 \times [0, 2w]} |u|.
$$

\comment{

$$|\nabla u(x,y,t)|, \, |\dot u(x,y,t)| \lesssim \left(\lambda+\lambda' \right)\sup_{t\in[0,2w]}|u|$$ and
$$
|\nabla^2 u(x,y,t)|, |\nabla \dot u(x,y,t)|, |\ddot u(x,y,t)| \lesssim \bigg(\lambda^2+(\lambda')^2 \bigg)\sup_{t\in[0,2w]}|u|.$$}

\end{enumerate}

\end{lemma}

  We note that $u=f$ in $\T^2\times(-\infty,0]$
 and $u=g$ in $\T^2\times[2w,\infty)$. We will say that Lemma \ref{PML1} allows to transform $f$ into $g$ within the set $\T^2\times[0,2w]$ via the solution $u$.

  Lemma \ref{PML1} is well known to specialists (we include the proof in Appendix \ref{theboundsforaanducounterexampleappendix} for the reader's convenience). We formally do not use Lemma \ref{PML1} in the proof of the main results, but we would like to include it here for building intuition and explaining a natural idea of how it can be used.  This lemma can be used to construct solutions to elliptic PDE with uniformly smooth coefficients, which decay faster than exponentially. It can also be used to construct solutions to elliptic PDE with Hölder coefficients, which vanish on an open set.

 One can apply the lemma to transform 
 $$f_{2n}= \cos(2nx)e^{-2nt} \hspace{0.5cm} \mbox{ into } \hspace{0.5cm} f_{2n+1}= \cos((2n+1)y)e^{-(2n+1)t}$$
 using a function $u_{2n}$ within $\T^2\times[0,T]$ with $T=2$ (i.e. $w=1)$ and all $n\geq n_0$
 and similarly transform $f_{2n+1}$ and $f_{2n+2}$ using a function $u_{2n+1}$.

A natural idea is to consequently transform $c_nf_n$ into $c_{n+1}f_{n+1}$ on $\T^2\times[nT,(n+1)T)$ for some 
constants $c_n>0$ so that
$$ c_n \sup_{x,y}|f_n(x,y,nT)|= c_{n+1} \sup_{x,y}|f_{n+1}(x,y,nT)|,$$
which is equivalent to $$ c_n\sup_{x,y}|\cos(nx)e^{-n^2T}|= c_{n+1} \sup_{x,y}|\cos((n+1)y)e^{-(n+1)nT}|,$$ 
$$ c_{n}e^{-n^2T}=c_{n+1}e^{-n(n+1)T}.$$
We make the choice 
$$ c_{n+1}= e^{n(n+1)T/2}$$
and put $u(x,y,t)=c_n \e^{-n^2T}u_n(x,y,t-nT) $ on $\T^2\times[nT,(n+1)T)$. 
The function $u$ is a solution to the PDE in divergence form
$$\ddot u+ \div(A\nabla u)=0 \quad \textup{in} \quad \T^2\times \{t>n_0T\}.$$ 
The matrix $A$ satisfies \eqref{props} with $w=1$.
In particular, $\|A-Id\|\lesssim \frac{1}{n}$ and the equation $\ddot u+ \div(A\nabla u)=0$ is elliptic with uniformly $C^1$ smooth coefficients  in the half-cylinder $\{t>n_0T\}$ if $n_0$ is sufficiently large. 
This demonstrates how one can apply a shifted version of lemma \ref{PML1} to construct a solution to a uniformly smooth and elliptic PDE, which is decaying like $e^{-ct^2}$.

\begin{corollary}
There exists a non-zero, real-valued solution $u$ to an elliptic equation in divergence form $\ddot u+ \div(A\nabla u)=0$ in the half cylinder $\T^2\times \mathbb{R}^+$ such that 
$A$ is uniformly $C^1$ smooth and uniformly elliptic in the half-cylinder and for all $t \gg 1,$
 $|u(x,y,t)| \leq e^{-ct^2}, t > 0$.    
\end{corollary}

\begin{rem}
The rate of decay $e^{-ct^2}$ is already non-trivial. By \cite{E20}, in $\T^2 \times \R^+$, non-zero solution to $-\Delta u + Vu=0$ with bounded $V$ can decay at most exponentially fast. However, in this non-constant coefficient setting, the Gaussian rate of decay given in the Corollary is not optimal and can be improved to double exponential decay using a different idea  (see Theorem \ref{theoremforhalfcylinder} and Section \ref{toolboxnew}).
\end{rem}

\pseudosection{Counterexample to unique continuation.}
This section is another application of the previous lemma and it explains the counterexample of Miller \cite{M74} to unique continuation of elliptic PDE with $\alpha$-Hölder coefficients for $\alpha \in (0,1/2)$. 

 We would like to construct a solution $u$ to $\ddot{u}+\div(A \nabla u)=0$ in $\T^2\times \R$ with elliptic $A$ such that $u=\cos(k_0x)\e^{- k_0 t}$ when $t \leq 0$  and $u= 0$ for $t\geq T$ for some $T>0$ and $k_0\gg 1$. To do that, we will choose an increasing sequence of positive numbers $\{a_n\}$  with a finite limit $T:=\lim_{n\to\infty} a_n$.  Then the interval  $[0,T]$ is split into infinitely many sub-intervals $[a_n,a_{n+1}]$.
The choice of $\{a_n\}$ is to be specified later.

 Consider the functions 
    \begin{align*}
        f_{2n}:=\cos(k_{2n}x) \e^{-k_{2n}t}, \hspace{0.5cm}   f_{2n+1}:=\cos(k_{2n+1}y) \e^{-k_{2n+1}t},
    \end{align*}
 where $\{k_n\}$ is an increasing sequence of integer numbers to be chosen later. Put $w_n= \frac{a_{n+1}-a_n}{2}$ and note that $\sum_{n=1}^\infty 2w_n =T$. 
If 
\begin{equation} \label{eq: we need}
 k_{n+1}-k_n \lesssim w_n^{-1} \lesssim k_n,
\end{equation}
 then we can apply 
Lemma \ref{PML1} to transform $f_n$ into $f_{n+1}$ using a function $u_n$ on $\T^2 \times [0,2w_n]$, 
and therefore transform $c_nf_n$ into $c_{n+1}f_{n+1}$ using a function $\tilde u_n$ on $\T^2 \times [a_n,a_{n+1}]$, where the sequence $\{c_n\}$ should satisfy 
    \begin{align}\label{seqeuncecn}
     c_ne^{-a_nk_n}=c_{n+1}e^{-a_nk_{n+1}}
    \end{align}
and $c_1$ can be chosen to be $1$.

 Our goal is to choose the sequences $\{a_n\}$ and $\{k_n\}$ so that the function 
\begin{align}\label{ulemmatwopointonetildeunzeroforbigt}
u=\tilde u_n \quad \textup{ on } \quad \T^2 \times[a_n, a_{n+1}], \hspace{0.5cm} u=0 \quad \textup{ on } \quad \T^2 \times[T, \infty)    
\end{align}

 is a solution to a divergence-form elliptic equation with $\alpha$-Hölder continuous coefficients through $\T^2 \times\{t=T\}$, where $\alpha$ can be chosen in $(0,1/2)$.

\pseudosection{The choice of the sequences ensuring the Hölder condition and the ellipticity} 
Recall that we denote the $\alpha$-Hölder semi-norm  by $|f|_{C^{0, \alpha}} := \sup_{x\neq y}\frac{|f(x)-f(y)|}{|x-y|^\alpha}$ (see definition \ref{semi-norm}).
\begin{claim}\label{claimholderfailuniquecont}
    If $a,b>0$ are real positive numbers and $\|\nabla f\|_{\infty} \leq a,$ $\|f\|_{\infty} \leq b$, then $|f|_{C^{0, \alpha}} \leq a^{\alpha} (2b)^{1-\alpha}.$
\end{claim}
\begin{proof}[Proof of the claim]
The following inequalities hold:
\begin{align*}
    |f(x)-f(y)| = |f(x)-f(y)|^{\alpha} |f(x)-f(y)|^{1-\alpha} \leq a^{\alpha} |x-y|^{\alpha} (2b)^{1-\alpha}.
\end{align*} Therefore, $$|f|_{C^{0,\alpha}} = \sup_{x \neq y} \frac{|f(x)-f(y)|}{|x-y|^\alpha}\leq a^\alpha (2b)^{1-\alpha}.$$
\end{proof}

In order to apply Lemma \ref{PML1} and transform $c_nf_n$ into $c_{n+1}f_{n+1}$ in $\T^2 \times [a_n,a_{n+1}]$ we need 
 $$k_{n+1}-k_n \lesssim w_n^{-1} \lesssim k_n.$$

Assuming that these inequalities hold, Lemma \ref{PML1} gives us the estimates 
\begin{align}\label{recallestimatelemmazeroone}
     \|A-Id\|\lesssim \frac{1}{w_n k_n}, \hspace{0.5cm} \|\nabla A\|\lesssim \frac{1}{w_n}, \hspace{0.5cm} \| \dot A \| \lesssim \frac{1}{w_n}
\end{align}
in $\T^2 \times [a_n,a_{n+1}]$.
Combining these three estimates with the Claim \eqref{claimholderfailuniquecont} ensures that  
\begin{align}\label{holdercoeffofa}
  \left|A-Id\right|_{C^{0, \alpha}} \lesssim  \frac{1}{w_n k_n^{1-\alpha}} .
\end{align}
Therefore, to ensure uniform $\alpha$-Hölder regularity of $A$, it is enough to have
\begin{align}\label{sufficientforunifomrholder}
   \frac{1}{w_n k_n^{1-\alpha} }\lesssim 1.
\end{align}
 To guarantee uniform ellipticity of $A,$ it is enough to have 
\begin{align}\label{sufficientforunifellipticity}
    \frac{1}{w_n k_n} \ll 1.
\end{align}
  as can be seen from the first inequality from \eqref{recallestimatelemmazeroone}.

The sequence $\{k_n\}$ is increasing and tends to infinity. So if $k_0$ is sufficiently large, then \eqref{sufficientforunifomrholder} implies \eqref{sufficientforunifellipticity} and we need to check one condition instead of two. Also if $\frac{1}{w_n k_n^{1-\alpha}} \lesssim 1$, if $k_0$ is sufficiently large and if $k_n \to \infty$, then $\|A-Id\|\lesssim \frac{1}{w_nk_n}=o(1)$ as $n\to \infty$ and the matrix $A(x,y,t)$ tends to the identity matrix as $t \to T$.

\textbf{Goal:} Find two sequences $\{w_n\} $ and $\{k_n\}$ such that 
\begin{align}\label{constraintsgoal}
    0 < k_{n+1}-k_n \lesssim w_n^{-1}, \hspace{0.5cm} \frac{1}{w_n k_n^{1-\alpha} }\lesssim 1, \hspace{0.5cm} \sum_{n \geq 1} w_n <\infty 
   \quad \textup{ and } \quad k_n \to \infty
\end{align}

In the next few lines, we explain why these two sequences don't exist for $\alpha\geq 1/2.$ Assuming these two sequences exist, we see that combining the first two constraints from \eqref{constraintsgoal} yields
\begin{align}\label{firsttwoconstraintsalltogether}
    k_{n+1}-k_n \lesssim k_n^{1-\alpha},
\end{align}
which yields $k_{n+1}/k_n \to 1$.
By the Mean Value Theorem and \eqref{firsttwoconstraintsalltogether}, $$k_{n+1}^{\alpha}-k_n^\alpha \leq \alpha(k_{n+1}-k_n)k_{n+1}^{\alpha-1} \sim \alpha(k_{n+1}-k_n)k_{n}^{\alpha-1} \lesssim 1.$$ 

Hence,
\begin{align}\label{estimateforlambdanalpha}
k_n^{\alpha} \lesssim n .   
\end{align}

Now, the second constraint from \eqref{constraintsgoal} together with \eqref{estimateforlambdanalpha} implies 
$$
w_n \gtrsim \frac{1}{k_n^{1-\alpha}}\gtrsim \frac{1}{n^{(1-\alpha)/\alpha}}.
$$

Finally, $\sum w_n < \infty$ implies that $\frac{1-\alpha}{\alpha}>1$, which is equivalent to $\alpha<\frac{1}{2}.$ This is precisely why this approach does not yield $\alpha$-Hölder regularity  for $\alpha \geq 1/2$.

Based on the above argument we can now explicitly find two sequences $\{k_n\}$ and $\{w_n\}$ satisfying the constraints \eqref{constraintsgoal} when $\alpha<\frac{1}{2}$. More precisely, the inequality \eqref{estimateforlambdanalpha} suggests to choose $k_n=n^{1/\alpha}$ and the second inequality from \eqref{constraintsgoal} suggests to choose $w_n=k_n^{\alpha-1}=n^{(\alpha-1)/\alpha}$. This is the content of the next Claim:

\begin{claim}\label{defoflambdanwn}
    By defining $k_n := (n+n_0)^{1/\alpha}, w_n:=(n+n_0)^{(\alpha -1)/\alpha}$ for some  $n_0\gg 1$, the constraints from \eqref{constraintsgoal} are satisfied.
\end{claim}

\begin{proof}[Proof of the Claim]
    The claim follows by direct computation. 
\end{proof}

 These conditions guarantee uniform $\alpha$-Hölder continuity and ellipticity of the coefficients of $A$ in $\T^2 \times (-\infty, T)$. As discussed before $A(x,y,t) \to Id$ as $t\to T$, so we define $A = Id$ for $t\geq T$.
We also should check that the constructed $u$ actually solves the divergence form equation through $\{t=T\}$.
 We include an explanation in the next section  \ref{appendixforcounerexampleendpoint}, which shows that $u$ and the derivatives of $u$ rapidly decay to zero as $t\to T$. In particular, the divergence-free vector field $A\nabla u$ in 
$\T^2 \times (-\infty, T)$ has a $C^1$ extension by zero  onto $\T^2 \times [T,+\infty)$.

\subsection{The equation near $\{t=T\}$}\label{appendixforcounerexampleendpoint}

This section explains why the equation $\ddot u + \div(A\nabla u)=0 $ is satisfied near $\{t=T\}$.
\\

\pseudosection{The equation near $\{t=T\}$}

Recall that $u \equiv 0$ for $t\geq T.$  We note that $u$ solves $\div(A\nabla u) +\ddot u =0$ in $\mathbb{T}^2\times [0, T)$, but it is not clear why the equation holds through $\{t=T\}$. In fact, as we will show below, the derivatives of $A$ (polynomially) blow up at $\{t=T\}$. However, $u$ and its derivatives decay to zero at $\{t=T\}$ as $e^{-d_1(T-t)^{-d_2}}$ for some positive constants $d_1, d_2$. Hence, we will be able to extend $A\nabla u$ into a $C^1$ function across $\{t=T\}$ in such a way that $A\nabla u$ and its derivatives are equal to zero at $\{t=T\}$. Using this extension of $A\nabla u$, we will be able to make sense of the equation $\div(A\nabla u)+\ddot{u}=0$  at $\{t=T\}$ 

\pseudosection{Exponential decay of $u$ at $\{t=T\}$} 

By the definition of $u$ \eqref{ulemmatwopointonetildeunzeroforbigt}, 
$$
\sup_{\T^2\times \{t=a_n\}}|u|=c_n \e^{-k_na_n}, \hspace{0.5cm} \sup_{\T^2\times \{t=a_{n+1}\}}|u|=c_{n+1} \e^{-k_{n+1}a_{n+1}}.
$$
By definition \eqref{seqeuncecn}, the sequence $\{c_n\}$ is given  by $c_n \e^{-k_na_n}=c_{n+1}\e^{-k_{n+1}a_n}$. Since $2w_n=a_{n+1}-a_n$ by definition, we get
\begin{equation}
\label{decay}
    \sup_{\T^2\times \{t=a_{n+1}\}}|u|= e^{-2k_{n+1}w_n} \sup_{\T^2 \times \{t=a_{n}\}}|u| \hspace{0.5cm} \mbox{ for all } n.
\end{equation}

By definition, $k_n = (n+n_0)^{1/\alpha}, w_n=(n+n_0)^{(\alpha -1)/\alpha}$ for some  $n_0\gg 1$ (see Claim \ref{defoflambdanwn}). Therefore, 
$$k_{n+1}w_n \asymp n, \quad \mbox{ for } n \gg 1.$$ 

Since $u$ solves $\ddot u + \div(A\nabla u) = 0$ between $a_n$ and $a_{n+1}$, the maximum principle applies and 
\begin{equation}\label{max}
    \sup_{\T^2\times [a_n,a_{n+1}]}|u|\leq \max\left(\sup_{\T^2 \times \{t=a_n\}}|u|, \sup_{\T^2 \times \{t=a_{n+1}\}}|u| \right) = \sup_{\T^2\times \{t=a_n\}} |u|
    \end{equation}
    by \eqref{decay}.
\comment{
\begin{rem}
    \label{maxim}
    Between $a_n$ and $a_{n+1}$ we have $\div\mathcal{A}\nabla u = 0.$ Thus, the maximum principle applies and
    \begin{equation}
    \label{max}
    \sup_{a_n \leq t\leq a_{n+1}}|u|\leq \sup_{t\in\{a_n,a_{n+1}\}}|u| = \sup_{t=a_n} |u|.
    \end{equation}

\end{rem}}

Let $t$ be close to $T$, that is , $t \in [a_n, a_{n+1}]$ for some $n \gg 1.$ By \eqref{max} and  by iterating \eqref{decay}, we get 
\begin{align}\label{iterateexponentialdeecayclosetobigt}
\sup_{\T^2} |u| \leq \sup_{\T^2\times \{t=a_n\}} |u| \lesssim \e^{-d_1n^2}    
\end{align}
for some constant $d_1>0.$ However, recalling that $w_n \asymp n^{(\alpha-1)/\alpha}$ (see Claim \ref{defoflambdanwn}), with $\alpha <1/2$, we have that 
\begin{align}\label{ncomparabletoinversepowerinterval}
T-t \approx \sum_{i=n}^\infty 2w_i \asymp n^{2-\frac{1}{\alpha}}    
\end{align}
and therefore, 
\begin{align}
 |u(x,y,t)| \lesssim e^{-d_1(T-t)^{-d_2}}   
\end{align}
for some constant $d_2>0$. This proves the exponential decay of $u$ as $t \rightarrow T$.
\\

\pseudosection{Exponential decay of the derivatives of $u$ at $\{t=T\}$}
Let $t$ be close to $T$, i.e., $t \in [a_n, a_{n+1}]$ for some $n \gg 1$. Let $\beta \in \mathbb N^3$ be a multi-index of order $|\beta|\leq 2$. By Lemma \ref{PML1}
$$|\partial^{\beta} u| \lesssim k_{n+1}^{|\beta|} \sup_{\T^2 \times [a_n, a_{n+1}]}|u|\lesssim n^{2/\alpha} e^{-d_1 n^2}$$
for some $d_1>0$. In the second inequality, we used \eqref{max} and \eqref{iterateexponentialdeecayclosetobigt} and that $k_n:=n^{1/\alpha}$, $0<\alpha <1/2$, by definition (see Claim \ref{defoflambdanwn}). Hence, by \eqref{ncomparabletoinversepowerinterval}, we have $$|\partial^{\beta} u|\lesssim (T-t)^{\frac{-d_2}{\alpha}}e^{-d_1(T-t)^{-d_2}}.$$
   Therefore all the derivatives of $u$ of order less or equal to 2 decay exponentially fast.
 \comment{
 By Lemma \ref{PML1}, 
 $$|\pl_{i,j}^2 u| + |\ddot u| \lesssim k_{n+1}^2 \sup_{\T^2 \times [a_n;a_{n+1}]}|u|$$. Therefore, by arguing as above, it is clear that

 which  implies, as above, that $$|\pl_{i,j}^2 u| + |\ddot u|\lesssim (T-t)^{-C}e^{-c(T-t)^C}$$
 for $t\in[a_n,a_{n+1})$ and for some constants $c, C>0.$  Therefore $\ddot{u}$ and $\pl_{i.j}^2 u$ also decay exponentially fast.
 \\
 }

\pseudosection{Polynomial growth of the derivatives of $A$}
By the inequalities \eqref{props} in Lemma \ref{PML1}, the coefficient of $A$ are uniformly bounded and the derivatives of $A$ in the interval $t\in[a_n,a_{n+1})$ are bounded by \begin{align}\label{polynomialgrowth}
|\nabla A|,|\dot{A}|\lesssim\frac{1}{w_n}.    
\end{align}
 Recall that $w_n=n^{(\alpha-1)/\alpha}$ and $0<\alpha<1/2$ by definition (see Claim \ref{defoflambdanwn}). For $t$ close to $T$ (or equivalently $t \in [a_n, a_{n+1}]$ with $n \gg 1$), we saw in \eqref{ncomparabletoinversepowerinterval} that $n\asymp (T-t)^{-d_2}$ for some constant $d_2>0$. Therefore, by \eqref{polynomialgrowth}, the derivatives of $A$ grow at most polynomially.

\phantom{bla-bla}\\\phantom{bla-bla}

\pseudosection{The extension}
Since $u$ and its derivatives decay exponentially fast and since $A$ and its derivative grow only polynomially fast, it is clear that $A\nabla u$ can be extended as a $C^1$ function at $\{t=T\}$ such that $A\nabla u$ and its derivatives are equal to zero at $t=T$. Therefore, $\div(A\nabla u) + \ddot{u}$ make sense through $\{t=T\}$  and the solution is equal to zero in $\{t>T\}$. In conclusion, we get that the equation is satisfied in $\T^2 \times \R$.

\comment{
\textcalor{red}{move this remark somewhere; potentially prove in Appendix or even here We would like to make a note (which will not be used) that a careful reader can extract that both $u$ and $A\nabla u$ will be $C^\infty$-smooth everywhere despite the fact that the derivatives of $A$ blow up near $t=T$. 
\begin{proof}
    It turns out that $A = \beta(t)e^{\lambda t}A_s = \beta(t)e^{\lambda t}\left(-s\begin{pmatrix}-\frac{\cos\lambda'y\cos\lambda x}{\lambda^2}&0\\0&0\end{pmatrix}+\begin{pmatrix}-\frac{-2\sin^2\lambda'y}{\lambda^2}&\frac{-2\sin\lambda'y\sin\lambda x}{\lambda\lambda'}\\0&0\end{pmatrix}\right).$ It's easy to see that the $k$th gradient of $A$ is equal to $O(\lambda^{k-1}w^{-1}).$ Moreover, the $k$th gradient of $u$ is estimated as $O(\lambda^k\|u\|).$ Thus the $k$th gradient of the product $A\nabla u$ is estimated as $O(\lambda^{k}w^{-1}\|u\|)$ and tends to zero since $\lambda$ and $w^{-1}$ increase polynomially and $\|u\|$ decays exponentially.   
\end{proof}
}
}
 
\begin{rem}
    With a bit more effort, one can show that $A\nabla u$ and $ u$ are also $C^\infty$-smooth across $\{t=T\}$ by showing that all derivatives of $u$ have exponential decay while derivatives of $A$ blow up at most polynomially.
     Thus the derivatives of $\binom{A\nabla u}{\dot{u}}$ also have exponential decay, and  $\binom{A\nabla u}{\dot{u}}$ is  $C^\infty$-smooth and divergence-free everywhere. 
\end{rem}

\section{Constructions of rapidly decaying solutions}\label{toolboxnew}
\subsection{Definitions, notations}
 We will say that a solution $u_1$ to an equation $L_1 u_1=0$ is transformed into a solution $u_2$ to an equation $L_2 u_2=0$ within a set $\{ (x,y,t): T_1 \leq t \leq T_2 \}$ via a solution $u$ to an equation $Lu=0$ if
$$ u=u_1, \hspace{0.2cm} L=L_1 \hspace{0.1cm}\textup { in } \{ (x,y,t):  t \leq T_1 \}\quad \mbox{ and }\quad u=u_2, \hspace{0.2cm} L=L_2 \hspace{0.1cm}\textup{ in } \{ (x,y,t):  t \geq T_2 \}$$ and $u$ is $C^2$. In the following, we will mostly consider operators $L$ of the form $L=\div(A \nabla \cdot)$.  In this case, we also ask $A$ to be $C^1$.

The difference $T=T_2-T_1$ will be called the time (duration) of the transformation. In addition, we will often say "transforming $u_1$ to $u_2$ in time $T$" to indicate that $u_1$ is transformed to $u_2$ within the set $\{ (x,y,t): 0 \leq t \leq T \}.$ 

\comment{
\begin{figure}[H]
	\includegraphics[scale=0.48]{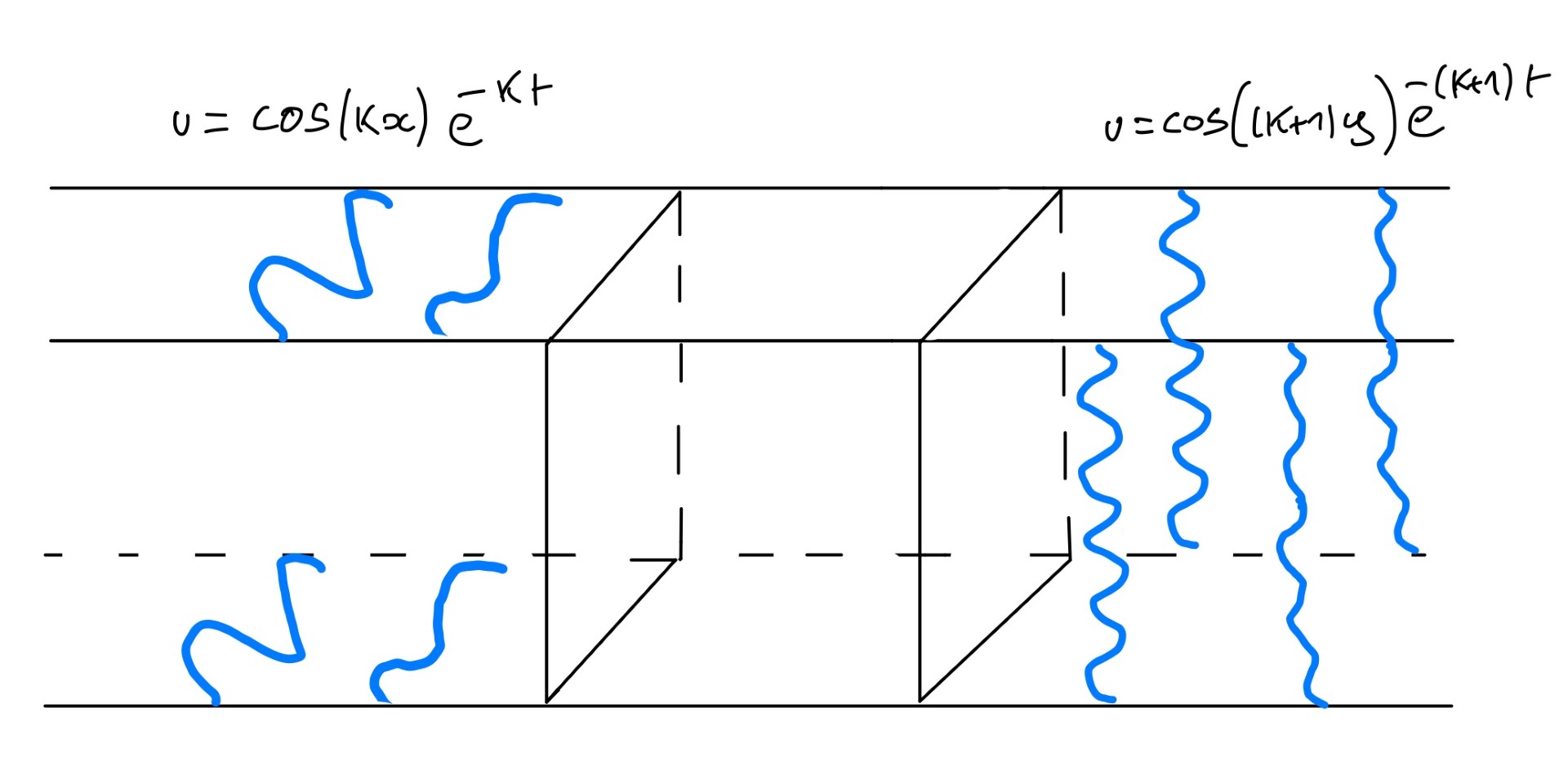}
	\centering
	\caption{The general case of transforming} 
\label{fig:main_step_kk}
\end{figure}}
We will also say that a solution $u$ to $Lu=0$ on the set $\{(x,y,t) : T_1 \leq t \leq T_2\}$ is obtained by gluing at time $T$ a solution $u_1$ to $L_1 u_1=0$ on the set $\{(x,y,t): T_1 \leq t \leq T\}$ and a solution $u_2$ to $L_2 u_2=0$ on the set $\{(x,y,t) : T \leq t \leq T_2)\}$ if
$$
 u=u_1, \hspace{0.2cm} L=L_1 \hspace{0.1cm}\textup { in } \{ (x,y,t):  T_1 \leq t \leq T \}\quad \mbox{ and } \quad u=u_2, \hspace{0.2cm} L=L_2 \hspace{0.1cm}\textup{ in } \{ (x,y,t): T \leq t \leq T_2 \}
$$
and $u$ is $C^2$ at $t=T$. In the case where $L=\div(A\nabla \cdot),$ we also ask $A$ to be $C^1$ at $t=T$.

We will say that the transformation has a certain regularity property if the operator $L$ has this property.
For instance, the transformation is called $C^1$ smooth with constant $C$ if the derivatives of the coefficients of $L$ are smaller than $C$ in absolute value.  

 Unless specified otherwise, we consider in this article transformations with $L$ in divergence form 
 $$Lu= \div(\mathcal{A}\nabla u)  + \mu u,$$
 where $\mu$ is a constant and $\mathcal{A}$ is a real elliptic matrix valued function.

\textbf{Regularity classes.} We say that a transformation belongs to the class $ R (\Lambda, C )$ if the 
the ellipticity constant of $\mathcal{A}$ is bounded by $\Lambda$: for any vector $\xi$ $$\frac{1}{\Lambda}  |\xi|^2 \leq (\mathcal{A}\xi, \xi) \leq \Lambda |\xi|^2$$
and the coefficients of $\mathcal{A}$  are $C^1$ smooth with constant $C$.

\comment{\textbf{ Time versus regularity.} Given two solutions $u_1$ and $u_2$ (and two equations), we will be interested in how the minimal time of transformation from $u_1$ to $u_2$ depends on the regularity of the transformation. 
The less regularity, the smaller time we need.}

\comment{
\begin{rem}[Remark about the function $u$ used]
\label{whatisu}
    Here and further we will use only functions $u$ that can be represented as a sum of one or two cosines: $u = f(t)\cos(kx)+g(t)\cos(k'y),$ $u = f(t)\cos(kx)$ or $u = g(t)\cos(k'y)$ on certain intervals. Changing the values of $k$ or $k'$ is done by having $f$ or $g$ become zero with zero derivatives, then one can safely glue, for example, the functions. \textbf{Whenever we say that a solution exists, it will have the form described above.} 
\end{rem}}

\subsection{Informal explanation of the proof of Theorem \ref{EigenTheorem}}\label{informalexplanationfullproof}
To make it easier for the reader, we first present the main steps in the proof of Theorem \ref{EigenTheorem}, that is, of the construction of a double-exponentially decaying eigenfunction to $\ddot u + \div(A\nabla u)=-\mu u$ in $\T^2\times \R$. The proof can be decomposed into three steps:
\begin{itemize}
    \item In the first step, we construct a solution $u$ to the divergence form equation $\ddot u + \div(A\nabla u)=0$ in $\T^2 \times [0, C]$ ("a block") for some universal $C>0$. The solution $u$ will be equal to $\cos(kx)\e^{-kt}$ for $t\leq 0$ and to $c(k)\cos(2ky)\e^{-2kt}$ for $t \geq C$, where $k \gg 1$ and $c(k)$ is a positive constant depending on $k$. At first sight, this step might look very similar to Lemma \ref{PML1}. The main difference lies in the frequency of oscillations: we can switch to a harmonic function with a much faster frequency compared to Lemma \ref{PML1} (for a duration of order $\mathcal{O}(1)$, while Lemma \ref{PML1} only allows to switch from frequency $k$ to frequency $k+\mathcal{O}(1)$).  To achieve our construction in a block, we use a new idea: we do not immediately transform $\cos(kx)\e^{-kt}$ into $c(k)\cos(2ky)\e^{-2kt}$. Instead, we first transform $\cos(kx)\e^{-kt}$ into a slower-decaying but faster oscillating function $c_1(k)\cos(2ky)\e^{\frac{-2k}{3}t}.$ Doing that allows us to control the regularity of the matrix $A$ uniformly in $k$ (if we were to immediately transform $\cos(kx)\e^{-kt}$ into $c(k)\cos(2ky)\e^{-2kt}$ in time $\mathcal{O}(1)$, we would not be able to get a uniformly $C^1$ matrix $A$). Once this transformation is done, we can accelerate the decay to transform $c_1(k)\cos(2ky)\e^{\frac{-2k}{3}t}$ into $c(k)\cos(2ky) \e^{-2kt}.$  This first step is explained in the section \ref{thebuildingblocknewlabel}.

    \item With the construction from the first step in hand, we can now shift it to transform $c_n \cos(k_nx)\e^{-k_nt}$ into $c_{n+1}\cos(2k_ny)\e^{-2k_nt}$ via a solution $u_n(x,y,t)$ to $\ddot u + \div(A\nabla u)=0,$ in the block $\T^2 \times [C(n-1), Cn]$. This allows us to define a solution $u$ in $\T^2\times \R^+$ equal to $u_n(x,y,t)$ when $t \in [C(n-1), Cn]$ and $n$ is odd and equal to $u_n(y,x,t)$ when $t \in [C(n-1), Cn]$ and $n$ is even. This function $u$ will solve an equation $\ddot u+\div(A\nabla u)=0$ in $\T^2 \times \R^+$, it will be uniformly $C^2$ and the matrix $A$ will be uniformly elliptic and uniformly $C^1$. Since in each block $\T^2\times [C(n-1), Cn]$, the frequency increases by a factor 2,  and since the length of a block is of order $\mathcal{O}(1),$ we get that after $n$ iterations, $t \asymp n$ and the frequency $k_n \asymp 2^n$, showing where the super-exponential decay comes from. We then lift this construction of a $A$-harmonic solution with double-exponential decay in $\T^2\times \R^+$  into a double-exponential decaying solution to an equation with constant potential $\ddot u + \div(A\nabla u)=-\mu u$ in the same set $\T^2\times \R^+$. To achieve that, we use that at any moment $t \in [C(n-1), Cn]$ (assuming $n$ is odd), the $A$-harmonic solution $u$ will be of the form $f_n(t)\cos(k_nx)+g_{n+1}(t)\cos(2k_ny).$ We then observe that since $u$ solves $\ddot u +\div(A\nabla u)=0$, then $u$ also solves $\ddot u + \div(A'\nabla u)=-\mu u$ when $A'=A+B$, $B=\begin{pmatrix}
        \mu/k_n^2 & 0 \\0 & \mu/(2k_n)^2
    \end{pmatrix}$ and $t\in [C(n-1), Cn]$. However, the matrix $A'$ is not even continuous in $\T^2\times \R^+$ while we want a $C^1$ matrix $A'$. To solve this issue, we use that $u$ will in fact have only one cosine term near the boundaries of the block $\T^2\times [C(n-1), Cn].$. This allows us to smoothly change the entry in $B$ related to the other space variable, thus making the matrix $B \in C^1$ in $\T^2\times \R^+$. This is the second step and we prove it in section \ref{sect4}.

    \item Now, after we have an eigenfunction in $\T^2 \times \R^+$, we would like to reflect it around zero to get an eigenfunction in $\T^2 \times \R$. However, we face an issue: the eigenfunction $u$ that we are building is of the form $\cos(k_1x)\e^{-k_1t}$ for $t$ close to zero, and we first transform $u$ into another function, which solves an eigenfunction equation with different coefficients and, as a function of $t$,  is a periodic trigonometric function. Then we will shift the constructed solution to make the time derivative at $t=0$ zero and after that the solution can be reflected in an even way. We prove this third step in section \ref{mainproof}. 
\end{itemize}

\comment{
Informally speaking, we use three technical lemmas and propositions (\ref{lemmachangingacoeff}, \ref{slownew} and \ref{accelnew}) to conclude that one can transform  $\cos(kx)e^{-kt}$ into $c\cos(k'y)e^{-k'(t-C)}$ within the set $\mathbb{T}^2\times[0;C]$ where $c$ is a small constant proportional to $e^{-O(k)}$ (Lemma \ref{blockzerocnewnew}). One can use a shifted version to transform $\cos(k_n x)e^{-k_n(t-C(n-1))}$ into $c_n\cos(k_{n+1} y)e^{-k_{n+1}(t-Cn)}$ within the set $\{(x,y,t): C(n-1)\leq t\leq Cn\}$. A chain of those trasformations ensures that the module rapidly decays. This yields an example proving Theorem \ref{theoremforhalfcylinder}.\\
In Section 4 we describe a way to transform said example into a solution of the equation $\ddot{u}+\div A'\nabla u = -\mu u$ so that $u$ decays in \textit{both} directions. First of all, we transform the matrix $A$ of the example described above into another matrix $A'$ so that $\ddot{u}+\div A'\nabla u = -\mu u.$ It is possible since $u$ is always a sum of two cosines and since $u$ is actually one cosine on rather long intervals, during which one can prepare to introducing the new frequency.\\
This allows one to nullify $\dot{u}$ at some moment. Then we move this moment to $t=0,$ reflect the example about $t=0$, move it away and join the example with its reflection.
}

\subsubsection{Plan of the proof}
We prove the main Theorem \ref{EigenTheorem} by successively reducing it to simpler statements. In section \ref{mainproof}, we reduce the proof of Theorem \ref{EigenTheorem}, to the proof of the existence of a fast decaying eigenfunction on a half-cylinder and to the \textit{symmetrization Lemma} that will allow us to reflect around zero. We then prove the symmetrization Lemma.
In section \ref{sect4}, we reduce the construction of a fast decaying eigenfunction on a half-cylinder to the construction of a solution in the block $\T^2 \times [0, C]$. In section \ref{thebuildingblocknewlabel}, we reduce the construction of a solution in the block $\T^2 \times [0, C]$ to three technical Propositions. We finally prove these three technical Propositions in section \ref{proofofpropslabel}.

\section{Proof of Theorem \ref{EigenTheorem}}
\label{mainproof}
In this section, we reduce the construction of an eigenfunction on $\T^2 \times \R$ with double-exponential decay in both directions, to the construction of a solution to a divergence form PDE with constant potential on $\T^2 \times \R^+$ with double-exponential decay and to the symmetrization Lemma (which allows us to reflect around zero). We will then prove the symmetrization Lemma.  The construction of a solution to a divergence form PDE with constant potential on $\T^2 \times \R^+$ with double-exponential decay in a half-cylinder will be proved in section \ref{sect4}.

\subsection{A fast-decaying eigenfunction on $\T^2 \times \R$}
We recall the Theorem that we want to prove:

\begin{thm*}[\ref{EigenTheorem}]
    In the 3-dimensional cylinder $\mathbb{T}^2\times\mathbb{R}$, for every $\mu>0$, there exists a uniformly $C^1$ and uniformly elliptic matrix $A=A_{\mu}(x,y,t)$ and  a non-zero uniformly $C^2$ real function $u=u_\mu(x,y,t)$ such that \begin{equation} \label{eigen2} \ddot u+ \div(A\nabla u)=-\mu u\end{equation}and  $u$ has double exponential decay: for any $T\gg 1,$
$$\sup_{\T^2\times \{|t|\geq T\}}|u(x,y,t)|\leq Ce^{-ce^{cT}}$$
for some numerical $c(\mu), C(\mu)>0$.
\end{thm*}

As mentioned above, we reduce Theorem \ref{EigenTheorem} to the construction of a solution to a divergence form PDE with constant potential on $\T^2\times \R^+$ (Theorem \ref{halfcylinderallmunew} below) and to the symmetrization Lemma (Lemma \ref{altabstractlemmaevaluenewnew} below). We first state these two results and then we proceed to the proof of this reduction.

\begin{thm}[A fast-decaying solution on a half-cylinder with constant potential]\label{halfcylinderallmunew}
    Let $\mu \in \R$. In the 3-dimensional cylinder $\mathbb{T}^2\times\mathbb{R}^+$, there exists a uniformly $C^1$ and uniformly elliptic matrix $A:=A_{\mu}(x,y,t)$ and a non-zero uniformly $C^2$ real function $u:=u_{\mu}(x,y,t)$ such that
    \begin{equation} \label{eigen} \ddot{u} + \div(A\nabla u)=-\mu u,\end{equation}
    and  $u$ has double exponential decay: for any $T\gg 1,$
$$ \sup_{\T^2 \times \{t\geq T\}}|u(x,y,t)|\leq Ce^{-ce^{c T}}$$ 
for some  $c(\mu), C(\mu)>0$. The matrix $A$ belongs to the regularity class $R(100, 61).$ Moreover, on $\T^2 \times \left[0, \frac{1}{100}\right],$ the function $u$ and the matrix $A$ are equal to
\begin{align}\label{solutionuclosetozeroinmainhalfcylinder}
    u(x,y,t)= \cos(kx)\e^{-kt}, \hspace{0.5cm} A= Id + \begin{pmatrix}
        \frac{\mu}{k^2} & 0\\
        0 & \frac{\mu}{4k^2}
    \end{pmatrix}.
\end{align}
for some fixed $k\gg \sqrt{|\mu|}$. \comment{ and the matrix $A$ is be equal to
\begin{align}\label{matrixaclosetozeroinainhalfcylinder}
    A= Id + \begin{pmatrix}
        \frac{\mu}{k^2} & 0\\
        0 & \frac{\mu}{4k^2}
    \end{pmatrix}.
\end{align}}
\end{thm}

\begin{rem}
    Theorem \ref{halfcylinderallmunew} implies our first main result Theorem \ref{theoremforhalfcylinder} by taking $\mu=0.$ 
\end{rem}

\begin{rem}
Theorem \ref{halfcylinderallmunew} holds for all $\mu \in \R$ but in Theorem \ref{EigenTheorem}, we will restrict to $\mu>0$. Indeed, for $\mu\leq 0$, there does not exist non-trivial, uniformly $C^2$, solutions $u$ to 
\begin{align}\label{generalpdenew}
\ddot u + \div(A\nabla u) = -\mu u    
\end{align}
with $u, \nabla u$ decaying to zero (fast enough) in both directions in $\T^2 \times \R$ and with $A$ uniformly $C^1$. To see that, one can multiply the equation \eqref{generalpdenew} by $u$ and integrate over $\T^2 \times \R$ to get $$\int_{\T^2\times \R} u \, \div(\mathcal{A}\nabla u)\, dx\,dy\,dt = -\mu \int_{\T^2 \times \R} |u|^2 \,dx\,dy\,dt$$ where the definition of $\mathcal{A}$ was given in \eqref{threedmatrixaugmented}. In particular, $\mathcal{A}$ is a $3\times 3$ matrix, which has $A$ as its $2\times 2$ principal minor and the last line and column are of the form $(0, 0, 1)$. We can then integrate by part the first term (the boundary terms vanish), to get $$\mu = \frac{\int_{\T^2\times \R} (\mathcal{A}\nabla u, \nabla u) }{\int_{\T^2\times \R} |u|^2} \geq \frac{\int_{\T^2\times \R} |\nabla u|^2 }{\Lambda \int_{\T^2\times \R} |u|^2} $$ by ellipticity. This last quantity is strictly positive since otherwise, $u$ would be a constant function decaying to zero at infinity i.e., the trivial function. 
\end{rem}

As we mentioned earlier, a natural idea to extend the construction from a half-cylinder to a full-cylinder is to glue at $t=0$, the solution $u$ to $\ddot{u}+ \div(A \nabla u) = -\mu u$, $\mu>0$, in $\T^2 \times \R^+$, with its reflection $\tilde u:= u(x,y,-t)$, solution to $\ddot{\tilde u} +\div(\tilde A \nabla \tilde u) = -\mu \tilde u$, $\mu>0$ in $\T^2 \times \R^-$. For this approach to yield a $C^2$ function around $t=0$, it is necessary that $\dot{u}(t=0)=0.$ However, the function that we construct in Theorem \ref{halfcylinderallmunew} will be of the form $u=\cos(k x) \e^{-kt }$, for some fixed $k\gg \sqrt{\mu}$ and  for $t$ close enough to 0. This function does not satisfy $\dot{u}(t=0)=0.$ The way out is provided by Lemma \ref{altabstractlemmaevaluenewnew} below.

\begin{lemma}[Symmetrization]\label{altabstractlemmaevaluenewnew}
Let $\mu>0$ and let $k \gg \sqrt{\mu}$. There exists a $C^2$ function $f(t):\R \rightarrow \R$ such that $\dot{f}(0)=0$ and there exist $C=C(\mu, k)>0$ and a time $t_0=t_0(\mu, k)>0$ satisfying $t_0-C>0,$ such that the function $u_1:=\cos(kx)f(t)$  solves  $$
\ddot u_1 +\div\left[ \begin{pmatrix}
    \frac{\mu}{2k^2} & 0\\
    0 & 1+\frac{\mu}{4k^2}
\end{pmatrix}\nabla u_1\right]=-\mu u_1 $$ for $t \leq t_0-C$
and also for the solution $u_2:=\cos(kx)e^{-k(t-t_0)}$ to the equation
$$
 \ddot u_2 +\div\left[ \begin{pmatrix}
    1+\frac{\mu}{k^2} & 0\\
    0 & 1+\frac{\mu}{4k^2}
\end{pmatrix}\nabla u_2\right]=-\mu u_2. 
$$
there is a transformation of $u_1$ into $u_2$ within the set $\T^2 \times [t_0-C, t_0]$, which solves  $\ddot u + \div(A\nabla u)=-\mu u$, where $A$ is in the regularity class $R(\frac{5k^2}{\mu}, 10).$ The duration $C$ of the transformation satisfies $C \leq \frac{2}{5}.$
\end{lemma}

\textbf{Heuristic idea of the reduction of Theorem \ref{EigenTheorem} to Theorem \ref{halfcylinderallmunew} and to Lemma \ref{altabstractlemmaevaluenewnew}}

Lemma \ref{altabstractlemmaevaluenewnew} tells us that we can transform  a shifted version (on $\T^2 \times [t_0, \infty)$)  of the eigenfunction constructed in Theorem \ref{halfcylinderallmunew}  into a new eigenfunction (on $\T^2 \times [0, t_0-C])$,  such that the latter has a vanishing $t$ derivative at $t=0$. This yields a new eigenfunction $V$ on $\T^2 \times \R^+$ which has a vanishing $t$ derivative at $t=0$  and a double exponential decay as $t \rightarrow \infty$.

Theorem \ref{EigenTheorem} now follows by gluing at $t=0$ the eigenfunction $V$ obtained in this way, with its reflection $\tilde V:= V(x,y,-t)$ defined for $t\leq 0$. The function $U$ obtained with this gluing will be an eigenfunction on $\T^2\times \R$ with double exponential decay in both directions. 
\\

We postpone the proofs of Theorem \ref{halfcylinderallmunew} and of Lemma \ref{altabstractlemmaevaluenewnew} and we proceed to the reduction of  Theorem \ref{EigenTheorem} to Theorem \ref{halfcylinderallmunew} to Lemma \ref{altabstractlemmaevaluenewnew}.
\begin{proof}[Proof of the reduction of Theorem \ref{EigenTheorem} to Theorem \ref{halfcylinderallmunew} and to Lemma \ref{altabstractlemmaevaluenewnew}] 

Let $\mu>0.$
    By Theorem \ref{halfcylinderallmunew}, we can construct a solution $u_1$ and a matrix $A_1$ such that 
    \begin{align}\label{fullcylinderfirstintroductionvandcnew}
     \ddot{u}_1 + \div(A_1\nabla u_1) = -\mu u_1  \hspace{0.3cm} \mbox{ in } \hspace{0.3cm} \T^2\times \R^+.
    \end{align}

Define the function $\tilde u_1$ and the matrix $\tilde A_1$ which are shifted version of $u_1$ and $A_1$:
\begin{align}\label{defuforshiftedfulcylindernew}
    \tilde u_1(x,t):= u_1(x, t-t_0)  \hspace{0.3cm} \mbox{ and } \hspace{0.3cm} \tilde A_1(x,y,t):= A_1(x,y,t-t_0) \hspace{0.3cm} \mbox{ for } \hspace{0.3cm} t\geq t_0.
\end{align}
where $t_0$  is given by Lemma \ref{altabstractlemmaevaluenewnew}. 
\comment{
Define also the matrix $\tilde A_1$ which is a shifted version of the matrix $A_1$ given in \eqref{fullcylinderfirstintroductionvandcnew}:
\begin{align}\label{defBnew}
    \tilde A_1(x,y,t):= A_1(x,y,t-t_0), \hspace{0.5cm} t\geq t_0.
\end{align}
}
In particular, by Theorem \ref{halfcylinderallmunew}, in the interval $[t_0, t_0+\frac{1}{100}],$ the matrix $\tilde A_1$ and the function $\tilde u_1$ are given by 
\begin{align}\label{defuandaclosetozeroforfullcylindernew}
 \tilde u_1(x,t)= \cos(kx) \e^{-(t-t_0) k}, \hspace{0.5cm} \tilde A_1=Id + \begin{pmatrix}
    \frac{\mu}{k^2} & 0 \\
    0 & \frac{\mu}{4k^2}
\end{pmatrix} 
\end{align}
for some fixed $k \gg \sqrt{\mu}.$

Let $C=C(\mu, k)$ be given by Lemma \ref{altabstractlemmaevaluenewnew}.  Consider also the function $u_0 :=\cos(kx) f(t)$ where $f$ is given by Lemma \ref{altabstractlemmaevaluenewnew}. 
Define the matrix 
\begin{align}\label{azeromatrixbottom}
    A_0:= \begin{pmatrix}
    \frac{\mu}{2k^2} & 0 \\
    0 & 1+\frac{\mu}{4k^2}
\end{pmatrix}
\end{align}

By Lemma \ref{altabstractlemmaevaluenewnew}, $u_0$ solves
$$
\ddot u_0 + \div(A_0\nabla u_0)=-\mu u_0 \hspace{0.3cm} \mbox{ for } \hspace{0.3cm} t \leq t_0-C.
$$

By \eqref{defuandaclosetozeroforfullcylindernew} and by Lemma \ref{altabstractlemmaevaluenewnew}, we can transform the solution $u_0$ into the solution $\tilde u_1$, via a function $v$ and a matrix $B$ such that $\ddot v + \div(B \nabla v)=-\mu v$, within the set $\T^2 \times [t_0-C, t_0].$ In particular, the function $v$ is $C^2$, the matrix $B$ is $C^1$ and they satisfy
\begin{align}\label{defvnew}
    v=\begin{cases}
      \cos(kx) f(t) & \mbox{ for } 0 \leq t \leq t_0-C, \\
      \tilde u_1 & \mbox{ for } t \geq t_0,
    \end{cases} \hspace{0.5cm} \mbox{ and } \hspace{0.5cm} B= \begin{cases}
        A_0 & \mbox{ for } 0 \leq t \leq t_0-C, \\
        \tilde A_1 & \mbox{ for } t \geq t_0.
    \end{cases}
\end{align}
The function $\tilde u_1$ and the matrix $\tilde A_1$ were defined in \eqref{defuforshiftedfulcylindernew}, and the matrix  $A_0$ was defined in \eqref{azeromatrixbottom}. \comment{ and the matrix $B$ is $C^1$ and satisfies
\begin{align}\label{defbnewnewnew}
    B= \begin{cases}
        A_0 & \mbox{ for } 0 \leq t \leq t_0-C, \\
        \tilde A_1 & \mbox{ for } t \geq t_0.
    \end{cases}
\end{align}}
By Theorem \ref{halfcylinderallmunew}, the matrix $\tilde A_1$ belongs to the regularity class $R(100, 61)$, by Lemma \ref{altabstractlemmaevaluenewnew} the matrix $B$ belongs to the regularity class $R(\frac{5k^2}{\mu}, 10)$ when $t_0-C\leq t \leq t_0$ and by the definition of $A_0$ \eqref{azeromatrixbottom}, it is clear that $A_0 \in R(\frac{5k^2}{\mu}, 1)$. Hence, we have that $B \in R(\frac{5k^2}{\mu}, 61)$ (since $k \gg \sqrt{\mu}$), i.e., $B$ is uniformly elliptic and uniformly $C^1$ on $\T^2 \times \R^+ $ (recall that $k, \mu$ are fixed numbers).  Note also that $v$ is uniformly $C^2$ by Lemma \ref{altabstractlemmaevaluenewnew} and by Theorem \ref{halfcylinderallmunew}

We now define a function $U$ and a matrix $A$ by 
\begin{align}\label{defUandAfinalfinalfinal}
  U(x,y,t)=  \left\{
\begin{array}{ll}
v(x,y,t) & \mbox{if } t \geq 0, \\
v(x,y,-t) & \mbox{if } t \leq 0,
\end{array}
\right. \hspace{0.3cm} \mbox{ and } \hspace{0.3cm} A(x,y,t)=  \left\{
\begin{array}{ll}
 B(x,y,t) & \mbox{if } t \geq 0, \\
 B(x,y,-t) & \mbox{if } t \leq 0.
\end{array}
\right.
\end{align}

By the definition of $v$ \eqref{defvnew} and by Lemma \ref{altabstractlemmaevaluenewnew}, $\dot{v}(t=0)=0$ and therefore, the function $U$ is $C^2$ at $t=0$. Moreover, by the definition of $v$ \eqref{defvnew}, and by Theorem \ref{halfcylinderallmunew}, $U$ has double exponential decay in both directions (since far away from 0, $U=\tilde u_1$ and $\tilde u_1$, defined in \eqref{defuforshiftedfulcylindernew}, has double exponential decay by Theorem \ref{halfcylinderallmunew}).  By the definition of $ B$ \eqref{defvnew}, for $t \in [0, t_0-C],$ $B$ has constant coefficients, therefore, the matrix $A$ defined in \eqref{defUandAfinalfinalfinal} is $C^1$ at $t=0.$

\comment{
We define a new matrix $A$ by 
\begin{align}\label{defmatrixaaroundzerosymmetryenew}
  A(x,y,t)=  \left\{
\begin{array}{ll}
 B(x,y,t) & \mbox{if } t \geq 0, \\
 B(x,y,-t) & \mbox{if } t \leq 0,
\end{array}
\right.
\end{align}
}

To conclude, assuming Theorem \ref{halfcylinderallmunew} and Lemma \ref{altabstractlemmaevaluenewnew}, we constructed a uniformly elliptic and uniformly $C^1$-smooth matrix $A$ defined on $\T^2 \times \R$ and a uniformly $C^2$ solution $U$ to the eigenvalue equation $\ddot{U} + \div(A\nabla U)=-\mu U$, such that $U$ decays double exponentially in both directions. This finishes the proof of Theorem \ref{EigenTheorem} assuming Theorem \ref{halfcylinderallmunew} and Lemma \ref{altabstractlemmaevaluenewnew}.
\end{proof}

\subsection{The symmetrization Lemma}
In this section, we prove the symmetrization Lemma that allowed us to extend the construction of a super-exponentially decaying eigenfunction on a half-cylinder $\T^2 \times \R^+$ to a super-exponentially decaying eigenfunction on the full-cylinder $\T^2 \times \R$. We recall the statement:

\begin{lemma*}[\ref{altabstractlemmaevaluenewnew}, Symmetrization]
Let $\mu>0$ and let $k \gg \sqrt{\mu}$. There exists a $C^2$ function $f(t):\R \rightarrow \R$ such that $\dot{f}(0)=0$ and there exist $C=C(\mu, k)>0$ and a time $t_0=t_0(\mu, k)>0$ satisfying $t_0-C>0,$ such that the function $u_1:=\cos(kx)f(t)$  solves  \begin{align}\label{problemuonenew}
\ddot u_1 +\div\left[ \begin{pmatrix}
    \frac{\mu}{2k^2} & 0\\
    0 & 1+\frac{\mu}{4k^2}
\end{pmatrix}\nabla u_1\right]=-\mu u_1    
\end{align}
 for $t \leq t_0-C.$

Consider also the function $u_2:=\cos(kx)e^{-k(t-t_0)}$ which is  a solution to the equation
\begin{align}\label{problemutwonew}
 \ddot u_2 +\div\left[ \begin{pmatrix}
    1+\frac{\mu}{k^2} & 0\\
    0 & 1+\frac{\mu}{4k^2}
\end{pmatrix}\nabla u_2\right]=-\mu u_2.    
\end{align}
 
Then, we can transform $u_1$ into $u_2$ within the set $\T^2 \times [t_0-C, t_0]$ while solving $\ddot u + \div(A\nabla u)=-\mu u$ and where $A$ is in the regularity class $R(\frac{5k^2}{\mu}, 10).$ The duration $C$ of the transformation satisfies $C \leq \frac{2}{5}.$
\end{lemma*}

\begin{proof}[Proof of Lemma \ref{altabstractlemmaevaluenewnew}]
Let $0<t_1< t_2$ and define $a(t)$ to be a smooth increasing function such that 
\begin{align}\label{defaoftbeginningnew}
   a(t)=\frac{\mu}{2k^2} \hspace{0.3cm} \mbox{ for } t \leq t_1, \hspace{0.5cm} a(t) = 1+\frac{\mu}{k^2} \hspace{0.3cm} \mbox{ for } t \geq t_2.
\end{align}

An explicit choice of $a(t)$ is given in Appendix \ref{appendixdurationsymmetry}, where we also explain how to choose $t_2-t_1$  such that $\dot{a}(t) \leq 10.$ Consider the unique solution $g(t)$ defined on $\R$ of the equation 
\begin{align}\label{secondorderodegnew}
    \ddot{g}(t) = \left( k^2 a(t) -\mu\right) g(t)
\end{align}
with initial conditions $g(t_2)=1$ and $\dot{g}(t_2)=-k.$ In particular, this solution $g$ is such that $g(t)=\e^{-(t-t_2)k}$ for $t \geq t_2$ and is such that $g(t)=\alpha \cos\left( \sqrt{\frac{\mu}{2}}t\right) + \beta \sin\left( \sqrt{\frac{\mu}{2}}t\right)  $ for $t \leq t_1$ and for some $\alpha, \beta \in \R.$ We note that the function $\alpha \cos\left( \sqrt{\frac{\mu}{2}}t\right) + \beta \sin\left( \sqrt{\frac{\mu}{2}}t\right)$ has its derivative equal to zero at points $\sqrt{\frac{2}{\mu}} \left(n \pi + \arctan\left( \frac{\beta}{\alpha}\right)\right)$, $n \in \mathbb{Z}$. By choosing $n=-1$, we can ensure that 
\begin{align}\label{positivetime}
\sqrt{\frac{2}{\mu}} \left(-\pi + \arctan\left( \frac{\beta}{\alpha}\right)\right)\leq 0 < t_1.    
\end{align}

We now define $f:\R \rightarrow \R$ which is a shifted version of $g$:
\begin{align}\label{deffunctionfnew}
    f(t):=g\left(t+\sqrt{\frac{2}{\mu}} \left(-\pi + \arctan\left( \frac{\beta}{\alpha}\right)\right)\right).
\end{align}

We note that $f(t)$ will be a linear combination of $\cos$ and $\sin$ for $t \leq t_1 - \sqrt{\frac{2}{\mu}} \left(-\pi + \arctan\left( \frac{\beta}{\alpha}\right)\right)$ which is a strictly positive time by \eqref{positivetime} and that $\dot{f}(0)=0$. We also note that  
\begin{align}\label{deftonew}
    f(t) = \e^{-(t-t_0)k} \hspace{0.5cm} \mbox{ for } t \geq t_0:= t_2-\sqrt{\frac{2}{\mu}} \left(-\pi + \arctan\left( \frac{\beta}{\alpha}\right)\right)>0.
\end{align}

Denote by $\tilde a$ the shifted version of $a$:
\begin{align}\label{tildeatnew}
    \tilde a(t) := a\left(t+\sqrt{\frac{2}{\mu}} \left(-\pi + \arctan\left( \frac{\beta}{\alpha}\right)\right)\right).
\end{align}

Since $g$ solves the second order ode \eqref{secondorderodegnew}, it is clear that $f$ solves the same shifted ode: $\ddot f(t) = (k^2 \tilde a(t) -\mu) f(t)$. By using that $f$ solves this second order ode, it follows by a direct computation that $u(x,t):=f(t) \cos(k x)$ solves the divergence form equation $\ddot{u} + \div\left[\begin{pmatrix}\tilde a(t) &0\\0&1+\frac{\mu}{4k^2}\end{pmatrix}\nabla u\right] = -\mu u$ on $\T^2 \times \R$. 
Indeed,
\begin{align}\label{eqtforunew}
    \ddot u + \div\left[\begin{pmatrix}\tilde a(t)&0\\0&1+\frac{\mu}{4k^2}\end{pmatrix}\nabla u\right]& = \ddot f(t) \cos(kx) - k^2 \tilde a(t) f(t) \cos(kx) \\
    &= k^2 \tilde a(t) f(t) \cos(kx) - \mu f(t) \cos(kx) - k^2 \tilde a(t) f(t) \cos(kx) \nonumber \\
    &=-\mu u. \nonumber
\end{align}

\textbf{Conclusion:} In \eqref{deffunctionfnew}, we constructed a uniformly $C^2$ function $f:\R \rightarrow \R$ such that $\dot f(0)=0.$ We also showed that there exists a time $t_0>0$ (see \eqref{deftonew}) and a number $C:=t_2-t_1>0$ such that $t_0-C>0$ (by \eqref{positivetime}) and such that the function $u(x,t):=\cos(kx)f(t)$ solves 
$$
\ddot u + \div\left[ \begin{pmatrix}
    \frac{\mu}{2k^2} & 0\\
    0 & 1+\frac{\mu}{4k^2}
\end{pmatrix}\nabla u\right]=-\mu u $$ for $t \leq t_0-C.$ Indeed, $u$ solves
$$
\ddot u + \div\left[ \begin{pmatrix}
    \tilde a(t) & 0\\
    0 & 1+\frac{\mu}{4k^2}
\end{pmatrix}\nabla u\right]=-\mu u $$ for all $t \in \R$ by \eqref{eqtforunew} and $\tilde a(t)\equiv \frac{\mu}{2k^2}$ for $t \leq t_0-C$ by the definition of $\tilde a$ (see \eqref{tildeatnew}), by the definition of $a$ (see \eqref{defaoftbeginningnew}) and of $t_0$ and $C$. Moreover, by the definition of $\tilde a$ (see \eqref{tildeatnew}), $\tilde a \equiv 1+\frac{\mu}{k^2}$ for $t \geq t_0$ and $u(x,t)=\cos(kx) \e^{-k(t-t_0)}$ for $t \geq t_0$ (by \eqref{deftonew}). Hence, as claimed in the symmetrization Lemma, we have been able to transform the solution and the matrix given in \eqref{problemuonenew} into the solution and the matrix given in \eqref{problemutwonew} within the set $\T^2 \times [t_0-C, t_0]$. The transformation matrix is given by 
\begin{align}\label{transfmatrixwithatilde}
\begin{pmatrix}
    \tilde a(t) & 0 \\
    0 & 1 + \frac{\mu}{4k2}
\end{pmatrix}.    
\end{align}

Since $\frac{\mu}{2k^2} \leq \tilde a(t) \leq 1+\frac{\mu}{k^2},$ (since the same is true for $a(t)$ by its definition \eqref{defaoftbeginningnew}), it is clear that the transformation matrix \eqref{transfmatrixwithatilde} has a uniform ellipticity constant of $\frac{5k^2}{\mu}.$ Finally, as mentioned in the beginning, we show in Appendix \ref{appendixdurationsymmetry} that $t_2-t_1 \leq \frac{2}{5}$. This bound is obtained by asking that $\dot{a}(t) \leq 10$ which is also proved in the Appendix \ref{appendixdurationsymmetry}. This finishes the proof of the symmetrization Lemma.

\comment{
The construction above shows that we can transform $f(t)\cos(k x)$ into $ \e^{-(t-t_0)k} \cos(k x)$ in time $t_2-t_1.$  As mentioned in the beginning, we show in Appendix \ref{appendixdurationsymmetry} that $t_2-t_1 \leq \frac{2}{5}$. This bound is obtained by asking that $\dot{a}(t) \leq 10$ which is also proved in the Appendix \ref{appendixdurationsymmetry}. Finally, since $\frac{\mu}{2k^2} \leq a(t) \leq 1+\frac{\mu}{k^2},$ it is clear that the matrix $\begin{pmatrix}
    a(t) & 0 \\
    0 & 1+\frac{\mu}{4k^2}
\end{pmatrix}$
has a uniform ellipticity constant of $\frac{5k^2}{\mu}.$ This finishes the proof of Lemma \ref{altabstractlemmaevaluenewnewrecall}.}

\end{proof}

To finish the proof of Theorem \ref{EigenTheorem}, it remains to prove Theorem \ref{halfcylinderallmunew}. This is done in the next three sections (section \ref{sect4}, section \ref{thebuildingblocknewlabel} and section \ref{proofofpropslabel}).

\section{The half-cylinder}
\label{sect4}
The goal of this section is to reduce Theorem \ref{halfcylinderallmunew} that we recall below on the construction of a super-exponentially decaying solution to a divergence form equation with constant potential on the half-cylinder $\T^2 \times \R^+$, to the construction of a $A$-harmonic function on $\T^2 \times [0, C]$, which oscillates and decay faster at $t=C$ compared to $t=0$. This reduction will be done in several steps.

\subsection{A fast-decaying solution on a half-cylinder with constant potential}

\begin{thm*}[\ref{halfcylinderallmunew}, A fast-decaying solution on a half-cylinder with constant potential] 
    Let $\mu \in \R$. In the 3-dimensional cylinder $\mathbb{T}^2\times\mathbb{R}^+$, there exists a $C^1$-uniformly smooth elliptic matrix $A:=A_{\mu}(x,y,t)$ and a non-zero uniformly $C^2$ real function $u:=u_{\mu}(x,y,t)$ such that \begin{equation} \label{eigen} \ddot{u} + \div(A\nabla u)=-\mu u\end{equation} and  $u$ has double exponential decay: for any $T\gg 1,$
$$ \sup_{\T^2 \times \{t\geq T\}}|u(x,y,t)|\leq Ce^{-ce^{c T}}$$
for some  $c(\mu), C(\mu)>0$. The matrix $A$ belongs to the regularity class $R(100, 61).$ Moreover, on $\T^2 \times \left[0, \frac{1}{100}\right],$ the function $u$ and the matrix $A$ are equal to
\begin{align}\label{solutionuclosetozeroinmainhalfcylinderbis}
    u(x,y,t)= \cos(kx)\e^{-kt}, \hspace{0.5cm} A= Id + \begin{pmatrix}
        \frac{\mu}{k^2} & 0\\
        0 & \frac{\mu}{4k^2}
    \end{pmatrix}.
\end{align}
for some fixed $k\gg \sqrt{|\mu|}$. \comment{ and the matrix $A$ is be equal to
\begin{align}\label{matrixaclosetozeroinainhalfcylinder}
    A= Id + \begin{pmatrix}
        \frac{\mu}{k^2} & 0\\
        0 & \frac{\mu}{4k^2}
    \end{pmatrix}.
\end{align}}
\end{thm*}

\comment{
\begin{thm}[A fast-decaying solution on a half-cylinder with constant potential]\label{halfcylinderallmunewrecall}
    Let $\mu \in \R$. In the 3-dimensional cylinder $\mathbb{T}^2\times\mathbb{R}^+$, there exists a $C^1$-uniformly smooth elliptic operator $A:=A_{\mu}(x,y,t)$ and a non-zero real function $u:=u_{\mu}(x,y,t)$ such that \begin{equation} \label{eigen} \ddot{u} + \div(A\nabla u)=-\mu u\end{equation} and  $u$ has double exponential decay:
$$ \sup_{t\geq T}|u(x,y,t)|\leq e^{-ce^{c T}}$$
for some  $c(\mu)>0$ and for $T\gg 1$. The matrix $A$ belongs to the regularity class $R(400, 11).$
\\

Moreover, on $\T^2 \times \left[0, \frac{1}{100}\right],$ the function $u$ is be equal to
\begin{align}\label{solutionuclosetozeroinmainhalfcylinder}
    u(x,y,t)= \cos(kx)\e^{-kt}.
\end{align}
for some fixed $k\gg \sqrt{|\mu|}$ and the matrix $A$ is be equal to
\begin{align}\label{matrixaclosetozeroinainhalfcylinder}
    A= Id + \begin{pmatrix}
        \frac{\mu}{k^2} & 0\\
        0 & \frac{\mu}{4k^2}
    \end{pmatrix}.
\end{align}
\end{thm}
}

We first reduce Theorem \ref{halfcylinderallmunew} to Theorem \ref{theoremaharmonicagainnew} below, which is a quantitative version of our first main Theorem \ref{theoremforhalfcylinder} and which states that we can construct a $A$-harmonic function on the half-cylinder $\T^2 \times \R^+$ with double-exponential decay.

\begin{thm}[A fast-decaying $A$-harmonic function on a half-cylinder]\label{theoremaharmonicagainnew}
    In the 3-dimensional cylinder $\mathbb{T}^2\times\mathbb{R^+}$, there exists an elliptic matrix $A=A(x,y,t)$ which is $C^1$-uniformly smooth and  a non-zero uniformly  $C^2$ real function $u=u(x,y,t)$ such that \begin{equation*} \ddot u + \div(A\nabla u)=0\end{equation*}and such that $u$ has double exponential decay: for any $T\gg 1,$
\begin{align}\label{decaystatementforevaluebutnotyet}
    \sup_{\T^2\times \{t \geq T\}}|u(x,y,t)|\leq Ce^{-ce^{cT}}
\end{align}
for some numerical $c, C>0$.

More precisely, there exists an increasing sequence of times $t_n \geq 0$ such that $t_1=0$ and such that $t_{n+1}-t_n > \frac{3}{100}$. We represent $\T^2 \times \R^+ = \bigsqcup_{n \geq 1} \T^2 \times [t_n, t_{n+1}]$. The following hold:
\begin{itemize}
    \item On $\T^2 \times [t_n, t_{n+1}],$ for $n$ odd, $u:=u_n(x,y,t)$ where
    \begin{align}\label{defunandknforhalfyclinderreductionharmonicfirst}
        u_n(x,y,t):= f_n(t) \cos(k_n x) + g_{n+1}(t)  \cos(k_{n+1} y)
    \end{align}
    where $k_n:=2^{n+n_0-1}$ where $n_0\gg 1$ of our choosing. Moreover, $f_n$ and $g_{n+1}$ are uniformly $C^2$ on $\R^+.$
    \item In particular, on $\T^2 \times [t_n, t_n+\frac{1}{100}],$ for $n$ odd,
    \begin{align}\label{unbeginnnngforhalfcylinderharmofirst}
    u_n(x,y,t)=c_n \cos(k_n x)\e^{-k_n t}, \hspace{0.5cm} A=Id     
    \end{align}
    
    and on $\T^2 \times [t_{n+1}-\frac{1}{100},t_{n+1}]$, for $n$ odd
    \begin{align}\label{unbeginnnngforhalfcylinderharmoendfirst}
        u_n(x,y,t)=c_{n+1} \cos(k_{n+1} y)\e^{-k_{n+1} t}, \hspace{0.5cm} A=Id 
    \end{align}
    where $\{c_n\}$ is a sequence of positive terms such that $c_1:=1$.
    
    \item On $\T^2 \times [t_n, t_{n+1}],$ for $n$ even, the role of $x$ and $y$ are exchanged in \eqref{defunandknforhalfyclinderreductionharmonicfirst}, \eqref{unbeginnnngforhalfcylinderharmofirst} and \eqref{unbeginnnngforhalfcylinderharmoendfirst}.
    
    \item $A$ belongs to the regularity class $R(80,60).$ 

    \item The constants $c, C$ from \eqref{decaystatementforevaluebutnotyet} depends on $n_0.$
\end{itemize}

\end{thm}

To ease the reader, we first present the heuristic idea of the reduction.
\pseudosection{Heuristic idea of the reduction of Theorem \ref{halfcylinderallmunew} to Theorem \ref{theoremaharmonicagainnew}.}
 Consider $u$, the solution to $\ddot{u} + \div(A\nabla u)=0$ from Theorem \ref{theoremaharmonicagainnew}. At any moment $t\in[t_n, t_{n+1}]$ (assuming $n$ is odd) the solution $u$ can be represented as 
    \begin{align}\label{uintntnplusonebisbisbisfirstnew}
        f_n(t) \cos(k_nx)+g_{n+1}(t)\cos(k_{n+1}y).
    \end{align}

    Consider the matrix $A' := A+B$ where 
    $$
    B:=\begin{pmatrix}\mu/k_n^2&0\\0&\mu/(k_{n+1})^2\end{pmatrix}.$$ 

    Then,    
    $$\ddot{u}+\div (A'\nabla u) = \ddot{u}+\div (A\nabla u)+\div(B\nabla u)= \div(B\nabla u).$$
    
    Moreover,
    $$\div(B\nabla u) = \frac{\mu}{k_n^2} \frac{\pl^2 u}{\pl x^2}+\frac{\mu}{(k_{n+1})^2}\frac{\pl^2 u}{\pl y^2} = -\mu u$$
by \eqref{uintntnplusonebisbisbisfirstnew}.

However, the matrix $A'$ that we constructed here is not even continuous at times $t=t_n$ (since $B$ has constant coefficients). 
To solve this issue, we use that $u$ will in fact be only one cosine towards the endpoints of a block $\T^2\times [t_n, t_{n+1}]$ (see \eqref{unbeginnnngforhalfcylinderharmofirst} and \eqref{unbeginnnngforhalfcylinderharmoendfirst}). This allows us to smoothly change the entry in $B$ related to the other space variable thus making the matrix $B \in C^1$ in $\T^2\times \R^+$. We now present the proof of this reduction.

\begin{proof}[Proof of the reduction of Theorem \ref{halfcylinderallmunew} to Theorem \ref{theoremaharmonicagainnew}]
    Let $\mu \in \R$. Consider the matrix $A$ and the double-exponentially decaying solution $u$ to $\ddot{u}+\div( A\nabla u) = 0$ given by Theorem \ref{theoremaharmonicagainnew}.  Define a matrix $A'$ by 
    \begin{align}\label{thematrixaprime}
        A':=A+B
    \end{align}
where $B$ is to be constructed.

\pseudosection{Step 1: The construction of $B$}
Let $\{t_n\}$ be the increasing sequence of times given by Theorem \ref{theoremaharmonicagainnew}. In particular,
\begin{align}\label{propertytnnew}
    t_1=0 \hspace{0.5cm} \mbox{ and } \hspace{0.5cm} t_{n+1}-t_n >\frac{3}{100}.
\end{align}
  Let $\{k_n\}$ be given by  Theorem \ref{theoremaharmonicagainnew}, that is,
\begin{align}\label{defknxonstructionb}
k_n:=2^{n_0+n-1}, \hspace{0.5cm } n\geq 1,
\end{align}
where $n_0$ is fixed and to be chosen later (we will choose it in the third step of the proof, see 
 \eqref{choosingnzero}, such that $2^{n_0} \gg \sqrt{|\mu|})$. For all $n \geq 1$ and for all $t \in [t_n, t_{n+1}]$, we define 
    \begin{align}\label{defbharmonicreducing}
        B_n:=
    \begin{pmatrix}
       a_{n+1}(t) & 0 \\ 0 & a_{n+2}(t)
    \end{pmatrix} \mbox{for $n$ odd } \hspace{0.3cm} \mbox{ and } \hspace{0.3cm} B_n:=
    \begin{pmatrix}
       a_{n+2}(t) & 0 \\ 0 & a_{n+1}(t)
    \end{pmatrix} \mbox{for $n$ even.}
    \end{align}

 The sequence $\{a_{n+1}(t)\}$ is defined for $n \geq 1$ by 
    \begin{align}\label{defanforcoeffinb}
       a_{n+1}(t) := \left(\frac{\mu}{k_{n}^2} - \frac{\mu}{k_{n+2}^2} \right) \theta\bigg(100(t-t_{n+1})+1\bigg) + \frac{\mu}{k_{n+2}^2} ,
    \end{align}
    
 where $\theta(\tau)$ is presented in picture \ref{fig:theta} (see also footnote \footnote{We choose $\theta(t):= 1-G(\tan(\pi(t-1/2)))$ where $G(x)=\frac{1}{\sqrt{\pi}}\int_{-\infty}^x \e^{-\eta^2} \ud{\eta}$. \label{footnoteblablattt}}) and goes smoothly from 1 to 0 as $\tau$ goes from 0 to 1. We continue the function $\theta$ so that $\theta(\tau)\equiv 1$ for $\tau \leq 0$.  In particular,
 \begin{align}\label{propanplusonenew}
     a_{n+1}(t) \mbox{ is smooth }, \hspace{0.3cm} a_{n+1}(t) \equiv \frac{\mu}{k_n^2} \mbox{ for } t \leq t_{n+1}-\frac{1}{100}, \hspace{0.3cm} \lim_{t \rightarrow t_{n+1}^-} a_{n+1}(t) =  \frac{\mu}{k_{n+2}^2}.
 \end{align}

    We also define the matrix $B$ by 
    \begin{align}\label{defBharmonicreducingwitouhoutn}
        B:=B_n \mbox{ for } t \in [t_n, t_{n+1}] \mbox{ and } n \geq 1.
    \end{align}

\pseudosection{Step 2: The equation $\div(B \nabla u)=-\mu u$}

   Without loss of generality, \underline{we assume $n$ is odd in what follows.}
\comment{
By Theorem \ref{theoremaharmonicagainnew}, at any moment $t\in[t_n, t_{n+1}]$ (since $n$ is odd) the solution can be represented as 
    \begin{align}\label{uintntnplusonebisbisbis}
        f_n(t) \cos(k_nx)+g_{n+1}(t)\cos(k_{n+1}y).
    \end{align}
}

\begin{enumerate}
    \item 
\underline{For $t \in [t_n, t_{n+1}]$:} In this case,
\begin{align}\label{defpropbnintntnplusone}
    B=B_n= \begin{pmatrix}
    a_{n+1}(t) & 0 \\
    0 & a_{n+2}(t) \end{pmatrix}.
\end{align}
The first equality comes from the definition of $B$ \eqref{defBharmonicreducingwitouhoutn}. The second equality comes from the definition of $B_n$ \eqref{defbharmonicreducing} and by the assumption that $n$ is odd. We split the discussion into two cases:

\begin{itemize}
    \item For $t \in [t_{n+1}-\frac{1}{100}, t_{n+1}],$
\begin{align}\label{bclosetotnplusoneminus}
B= \begin{pmatrix}
    a_{n+1}(t) & 0 \\
    0 & \mu/k_{n+1}^2
\end{pmatrix}.    
\end{align}
Indeed, $t \leq t_{n+1} <t_{n+2}-\frac{1}{100}$ since $t_{n+2}-t_{n+1}>\frac{3}{100}$ by the property of $\{t_n\}$ \eqref{propertytnnew}. Therefore, by \eqref{propanplusonenew}, $a_{n+2}(t) \equiv \frac{\mu}{k_{n+1}^2}$ and this yields  the equality \eqref{bclosetotnplusoneminus}.

Also, by Theorem \ref{theoremaharmonicagainnew}, when $t \in [t_{n+1}-\frac{1}{100}, t_{n+1}],$ and since $n$ is odd by assumption, $u=g_{n+1}(t) \cos(k_{n+1}y)$. Therefore, $\div(B \nabla u)=-\mu u.$

\item For $t \in [t_n, t_{n+1}-\frac{1}{100}],$ 
\begin{align}\label{defbawayfromrightendpoint}    
   B=\begin{pmatrix}
        \mu/k_n^2 & 0 \\
        0 & \mu/k_{n+1}^2
    \end{pmatrix}    .
    \end{align}
Indeed, since $t \leq t_{n+1}-\frac{1}{100},$ $a_{n+1}(t) \equiv \frac{\mu}{k_n^2}$ and $a_{n+2}(t) \equiv \frac{\mu}{k_{n+1}^2}$  by \eqref{propanplusonenew}. So the equality \eqref{defbawayfromrightendpoint} follows from \eqref{defpropbnintntnplusone}. By Theorem \ref{theoremaharmonicagainnew}, when $t \in [t_n, t_{n+1}],$ $u=f_n(t)\cos(k_nx)+g_{n+1}(t)\cos(k_{n+1}y)$ (since $n$ is assumed to be odd) and therefore $\div(B\nabla u)=-\mu u$. 

\pseudosection{Verification of \eqref{solutionuclosetozeroinmainhalfcylinderbis}}
Moreover, \eqref{defbawayfromrightendpoint} shows that 
    \begin{align}
        B=\begin{pmatrix}
         \mu/k_1^2 & 0 \\
         0 & \mu/k_2^2
        \end{pmatrix}
    \end{align}
for $t \in [0, \frac{1}{100}]$ (because $[0, \frac{1}{100}] = [t_1, t_1+\frac{1}{100}] \subset [t_1, t_2-\frac{1}{100}]$ since $t_1=0$ and $t_{2}-t_1 > \frac{3}{100}$ by \eqref{propertytnnew}). 

Therefore, since $A'=A+B$  by definition \eqref{thematrixaprime} and since $A=Id$ on $[0, \frac{1}{100}]$ by Theorem \ref{theoremaharmonicagainnew}, we get that 
\begin{align}\label{aprimeclosetozeroproof}
    A'=Id+\begin{pmatrix}
         \mu/k_1^2 & 0 \\
         0 & \mu/4k_1^2
        \end{pmatrix}
\end{align}
on $[0, \frac{1}{100}]$ (we used that $k_{n+1}=2k_n$ by definition \eqref{defknxonstructionb}). Also by Theorem \ref{theoremaharmonicagainnew}, for $t \in [0, \frac{1}{100}],$
       \begin{align}\label{uoneclosetozeroproof}
           u_1(x,y,t)=\cos(k_1 x) \e^{-k_1t}
       \end{align} 
       since $c_1:=1.$
Hence, by \eqref{aprimeclosetozeroproof}, by \eqref{uoneclosetozeroproof} and by the fact that $k_1=2^{n_0}$  (see \eqref{defknxonstructionb} and note that $n_0$  will be chosen in \eqref{choosingnzero} such that $2^{n_0} \gg \sqrt{|\mu|}$), we obtain \eqref{solutionuclosetozeroinmainhalfcylinderbis} that was claimed in Theorem \ref{halfcylinderallmunew}. 

\end{itemize}

\item 
\underline{For $t \in [t_{n+1}, t_{n+2}]$:} In this case, 
\begin{align}\label{bdefpropintnrplusonetnplusrwo}
    B=B_{n+1} = \begin{pmatrix}
        a_{n+3}(t) & 0 \\
        0 & a_{n+2}(t)
    \end{pmatrix}.
\end{align}
The first equality comes from the definition of $B$ \eqref{defBharmonicreducingwitouhoutn}. The second equality comes from the definition of $B_n$ \eqref{defbharmonicreducing} and by the assumption that $n$ is odd (hence $n+1$ is even). We again split the discussion into two cases:
\begin{itemize}
    \item For $t \in [t_{n+1}, t_{n+2}-\frac{1}{100}],$
    \begin{align}\label{bontplusoneplus}
         B= \begin{pmatrix}
        \mu/k_{n+2}^2 & 0 \\
        0 & \mu/k_{n+1}^2
    \end{pmatrix} . 
    \end{align}
    Indeed, since $t \leq t_{n+2}-\frac{1}{100},$ $a_{n+3}(t) \equiv \frac{\mu}{k_{n+2}^2}$ and $a_{n+2}(t)\equiv \frac{\mu}{k_{n+1}^2} $  by \eqref{propanplusonenew}. So, the equality \eqref{bontplusoneplus} now follows from \eqref{bdefpropintnrplusonetnplusrwo}. By Theorem \ref{theoremaharmonicagainnew}, for $t \in [t_{n+1}, t_{n+2}-\frac{1}{100}], $ $u$ is of the form $f_{n+1}(t) \cos(k_{n+1}y) + g_{n+2}(t)\cos(k_{n+2}x)$ (since $n+1$  is even) and therefore $\div(B\nabla u) = -\mu u.$

    \item For $t \in [t_{n+2}-\frac{1}{100}, t_{n+2}]$,
    \begin{align}\label{eqtforbnlastoneblablanew}
    B=  \begin{pmatrix}
        \mu/k_{n+2}^2 & 0 \\
        0 & a_{n+2}(t)
    \end{pmatrix}.     
    \end{align}
\end{itemize}
Indeed, since $t \leq t_{n+2} < t_{n+3}-\frac{1}{100}$ (since by  \eqref{propertytnnew}, $t_{n+3}-t_{n+2}>\frac{3}{100}$), \eqref{propanplusonenew} gives $a_{n+3}(t) \equiv \mu/k_{n+2}^2.$ So the equality \eqref{eqtforbnlastoneblablanew} now follows from \eqref{bdefpropintnrplusonetnplusrwo}. Since for $t \in [t_{n+2}-\frac{1}{100}, t_{n+2}],$ $u$ is of the form $g_{n+2}(t) \cos(k_{n+2}x)$ by Theorem \ref{theoremaharmonicagainnew}, we again have $\div(B\nabla u)=-\mu u.$ 
\\
\end{enumerate}

\pseudosection{Step 3: The regularity}

\begin{enumerate}
    \item 
\underline{The uniform ellipticity of $A'$:}
Recall that we defined $A':=A+B$ (see \eqref{thematrixaprime}). The uniform ellipticity of $A'$ is due to the following: 
for any $n \geq 1$ and for any $t \in [t_n, t_{n+1}],$ 
\begin{align}\label{ineqforellipticboundofaprime}
|A'-A| \leq \frac{|\mu|}{k_1^2} .    
\end{align}
This comes from the definition $A'=A+B$ and from the definition of $B$ \eqref{defbharmonicreducing}. Indeed, for $t \in [t_n, t_{n+1}],$ $B$ is a diagonal matrix with entries $a_{n+1}(t)$ and $a_{n+2}(t)$. By definition,  $$ a_{n+1}(t) := \left(\frac{\mu}{k_{n}^2} - \frac{\mu}{k_{n+2}^2} \right) \theta\bigg(100(t-t_{n+1})+1\bigg) + \frac{\mu}{k_{n+2}^2}.$$ Since $0 \leq \theta \leq 1,$ and since $\{k_n\}$ is an increasing sequence (see \eqref{defknxonstructionb}), we have for all $n \geq 1$ and for any $t \in [t_n, t_{n+1}],$
$$
\frac{|\mu|}{k_{n+2}^2} \leq |a_{n+1}(t)| \leq \frac{|\mu|}{k_n^2} \leq \frac{|\mu|}{k_1^2}, \hspace{0.5cm} \frac{|\mu|}{k_{n+3}^2} \leq |a_{n+2}(t)| \leq \frac{|\mu|}{k_{n+1}^2} \leq \frac{|\mu|}{k_1^2}.
$$
Since  $k_1:=2^{n_0}$ and since $n_0$ is us to choose, we choose it such that 
\begin{align}\label{choosingnzero}
2^{n_0} \gg \sqrt{|\mu|},    
\end{align}

which  will ensure thanks to \eqref{ineqforellipticboundofaprime} above that 
$$
|A'-A| \leq\frac{1}{1000}
$$

Since $A \in R(80,60)$ by Theorem \ref{theoremaharmonicagainnew}, i.e., $\frac{1}{80}|\xi|^2 \leq (A\xi,\xi)\leq 80 |\xi|^2,$ we get that 
\begin{align}\label{ellipticityofaprime}
\frac{|\xi|^2}{100} \leq (A'\xi, \xi) \leq 100 |\xi|^2.    
\end{align}

The argument is similar in $[t_{n+1}, t_{n+2}]$ and we skip it. This proves the uniform ellipticity of $A'.$
\\

\item
 \underline{The $C^1$ boundedness of $A'$:}
 \begin{itemize}
     \item \underline{The endpoints:} By \eqref{bclosetotnplusoneminus},
 for $t \in [t_{n+1}-\frac{1}{100}, t_{n+1}]$,
 \begin{align*}
     B=\begin{pmatrix}
         a_{n+1}(t) & 0 \\
         0 & \mu/k_{n+1}^2
     \end{pmatrix} 
 \end{align*}
 and by \eqref{bontplusoneplus}, for $t \in [t_{n+1}, t_{n+2}-\frac{1}{100}],$
 \begin{align*}
     B=\begin{pmatrix}
         \mu/k_{n+2}^2 & 0 \\
         0 & \mu/k_{n+1}^2
     \end{pmatrix}.
 \end{align*}

 Since $\lim_{t \rightarrow t_{n+1}^-}a_{n+1}(t) =\mu/k_{n+2}^2 $ by \eqref{propanplusonenew}, $B$ is continuous at $t_{n+1}.$  Note that by definition (see footnote \footnote{We choose $\theta(t):= 1-G(\tan(\pi(t-1/2)))$ where $G(x)=\frac{1}{\sqrt{\pi}}\int_{-\infty}^x \e^{-\eta^2} \ud{\eta}$. \label{footnoteblablabla}}), $\theta$ satisfies 
\begin{align}\label{propthetaforbn}
\lim_{\tau \rightarrow 0} \dot{\theta}(\tau) = 0, \hspace{0.5cm} \lim_{\tau \rightarrow 1} \dot{\theta}(\tau)=0, \hspace{0.5cm} \sup_{[0, 1]}|\dot \theta(\tau)|\leq \sqrt{\pi}  .  
\end{align}

Since by definition (see \eqref{defanforcoeffinb})
$$
a_{n+1}(t) := \left(\frac{\mu}{k_{n}^2} - \frac{\mu}{k_{n+2}^2} \right) \theta\bigg(100(t-t_{n+1})+1\bigg) + \frac{\mu}{k_{n+2}^2},
$$
we have $\lim_{t \rightarrow t_{n+1}^-}\dot a_{n+1}(t) =0$. Hence, $B$ is $C^1$ across $t_{n+1}.$ A similar argument applies across $t_{n+2}.$ Hence, $B$ is $C^1$ across all $t_n$, $n \geq 2.$

\item \underline{In $[t_n, t_{n+1}]$} The $C^1$ regularity of $B$ on $t \in [t_n, t_{n+1}]$ is due to the following: $B$ is a diagonal matrix, with entries $a_{n+1}(t)$ and $a_{n+2}(t)$. Hence, it does not depend on $x,y$ so $\frac{\pl B}{\pl x}=\frac{\pl B}{\pl y} = 0.$

Moreover, for $t \in [t_n, t_{n+1}]$, 

\begin{align}\label{derivativebintntnplusonegood}
     \left|\frac{d a_{n+1}(t)}{dt}\right| = \left|\frac{\mu}{k_{n}^2} - \frac{\mu}{k_{n+2}^2}\right| 100 \big|\dot{\theta}\big(100(t-t_{n+1})+1)\big)\big|.
\end{align}
By \eqref{propthetaforbn}, $\left| \dot{\theta}\big(100(t-t_{n+1})+1)\big)\right| \leq \sqrt{\pi}$ on $[t_n, t_{n+1}]$ (recall that we extended $\theta(\tau)$ by 1 for $\tau \leq 0$).
\comment{
$
\dot{\theta}\big(100(t-t_{n+1})-1)\big)$ goes to zero as $t$ goes to $t_{n+1}-\frac{1}{100}$ and to $t_{n+1}$. Moreover, on $[t_n, t_{n+1}],$ $\left| \dot{\theta}\big(100(t-t_{n+1})-1)\big)\right| \leq \sqrt{\pi}$
}

Therefore, by \eqref{derivativebintntnplusonegood}, and by recalling that $k_n=2^{n+n_0-1},$ we get that
\begin{align*}
\left|\frac{d a_{n+1}(t)}{d t}\right| \leq \frac{1}{100}    
\end{align*}
by choosing $n_0$ such that $2^{n_0}\gg \sqrt{|\mu|}.$ A similar argument applies for $a_{n+2}(t)$ hence,
\begin{align}\label{coneboundofb}
   \left|\frac{d B}{d t}\right| \leq \frac{1}{100} 
\end{align}

in $[t_{n}, t_{n+1}].$ A similar argument applies for $t \in [t_{n+1}, t_{n+2}]$ and we skip it. We have therefore shown that $B$ is uniformly $C^1$ on all of $\R^+.$ 

\comment{
\item \textcolor{red}{At any interval the function $u$ is equal to $f(t)\cos(kx)+g(t)\cos(k'y)$. Here $f(t)$ and $g(t)$ decay exponentially everywnere save for the introduction or eradication. In said intervals the $\alpha$-th derivatives of $f$ and $g$ are estimated as $C_\alpha k^\alpha \max(f,g).$ During the introduction itself $g$ is introduced or $f$ is eradicated while being $\epsilon$ times smaller than the other function. The $\alpha$th time derivative is estimated as $\epsilon^{\frac{\alpha}{3}-1}\max(|f|,|g|).$ Then the derivatives up to order 3 satisfy the necessary bounds.}}
 \end{itemize}

\item \underline{The conclusion} 

Since $A$ is uniformly $C^1$ on $\T^2 \times \R^+$ by Theorem \ref{theoremaharmonicagainnew} we get that $A':=A+B$ is uniformly $C^1$ on $\T^2 \times \R^+$ since $B$ is uniformly $C^1$ by step 3, as claimed in Theorem \ref{halfcylinderallmunew}. 

In particular, since $A'$ has a uniform ellipticity constant of $100$ by \eqref{ellipticityofaprime}, since $B$ has $C^1$ bound of $\frac{1}{100}$ by \eqref{coneboundofb} and since $A \in R(80, 60)$ by Theorem \ref{theoremaharmonicagainnew}, we get that $A' \in R(100, 61)$ as claimed in Theorem \ref{halfcylinderallmunew}.

By Theorem \ref{theoremaharmonicagainnew}, $u$ is uniformly $C^2$ smooth on $\T^2\times \R^+$ and has double exponential decay at infinity. Hence, the claimed $C^2$ regularity and double exponential decay in Theorem \ref{halfcylinderallmunew} is true as well since we did not modify the solution $u$ from Theorem \ref{theoremaharmonicagainnew}. Note also that by Theorem \ref{theoremaharmonicagainnew}, the constants $c, C$ depend on $n_0.$ Since $n_0$ is now chosen so that $2^{n_0}\gg \sqrt{|\mu|},$ the constants $c, C$ now depends on $\mu$ as claimed.

Finally, by step 2,  $\div(B\nabla u)=-\mu u$ on $\T^2 \times \R^+$. Since $u$ satisfies $\ddot u+ \div(A\nabla u)=0$ in $\T^2\times \R^+$  we get that $\ddot u+ \div(A'\nabla u)=-\mu u$ in $\T^2\times \R^+.$ This finishes the proof of the reduction of Theorem \ref{halfcylinderallmunew} to Theorem \ref{theoremaharmonicagainnew}.

\end{enumerate}

    \end{proof}

\subsection{A fast decaying $A$-harmonic function on a half-cylinder}
\label{harmony}
In this section, we reduce Theorem \ref{theoremaharmonicagainnew}, which we recall below, to the construction of a $A$-harmonic function on a block $\T^2 \times [t_1, t_1+C]$ for any $t_1\geq 0$ and where $C$ is a universal constant. 
\begin{thm*}[\ref{theoremaharmonicagainnew}, A fast-decaying $A$-harmonic function on a half-cylinder]
    In the 3-dimensional cylinder $\mathbb{T}^2\times\mathbb{R^+}$, there exists an elliptic operator $A=A(x,y,t)$ which is $C^1$-uniformly smooth in $x,y,t$ and  a non-zero uniformly  $C^2$ real function $u=u(x,y,t)$ such that \begin{equation*} \ddot u+ \div(A\nabla u)=0\end{equation*}and such that $u$ has double exponential decay: for any $T\gg 1,$
\begin{align}\label{decayestimatestatementhalfharmonic}
\sup_{\T^2\times \{t \geq T\}}|u(x,y,t)|\leq Ce^{-ce^{cT}}    
\end{align}
for some numerical $c, C>0$.

More precisely, there exists an increasing sequence of times $t_n \geq 0$ such that $t_1=0$ and such that $t_{n+1}-t_n > \frac{3}{100}$. We represent $\T^2 \times \R^+ = \bigsqcup_{n \geq 1} \T^2 \times [t_n, t_{n+1}]$. The following hold:
\begin{itemize}
    \item On $\T^2 \times [t_n, t_{n+1}],$ for $n$ odd, $u:=u_n(x,y,t)$ where
    \begin{align}\label{defunandknforhalfyclinderreductionharmonic}
        u_n(x,y,t):= f_n(t) \cos(k_n x) + g_{n+1}(t)  \cos(k_{n+1} y)
    \end{align}
    where $k_n:=2^{n+n_0-1}$ where $n_0\gg 1$ of our choosing. Moreover, $f_n$ and $g_{n+1}$ are uniformly $C^2$ on $\R^+.$
    \item In particular, on $\T^2 \times [t_n, t_n+\frac{1}{100}],$ for $n$ odd,
    \begin{align}\label{unbeginnnngforhalfcylinderharmo}
    u_n(x,y,t)=c_n \cos(k_n x)\e^{-k_n t}, \hspace{0.5cm} A=Id     
    \end{align}
    
    and on $\T^2 \times [t_{n+1}-\frac{1}{100},t_{n+1}]$, for $n$ odd
    \begin{align}\label{unbeginnnngforhalfcylinderharmoend}
        u_n(x,y,t)=c_{n+1} \cos(k_{n+1} y)\e^{-k_{n+1} t}, \hspace{0.5cm} A=Id 
    \end{align}
    where $\{c_n\}$ is a sequence of positive terms such that $c_1:=1$.
    
    \item On $\T^2 \times [t_n, t_{n+1}],$ for $n$ even, variables $x$ and $y$ are exchanged in \eqref{defunandknforhalfyclinderreductionharmonic}, \eqref{unbeginnnngforhalfcylinderharmo} and \eqref{unbeginnnngforhalfcylinderharmoend}.
    
    \item $A$ belongs to the regularity class $R(80,60).$ 

    \item The constants $c, C$ from \eqref{decayestimatestatementhalfharmonic} depends on $n_0$.
\end{itemize}
\end{thm*}

We reduce Theorem \ref{theoremaharmonicagainnew} to Lemma \ref{lemmaniceconnectingtone} below.

\begin{lemma}\label{lemmaniceconnectingtone}
    Let $1\ll k<k'\leq 2k$. There exists $C\geq 2$ independent of $k, k'$ such that for any $t_1 \geq 0$ and for any $c_1>0$,  there exists $ c_2>0$ such that one can transform $c_1 \cos(kx) \e^{-kt}$ into $c_2 \cos(k'y) \e^{-k't}$ within the set $\mathbb{T}^2\times[t_1, t_1+C]$, via a solution $u$ to $\ddot u + \div(A\nabla u)=0$ and where $A$ is in the  regularity class $R(80, 60)$ and $A$ is identity near the boundary of $\mathbb{T}^2\times[t_1, t_1+C]$. The constants $c_1$ and $c_2$ are related by $$c_1\, c\,  \e^{-kt_1} = c_2 \e^{-k't_1}, \hspace{0.5cm} c=\e^{\frac{-k}{2} + \frac{5k'}{6}}.$$

In particular, on $[t_1, t_1+C]$, the function $u$ is of the form 
\begin{align}\label{generalformtonetoneplusc}
    f(t)  \cos(kx) + g(t)  \cos(k'y)
\end{align}
      where $f,g \in C^2$ satisfy for all $0\leq \alpha \leq 2,$
     \begin{align}\label{generalformtonetonepluscforfg}
     |f^{(\alpha)}(t)| \lesssim c_1 k^{\frac{7\alpha}{3}} \e^{-kt}, \hspace{0.5cm} |g^{(\alpha)}(t)| \lesssim c_2\, (k')^{\alpha}  \e^{\frac{-k'}{3}t} \e^{\frac{-2k'}{3}t_1}.    
     \end{align}

      Moreover, on $[t_1, t_1+\frac{1}{100}]$,
     \begin{align}\label{formuclosebegtone}
         u(x,y,t)=c_1\cos(kx) \e^{-kt}, \hspace{0.5cm} A=Id
     \end{align}
      and on $[t_1+C-\frac{1}{100}, t_1+C]$,
      \begin{align}\label{formucloseendttwo}
          u(x,y,t)=c_2 \cos(k'y) \e^{-k't}, \hspace{0.5cm} A=Id.
      \end{align}
\end{lemma}

\begin{proof}[Proof of the reduction of Theorem \ref{theoremaharmonicagainnew} to Lemma \ref{lemmaniceconnectingtone}]
    The reduction will be done in several steps. Using Lemma \ref{lemmaniceconnectingtone}, we construct $u$ and $A$. Then we show that $A$ is uniformly $C^1$. Finally, we show that $u$ is uniformly $C^2$ and that it has a double exponential decay at infinity.

    \pseudosection{Construction of the solution $u$ and the matrix $A$.}
    
    Consider the sequence of times
    \begin{align}\label{deftimestnnew}
        t_n := C(n-1), \hspace{0.5cm} n \geq 1,
    \end{align}
    where $C\geq 2$ is given by Lemma \ref{lemmaniceconnectingtone}. We decompose $\R^+$ as the almost disjoint union of intervals $[t_n, t_{n+1}].$ Choose 
    \begin{align}\label{defkn}
        k_n:=2^{n_0+n-1}, \hspace{0.5cm} n \geq 1,
    \end{align}
    where $n_0\gg 1$ is fixed of our choosing. Let $c_1:=1$ and inductively define the sequence $\{c_n\}$ by 
    \begin{align}\label{relationcncnplusone}
        c_n \,d_n\, \e^{-k_nt_n} = c_{n+1} \e^{-k_{n+1}t_n}, \hspace{0.5cm} d_n=\e^{\frac{-k_n}{2} + \frac{5k_{n+1}}{6}}.
    \end{align}

    By Lemma \ref{lemmaniceconnectingtone}, we can transform $c_n \cos(k_n x) \e^{-k_nt}$ into $c_{n+1} \cos(k_{n+1} y) \e^{-k_{n+1}t}$ within $\T^2 \times [t_n, t_{n+1}]$ via a solution $\tilde u_n(x,y,t)$ of $\ddot{\tilde u}_n + \div( \tilde A_n\nabla \tilde u_n)=0$ and with regularity $R(80, 60).$  In particular, by Lemma \ref{lemmaniceconnectingtone}, for $t \in [t_n, t_n+\frac{1}{100}]$, 
\begin{align}\label{tildeunatthebeginning}
\tilde u_n(x,y,t)=c_n \cos(k_nx) \e^{-k_n t}, \hspace{0.5cm} \tilde A_n = Id
\end{align}
 and  for $t \in [t_{n+1}-\frac{1}{100}, t_{n+1}],$
\begin{align}\label{tildeunatthebeginningend}
   \tilde u_n(x,y,t)=c_{n+1} \cos(k_{n+1}y) \e^{-k_{n+1} t} , \hspace{0.5cm} \tilde A_n = Id.
\end{align}

Moreover, by Lemma \ref{lemmaniceconnectingtone}, for $t \in [t_n, t_{n+1}],$ the function $\tilde u_n$ is of the form 
\begin{align}\label{shapetildeunhalfyclinderharmonic}
\tilde u_n(x,y,t) = f_n(t) \cos(k_nx) + g_{n+1}(t) \cos(k_{n+1}y)    
\end{align}
where $f_n, g_{n+1} \in C^2$ satisfy for all $0\leq \alpha\leq 2$ and for all $t \in [t_n, t_{n+1}],$
\begin{align}\label{regularityfngnplusonehalfcylinderhamronic}
|f_n^{(\alpha)}(t)| \lesssim c_n (k_n)^{\frac{7\alpha}{3}} \e^{-k_n t}, \hspace{0.5cm} |g^{(\alpha)}(t)| \lesssim c_{n+1} (k_{n+1})^{\alpha} \e^{\frac{-k_{n+1}t}{3}} \e^{\frac{-2k_{n+1} t_n}{3}}.    
\end{align}

    For $t \in [t_n, t_{n+1}],$ define the functions 
    \begin{align}\label{defunoddeven}
       u_n(x,y,t):= \begin{cases}\wt{u}_n(x,y,t),&n \mbox{ odd,}\\\wt{u}_n(y,x,t),&n \mbox{ even},\end{cases} \hspace{0.5cm}  A_n(x,y,t):= \begin{cases}\wt{A}_n(x,y,t),&n \mbox{ odd,}\\\wt{A}^t_n(y,x,t),&n \mbox{ even}\end{cases}
    \end{align}
   \comment{ and 
    \begin{align}\label{defAnoddeven}
       A_n(x,y,t):= \begin{cases}\wt{A}_n(x,y,t),&n \mbox{ odd,}\\\wt{A}_n(y,x,t),&n \mbox{ even}\end{cases}
    \end{align}}
    where $\tilde{A}^t_n(y,x,t)$, denotes the matrix obtained from $\tilde{A}_n(y,x,t)$ by exchanging the position of the two rows and by exchanging the position of the two columns. We will now glue all these functions $u_n$ and all these matrices $A_n$ to get a function $u$ and a matrix $A$ defined on $\T^2\times \R^+$: for all $n\geq 1$ and for all $t \in [t_n, t_{n+1}],$ we define 
\begin{align}\label{defu}
u(x,y,t):= u_n(x,y,t), \hspace{0.5cm} A(x,y,t):= A_n(x,y,t).
\end{align}

\pseudosection{The regularity of $A$}
We verify that the matrix $A$ defined in \eqref{defu} is $C^1$ on $\T^2 \times \R^+$. Let $t \in [t_n, t_{n+1}]$. To fix the idea, assume $n$ is odd. By the definition of $A$ \eqref{defu} and by the definition of $A_n$ \eqref{defunoddeven}, $A(x,y,t) = \tilde A_n (x,y,t).$ Since $\tilde A_n$ is the transformation matrix from Lemma \ref{lemmaniceconnectingtone}, we saw above that $\tilde A_n$ is in the regularity class $R(80, 60).$ Hence, $ A\in R(80, 60)$ for $t \in [t_n, t_{n+1}]$.

It remains to verify that $A$ is $C^1$ across the endpoints $t_n$, $n \geq 2$. By the definition of $A$ \eqref{defu}, by the definition of $A_n$ \eqref{defunoddeven} and \eqref{tildeunatthebeginning} and \eqref{tildeunatthebeginningend}, 
    \begin{align}\label{defaclosetotnplusone}
        \begin{cases}
            A(x,y,t) = Id  & t \in [t_{n+1}-\frac{1}{100}, t_{n+1}], \\
            A(x,y,t) = Id , & t \in [t_{n+1}, t_{n+1}+\frac{1}{100}].
        \end{cases}
    \end{align}

   \comment{ By Lemma \ref{lemmaniceconnectingtoneglue}, 
    $$
    \lim_{t \rightarrow t_{n+1}^-} C_n = \lim_{t \rightarrow t_{n+1}^+} B_{n+1}=0 
    $$
    and therefore by \eqref{defaclosetotnplusone} $A$ is continuous at $t_{n+1}.$

    Also, by Lemma \ref{lemmaniceconnectingtoneglue},
    $$
    \lim_{t \rightarrow t_{n+1}^-} \partial_i C_n = \lim_{t \rightarrow t_{n+1}^+} \partial_i B_{n+1}=0 
    $$
    for $i\in \{x, y,t\}$. }
    Therefore, $A$ is $C^1$ at $t_{n+1}$. In conclusion, the matrix $A$ defined by \eqref{defu} is  $C^1$ on $\T^2 \times \R^+$ and more precisely, $A \in R(80, 60).$

\pseudosection{The regularity of $u$}
We will show that $u$ defined in \eqref{defu} is uniformly $C^2$ on $\T^2 \times \R^+.$ The proof will be split in several steps:
\\
\pseudosection{Step 1: The regularity of $u$ across $t_n,$ $n \geq 2.$} Without loss of generality, we assume $n$ is odd. By the definition of $u$ \eqref{defu}, by the definition of $u_n$ \eqref{defunoddeven}, and by \eqref{tildeunatthebeginning} and \eqref{tildeunatthebeginningend},  
    \begin{align}\label{unintheinterval}
    \begin{cases}
       u(x,y,t)=  c_{n+1}\cos(k_{n+1}y) \e^{-k_{n+1}t} & t \in [t_{n+1}-\frac{1}{100}, t_{n+1}] ,\\
        u(x,y,t)=c_{n+1} \cos(k_{n+1} y) \e^{-k_{n+1}t}   & t \in [t_{n+1}, t_{n+1}+\frac{1}{100}] 
    \end{cases}    
    \end{align}
    \comment{
and   $$
   \lim_{t \rightarrow t_{n+1}^-} g_{n+1}(t) =1, \hspace{0.5cm}  \lim_{t \rightarrow t_{n+1}^+} f_{n+1}(t) =1.
   $$
   
    Therefore, by  \eqref{unintheinterval}, $u$ is continuous at $t_{n+1}.$

    By Lemma \ref{lemmaniceconnectingtoneglue}, 
    $$
    \lim_{t \rightarrow t_{n+1}^+}  f^{(k)}_{n+1}(t) = \lim_{t \rightarrow t_{n+1}^-} g^{(k)}_{n+1}(t) =0.
    $$
    for $k \in \{1, 2\}$.}
    Therefore,  $u$ is $C^2$ at $t_{n+1}.$ 

    \pseudosection{Step 2: The regularity of $u$ on $\T^2\times [t_n, t_{n+1}]$.} We again assume $n$ is odd. By the definition of $u$ \eqref{defu}, by the definition of $u_n$ \eqref{defunoddeven} and by \eqref{shapetildeunhalfyclinderharmonic}, for $t \in [t_n, t_{n+1}],$
$$
u(x,y,t) = f_n(t) \cos(k_nx) + g_{n+1}(t) \cos(k_{n+1}y)    
$$
where $f_n, g_{n+1} \in C^2$ . Therefore, $u$ is $C^2$ on $\T^2 \times [t_n, t_{n+1}].$

\textbf{Partial conclusion:} By Step 1 and Step 2 above, we can conclude that $u$ is $C^2$ on $\T^2 \times \R^+$.  Moreover, by construction, $u$ solves 
\begin{align}\label{partialconclusionaftersometime}
    \ddot u + \div(A\nabla u)=0
\end{align}
 on $\T^2 \times \R^+$. Note that we do not claim yet that $u$ is \textit{uniformly} $C^2$.

\pseudosection{Step 3: The super-exponential decay of $u$ and the uniform $C^2$ boundedness.}
In this step, we show the following  result: for any multi-index $\alpha \in \mathbb{N}^3$ of order $|\alpha|\leq 2$, the following holds: for any $T \gg 1$, 
\begin{align}\label{maindecayresult}
\sup_{\T^2 \times \{t\geq T\}} |\partial^{\alpha} u(x,y,t)|\leq C \e^{-c\e^{cT}} 
\end{align}
for some numerical constants $C, c>0.$ This result shows that $u$ and all its derivatives of order 2 or less, decay super-exponentially at infinity. In particular, \eqref{maindecayresult} shows that $u$ has super-exponential decay at infinity and that  $u$ is uniformly $C^2$ on $\T^2 \times \R^+.$

\begin{proof}[Proof of \eqref{maindecayresult}]

Let $t \in [t_n, t_{n+1}]$ for some $n \geq 1$. Suppose without loss of generality that $n$  is odd. By the definition of $u$ \eqref{defu} (see also \eqref{defunoddeven} and  \eqref{shapetildeunhalfyclinderharmonic}), for $t \in [t_n, t_{n+1}],$ the function $u$ is of the form 
\begin{align}\label{shapetildeunhalfyclinderharmonicforproof}
u(x,y,t) = f_n(t) \cos(k_nx) + g_{n+1}(t) \cos(k_{n+1}y)    
\end{align}
with $f_n, g_{n+1} \in C^2.$
\comment{
satisfying for all $0\leq \alpha\leq 2$ and for all $t \in [t_n, t_{n+1}],$
\begin{align}\label{regularityfngnplusonehalfcylinderhamronicforproof}
|f_n^{(\alpha)}(t)| \lesssim c_n (k_n)^{\frac{7\alpha}{3}} \e^{-k_n t}, \hspace{0.5cm} |g_{n+1}^{(\alpha)}(t)| \lesssim c_{n+1} (k_{n+1})^{\alpha} \e^{\frac{-k_{n+1}t}{3}} \e^{\frac{-2k_{n+1} t_n}{3}}.    
\end{align}
}

\textbf{Claim 1:} Let $t \in [t_n, t_{n+1}]$ for some $n \gg 1.$ Then, 
\begin{align}\label{claimfourstatetemtn}
   |f_n^{(\beta)}(t)| \lesssim \e^{k_1(2C-\frac{7}{6})} \e^{-(C-\frac{7}{6})k_n},  \hspace{0.5cm}  |g_{n+1}^{(\beta)}(t)| \lesssim  \e^{k_1(2C-\frac{7}{6})} \e^{-(C-\frac{7}{6})k_n}.
\end{align}
for all $0\leq \beta\leq 2$ and where $C=t_{n+1}-t_n\geq 2$ is a universal constant. We postpone the proof of Claim 1 and we continue with the proof of \eqref{maindecayresult}.

Let $T\gg 1$, that is, $T\in [t_{n_1}, t_{n_1+1} ]$ for some $n_1\gg 1.$ For any $t \geq T$, $t$ will be in some interval $t \in [t_{n_2}, t_{n_2+1}]$, $n_2\geq n_1$. By \eqref{shapetildeunhalfyclinderharmonicforproof}, for any $0\leq \beta \leq 2$ and for $t \in [t_{n_2}, t_{n_2+1}]$
\begin{align}\label{sueperdecaytderivatives}
|\partial^{\beta}_t u(x,y,t)| \leq |f_{n_2}^{(\beta)}(t)| + |g_{n_2+1}^{(\beta)}(t)| &\lesssim \e^{k_1(2C-\frac{7}{6})} \e^{-(C-\frac{7}{6})k_{n_2}} \nonumber \\
&\leq  \e^{k_1(2C-\frac{7}{6})} \e^{-(C-\frac{7}{6})k_{n_1}} \nonumber \\
&=\e^{2^{n_0}(2C-\frac{7}{6})} \e^{-(C-\frac{7}{6})2^{n_0+n_1-1}}.    
\end{align}
The second inequality comes from Claim 1 \eqref{claimfourstatetemtn} and the third and fourth inequalities hold since $n_2\geq n_1$ and $k_n=2^{n_0+n-1}$ by definition (see \eqref{defkn}) where $n_0\gg 1$ is fixed and of our choosing. Since $T \in [t_{n_1}, t_{n_1+1}]$ and since $t_n=C(n-1)$ with $C\geq 2$ (see \eqref{deftimestnnew}), we have $\frac{T}{C}\leq n_1.$ Therefore, by \eqref{sueperdecaytderivatives},
\begin{align}\label{finalfort}
|\partial^{\beta}_t u(x,y,t)| \lesssim \e^{2^{n_0}(2C-\frac{7}{6})} \e^{-(C-\frac{7}{6})2^{n_0-1} 2^{\frac{T}{C}}}      
\end{align}
for any $0\leq \beta \leq 2$ and for any $t\geq T.$ Since $n_0$ is fixed, and since \eqref{finalfort} holds for any $t \geq T,$ we get the claimed inequality \eqref{maindecayresult} when $\alpha=(0, 0, \beta)$. By combining \eqref{shapetildeunhalfyclinderharmonicforproof} with Claim 1 \eqref{claimfourstatetemtn}, it is straightforward to see that a similar argument ensures that \eqref{maindecayresult} holds for any multi-index $\alpha$ of order $|\alpha|\leq 2.$ This finishes the proof of \eqref{maindecayresult} given Claim 1.
\end{proof}

We now prove Claim 1.
\begin{proof}[Proof of Claim 1] 
Let $t \in [t_n, t_{n+1}]$ for some $n \geq 1$. Suppose without loss of generality that $n$  is odd. By the definition of $u$ \eqref{defu} (see also \eqref{defunoddeven}, \eqref{shapetildeunhalfyclinderharmonic} and \eqref{regularityfngnplusonehalfcylinderhamronic}), for $t \in [t_n, t_{n+1}],$ the function $u$ is of the form 
\begin{align}\label{shapetildeunhalfyclinderharmonicforproofbis}
u(x,y,t) = f_n(t) \cos(k_nx) + g_{n+1}(t) \cos(k_{n+1}y)    
\end{align}
with $f_n, g_{n+1} \in C^2$ satisfying for all $0\leq \beta\leq 2$ and for all $t \in [t_n, t_{n+1}],$
\begin{align}\label{regularityfngnplusonehalfcylinderhamronicforproof}
|f_n^{(\beta)}(t)| \lesssim c_n (k_n)^{\frac{7\beta}{3}} \e^{-k_n t}, \hspace{0.5cm} |g_{n+1}^{(\beta)}(t)| \lesssim c_{n+1} (k_{n+1})^{\beta} \e^{\frac{-k_{n+1}t}{3}} \e^{\frac{-2k_{n+1} t_n}{3}}.    
\end{align}
We start with $f_n^{(\beta)}$. By \eqref{regularityfngnplusonehalfcylinderhamronicforproof}, we get for $t \in [t_n, t_{n+1}]$, $n \geq 1,$ 
\begin{align}\label{claimfourfirstresult}
    |f_n^{(\beta)}(t)| \lesssim (k_n)^{\frac{7\beta}{3}} c_n\e^{-k_n t} \leq (k_n)^{\frac{7\beta}{3}} c_n\e^{-k_n t_n}.
\end{align}
    for any $0\leq \beta \leq 2$. 
    
   Now, for $g_{n+1}^{(\beta)},$  recall that by \eqref{regularityfngnplusonehalfcylinderhamronicforproof}, for all $t \in [t_n, t_{n+1}]$, $n \geq 1,$
    $$
    |g_{n+1}^{(\beta)}(t)| \lesssim c_{n+1} (k_{n+1})^{\beta} \e^{\frac{-k_{n+1}t}{3}} \e^{\frac{-2k_{n+1} t_n}{3}}
    $$
    for all $0\leq \beta \leq 2.$ Recall that $ 
        c_{n+1}=c_n \,d_n\, \e^{(k_{n+1}-k_n)t_n} = c_n \,d_n\, \e^{k_nt_n}$. The first equality comes from the definition of $c_{n+1}$ \eqref{relationcncnplusone} and the second comes from $k_{n+1}-k_n=k_n$ by definition \eqref{defkn}. Therefore,  we have that
\begin{align*}
    |g_{n+1}^{(\beta)}(t)| &\lesssim (k_{n+1})^{\beta} c_{n} \,d_n\, \e^{k_n t_n}  \e^{\frac{-2k_{n+1} t_n}{3}} \e^{\frac{-k_{n+1}t}{3}} 
\end{align*}

Using again  $k_{n+1}=2k_n$, we have that $\e^{k_n t_n}\e^{\frac{-2k_{n+1} t_n}{3}} = \e^{\frac{-k_{n} t_n}{3}} $ and the estimate above  becomes
\begin{align*}
   |g_{n+1}^{(\beta)}(t)| \lesssim (k_{n+1})^{\beta} c_{n} \,d_n\, \e^{\frac{-k_{n} t_n}{3}}  \e^{\frac{-k_{n+1}t}{3}} . 
\end{align*}

Recall that $d_n=\e^{\frac{-k_n}{2}+\frac{5k_{n+1}}{6}}$  by definition \eqref{relationcncnplusone}, which becomes $d_n=\e^{\frac{7k_n}{6}}$ by using $k_{n+1}=2k_n$. Hence, 
\begin{align*}
    |g_{n+1}^{(\beta)}(t)| \lesssim  (k_{n+1})^{\beta} c_{n} \e^{\frac{7}{6}k_n} \e^{\frac{-k_nt_n}{3}} \e^{\frac{-2k_nt}{3}}
\end{align*}
where we used again that $k_{n+1}=2k_n.$ Finally, since $t \in [t_n, t_{n+1}],$ we get
\begin{align}\label{estimateforgnplusonebeforeintermediateclaim}
|g_{n+1}^{(\beta)}(t)| \lesssim  (k_{n+1})^{\beta} c_{n} \e^{-k_n(t_n-\frac{7}{6})}.    
\end{align}

\textbf{Partial conclusion:} Combining \eqref{claimfourfirstresult} and \eqref{estimateforgnplusonebeforeintermediateclaim}, we showed for $t \in [t_n, t_{n+1}],$ $n \geq 1$ and for all $0\leq \beta \leq 2,$
\begin{align}\label{intermediateconclusionnicefngnplusone}
|f_n^{(\beta)}(t)| \lesssim (k_n)^{\frac{7\beta}{3}} c_n\e^{-k_n t_n},  \hspace{0.5cm}  |g_{n+1}^{(\beta)}(t)| \lesssim  (k_{n+1})^{\beta} c_{n} \e^{-k_n(t_n-\frac{7}{6})}. 
\end{align}

\textbf{Claim 2:} Define $C_n:=c_n \e^{-k_n(t_n-\frac{7}{6})}$. Then $C_n= \e^{k_1(2C-\frac{7}{6})} \e^{-(2C-\frac{14}{6})k_n}, $ where we recall that $C=t_{n+1}-t_n \geq 2$ is a universal constant.
\\

We postpone the proof of Claim 2 and we finish the proof of Claim 1 \eqref{claimfourstatetemtn}. Since $\e^{-k_n t_n} \leq \e^{-k_n(t_n-\frac{7}{6})},$ it comes by \eqref{intermediateconclusionnicefngnplusone} and by Claim 2 that 
\begin{align}\label{claimforclaimoneblabla}
|f_n^{(\beta)}(t)| \lesssim (k_n)^{\frac{7\beta}{3}} \e^{k_1(2C-\frac{7}{6})} \e^{-(2C-\frac{14}{6})k_n},  \hspace{0.5cm}  |g_{n+1}^{(\beta)}(t)| \lesssim  (k_{n+1})^{\beta} \e^{k_1(2C-\frac{7}{6})} \e^{-(2C-\frac{14}{6})k_n}    
\end{align}
for all $t \in [t_n, t_{n+1}],$ $n \geq 1$ and for all $0\leq \beta \leq 2.$  Now, let $t\gg 1$, that is, $t \in [t_n, t_{n+1}]$ for $n\gg 1.$ Since $k_n=2^{n_0+n-1}$ by definition (see \eqref{defkn}) and since $C=t_{n+1}-t_n\geq 2$ (see \eqref{deftimestnnew}) we have by \eqref{claimforclaimoneblabla} that,
$$
|f_n^{(\beta)}(t)| \lesssim  \e^{k_1(2C-\frac{7}{6})} \e^{-(C-\frac{7}{6})k_n},  \hspace{0.5cm}  |g_{n+1}^{(\beta)}(t)| \lesssim  \e^{k_1(2C-\frac{7}{6})} \e^{-(C-\frac{7}{6})k_n}    
$$
for all $0\leq \beta \leq 2.$ This proves Claim 1.
\end{proof}

We now prove Claim 2.
\begin{proof}[Proof of Claim 2]
Recall from definition \eqref{relationcncnplusone} that $d_n=\e^{\frac{-k_n}{2}+\frac{5k_{n+1}}{6}} = \e^{\frac{7}{6}k_n}$ since $k_{n+1}=2k_n.$ Therefore, since $c_{n+1}=c_n d_n \e^{k_nt_n},$ we have
\begin{align*}
    c_{n+1} = c_n \e^{k_n(\frac{7}{6}+t_n)}
\end{align*}
for all $ n \geq 1.$ Therefore, 
$$
C_{n+1} = c_{n+1}\e^{-k_{n+1}(t_{n+1}-\frac{7}{6})} = c_n \e^{k_n(\frac{7}{6}+t_n)}\e^{-k_{n+1}(t_{n+1}-\frac{7}{6})}.
$$

Since $k_{n+1}=2k_n$ (see \eqref{defkn}) and since $t_{n+1}=t_{n}+C$ (see \eqref{deftimestnnew}), this estimates becomes
$$
C_{n+1} = c_n \e^{k_n(\frac{7}{6}+t_n)} \e^{-2k_nt_n-2k_nC +k_n\frac{14}{6}} = c_n \e^{-k_n(t_n-\frac{7}{6})} \e^{-k_n(2C-\frac{14}{6})}=C_n \e^{-k_n(2C-\frac{14}{6})}.
$$
So, for all $n \geq 1,$
$$
C_n=C_1\e^{-(2C-\frac{14}{6})\sum_{l=1}^{n-1}k_l}
$$
Since $k_l=2^{n_0+l-1}$ by definition \eqref{defkn}, we see that $\sum_{l=1}^{n-1}k_l=2^{n_0+n-1}-2^{n_0}=k_n-k_1.$ Therefore, for all $n \geq 1,$
\begin{align}\label{cnaftersometime}
C_n = C_1 \e^{-(2C-\frac{14}{6})(k_n-k_1)}.    
\end{align}

By definition (see Claim 2), $C_1=c_1\e^{-k_1(t_1-\frac{7}{6})}$. Since  $c_1=1$ by definition \eqref{relationcncnplusone} and since $t_1=0$ also by definition \eqref{deftimestnnew}, we obtain $C_1=\e^{\frac{7}{6}k_1}$. And therefore, by \eqref{cnaftersometime}, 
$$
C_n = \e^{k_1(2C-\frac{7}{6})} \e^{-(2C-\frac{14}{6})k_n}
$$
as claimed.   This finishes the proof of Claim 2.

\comment{
Note that $k_1=2^{n_0}$ for some fixed $n_0\gg 1$ of our choosing (see \eqref{defkn}), that $c_1=1$ by definition \eqref{relationcncnplusone} and that $t_1=0$ also by definition \eqref{deftimestnnew} . Therefore, $C_1=c_1\e^{-k_1(t_1-\frac{7}{6})}=\e^{\frac{7}{6}k_1}$ and
$C_1 \e^{(2C-\frac{14}{6})k_1}=\e^{k_1(2C-\frac{7}{6})}$. In conclusion, we showed 
$$
C_n = \e^{k_1(2C-\frac{7}{6})} \e^{-(2C-\frac{14}{6})k_n}
$$
as claimed.   This finishes the proof of Claim 2.}
\end{proof}

\pseudosection{What we proved:}
\begin{itemize}
    \item We constructed a function $u$ and a matrix $A$ (see \eqref{defu}).
    \item In \eqref{defaclosetotnplusone} , we showed that $A$ is uniformly elliptic and uniformly $C^1$ in $\T^2 \times \R^+.$ More precisely, $A\in R(80, 60).$
    \item By the partial conclusion \eqref{partialconclusionaftersometime}, and by \eqref{maindecayresult} in Step 3, we get that $u$ is uniformly $C^2$ in $\T^2\times \R^+$, that it has super-exponential decay and that $\ddot u + \div(A\nabla u)=0$ in $\T^2\times \R^+.$
\end{itemize}

\underline{It only remains to verify some claims from Theorem \ref{theoremaharmonicagainnew}.} \begin{itemize}
\item  By definition, $t_n:=C(n-1)$ (see \eqref{deftimestnnew}) so $t_1=0$. By Lemma \ref{lemmaniceconnectingtone}, $t_{n+1}-t_n=C\geq 2$. 

\comment{
    \item By the definition of $u$ and $A$ \eqref{defu}, \eqref{defA}, by the definition of $u_n$ and $A_n$ \eqref{defunoddeven}, \eqref{defAnoddeven} and by \eqref{tildeunatthebeginning} and \eqref{tildeantosatisfyatthebegiining}, and using that $t_1=0$ ($t_n$ was defined at the very beginning of the proof), we have that on $[0, \frac{1}{100}]$,
\begin{align}
    u=u_1=\cos(k_1 x)\e^{-k_1t}, \hspace{0.5cm} A=Id
\end{align}
which proves a claim made in Theorem \ref{theoremaharmonicagain}.}

\item For $n$ odd, we saw (see \eqref{defu}, \eqref{defunoddeven} and \eqref{shapetildeunhalfyclinderharmonic}) that for $t \in \T^2 
\times [t_n, t_{n+1}]$,
\begin{align}\label{generalformuinsideablockforlemmablabla}
u(x,y,t) = f_n(t)  \cos(k_n x) + g_{n+1}(t)  \cos(k_{n+1}y)   
\end{align}
with $k_n=2^{n_0+n-1}$, $n_0\gg 1$ of our choosing (see \eqref{defkn}). We also showed in \eqref{claimfourstatetemtn} that $f_n, g_{n+1}$ are uniformly $C^2$ on $\R^+.$

\item  By the definition of $u$ and $A$ \eqref{defu}, by the definition of $u_n$ and $A_n$ \eqref{defunoddeven} and by \eqref{tildeunatthebeginning} (respectively \eqref{tildeunatthebeginningend}), we get \eqref{unbeginnnngforhalfcylinderharmo} (respectively \eqref{unbeginnnngforhalfcylinderharmoend}).

\item For $n$ even, by the definition of $u$ \eqref{defu}, by the definition of $u_n$ \eqref{defunoddeven} and by \eqref{shapetildeunhalfyclinderharmonic}, we have in $\T^2 
\times [t_n, t_{n+1}]$,
\begin{align}\label{generalformuinsideablockforlemmablablaeven}
u(x,y,t) = f_n(t)  \cos(k_n y)  + g_{n+1}(t)  \cos(k_{n+1}x)    
\end{align}
with $k_n=2^{n_0+n-1}$, $n_0\gg 1$ (see \eqref{defkn}). 

\item By \eqref{finalfort}, the constants $c, C$ in the decay estimate \eqref{decayestimatestatementhalfharmonic} depends on $n_0$ as claimed.
\end{itemize}

This finishes the proof of the reduction of Theorem \ref{theoremaharmonicagainnew} to Lemma \ref{lemmaniceconnectingtone}.

\end{proof}

\subsection{The construction in a block }\label{reductioninablock}
In this section, we reduce Lemma \ref{lemmaniceconnectingtone}, which we recall below, to the construction of a $A$-harmonic function in a block $\T^2 \times [0, C]$ for a universal constant $C>0$.
\begin{lemma*}[\ref{lemmaniceconnectingtone}]
     Let $k$ and $k'$ be any integers such that $1\ll k<k'\leq 2k$. There exists $C\geq 2$ independent of $k, k'$ such that for any $t_1 \geq 0$ and for any $c_1>0$,  there exists $ c_2>0$ such that one can transform $c_1 \cos(kx) \e^{-kt}$ into $c_2 \cos(k'y) \e^{-k't}$ within the set $\mathbb{T}^2\times[t_1, t_1+C]$, via a solution $u$ to $\ddot u + \div(A\nabla u)=0$ and where $A$ is in the  regularity class $R(80, 60).$ The constants $c_1$ and $c_2$ are related by $$c_1\, c\,  \e^{-kt_1} = c_2 \e^{-k't_1}, \hspace{0.5cm} c=\e^{\frac{-k}{2} + \frac{5k'}{6}}.$$

In particular, on $[t_1, t_1+C]$, the function $u$ is of the form 
\begin{align}\label{generalformtonetoneplusc}
    f(t)  \cos(kx) + g(t)  \cos(k'y)
\end{align}
      where $f,g \in C^2$ satisfy for all $0\leq \alpha \leq 2,$
     \begin{align}\label{generalformtonetonepluscforfg}
     |f^{(\alpha)}(t)| \lesssim c_1 k^{\frac{7\alpha}{3}} \e^{-kt}, \hspace{0.5cm} |g^{(\alpha)}(t)| \lesssim c_2\, (k')^{\alpha}  \e^{\frac{-k'}{3}t} \e^{\frac{-2k'}{3}t_1}.    
     \end{align}

      Moreover, on $[t_1, t_1+\frac{1}{100}]$,
     \begin{align}\label{formuclosebegtone}
         u(x,y,t)=c_1\cos(kx) \e^{-kt}, \hspace{0.5cm} A=Id
     \end{align}
      and on $[t_1+C-\frac{1}{100}, t_1+C]$,
      \begin{align}\label{formucloseendttwo}
          u(x,y,t)=c_2 \cos(k'y) \e^{-k't}, \hspace{0.5cm} A=Id.
      \end{align}
\end{lemma*}

We will reduce this Lemma to the \textit{building block}, Lemma \ref{blockzerocnewnew}, that we present below.
\begin{lemma}[The building block]\label{blockzerocnewnew}
     Let $1\ll k<k'\leq 2k$. There exists $C\geq 2$ independent of $k, k'$ such that  one can transform $ \cos(kx) \e^{-kt}$ into $c\cos(k'y) \e^{-k't}$ within the set $\{(x,y,t): 0\leq t \leq C\}$, via a solution $u$ to $\ddot u + \div(A\nabla u)=0$ and where $A$ is in the  regularity class $R(80, 60).$  The constant $c=c(k,k')$ is given by 
   $$
   c=\e^{\frac{-k}{2} + \frac{5k'}{6}}.
   $$
  On $[0, C]$, the function $u$ is of the form \begin{align}\label{generalshapeonzeroc}
       u(x,y,t)=f(t)  \cos(kx) + g(t) \cos(k'y)
  \end{align}
     where $f,g \in C^2$ satisfy for all $0\leq \alpha \leq 2,$
     $$
     |f^{(\alpha)}(t)| \lesssim k^{\frac{7\alpha}{3}} \e^{-kt}, \hspace{0.5cm} |g^{(\alpha)}(t)| \lesssim (k')^{\alpha} c\, \e^{\frac{-k'}{3}t}.
     $$
     Moreover, on $[0, \frac{1}{100}]$,
     \begin{align}\label{formuclosebegtonezeroc}
         u(x,y,t)=\cos(kx) \e^{-kt}, \hspace{0.5cm} A=Id 
     \end{align}
      and on $[C-\frac{1}{100}, C]$,
      \begin{align}\label{formucloseendttwozeroc}
          u(x,y,t)= c \cos(k'y) \e^{-k't}, \hspace{0.5cm} A=Id.
      \end{align}
\end{lemma}

\begin{proof}[Proof of the reduction of Lemma \ref{lemmaniceconnectingtone}  to Lemma \ref{blockzerocnewnew}]

    Let $C\geq 2$ be given by Lemma \ref{blockzerocnewnew} and let  $u$ and $A$ be the solution and the matrix from Lemma \ref{blockzerocnewnew}:  $\ddot{u}+\div(A\nabla u)=0$ and $u$ transforms  $ \cos(kx) \e^{-kt}$ into $ c \cos(k'y) \e^{-k't}$ within the set $\{(x,y,t): 0\leq t \leq C\}$ where $A$ is in the regularity class $R(80, 60)$ and where $c=c(k, k')=\e^{\frac{-k}{2} + \frac{5k'}{6}}$.

    Let $t_1\geq 0$ and let $t \in [t_1, t_1+C]$. Denote 
    \begin{align}\label{deftildeutildeaforshitedproof}
    \tilde u(t):= u(t-t_1), \hspace{0.5cm} \tilde A(t):=A(t-t_1).    
    \end{align}
    
    Then, $ \tilde u$ transforms $ \cos(kx) \e^{-k(t-t_1)}$ into $ c \cos(k'y) \e^{-k'(t-t_1)}$ within the set $\T^2\times [t_1, t_1+C]$. It is a solution of $\ddot{\tilde u} + \div(\tilde A \nabla \tilde u)=0$ and $\tilde A$ is in the regularity class $R(80, 60).$  Let $c_1>0$ and define $c_2$ by 
    \begin{align}\label{defconstantzeroc}
    c_1\, c\, \e^{-kt_1} = c_2 \e^{-k't_1}.    
    \end{align}
    
    Denote 
    \begin{align}\label{defUshifted}
    U(t):=c_1 \e^{-kt_1} \tilde u(t).        
    \end{align}
    
    Then, using \eqref{defconstantzeroc}, we see that $U$ transforms $c_1 \cos(kx) \e^{-kt}$ into $c_2\cos(k'y)\e^{-k't}$  within the set $\T^2\times [t_1, t_1+C]$. It is a solution of 
    \begin{align}\label{eqtintonetwpprofoo}
    \ddot{U} + \div(\tilde A \nabla U)=0    
    \end{align}
    for $t \in [t_1, t_1+C]$, where $\tilde A$ is in the regularity class $R(80, 60).$ The solution $U$ and the matrix $\tilde A$ are the solution and matrix in Lemma \ref{lemmaniceconnectingtone}.

\textbf{Proof of \eqref{generalformtonetoneplusc} and \eqref{generalformtonetonepluscforfg}}
 Consider the solution $U$  on $[t_1, t_1+C]$ that we just constructed. By \eqref{deftildeutildeaforshitedproof} and by \eqref{defUshifted}, $U$ is given, for $t \in [t_1, t_1+C]$, by 
\begin{align}\label{biguconstructedproof}
    U(x,y,t) = c_1 \e^{-kt_1}u(x,y,t-t_1)
\end{align}
where $u(x,y,t)$ is the solution on $[0, C]$ from Lemma \ref{blockzerocnewnew}.

By Lemma \ref{blockzerocnewnew}, for $t \in [0, C]$, 
\begin{align}\label{recallfromlittleuzeroc}
    u(x,y,t) = f(t) \cos(kx)  + g(t) \cos(k'y) 
\end{align}
where $f, g \in C^2.$ Hence, by \eqref{biguconstructedproof} and by \eqref{recallfromlittleuzeroc},  for $t \in [t_1, t_1+C]$,
\begin{align}
    U(x,y,t) &= c_1 \e^{-kt_1} f(t-t_1) \cos(kx) + c_1 \e^{-kt_1} g(t-t_1) \cos(k'y)  \nonumber \\
    &=: F(t) \cos(kx) + G(t)\cos(k'y).
\end{align}

By  Lemma \ref{blockzerocnewnew}, $f, g$ satisfy for all $0\leq \alpha \leq 2$ and for $t \in [0, C]$
     $$
     |f^{(\alpha)}(t)| \lesssim k^{\frac{7\alpha}{3}} \e^{-kt}, \hspace{0.5cm} |g^{(\alpha)}(t)| \lesssim (k')^{\alpha} c\, \e^{\frac{-k'}{3}t} .
     $$
Therefore, for $t \in [t_1, t_1+C,]$ $F, G$ are $C^2$ and satisfy for all $0\leq \alpha \leq 2,$
$$
|F^{(\alpha)}(t)| \lesssim c_1 k^{\frac{7\alpha}{3}} \e^{-kt}, \hspace{0.5cm} |G^{(\alpha)}(t)|\lesssim (k')^{\alpha} c_2 \e^{\frac{-k'}{3}t} \e^{\frac{-2k'}{3}t_1}.
$$
     Indeed, the inequality for $F$ is straightforward. For $G$, we use that $c_1\, c\, \e^{-kt_1} = c_2 \e^{-k't_1}$ by \eqref{defconstantzeroc}. This proves \eqref{generalformtonetoneplusc} and \eqref{generalformtonetonepluscforfg} from Lemma \ref{lemmaniceconnectingtone}.   
\\

\textbf{Proof of \eqref{formuclosebegtone}}

By Lemma \ref{blockzerocnewnew}, for $t \in [0, \frac{1}{100}]$, 
$$
u(x,y,t) = \cos(kx)e^{-kt}, \hspace{0.5cm} A(x,y,t)=Id.
$$
Since $U(x,y,t) = c_1 \e^{-kt_1}u(x,y,t-t_1)$ by \eqref{biguconstructedproof} and since $\tilde A= A(t-t_1)$ by \eqref{deftildeutildeaforshitedproof}, it is clear that 
$$
U(x,y,t) = c_1 \cos(kx) \e^{-kt}, \hspace{0.5cm} \tilde A(x,y,t) = Id
$$
for $t \in [t_1, t_1+\frac{1}{100}]$. This was claimed in \eqref{formuclosebegtone} in Lemma \ref{lemmaniceconnectingtone}. To prove \eqref{formucloseendttwo}, we argue similarly and we use the relation \eqref{defconstantzeroc} $c_1 c \e^{-kt_1} = c_2 \e^{-k't_1}$. This finishes the proof of the reduction of Lemma \ref{lemmaniceconnectingtone} to Lemma \ref{blockzerocnewnew}.

\end{proof}

To finish the proof of Theorem \ref{EigenTheorem}, it only remains to prove Lemma \ref{blockzerocnewnew}. We present the proof in the next section.

\section{The building block}\label{thebuildingblocknewlabel}
So far, we have reduced Theorem \ref{EigenTheorem} to Lemma \ref{blockzerocnewnew}, which we recall below, on the construction of a solution to a divergence form equation in the elementary block $\T^2 \times [0, C]$ for some universal $C>0$. In this section, we reduce Lemma \ref{blockzerocnewnew} to three technical Propositions. 

\begin{lemma*}[\ref{blockzerocnewnew}, The building block]
   Let $1\ll k<k'\leq 2k$. There exists $C\geq 2$ independent of $k, k'$ such that  one can transform $ \cos(kx) \e^{-kt}$ into $c\cos(k'y) \e^{-k't}$ within the set $\{(x,y,t): 0\leq t \leq C\}$, via a solution $u$ to $\ddot u + \div(A\nabla u)=0$ and where $A$ is in the  regularity class $R(80, 60).$  The constant $c=c(k,k')$ is given by 
   $$
   c=\e^{\frac{-k}{2} + \frac{5k'}{6}}.
   $$
  For $t \in [0, C]$, the function $u$ is of the form \begin{align}\label{generalshapeonzeroc}
       u(x,y,t)=f(t)  \cos(kx) + g(t) \cos(k'y)
  \end{align}
     where $f,g \in C^2$ satisfy for all $0\leq \alpha \leq 2,$
     $$
     |f^{(\alpha)}(t)| \lesssim k^{\frac{7\alpha}{3}} \e^{-kt}, \hspace{0.5cm} |g^{(\alpha)}(t)| \lesssim (k')^{\alpha} c\, \e^{\frac{-k'}{3}t}.
     $$
     Moreover, on $[0, \frac{1}{100}]$,
     \begin{align}\label{formuclosebegtonezeroc}
         u(x,y,t)=\cos(kx) \e^{-kt}, \hspace{0.5cm} A=Id 
     \end{align}
      and on $[C-\frac{1}{100}, C]$,
      \begin{align}\label{formucloseendttwozeroc}
          u(x,y,t)= c \cos(k'y) \e^{-k't}, \hspace{0.5cm} A=Id.
      \end{align}
\end{lemma*}

This construction relies on three Propositions that we present below.

\begin{prop}[Changing a coefficient]\label{lemmachangingacoeff}
 Let $k \geq 1$, let $a,b \in (\frac{1}{10}, 10)$, let $t_1\geq 0$ and let also $C>0$ be a universal constant. Consider  the equations
    $$
    \ddot u + \div\left[\begin{pmatrix}
        a & 0 \\
        0 & a
    \end{pmatrix} \nabla u\right]=0, \hspace{0.5cm} \mbox{ and } \hspace{0.5cm} \ddot u + \div\left[\begin{pmatrix}
        a & 0 \\
        0 & b
    \end{pmatrix} \nabla u\right]=0.
    $$
    There exists a uniformly $C^1$ and uniformly elliptic matrix-valued function $A$ on $\T^2 \times \R,$ such that  $$A=\begin{pmatrix}
        a & 0 \\
        0 & a
    \end{pmatrix} \hspace{0.3cm} \mbox{ for } t\leq t_1 \hspace{0.5cm} \mbox{ and } \hspace{0.3cm}  A= \begin{pmatrix}
        a & 0 \\
        0 & b
    \end{pmatrix} \hspace{0.3cm} \mbox{ for } t\geq t_1+C$$
     and such that $A$ belongs to the regularity class $R(10, \frac{10\sqrt{\pi}}{C})$ and the function $u:=\cos(kx) \e^{-k\sqrt{a}t}$ is a solution to $\ddot u + \div(A\nabla u)=0$ in $\T^2 \times \R$. 
\end{prop}

\begin{prop}[slowing down]\label{slownew}
Let $1 \ll k < k' \leq 2k$ and let $a, b \in (\frac{1}{10}, 10).$ 
Consider two solutions $\cos(kx) \e^{-k\sqrt{a}t}$ and $\cos(k'y) \e^{-k'\sqrt{b}t}$ to the same equation 
$$
\ddot u + \div \left [\begin{pmatrix}a&0\\0&b\end{pmatrix} \nabla u \right ] = 0.
$$
If 
$$
k' \sqrt{b} < k \sqrt{a},
$$
then for any $t_1\geq 0$ and for any $c_1>0$, there exists $c_2>0$ such that one can transform $c_1 \cos(kx) \e^{-k\sqrt{a}t}$ into $c_2\cos(k'y) \e^{-k'\sqrt{b}t}$ within the set $\T^2 \times [t_1, t_1+C],$ via a solution $u$ to $\ddot u + \div(A\nabla u)=0,$  where $A$ is in the regularity class $R(20,10)$. The constants $c_1$ and $c_2$ are related by 
\begin{align}\label{defconstantslowdownconectwonew}
    c_1 e^{-k\sqrt{a}t_1} = c_2 e^{-k'\sqrt{b}t_1}.
\end{align}

The duration $C=C(k, k')$ of the transformation is 
\begin{align}\label{durationslowdowntoneprop}
C=\frac{1}{k^{4/3}} + \frac{8 \ln k}{k\sqrt{a}-k'\sqrt{b}} + \frac{\sqrt{4\ln k}}{(k')^{1/3}}+\frac{1}{100}.    
\end{align}

Moreover, for $t \in [t_1, t_1+C],$ the function $u$ is of the form 
\begin{align}\label{shapesolslowdownproptone}
u(x,y,t)= f(t) \cos(kx)  +  g(t) \cos(k'y)    
\end{align}
 where $f, g \in C^2$ satisfy $|f^{(\alpha)}(t)| \lesssim c_1 k^{\frac{7\alpha}{3}} \e^{-k\sqrt{a}t}$ and $|g^{(\alpha)}(t)| \lesssim c_2 (k')^{\alpha} \e^{-k'\sqrt{b}t}$ for all $0 \leq \alpha \leq 2.$
\end{prop}

\begin{prop}[acceleration]
\label{accelnew}
Let $k \gg 1 $. Let also $a, b, b' \in (\frac{1}{10}, 10)$ with $b \leq b'$. Consider two functions $u_1: = \cos(ky) \e^{-k\sqrt{b}t}$ and $u_2:=\cos(ky) \e^{-k\sqrt{b'}t}$ which are solutions to the equations
$$\ddot{u}_1+\div \left[ \begin{pmatrix}a&0\\0&b\end{pmatrix}\nabla u_1 \right] = 0 \text{ \quad and \quad }\ddot{u}_2+\div \left [\begin{pmatrix}a&0\\0&b'\end{pmatrix}\nabla u_2 \right]= 0.$$
For any $t_1\geq 0$ and for any $c_1>0$, there exists $c_2>0$ such that we can transform $c_1 u_1$ into $c_2 u_2$ within the set $\T^2 \times [t_1, t_1+C]$ via a solution $u$ to $\ddot u + \div(A\nabla u)=0$ and where $A$ is in the regularity class $R(80, 10).$ The constants $c_1$ and $c_2$ are related by $$c_1 \e^{-k\sqrt{b}t_1}= c_2 \e^{-k\sqrt{b'}t_1}.$$

The time of the transformation is $C:=400.$ Moreover, for $t \in [t_1, t_1+C],$ the solution $u$ is of the form $$
        u(x,y,t)=g(t) \cos(ky) 
$$    where $g \in C^2$ satisfies  $|g^{(\alpha)}(t)| \lesssim c_1 k^{\alpha} \e^{-k\sqrt{b}t}$ for all $0 \leq \alpha \leq 2$.
\end{prop}

We prove these three Propositions respectively in sections \ref{proof33}, \ref{proof34} and \ref{proof35}. In the following, we show that Lemma \ref{blockzerocnewnew} can be reduced to these three Propositions.

\pseudosection{Heuristic idea of the reduction of Lemma \ref{blockzerocnewnew}  to Proposition \ref{lemmachangingacoeff}, Proposition \ref{slownew} and Proposition \ref{accelnew}:}

We want to transform the harmonic function $\cos(kx) \e^{-kt}$ into the faster oscillating and faster decaying harmonic function $\cos(k'y) \e^{-k't}$ for $1\ll k <k'\leq 2k$.

To do that, we start with the harmonic function $u_1:=\cos(k x) \e^{-k t}$ solution of $\ddot u + \div(A\nabla u)=0$ with $A=Id$. In the first step, we use Proposition \ref{lemmachangingacoeff} to change a coefficient of $A$ to go from $A=Id$ to 
\begin{align}\label{Aheuristicblock}
    A=\begin{pmatrix}
    1 & 0 \\
    0 & b
\end{pmatrix}
\end{align} 
for some $b\leq 1$ to be chosen. At the end of this step, we have ensured that the functions $u_1=\cos(kx)\e^{-kt}$ and $u_2:= \cos(k'y) \e^{-k'\sqrt{b}t}$ are both solutions to $\ddot u + \div(A\nabla u)=0$ where $A$ is given by \eqref{Aheuristicblock}. By choosing $b$ such that  $k'\sqrt{b}<k$, Proposition \ref{slownew} ensures that we can transform $u_1$ into $u_2$ via a solution $u$ to a divergence form equation within some time $C>0.$ At the end of this second step, we only have the function $u_2$ and the matrix $A$ given in \eqref{Aheuristicblock}. Then, by using Proposition \ref{accelnew}, we can transform $u_2$ into the harmonic function $\cos(k'y) \e^{-k't}$ via a solution $u$ to a divergence form equation within some time $C>0.$ This  will yield Lemma \ref{blockzerocnewnew}.

\begin{rem} \label{Remark:transition to slower decaying solution}
    We see that the idea is to first go to a slower decaying but faster oscillating function, and then to accelerate the decay without changing the oscillation. Hence, we do not increase the rate of decay monotonically. This approach severely uses the flexibility of changing the coefficients of the equation.
\end{rem}

We now present the reduction of Lemma \ref{blockzerocnewnew}  to Proposition \ref{lemmachangingacoeff}, Proposition \ref{slownew} and Proposition \ref{accelnew}.

\begin{proof}[Proof of the reduction of Lemma \ref{blockzerocnewnew}  to Proposition \ref{lemmachangingacoeff}, Proposition \ref{slownew} and Proposition \ref{accelnew}]

 We will construct a function $u$ and a matrix $A$ that transforms $u_1:=\cos(kx) \e^{-kt}$ into $u_2:=c\cos(k'y) \e^{-k't}$, within the set $\{(x,y,t): 0\leq t \leq C\}$ for some $C\geq 2$ universal and for some $c=c(k, k')$. Moreover, $A$ will be in the regularity class $R(80, 60).$ 

\begin{figure}[H]
	\includegraphics[scale=0.27]{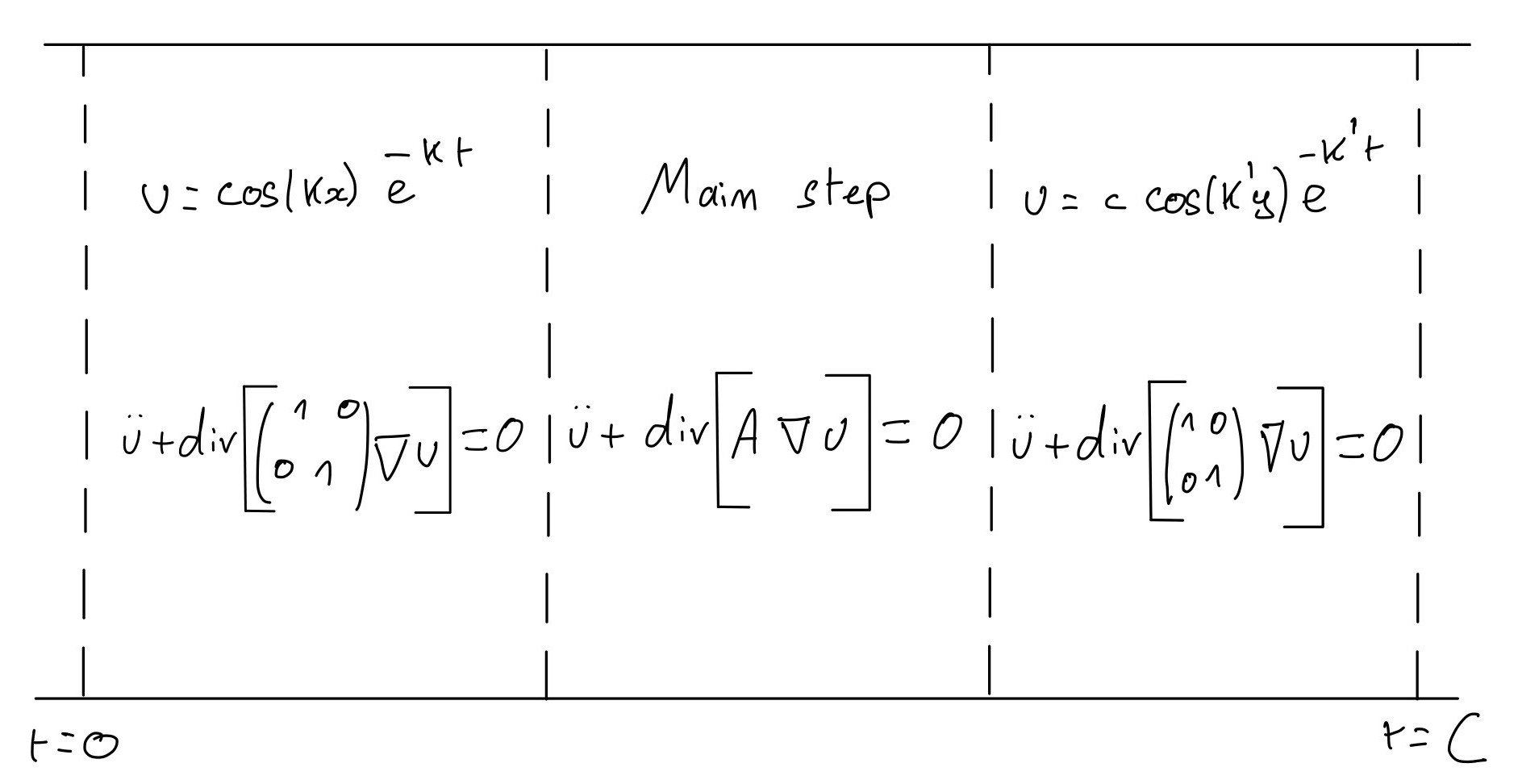}
	\centering
	\caption{Main step} 
\label{fig:main_step_kk}
\end{figure}

\pseudosection{Step 1: changing a coefficient}

We start with $A_0=Id$ in $\T^2 \times [0, \frac{1}{100}].$  By Proposition \ref{lemmachangingacoeff}, there exists a uniformly elliptic and uniformly $C^1$ matrix $\tilde A_0$ such that 
$$
\tilde A_0 = Id \hspace{0.3cm} \mbox{ for } 0\leq t \leq \frac{1}{100}, \hspace{0.5cm} \tilde A_0 = \begin{pmatrix}
    1 & 0 \\
    0 & \frac{1}{9}
\end{pmatrix} \hspace{0.3cm} \mbox{ for } t \geq \frac{1}{3}.
$$
By Proposition \ref{lemmachangingacoeff}, $u_1:=\cos(kx) \e^{-kt}$ is a solution to $\ddot u_1+ \div(\tilde A_0 \nabla u_1)=0 $ on $\T^2 \times \R^+ $ and the matrix $\tilde A_0$ belongs to the regularity class $R(10, 60).$

We define a matrix $A_1$ by 
\begin{align}\label{defofaonesteponebuildingblockc}
    A_1:= \begin{cases}
        \tilde A_0 &\mbox{ for } t \in [0, \frac{1}{3}], \\
\begin{pmatrix}
    1 & 0 \\
    0 & \frac{1}{9} 
\end{pmatrix} & \mbox{ for } t \in [\frac{1}{3}, \frac{1}{2}].
    \end{cases}
\end{align}

We note that $A_1$ has the same regularity as $\tilde A_0$: it is uniformly elliptic and uniformly $C^1$ on $[0, \frac{1}{2}]$ and it belongs to the regularity class $R(10, 60)$. Also,
\begin{align}\label{defuonefirststepnice}
u_1:=\cos(kx) \e^{-kt} \hspace{0.3cm} \mbox{ solves }  \hspace{0.3cm}  \ddot u_1+ \div( A_1 \nabla u_1)=0 \hspace{0.3cm} \mbox{ on }  \hspace{0.3cm} \T^2 \times \left[0, \frac{1}{2}\right].
\end{align}
It is also clear that for $t \in [0, \frac{1}{2}],$ $u_1$ is of the form $u_1=\cos(kx)f(t)$ for a function $f \in C^2$ satisfying
\begin{align}\label{uonebockzerocstepone}
    |f^{(\alpha)}(t)|\leq k^{\alpha} \e^{-kt}
\end{align}
for all $0\leq \alpha \leq 2.$

\begin{figure}[H]
	\includegraphics[scale=0.36]{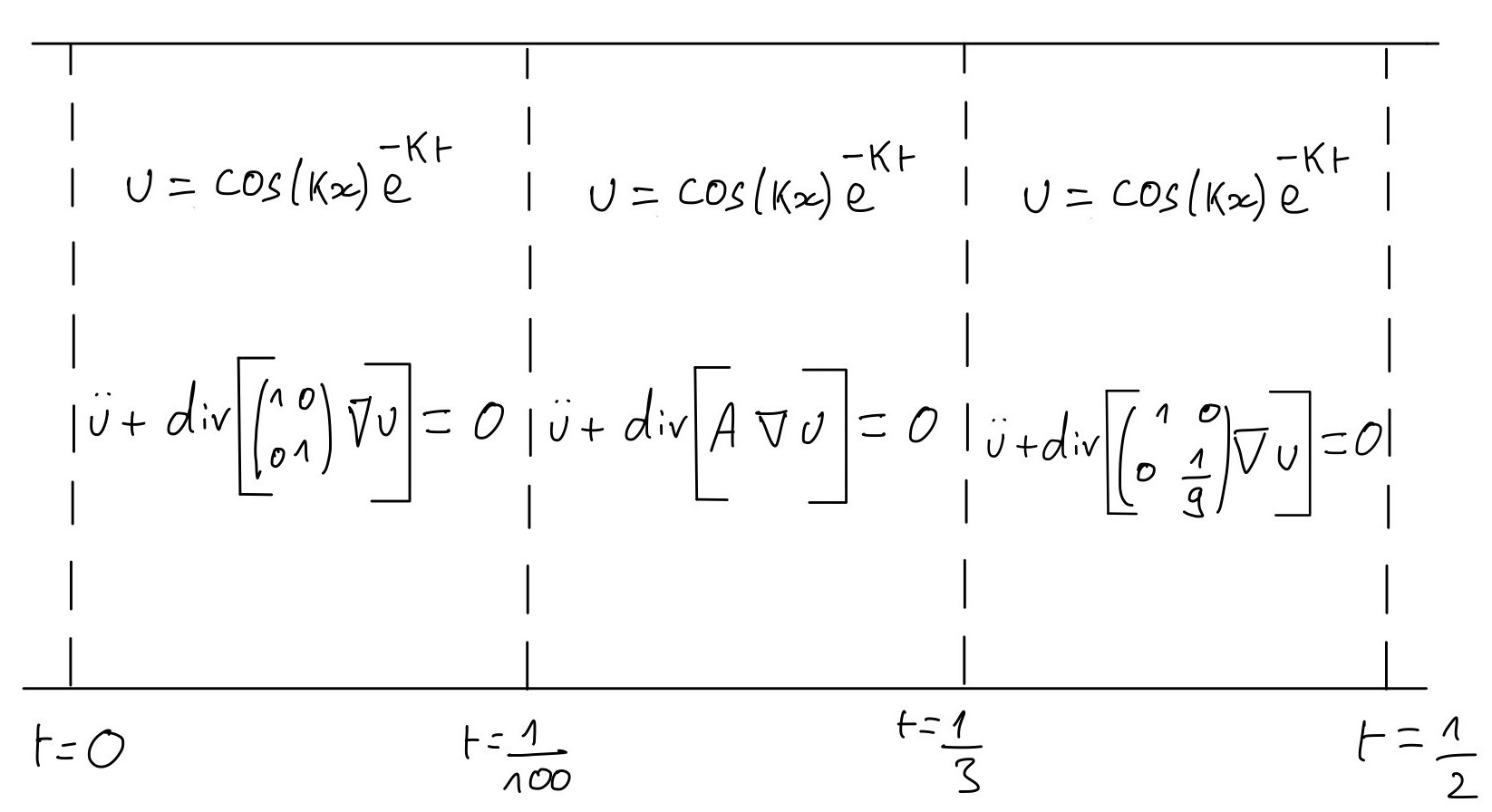}
	\centering
	\caption{Changing a coefficient} 
\label{fig:changing}
\end{figure}

\pseudosection{Step 2: transformation into a slower-decaying solution}

Define 
\begin{align}\label{tildeutwodefsteptwo}
\tilde u_2:= \cos(k'y) \e^{\frac{-k'}{3}t}    
\end{align}
 and note that it is a solution to $\ddot u+\div\left[\begin{pmatrix}
    1 & 0 \\
    0 & \frac{1}{9}
\end{pmatrix}\nabla u\right]= 0$ and that $\tilde u_2$ oscillates faster but decays slower that $u_1=\cos(kx) \e^{-kt}$ defined in the previous step. Indeed, by assumption, $1\ll k < k' \leq 2k$. Hence, $\frac{k'}{3}<k$. By the first step, for $t \in [\frac{1}{3}, \frac{1}{2}]$, the matrix $A_1$  defined in \eqref{defofaonesteponebuildingblockc} and the solution $u_1$ \eqref{defuonefirststepnice},  are given by 
\begin{align}\label{aoneuoneendlastblabla}
A_1 = \begin{pmatrix}
    1 & 0 \\
    0 & \frac{1}{9} 
\end{pmatrix}, \hspace{0.5cm} u_1=\cos(kx) \e^{-kt} .
\end{align}
Hence, $u_1$ \eqref{aoneuoneendlastblabla} and $\tilde u_2$ \eqref{tildeutwodefsteptwo} are solutions to the same equation and as we saw, $\frac{k'}{3}<k$. So, Proposition \ref{slownew} applies:  we can transform $u_1$ into the slower-decaying function $c_2\tilde u_2 = c_2 \cos(k'y) \e^{\frac{-k'}{3}t}$ within the set $\T^2 \times [\frac{1}{2}, \frac{1}{2}+C]$ via a solution $ u_2$ to $\ddot{ u}_2 + \div(A_2\nabla  u_2)=0$ where $A_2 \in R(20, 10)$. We recall that by transforming  we mean that $A_2$ is a uniformly elliptic and uniformly $C^1$ matrix on $\T^2 \times [\frac{1}{3}, \infty)$, such that $A_2=\begin{pmatrix}
    1 & 0 \\
    0 & \frac{1}{9}
\end{pmatrix}$ for $t \in [\frac{1}{3}, \frac{1}{2}] \cup [\frac{1}{2}+C, \infty)$ and such that $u_2$ is a $C^2$ functions satisfying $u_2=u_1$ for $t \in [\frac{1}{3}, \frac{1}{2}]$ and satisfying $u_2=c_2\tilde u_2$ for $t \in [\frac{1}{2}+C, \infty).$ By Proposition \ref{slownew}, the constant $c_2$ is defined by 
\begin{align}\label{defctwosteptworef}
 c_2:=\e^{-\frac{1}{2}(k-\frac{k'}{3})}.   
\end{align}
Also, by Proposition \ref{slownew}, the function $u_2$ is of the form $u_2=f(t)\cos(kx)+g(t)\cos(k'y)$ where $f, g \in C^2$ satisfy for all $0\leq \alpha \leq 2$
\begin{align}\label{utwoblockzerocsteptwo}
    |f^{(\alpha)}(t)| \lesssim  k^{\frac{7\alpha}{3}} \e^{-kt}, \hspace{0.5cm} |g^{(\alpha)}(t)|\lesssim c_2\, (k')^{\alpha} \e^{\frac{-k'}{3}t} \lesssim (k')^{\alpha} \e^{\frac{-k'}{3}t}
\end{align}
since $c_2=\e^{-\frac{1}{2}(k-\frac{k'}{3})}$ and $\frac{k'}{3}<k$ (since by assumption $k'\leq 2k$).  Moreover, the duration $C$ is given by
$$
C=\frac{1}{k^{4/3}} + \frac{8 \ln k}{k-k'/3} + \frac{\sqrt{4\ln k}}{k'^{1/3}}+\frac{1}{100} < \frac{1}{2}
$$
 since by assumption $1\ll k<k' \leq 2k.$ For $t \in [0, 1],$ we define a new matrix $\tilde A$ and a new function $\tilde u$ by
\begin{align}\label{newmatrixatildeforgluingatonethird}
    \tilde A :=
    \begin{cases}
        A_1 & \mbox{ for } t \in [0, \frac{1}{2}], \\
        A_2 & \mbox{ for } t \in [\frac{1}{2}, \frac{1}{2}+C], \\
        \begin{pmatrix}
        1 & 0 \\
        0 & \frac{1}{9}
    \end{pmatrix} & \mbox{ for } t \in [\frac{1}{2}+C, 1],
    \end{cases} \hspace{0.5cm} \tilde u:= \begin{cases}
        u_1 & \mbox{ for } t \in [0, \frac{1}{2}],\\
        u_2& \mbox{ for } t \in [\frac{1}{2}, \frac{1}{2}+C], \\
        c_2 \tilde u_2 & \mbox{ for } t \in [\frac{1}{2}+C, 1].
    \end{cases}
\end{align}

The matrix $A_1$ was defined in the first step \eqref{defofaonesteponebuildingblockc} and the matrix  $A_2$ is the transformation matrix that we used in this second step. We note that $\tilde A$ is a uniformly elliptic and uniformly $C^1$ (since $A_1$ is by the first step and since $A_2$ is by Proposition \ref{slownew}). In particular, $\tilde A$ belongs to the regularity class $R(20, 60)$ (since $A_1 \in R(10, 60)$ by the first step and $A_2 \in R(20, 10)$ by Proposition \ref{slownew}).

\comment{
We also define a new function $\tilde u$ for $t \in [0, \frac{1}{2}+C]$, by
\begin{align}\label{newfunctiondefforgluingatonethird}
    \tilde u:= \begin{cases}
        u_1 & \mbox{ for } t \in [0, \frac{1}{2}],\\
        u_2& \mbox{ for } t \in [\frac{1}{2}, \frac{1}{2}+C], \\
        c_2 \tilde u_2 & \mbox{ for } t \in [\frac{1}{2}+C, 1].
    \end{cases}
\end{align}}

We also note that $\tilde u$ is a $C^2$ function, since $u_1$ and $\tilde u_2$ are and since $u_2$ is the transformation function that we used in this second step. Moreover, the equation $\ddot{\tilde u} + \div(\tilde A \nabla \tilde u)=0$ is satisfied on $\T^2 \times [0, 1].$ Finally, for $t\in [0, 1],$ $\tilde u$ is of the form $\tilde u = f(t) \cos(kx) + g(t) \cos(k'y)$ where $f, g \in C^2$ satisfy for all $0\leq \alpha \leq 2,$
\begin{align}\label{utilderegsteptwoconclusion}
    |f^{(\alpha)}(t)| \lesssim k^{\frac{7\alpha}{3}} \e^{-kt}, \hspace{0.5cm} |g^{(\alpha)}(t)| \lesssim (k')^{\alpha} \e^{\frac{-k'}{3}t}.
\end{align}
Indeed: 
\begin{itemize}
    \item By \eqref{uonebockzerocstepone} in the first step, $u_1=\cos(kx)f(t)$ with $f\in C^2$ satisfying  $|f^{(\alpha)}(t)| \lesssim k^{\alpha} \e^{-kt}$.
    \item  By \eqref{utwoblockzerocsteptwo} in this second step, $u_2=f(t)\cos(kx)+g(t)\cos(k'y)$ where $f, g \in C^2$ satisfy  $|f^{(\alpha)}(t)| \lesssim k^{\frac{7\alpha}{3}} \e^{-kt}$ and $ |g^{(\alpha)}(t)| \lesssim (k')^{\alpha} \e^{\frac{-k'}{3}t}.$
    \item  And finally, by the definition \eqref{tildeutwodefsteptwo}, $c_2\tilde u_2 = c_2 \e^{\frac{-k'}{3}t}\cos(k'y)$ and $c_2\leq 1$ (see \eqref{defctwosteptworef}).
\end{itemize}

\begin{figure}[H]
	\includegraphics[scale=0.4]{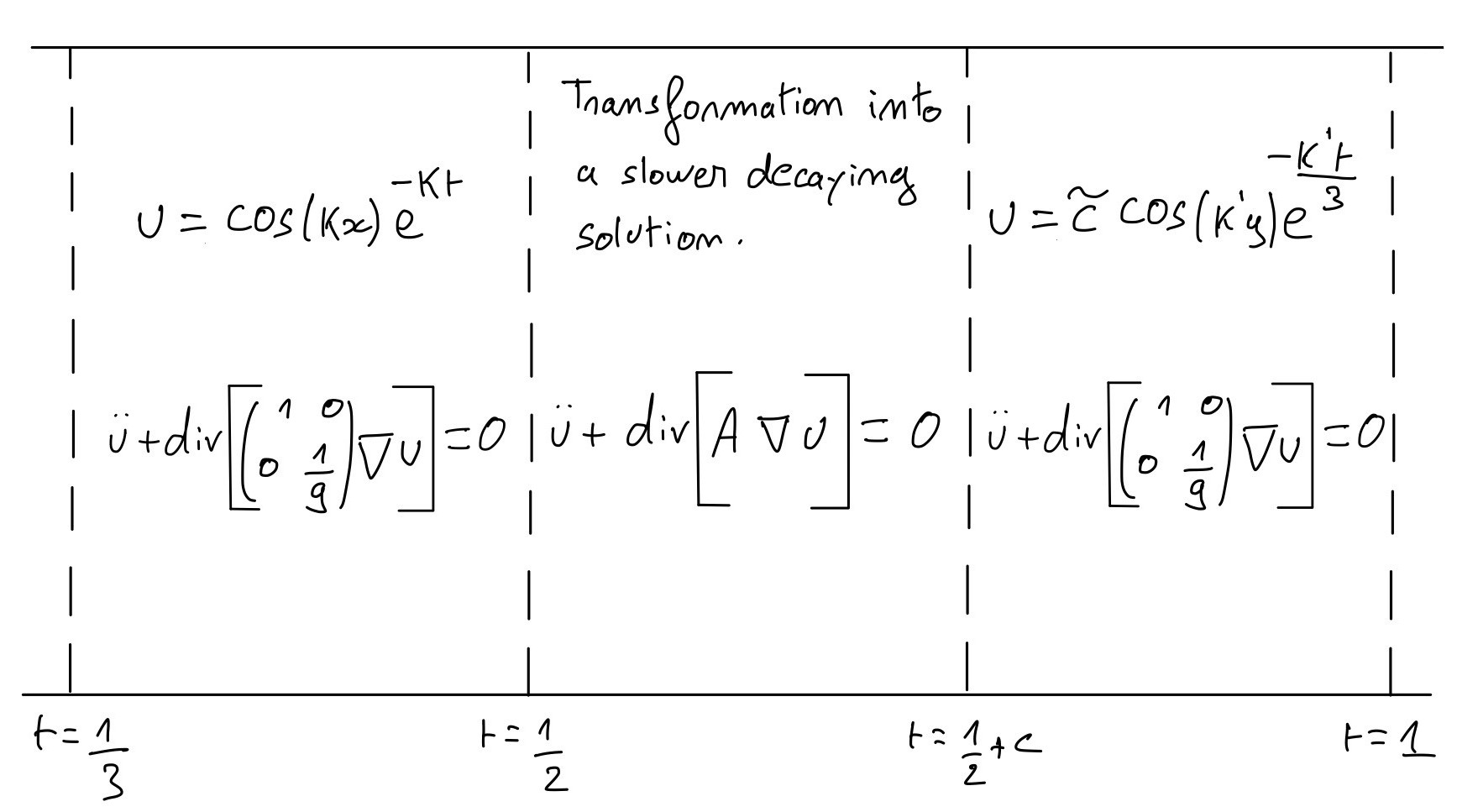}
	\centering
	\caption{Transformation into a slower decaying but faster oscillating solution} 
\label{fig:slower}
\end{figure}

\pseudosection{Step 3: the acceleration}

By the definition of $\tilde A$ and $\tilde u$ \eqref{newmatrixatildeforgluingatonethird}, for $t \in [\frac{1}{2}+C, 1]$,
\begin{align}\label{gluingatabovesteptwoblabladfblabla}
    \tilde u = c_2 \cos(k'y) \e^{\frac{-k'}{3}t}, \hspace{0.5cm} \tilde A = \begin{pmatrix}
    1 & 0 \\
    0 & \frac{1}{9}
\end{pmatrix}
\end{align}

By Lemma \ref{accelnew}, we can accelerate the decay of $\tilde u$ and transform $\tilde u$ into the faster decaying harmonic function $\tilde u_3:= c_3\cos(k'y) \e^{-k't}$ within the set $\T^2 \times [1, 1+C_1],$ via a solution $u_3$ of $\ddot u_3 + \div(A_3 \nabla u_3)=0$ where $A_3 \in R(80, 10).$ In particular, it means that $A_3$ is a uniformly $C^1$ and uniformly elliptic matrix satisfying $A_3=\tilde A$ for $t \in [\frac{1}{2}+C, 1]$ and $A_3= Id$ for $t \geq 1+C_1$. It also means that $u_3$ is a $C^2$ function such that $u_3=\tilde u$ for $t \in [\frac{1}{2}+C,1]$ and $u_3=c_3\cos(k'y) \e^{-k't}$ for $t \geq 1+C_1.$ By Proposition \ref{accelnew}, the constant $c_3$ is related to the constant $c_2$ by 
\begin{align}\label{defconstantcthreeusedinbuildingblock}
c_2 \e^{-\frac{k'}{3}} = c_3 \e^{-k'} .   
\end{align}
Moreover, by Proposition \ref{accelnew}, the duration $C_1$ is given by $C_1=400$, which is independent of $k, k'.$ Moreover, still by Proposition \ref{accelnew}, the solution $u_3$ is of the form $u_3=g(t)\cos(k'y)$ with $g\in C^2$ satisfying 
\begin{align}\label{uthreeinstepthreestimateg}
  |g^{(\alpha)}(t)|\lesssim c_2 (k')^{\alpha} \e^{\frac{-k'}{3}t} \lesssim (k')^{\alpha} \e^{\frac{-k'}{3}t}   
\end{align}
since we saw in step 2 that   $c_2\leq 1$ (see \eqref{defctwosteptworef}).

 \comment{
\begin{align}\label{defconstantconedurationstepthree}
C_1=400    
\end{align}
which is independent of $k, k'.$}

We finally define a new matrix $A$ and a new function $u$ for $t \in [0, 2+C_1]$ by 
\begin{align}\label{newmatrixatildeforgluingatonethirdaccel}
     A :=
    \begin{cases}
        \tilde A & \mbox{ for } t \in [0, 1], \\
        A_3 & \mbox{ for } t \in [1, 1+C_1]\\
        Id & \mbox{ for } t \in [1+C_1,2+C_1],
    \end{cases} \hspace{0.5cm} u:= \begin{cases}
        \tilde u & \mbox{ for } t \in [0, 1],\\
        u_3& \mbox{ for } t \in [1, 1+C_1],\\
        \tilde u_3 & \mbox{ for } t \in [1+C_1, 2+C_1].
    \end{cases}
\end{align}
The matrix  $\tilde A$ was defined in the second step \eqref{newmatrixatildeforgluingatonethird} while $A_3$ is the transformation matrix used in this third step. Since $\tilde A \in R(20, 60)$ and since $A_3 \in R(80, 10)$, it is clear that $A \in R(80, 60).$ The function $\tilde u$ was defined in the second step \eqref{newmatrixatildeforgluingatonethird}, the function  $u_3$ is the transformation function used in this third step and the function $\tilde u_3$ was defined above by $\tilde u_3=c_3\cos(k'y) \e^{-k't}$.

\comment{
We also define a new function $ u$ for $t \in [0, 2+C_1]$, by
\begin{align}\label{newfunctiondefforgluingatonethirdaccel}
     u:= \begin{cases}
        \tilde u & \mbox{ for } t \in [0, 1],\\
        u_3& \mbox{ for } t \in [1, 1+C_1]\\
        \tilde u_3 & \mbox{ for } t \in [1+C_1, 2+C_1]
    \end{cases}
\end{align}
}

We note that the solution $u$ defined in \eqref{newmatrixatildeforgluingatonethirdaccel} is of the form $u=f(t)\cos(kx) + g(t) \cos(k'y)$ where $f, g \in C^2$ satisfy for $t \in [0, 2+C_1]$ and for all $0\leq \alpha \leq 2,$
\begin{align}\label{finalestimatesforufandgblockzeroc}
|f^{(\alpha)}(t)| \lesssim k^{\frac{7\alpha}{3}} \e^{-kt}, \hspace{0.5cm} |g^{(\alpha)}(t)| \lesssim (k')^{\alpha} c_3 \e^{\frac{-k'}{3}t}, \hspace{0.5cm} c_3=\e^{\frac{-k}{2} + \frac{5k'}{6}}
\end{align}
as claimed in Lemma \ref{blockzerocnewnew}. Indeed, for $t\in [0, 1],$ $u=\tilde u=f(t)\cos(kx) + g(t) \cos(k'y)$ with $f, g \in C^2$ satisfying  
$$
|f^{(\alpha)}(t)| \lesssim k^{\frac{7\alpha}{3}} \e^{-kt}, \hspace{0.5cm} |g^{(\alpha)}(t)| \lesssim (k')^{\alpha} \e^{\frac{-k'}{3}t}
$$
(see \eqref{utilderegsteptwoconclusion}). We also saw in step 3 that for $t \in [1, 1+C_1],$ $u=u_3=g(t)\cos(k'y)$ with $g\in C^2$ satisfying
$$  |g^{(\alpha)}(t)| \lesssim (k')^{\alpha} \e^{\frac{-k'}{3}t}   $$ (see \eqref{uthreeinstepthreestimateg}). And finally, for $t \in [1+C_1, 2+C_1],$ $u=\tilde u_3 = c_3\cos(k'y)\e^{-k't}$ where $c_3=c_2 \e^{\frac{2}{3}k'} = \e^{\frac{-k}{2} + \frac{5k'}{6}}\geq 1$. The first equality comes from \eqref{defconstantcthreeusedinbuildingblock} and the second equality comes from $c_2= \e^{-\frac{1}{2}(k-\frac{k'}{3})}$ by  \eqref{defctwosteptworef}. \comment{ In particular,
\begin{align}\label{valueconstantcthreeend}
    \e^{\frac{k}{3}}\leq c_3\leq \e^{\frac{7}{12}k'}.
\end{align}
 Indeed,  $1\ll k \leq k' \leq 2k$ by assumption, which implies $\frac{k}{3}\leq \frac{-k}{2} + \frac{5k'}{6} \leq \frac{7}{12}k',$ which yields \eqref{valueconstantcthreeend} and \eqref{finalestimatesforufandgblockzeroc} (since $c_3\geq 1)$ .}

\begin{figure}[H]
	\includegraphics[scale=0.38]{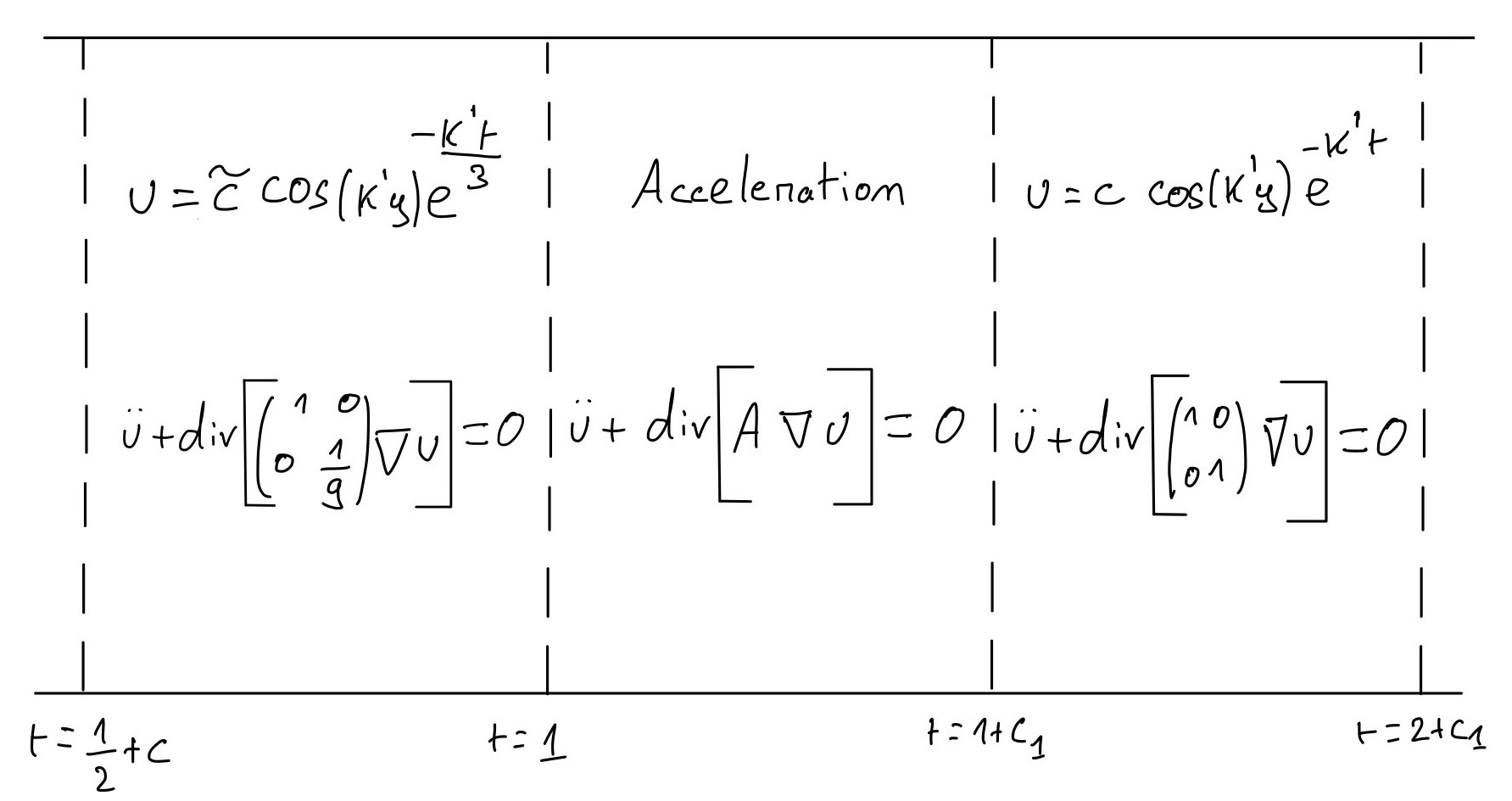}
	\centering
	\caption{Acceleration} 
\label{fig:accel}
\end{figure}

\textbf{Conclusion}
\begin{itemize}
    \item  We define the total duration $C_2:=2+C_1 \geq 2$ which is independent of $k, k'$ (as claimed in Lemma \ref{blockzerocnewnew}). We have been able to transform $\cos(kx)\e^{-kt}$ into $c_3\cos(k'y)\e^{-k't}$ via a solution $u$ to $\ddot u + \div(A\nabla u)=0$. We saw that $A$ is uniformly elliptic and uniformly $C^1$: $A$ belongs to the regularity class $R(80, 60)$.
    \item  By \eqref{finalestimatesforufandgblockzeroc}, for $t \in [0, C_2]$, the solution $u$ is of the form $f(t)\cos(kx) + g(t)\cos(k'y)$ with $f, g \in C^2$ satisfying  for all $0\leq \alpha \leq 2,$
$$
|f^{(\alpha)}(t)| \lesssim k^{\frac{7\alpha}{3}} \e^{-kt}, \hspace{0.5cm} |g^{(\alpha)}(t)| \lesssim (k')^{\alpha}c_3 \e^{\frac{-k'}{3}t}, \hspace{0.5cm} c_3=\e^{\frac{-k}{2} + \frac{5k'}{6}}.
$$
    \item Moreover, by the first step (see \eqref{defofaonesteponebuildingblockc} and \eqref{defuonefirststepnice}), the solution $u$ and the matrix $A$ satisfy for $t\in [0, \frac{1}{100}],$
$$
         u(x,y,t)=\cos(kx) \e^{-kt}, \hspace{0.5cm} A=Id.
$$

\item And by the third step (see \eqref{newmatrixatildeforgluingatonethirdaccel}),
       for $t \in [C_2-\frac{1}{100}, C_2]$, the solution $u$ and the matrix $A$ satisfy 
      
$$
u(x,y,t)= c_3 \cos(k'y) \e^{-k't}, \hspace{0.5cm} A=Id
$$
where $c_3=\e^{\frac{-k}{2} + \frac{5k'}{6}}$ (see \eqref{finalestimatesforufandgblockzeroc}).

\end{itemize}

This finishes the proof of Lemma \ref{blockzerocnewnew}.
\end{proof}

To finish the proof of Theorem \ref{EigenTheorem}, it only remains to prove the three technical Propositions. This is done in the next sections \ref{proof33}, \ref{proof34} and \ref{proof35}.

\section{The three technical Propositions }\label{proofofpropslabel}
\subsection{Proof of Proposition \ref{lemmachangingacoeff}}
\label{proof33}
For the reader's convenience, we recall the Proposition we want to prove:

\begin{prop*}[\ref{lemmachangingacoeff}, Changing a coefficient]
    Let $k \geq 1$, let $a,b \in (\frac{1}{10}, 10)$, let $t_1\geq 0$ and let also $C>0$ be a universal constant. Consider  the equations
    $$
    \ddot u + \div\left[\begin{pmatrix}
        a & 0 \\
        0 & a
    \end{pmatrix} \nabla u\right]=0, \hspace{0.5cm} \mbox{ and } \hspace{0.5cm} \ddot u + \div\left[\begin{pmatrix}
        a & 0 \\
        0 & b
    \end{pmatrix} \nabla u\right]=0.
    $$
    There exists a uniformly $C^1$ and uniformly elliptic matrix $A$ on $\T^2 \times \R,$ such that  $$A=\begin{pmatrix}
        a & 0 \\
        0 & a
    \end{pmatrix} \hspace{0.3cm} \mbox{ for } t\leq t_1 \hspace{0.5cm} \mbox{ and } \hspace{0.3cm}  A= \begin{pmatrix}
        a & 0 \\
        0 & b
    \end{pmatrix} \hspace{0.3cm} \mbox{ for } t\geq t_1+C$$
     and such that $A$ belongs to the regularity class $R(10, \frac{10\sqrt{\pi}}{C}).$ The function $u:=\cos(kx) \e^{-k\sqrt{a}t}$ is a solution to $\ddot u + \div(A\nabla u)=0$ in $\T^2 \times \R$. 
\end{prop*}

\begin{proof}[Proof of Proposition \ref{lemmachangingacoeff}]
    Denote  $A := \begin{pmatrix}
        a & 0 \\
        0 & a
    \end{pmatrix}$ and  $B := \begin{pmatrix}
        a & 0 \\
        0 & b
    \end{pmatrix}$. Consider the matrix $A_0$ of the form 
    \begin{align}\label{defazerocoeffc}
    A_0 := \begin{pmatrix}
        a & 0 \\
        0 & c(t)
    \end{pmatrix} , \hspace{0.5cm}   c(t):=(a-b) \theta\left( \frac{t}{C} - \frac{t_1}{C}\right) + b
    \end{align}
    
     where 
     \comment{\begin{align}\label{defcoftforlemmachangingcoeff}
         c(t):=(a-b) \theta\left( \frac{t}{C} - \frac{t_1}{C}\right) + b
     \end{align}
    and }  $\theta(\tau)$ is presented in Picture \ref{fig:theta} (see also footnote \footnote{We choose $\theta(t):= 1-G(\tan(\pi(t-1/2)))$ where $G(x)=\frac{1}{\sqrt{\pi}}\int_{-\infty}^x \e^{-\eta^2} \ud{\eta}$. \label{footnotebuildingblockcoeffc}}). In particular, the function $\theta$ is smooth and has the following properties:
     
\begin{align}\label{prothetabuildingblocklemmac}
  \lim_{\tau \rightarrow 0}  \theta(\tau) = 1, \hspace{0.5cm} \lim_{\tau \rightarrow 1}  \theta(\tau) = 0
\end{align}
and 
\begin{align}\label{prothetabuildingblockendlemac}
 \lim_{\tau \rightarrow 0}  \dot \theta(\tau) = 0 , \hspace{0.5cm} \lim_{\tau \rightarrow 1}  \dot \theta(\tau) = 0, \hspace{0.5cm} |\dot \theta(\tau)| \leq \sqrt{\pi} \mbox{ for } \tau \in [0, 1].
\end{align}      
     
     Define the matrix $A_1$:
\begin{align}\label{defofaonesteponebuildingblock}
    A_1:= \begin{cases}
        A & \mbox{ for } t \leq t_1, \\ 
        A_0 &\mbox{ for } t \in [t_1, t_1+C], \\
        B  & \mbox{ for } t \geq t_1+C.
    \end{cases}
\end{align}

\textbf{The regularity of $A_1$}
\begin{itemize}
    \item \underline{The continuity of $A_1$.} By \eqref{prothetabuildingblocklemmac}, $c(t) =(a-b) \theta\left( \frac{t}{C} - \frac{t_1}{C}\right) + b $ is going from $a$ to $b$ as $t$ moves in $\left[t_1, t_1+C\right].$ Hence, the matrix $A_0 = \begin{pmatrix}
    a & 0 \\
    0 & c(t)
\end{pmatrix} $ defined in \eqref{defazerocoeffc} converges to $A$  (respectively $B$) as $t$ goes to $t_1$ (respectively $t_1+C$). Moreover, since $\theta$ is continuous, the matrix $A_1$ is continuous on $\T^2 \times \R$.

\item \underline{The $C^1$ regularity of $A_1$.} By \eqref{prothetabuildingblockendlemac}, $\dot c(t)$ is going to 0 as $t$ goes to $t_1$ and $t_1+C.$ Since $A$ and $B$ have constant coefficients, the matrix $A_1$ is $C^1$ at $t=t_1$ and $t=t_1+C.$ Moreover, since $\theta$ is $C^1$, the matrix $A_1$ is $C^1$ on $\T^2\times \R.$ More precisely, the following estimate holds: for $t \in [t_1, t_1+C],$
$$
|\dot c(t)| \leq \frac{1}{C}|a-b| \left|\dot \theta\left(\frac{t}{C}-\frac{t_1}{C}\right)\right| \leq \frac{\sqrt{\pi}}{C} |a-b| \leq \frac{10\sqrt{\pi}}{C}
$$ since $|\dot \theta(\tau)|\leq \sqrt{\pi}$ for $\tau \in [0, 1]$ by \eqref{prothetabuildingblockendlemac} and since by assumption $a,b \in (\frac{1}{10}, 10)$.

\item \underline{The ellipticity of $A_1$.} Since $c(t)=(a-b)\theta\left( \frac{t}{C} - \frac{t_1}{C}\right) +b$, and since $0\leq \theta \leq 1,$ we have for $t \in [t_1, t_1+C]$, $ 10 \geq \max(a,b)\geq c(t) \geq \min(a,b)\geq \frac{1}{10}$ since we assumed $a, b \in (\frac{1}{10}, 10)$.  Since $A=\begin{pmatrix}
    a & 0\\
    0 & a
\end{pmatrix}$ and $B=\begin{pmatrix}
    a & 0 \\
    0 & b
\end{pmatrix}$, it is clear that $ A_1$ defined in \eqref{defofaonesteponebuildingblock} satisfies 
$$\frac{1}{10} |\xi|^2\leq (A_1\xi, \xi) \leq 10 |\xi|^2$$ for all $\xi \in \R^2.$

    \item In conclusion, we proved that the matrix $A_1$ is $C^1$ on $\T^2 \times \R$ and belongs to the regularity class $R(10, \frac{10\sqrt{\pi}}{C}).$
\end{itemize}

\textbf{The solution of the equation.}
Finally, consider $u:=\cos(kx)\e^{-k\sqrt{a}t}$. It depends on $x,t$ only. Since $A_1$ is a diagonal matrix with a constant coefficient equal to $a$ in position $(x,x)$ \big(we only changed the coefficient in position $(y,y)$\big), it is clear that $\ddot u+ \div(A\nabla u)=0$ on $\T^2 \times \R.$ This finishes the proof of Proposition \ref{lemmachangingacoeff}.
\end{proof}

\subsection{Proof of Proposition \ref{slownew}}
\label{proof34}
We recall the Proposition that we want to prove for the reader's convenience:
\begin{prop*}[\ref{slownew}, slowing down]
Let $1 \ll k < k' \leq 2k$ and let $a, b \in (\frac{1}{10}, 10).$ 
Consider two solutions $\cos(kx) \e^{-k\sqrt{a}t}$ and $\cos(k'y) \e^{-k'\sqrt{b}t}$ to the same equation 
$$
\ddot u + \div \left [\begin{pmatrix}a&0\\0&b\end{pmatrix} \nabla u \right ] = 0.
$$
If 
$$
k' \sqrt{b} < k \sqrt{a},
$$
then for any $t_1\geq 0$ and for any $c_1>0$, there exists $c_2>0$ such that one can transform $c_1 \cos(kx) \e^{-k\sqrt{a}t}$ into $c_2\cos(k'y) \e^{-k'\sqrt{b}t}$ within the set $\T^2 \times [t_1, t_1+C],$ via a solution $u$ to $\ddot u + \div(A\nabla u)=0,$  where $A$ is in the regularity class $R(20,10)$. The constants $c_1$ and $c_2$ are related by 
\begin{align}\label{defconstantslowdownconectwonew}
    c_1 e^{-k\sqrt{a}t_1} = c_2 e^{-k'\sqrt{b}t_1}.
\end{align}

The duration $C=C(k, k')$ of the transformation is 
\begin{align}\label{durationslowdowntoneprop}
C=\frac{1}{k^{4/3}} + \frac{8 \ln k}{k\sqrt{a}-k'\sqrt{b}} + \frac{\sqrt{4\ln k}}{(k')^{1/3}}+\frac{1}{100}.    
\end{align}

Moreover, for $t \in [t_1, t_1+C],$ the function $u$ is of the form 
\begin{align}\label{shapesolslowdownproptone}
u(x,y,t)= f(t) \cos(kx)  +  g(t) \cos(k'y)    
\end{align}
 where $f, g \in C^2$ satisfy $|f^{(\alpha)}(t)| \lesssim c_1 k^{\frac{7\alpha}{3}} \e^{-k\sqrt{a}t}$ and $|g^{(\alpha)}(t)| \lesssim c_2 (k')^{\alpha} \e^{-k'\sqrt{b}t}$ for all $0 \leq \alpha \leq 2.$
\end{prop*}

\comment{
First of all, we notice that one can easily reduce the Proposition \ref{slownew} to a special case: \begin{prop}
   \label{slownewzeroc}
   Proposition \ref{slownew} is true when $t_1=0, c_1=c_2=1.$
\end{prop}
The proof of this reduction is similar to the proof of the reduction in section \ref{reductioninablock} and we do not include it.
}

\subsubsection{A first reduction}
\label{reduce34}
We note that Proposition \ref{slownew} can be reduced to a special case: \begin{prop}
   \label{slownewzeroc}
   Proposition \ref{slownew} is true when $t_1=0, c_1=c_2=1.$
\end{prop}

\begin{proof}[Reduction of Proposition \ref{slownew} to Proposition \ref{slownewzeroc}]
Consider the function $u$, the matrix $A$ and the duration $C$ from Proposition \ref{slownewzeroc}. Then, $u$ solves
     $\ddot u + \div(A\nabla u)=0$, $u$ transforms $\cos(kx) \e^{-k\sqrt{a}t}$ into $\cos(k'y)\e^{-k'\sqrt{b}t}$ within the set $\T^2\times [0,C]$ and $A$ is in the regularity class $R(20, 10).$

    Let $t_1 \geq 0$. Denote \begin{align}\label{shiftforslowingdown}
        \tilde u(t):= u(t-t_1), \hspace{0.5cm} \tilde A(t) := A(t-t_1).
    \end{align}

Then, $\tilde u$ transforms $\cos(kx) \e^{-k\sqrt{a}(t-t_1)}$ into $\cos(k'y) \e^{-k'\sqrt{b}(t-t_1)}$ within the set $\mathbb{T}^2\times[t_1, t_1+C].$ It is a solution of $\ddot{\tilde u} + \div(\tilde A \nabla \tilde u) =0$ and $\tilde A$ is in the regularity class $R(20,10).$ Define $c_2$ by 
\begin{align}\label{defconectwoslowingdownproof}
    c_1 \e^{-k\sqrt{a}t_1} = c_2 \e^{-k'\sqrt{b}t_1}.
\end{align}
Denote 
\begin{align}\label{defbiguslowingdownproof}
    U(x,y,t):= c_1\e^{-k\sqrt{a}t_1} \tilde u(x,y,t).
\end{align}

Then, using \eqref{defconectwoslowingdownproof}, we see that $U$ transforms $c_1 \cos(kx) \e^{-k\sqrt{a}t}$ into $c_2\cos(k'y) \e^{-k'\sqrt{b}t}$ within the set $\mathbb{T}^2\times[t_1, t_1+C]$ and $\tilde A \in R(20, 10).$ Moreover, $U$ is a solution to 
\begin{align*}
    \ddot U + \div(\tilde A \nabla U)=0
\end{align*}
for $t \in [t_1, t_1+C]$. 

\textbf{Conclusion:} The matrix $\tilde A$ and the solution $U$ are the matrix and solution in Proposition \ref{slownew}. The duration $C(k, k')$ in Proposition \ref{slownew} is the same as the duration in Proposition \ref{slownewzeroc} since the shift does not change the duration of the transformation. By Proposition \ref{slownewzeroc}, for $t \in [0, C]$,
$$
u(x,y,t) = f(t) \cos(kx) + g(t) \cos(k'y) 
$$
where $f, g \in C^2$ satisfy $|f^{(\alpha)}(t)|\lesssim k^{\frac{7\alpha}{3}} \e^{-k\sqrt{a}t}$ and $|g^{(\alpha)}(t)| \lesssim (k')^{\alpha} \e^{-k'\sqrt{b}t}$ for all $0 \leq \alpha \leq 2.$ Hence, since $U(x,y,t) = c_1\e^{-k\sqrt{a}t_1}  u(x,y,t-t_1)$ by \eqref{defbiguslowingdownproof} and \eqref{shiftforslowingdown}, we get that $U$ is of the form 
$$
U(x,y,t) = c_1\e^{-k\sqrt{a}t_1} f(t-t_1)\cos(kx) + c_2 \e^{-k'\sqrt{b}t_1} g(t-t_1)\cos(k'y)
$$
where we used $c_1 \e^{-k\sqrt{a}t_1} = c_2 \e^{-k'\sqrt{b}t_1}$ by \eqref{defconectwoslowingdownproof}. Hence, $U$ is of the form $$U(x,y,t) = F(t) \cos(kx) + G(t) \cos(k'y)$$ where $F, G \in C^2$ satisfy $|F^{(\alpha)}(t)| \lesssim c_1 k^{\frac{7\alpha}{3}} \e^{-k\sqrt{a}t}$ and $|G^{(\alpha)}(t)| \lesssim c_2 (k')^{\alpha} \e^{-k'\sqrt{b}t}$ for all $0 \leq \alpha \leq 2.$ This finishes the proof of Proposition \ref{slownew} assuming Proposition \ref{slownewzeroc}.    
\end{proof}

\subsubsection{A second reduction}
In this section, we will reduce Proposition \ref{slownewzeroc} to the following Proposition:
\begin{prop}[adding a small perturbation]\label{perturbationnew}
\comment{Let $\theta(t)$ be the smooth function equal to 1 for $t \leq 0$, to 0 for $t \geq 1$ and that monotonically decreases for $t \in (0, 1)$ (see \footnote{We choose $\theta(t):= 1-G(\tan(\pi(t-1/2)))$ where $G(x)=\frac{1}{\sqrt{\pi}}\int_{-\infty}^x \e^{-\eta^2} \ud{\eta}$. \label{footnoteblablanewblablasecondred}}).}
     Consider  the equation $\ddot{u}+\div( A\nabla u) = 0,$ where $A := \begin{pmatrix}a&0\\0&b\end{pmatrix}$ is a matrix with constant coefficients $a,b$ in $(\frac{1}{10}, 10)$. Define two solutions $$u_1(x,t) := \cos(kx)e^{-k\sqrt{a} t} \textup{ and } u_2(y,t) := \cos(k'y)e^{-k'\sqrt{b}t}$$
    where  $ 1 \leq k \leq k' \leq 2k$   .    

\begin{itemize}
\item There exists a universal constant $D>0$ such that for any $0<\epsilon<D$, if $\epsilon^{-\frac{1}{4}} \leq k, k'\leq \epsilon^{-\frac{1}{3}}$, then we can transform $u_1$ into $u_1+\epsilon u_2$ via a solution $u$ to $\ddot u + \div(\tilde A\nabla u)=0$ in $\epsilon^{\frac{1}{3}}$ time
(within the set $\T^2 \times [0, \epsilon^{1/3}]$) with $\tilde A$ in the regularity class $R(20,10)$. 

\item For $t \in [0, \epsilon^{1/3}],$ the function $u$ is of the form 
\begin{align}\label{perturbationgneralfomrstatementone}
    u(x,y,t)= f(t) \cos(kx)  +  g(t) \cos(k'y)
\end{align}
where $f, g \in C^2$ satisfy $|f^{(\alpha)}(t)| \lesssim k^{\alpha} \e^{-k\sqrt{a}t}$ and $|g^{(\alpha)}(t)|  \lesssim (k')^{\alpha}\epsilon^{\frac{3-\alpha}{3}} \e^{-k'\sqrt{b}t}$ for any $0 \leq \alpha\leq 2.$
\item Similarly, we can go from $\epsilon u_1+u_2$ to $u_2$ in the same amount of time (within the same set), under the same conditions on $\epsilon, k, k'$ and with the same regularity of the transformation $u$. The solution $u$ will be of the same form $u(x,y,t)= f(t) \cos(kx)  +  g(t) \cos(k'y)$ (but $f, g$ will be different from $f, g$ in \eqref{perturbationgneralfomrstatementone}) and the following will hold: $f, g \in C^2$ and  $|f^{(\alpha)}(t)| \lesssim k^{\alpha} \epsilon^{\frac{3-\alpha}{3}} \e^{-k\sqrt{a}t}$ and $|g^{(\alpha)}(t)|  \lesssim (k')^{\alpha} \e^{-k'\sqrt{b}t}$ for any $0 \leq \alpha\leq 2.$
\end{itemize}
\end{prop}

\pseudosection{Heuristic proof of the reduction of Proposition \ref{slownewzeroc}  to Proposition \ref{perturbationnew}}
The reduction is done in four steps. In the first step, we start with the harmonic function $u_1:=\cos(kx) \e^{-k\sqrt{a}t}$ and using Proposition \ref{perturbationnew}, we will transform $u_1$ into $u_1 + \epsilon u_2$ where $\epsilon \ll 1$ and where $u_2:=\cos(k'y) \e^{-k'\sqrt{b}t}$ is a faster oscillating but slower decaying harmonic function (recall that by assumption in Proposition \ref{slownewzeroc}, $k'\sqrt{b}<k\sqrt{a}$). In the second step, we wait. We wait until $u_1$ becomes smaller than $\epsilon u_2$. This is possible since $u_1$ has a faster decay than $\epsilon u_2$. Once $u_1$ becomes smaller than $\epsilon u_2$, we will again be in the small perturbative regime and we can again apply Proposition \ref{perturbationnew} to remove $u_1$: we are now left with $\epsilon u_2.$ The last step now consists in transforming $\epsilon u_2$ into $u_2$, which will then finish the proof of the reduction.

\begin{proof}[Proof of the reduction of Proposition \ref{slownewzeroc} to Proposition \ref{perturbationnew}]

By assumption, $1 \ll k <k' \leq 2k,$ $a,b \in (\frac{1}{10}, 10),$ $k'\sqrt{b}<k\sqrt{a}$. Denote $u_1:=\cos(kx)\e^{-k\sqrt{a}t}$ and $u_2:=\cos(k'y)\e^{-k'\sqrt{b}t}$, two solutions to the same equation
$$
\ddot u+ \div\left[\begin{pmatrix}
    a & 0 \\
    0 & b
\end{pmatrix} \nabla u\right]=0.
$$
Note that $u_2$ oscillates faster but decays slower than $u_1$ (since $k'>k$ and $k'\sqrt{b}<k\sqrt{a}$). As heuristically explained above, the reduction will be done in four steps.
\\

\textbf{Step 1: Introduction of the slower decaying function.} We will apply Proposition \ref{perturbationnew} with $\epsilon=\frac{1}{k^4}$; since $k\gg 1$, we can ensure that $\frac{1}{k^4} <D$ where $D$ is the universal constant from Proposition \ref{perturbationnew}. Moreover, since $k\gg 1$, we also have $k \leq k^{4/3}.$ Also, since $k'\leq 2k$, we have $k'\leq k^{4/3}.$ Hence, 
 Proposition \ref{perturbationnew} applies with $\epsilon =\frac{1}{k^4}$ and  we can transform  
\begin{align}
    u_1 \hspace{0.5cm} \mbox{ into } \hspace{0.5cm} u_1+ \frac{1}{k^4}u_2
\end{align}
in $\frac{1}{k^{4/3}}$ time (within the set $\T^2 \times [0, k^{-4/3}]$), via a solution $u$ to $\ddot u+\div(A\nabla u)=0$ and where $A$ is in the regularity class $R(20, 10)$.

Throughout this step, the solution $u$ is of the form $u=f(t)\cos(kx) + g(t)\cos(k'y)$ where $f, g \in C^2$ satisfy for all $0\leq \alpha \leq 2,$ $|f^{(\alpha)}(t)|\lesssim k^{\alpha} \e^{-k\sqrt{a}t}$ and $|g^{(\alpha)}(t)|\lesssim (k')^{\alpha} \left(\frac{1}{k^4}\right)^{\frac{3-\alpha}{3}} \e^{-k'\sqrt{b}t} \lesssim (k')^{\alpha} \e^{-k'\sqrt{b}t}$ (since $k \geq 1$).  At the end of this step, the solution $u$ is $u=u_1+\frac{1}{k^4}u_2$ and the matrix $A$ is $A=\begin{pmatrix}
    a & 0 \\0 & b
\end{pmatrix}.$
\\

\textbf{Step 2: We wait.} This phase, which might seem surprising at first, is key. It explains why we considered a slower decaying function. Indeed, since $u_2$ decays slower than $u_1$, then, if we wait long enough, our initial function $u_1$ (which has a faster decay) will become much smaller than $u_2$. 

More precisely, we wait until the moment when $u_1$ becomes $\frac{1}{k^4}$ times smaller than $\frac{1}{k^4} u_2$:  $\frac{\sup_{\T^2}|u_1|(t)}{\sup_{\T^2}k^{-4}|u_2|(t)}=\frac{1}{k^4}$, i.e.,
\begin{align}\label{timetzerotowaitnew}
    \frac{\e^{-k \sqrt{a}t}}{k^{-4} \e^{-k' \sqrt{b}t}} = \frac{1}{k^4}.
\end{align}
This happens at time $t_0:=\frac{8 \ln k}{k \sqrt{a} - k' \sqrt{b}}>0.$ Since at the end of step 1 we were at time $t=\frac{1}{k^{4/3}}$, we then wait in this step 2 for a total time of $\frac{8 \ln k}{k \sqrt{a} - k' \sqrt{b}}-\frac{1}{k^{4/3}}>0$ since $k \gg 1.$

Hence, at the end of step 2, we are at time $t_0=\frac{8 \ln k}{k \sqrt{a} - k' \sqrt{b}}$ and it holds that 
\begin{align}\label{newtimephd}
     \frac{\e^{-k \sqrt{a}t_0}}{k^{-4} \e^{-k' \sqrt{b}t_0}} = \frac{1}{k^4}.
\end{align}
We do not change the matrix $A$ or the solution $u$ in this phase.
\\

\textbf{Step 3: We remove the initial function.} At the beginning of this step, the situation is the following: we have our function 
\begin{align}
u= \cos(k x) \e^{-k \sqrt{a} t}+ \frac{1}{k^4} \cos(k' y) \e^{-k' \sqrt{b} t},    
\end{align}
we have the matrix $A=\begin{pmatrix}
    a & 0 \\
    0 & b
\end{pmatrix}$, and we are at time $t=t_0=\frac{8 \ln k}{k \sqrt{a} - k' \sqrt{b}}$. We will use Proposition \ref{perturbationnew} to remove the initial function $u_1$. To this end, we note that $u$ can be written as
\begin{align*}
    u= \e^{-k \sqrt{a} t_0} \left( \cos(k x) \e^{-k \sqrt{a} (t-t_0)} + \cos(k'y) \e^{-k'\sqrt{b}(t-t_0)} \frac{\e^{-k'\sqrt{b}t_0}}{k^4 \e^{-k \sqrt{a}t_0}}\right).
\end{align*}

By \eqref{newtimephd}, we can rewrite $u$ as 
\begin{align}\label{rewriteutoremoveslowdown}
    u&= \e^{-k \sqrt{a} t_0} \left( \cos(k x) \e^{-k\sqrt{a} (t-t_0)} + \cos(k'y) \e^{-k'\sqrt{b}(t-t_0)} k^4\right) \nonumber\\
    &=k^4\e^{-k \sqrt{a} t_0} \left( \frac{1}{k^4}\cos(k x) \e^{-k \sqrt{a} (t-t_0)} + \cos(k'y) \e^{-k'\sqrt{b}(t-t_0)} \right).
\end{align}
We  introduce the new variable $t':=t-t_0\geq 0$ and apply Proposition \ref{perturbationnew} to transform the function
\begin{align}\label{removefromslowdownnew}
    \frac{1}{k^4}\cos(k x) \e^{-k \sqrt{a} t'} + \cos(k'y) \e^{-k'\sqrt{b}t'}
\end{align}
into the function 
\begin{align}\label{removetoslowdownnew}
    \cos(k'y) \e^{-k'\sqrt{b}t'}
\end{align}
via a function $u$ solution to $\ddot u + \div(A\nabla u)=0$, in time $\frac{1}{k^{4/3}}$ (within the set $\{(x,y,t') : 0 \leq t' \leq k^{-4/3}\}$, where $A$ is in the regularity class $R(20,10)$. The solution $u$ is of the form $f(t')\cos(kx) + g(t') \cos(k'y)$ where $f,g \in C^2$ satisfy $|f^{(\alpha)}(t')| \lesssim k^{\alpha}\left(\frac{1}{k^4}\right)^{\frac{3-\alpha}{3}} \e^{-k\sqrt{a}t'}$ and $|g^{(\alpha)}(t')| \lesssim (k')^{\alpha} \e^{-k'\sqrt{b}t'}$ for all $0\leq \alpha \leq 2.$

We can certainly go back to our old variable $t$ (since the pde $\ddot u + \div(A\nabla u)=0$ is unaffected by this change of variable): hence, we have been able to transform 
\begin{align}\label{removuonenew}
\frac{1}{k^4}\cos(k x) \e^{-k \sqrt{a} (t-t_0)} + \cos(k'y) \e^{-k'\sqrt{b} (t-t_0)}    
\end{align}
into 
\begin{align}\label{removeuonenewbis}
    \cos(k'y) \e^{-k'\sqrt{b} (t-t_0)} 
\end{align}
via a function $u$ solution to $\ddot u + \div(A\nabla u)=0$, within the set $\{(x,y,t): t_0\leq t \leq t_0+k^{-4/3}\}$ where $A$ is in the regularity class $R(20, 10).$ The solution $u$ is of the form $f(t)\cos(kx) + g(t) \cos(k'y)$ where $f,g \in C^2$  satisfy $|f^{(\alpha)}(t)| \lesssim k^{\alpha}\left(\frac{1}{k^4}\right)^{\frac{3-\alpha}{3}} \e^{-k\sqrt{a}(t-t_0)}$ and $|g^{(\alpha)}(t)| \lesssim (k')^{\alpha} \e^{-k'\sqrt{b}(t-t_0)}$ for all $0\leq \alpha \leq 2.$

We also notice that by the linearity of the pde $\ddot u + \div(A\nabla u)=0$,  we can multiply both \eqref{removuonenew} and \eqref{removeuonenewbis} by $ k^4 \e^{-k \sqrt{a}t_0}$ to transform
$$
\cos(k x) \e^{-k \sqrt{a} t} + k^4 \e^{-k \sqrt{a}t_0} \cos(k'y) \e^{-k'\sqrt{b} (t-t_0)}    
$$
into 
$$
k^4 \e^{-k \sqrt{a}t_0} \cos(k'y) \e^{-k'\sqrt{b} (t-t_0)} 
$$
via a function $u$ solution to $\ddot u + \div(A\nabla u)=0$, within the set $\{(x,y,t): t_0\leq t \leq t_0+k^{-4/3}\}$, where $A$ is in the regularity class $R(20, 10).$  The solution $u$ is of the form $f(t)\cos(kx) + g(t) \cos(k'y)$ where $f,g \in C^2$ satisfy $|f^{(\alpha)}(t)| \lesssim k^{\alpha}k^{\frac{4\alpha}{3}} \e^{-k\sqrt{a}t}$ and $|g^{(\alpha)}(t)| \lesssim (k')^{\alpha} k^4 \e^{-k\sqrt{a}t_0}\e^{-k'\sqrt{b}(t-t_0)}$ for all $0\leq \alpha \leq 2.$

Remembering the relation \eqref{newtimephd}
    $k^4\e^{-k \sqrt{a}t_0}=k^{-4} \e^{-k' \sqrt{b}t_0}$, we just showed that we can transform 
$$
\cos(k x) \e^{-k \sqrt{a} t} + \frac{1}{k^4}\cos(k'y) \e^{-k'\sqrt{b} t}    
$$
into 
$$
\frac{1}{k^4}  \cos(k'y) \e^{-k'\sqrt{b} t} 
$$
via a function $u$ solution to $\ddot u + \div(A\nabla u)=0$, within the set $\{(x,y,t): t_0\leq t \leq t_0+k^{-4/3}\}$, where $A$ is in the regularity class $R(20, 10).$ The solution $u$ is of the form $f(t)\cos(kx) + g(t) \cos(k'y)$ where $f,g \in C^2$ satisfy $|f^{(\alpha)}(t)| \lesssim k^{\frac{7\alpha}{3}} \e^{-k\sqrt{a}t}$ and $|g^{(\alpha)}(t)| \lesssim (k')^{\alpha} \frac{1}{k^4} \e^{-k'\sqrt{b}t} \lesssim (k')^{\alpha} \e^{-k'\sqrt{b}t}$ 
(since $k \geq 1$), for all $0\leq \alpha \leq 2$.

At the end of this step, the solution is $u=\frac{1}{k^4}\cos(k'y)\e^{-k'\sqrt{b}t}$ and the matrix is $A=\begin{pmatrix}
    a & 0\\
    0 & b
\end{pmatrix}.$
\\

\textbf{Partial conclusion 1:} We have been able to transform 
\begin{align}\label{patialconclusionslowblablanew}
    \cos(k x) \e^{-k \sqrt{a} t} \hspace{0.5cm} \mbox{ into } \hspace{0.5cm} \frac{1}{k^{4}} \cos(k'y) \e^{-k'\sqrt{b}t}
\end{align}
via a function $u$ solution to $\ddot u + \div(A\nabla u)=0$ in $\frac{1}{k^{4/3}} + \frac{8 \ln k}{k \sqrt{a} - k' \sqrt{b}} $ amount of time with $A$ in the regularity class $R(20,10).$ Moreover, the solution $u$ is of the form $f(t)\cos(kx) + g(t) \cos(k'y)$ where $f, g \in C^2$ satisfy $|f^{(\alpha)}(t)| \lesssim k^{\frac{7\alpha}{3}} \e^{-k\sqrt{a}t}$ and $|g^{(\alpha)}(t)|  \lesssim (k')^{\alpha} \e^{-k'\sqrt{b}t}$  for all $0\leq \alpha \leq 2$.
\\

\textbf{The last step:} We will now explain how to transform 
\begin{align}\label{laststepslowingdown}
    \frac{1}{k^{4}} \cos(k'y) \e^{-k'\sqrt{b}t} \hspace{0.5cm} \mbox{ into } \hspace{0.5cm}  \cos(k'y) \e^{-k'\sqrt{b}t}
\end{align}
via a solution $u$ to $\ddot u + \div(A\nabla u)=0$ in $\frac{\sqrt{4\ln(k)}}{k'^{1/3}} $ amount of time, where $A$ is in the regularity class $R(20,10).$ Moreover, we will show that throughout this phase, $u$ is of the form $u=g(t) \cos(k'y)$ with $g \in C^2$ satisfying $|g^{(\alpha)}(t)| \lesssim (k')^{\alpha} \e^{-k'\sqrt{b}t}$ for all $0\leq \alpha \leq 2.$ The matrix $A$ at the end of this step will be $A=\begin{pmatrix}
    a & 0 \\
    0 & b
\end{pmatrix}.$
\\

So far, the situation is the following: we have 
\begin{align}\label{uaftersometimebeginningphaselaststep}
   u= \frac{1}{k^{4}} \cos(k'y) \e^{-k'\sqrt{b}t}, \quad A=\begin{pmatrix}
    a &0 \\
    0 & b
\end{pmatrix}
\end{align}
and we are at the time 
\begin{align}\label{timetzerotildeaftertime}
\tilde t_0= \frac{1}{k^{4/3}} + \frac{8 \ln(k)}{k \sqrt{a} -k' \sqrt{b}}.    
\end{align}
We keep the function $u=\frac{1}{k^{4}} \cos(k'y) \e^{-k'\sqrt{b}t} $ and the matrix $A=\begin{pmatrix}
    a &0 \\
    0 & b
\end{pmatrix}$ until time 
\begin{align}\label{newtonephd}
t_1:=\tilde t_0+\frac{1}{100}.  
\end{align}
For $t \geq t_1$, denote by $\alpha$ the function 
 \begin{align}\label{defalphagslowingdownremovec}
     \alpha(t)=\theta\left(\frac{t-t_1}{w} \right) \hspace{0.5cm} \mbox{ where } \hspace{0.5cm} w:=\frac{\sqrt{4\ln(k)}}{k'^{1/3}}
 \end{align}
and where $\theta(\tau)$ is the smooth function going from 1 to 0 as $\tau$ goes from 0 to 1, which is presented in Picture \ref{fig:theta} (see also footnote \footnote{We choose $\theta(t):= 1-G(\tan(\pi(t-1/2)))$ where $G(x)=\frac{1}{\sqrt{\pi}}\int_{-\infty}^x \e^{-\eta^2} \ud{\eta}$. \label{footnotebuildingblockcoeffcblanewrenew}}). Consider also the function 
\begin{align}\label{defgslowingdownremovec}
    h(t):=-k'\sqrt{b}t +\ln\left(\frac{1}{k^4}\right) \alpha(t).
\end{align}
Then, the function 
\begin{align}\label{functionuremoveconstantforestimation}
u_3:=\cos(k'y) \e^{h(t)}    
\end{align}
 goes from 
\begin{align*}
    \frac{1}{k^{4}} \cos(k'y) \e^{-k'\sqrt{b}t} \hspace{0.5cm} \mbox{ to } \hspace{0.5cm} \cos(k'y) \e^{-k'\sqrt{b}t} 
\end{align*}
as $t$ goes from $t_1$ to $t_1+w$. We want $u_3$ to solve a divergence form equation. To this end, we look for a function $\tilde b$ depending on $t$ only and such that
\begin{align}\label{connectingequationremoveconstantslowingdownnew}
    \ddot{u}_3 + \div \left [\begin{pmatrix}a&0\\0&\tilde b\end{pmatrix} \nabla u_3 \right ] = 0.
\end{align}

Plugging $u_3=\cos(k'y) \e^{h(t)} $ into \eqref{connectingequationremoveconstantslowingdownnew}, we see that the connecting equation \eqref{connectingequationremoveconstantslowingdownnew} is equivalent to
\begin{align}\label{btildeforconneceqtslowdiwnnew}
    \tilde b := \frac{\ddot{h}}{k'^2} + \left(\frac{\dot{h}}{k'}\right)^2.
\end{align}

We then define a function $U$ and a matrix $B$ by 
\begin{align}\label{defbiguaftersometime}
U=
\left\{
\begin{array}{rl}
u & \mbox{if } t \leq t_1, \\
u_3 & \mbox{if } t \geq t_1
\end{array}
\right. ,
\quad 
B=
\left\{
\begin{array}{cl}
A & \mbox{if } t \leq t_1, \\
\begin{pmatrix}
    a & 0 \\
    0 & \tilde b
\end{pmatrix} & \mbox{if } t \geq t_1
\end{array}
\right. 
\end{align}
where $u$, $A$ and $u_3$ are defined in \eqref{uaftersometimebeginningphaselaststep} and \eqref{functionuremoveconstantforestimation} and where $t_1$ was defined in \eqref{newtonephd}.

Before discussing the regularity of $U$ and $B$, we state some useful facts about the function $\alpha(t)$ defined in \eqref{defalphagslowingdownremovec} and which appears in the definition of $u_3$ and $\tilde b$.

\textbf{Facts:} For $t \in [t_1, t_1+w],$
\begin{align}\label{propalphagreatgreat}
0 \leq \alpha(t) \leq 1, \hspace{0.3cm} |\dot{\alpha}(t)| \lesssim w^{-1}, \hspace{0.3cm} |\ddot{\alpha}(t)| \lesssim w^{-2}.
\end{align}
where $w=\frac{\sqrt{4\ln(k)}}{k'^{1/3}}$ (see \eqref{defalphagslowingdownremovec}). The first fact follows directly from the definition of $\theta$ (footnote \ref{footnotebuildingblockcoeffcblanewrenew}). To see the second and third fact, we use that $\theta$ satisfies $\sup_{\tau \in [0, 1]} |\theta^{(n)}(\tau)| \leq C_n$ for some universal constants $C_n$ and for all $n \geq 0.$ Indeed, note that for any $\tau \in [0, 1]$ and for any $n \geq 1,$
\begin{align}\label{propthetaaftertime}
\theta^{(n)} (\tau) = \e^{-\tan^2(\pi(\tau-1/2))} P_n\big(\tan(\pi(\tau-1/2))\big)    
\end{align}
where $P_n$ is a polynomial. Moreover, it holds that 
\begin{align}\label{propalphaaftersometime}
    \left\{
\begin{array}{rl}
\alpha(t) \to 1 & \mbox{when } t \to t_1, \\
\alpha(t) \to 0 & \mbox{when } t \to t_1+w
\end{array}
\right. ,
\quad 
   \left\{
\begin{array}{rl}
\alpha^{(n)}(t) \to 0 & \mbox{when } t \to t_1, \\
\alpha^{(n)}(t) \to 0 & \mbox{when } t \to t_1+w
\end{array}
\right.
\quad 
\end{align}
for $n \in \{1,2\}$. Formulas \eqref{propalphaaftersometime} immediately follow from the definition of $\alpha$ \eqref{defalphagslowingdownremovec} and from \eqref{propthetaaftertime}. 

\comment{
\pseudosection{The regularity estimates for $U$}
We first compute the derivatives of $u_3$ that we will need later. By \eqref{functionuremoveconstantforestimation},  the function $u_3$ is of the form $u_3=\e^{h(t)} \cos(k'y)$ where $h$ is given in \eqref{defgslowingdownremovec} by $h(t):=-k'\sqrt{b}t +\ln\left(\frac{1}{k^4}\right) \alpha(t).$  The derivatives of $g(t):=\e^{h(t)}$ are given by 
\begin{align}\label{derivatovesgremoveconstantgreat}
\dot{g}(t) = \left(-k'\sqrt{b}-4\ln(k) \dot{\alpha}(t)\right) g(t), \hspace{0.5cm} \ddot{g}(t) = \left(-k'\sqrt{b}-4\ln(k) \dot{\alpha}(t)\right)^2 g(t) -4\ln(k) \ddot{\alpha}(t) g(t).    
\end{align}
Also, since $g(t) = \e^{-k'\sqrt{b}t - 4\ln(k) \alpha(t)}$, we have that $|g(t)| \leq \e^{-k'\sqrt{b}t}$ because $k \geq 1$ (by assumption) and $\alpha(t) \geq 0$ (see \eqref{propalphagreatgreat}). 
}
\pseudosection{1) The regularity of $U$ at $t_1$}
We observe that by the definition of $U$ \eqref{defbiguaftersometime}, of $u$ \eqref{uaftersometimebeginningphaselaststep} and of $u_3$ \eqref{functionuremoveconstantforestimation}, we have  
$$
U= \left\{
\begin{array}{rl}
\frac{1}{k^4}\cos(k'y)e^{-k'\sqrt{b}t} & \mbox{when } t \in [t_1-1/100, t_1], \\
\cos(k'y)e^{h(t)} & \mbox{when } t \in [t_1, t_1+w]
\end{array}
\right. 
$$

Since $h(t)=-k'\sqrt{b}t + \ln(\frac{1}{k^4})\alpha(t)$ by definition \eqref{defgslowingdownremovec}, it is straightforward to check that $U$ is $C^2$ at $t=t_1$ by using \eqref{propalphaaftersometime}.  

\comment{
Recall that $t_1$ (defined in \eqref{newtonephd}) is the beginning of the phase. Since the function $\alpha$ defined in footnote \ref{footnotebuildingblockcoeffcblanewrenew} satisfies 
\begin{align}\label{deralphaphd}
\lim_{t \to t_1^+} \alpha^{(k)}(t)=0, \quad   k \in \{1, 2\}
\end{align}
We already saw in \eqref{functionuremoveconstantforestimation} that $u=\cos(k'y)\e^{h(t)}$ goes to $\frac{1}{k^4}\cos(k'y)\e^{-k'\sqrt{b}t}$ when $t \to t_1^+$. Hence, $u$ is continuous at $t_1.$ 

By writing $u=g(t)\cos(k'y),$ we have $\dot u = \dot g \cos(k'y)$ and $\ddot u = \ddot g \cos(k'y).$ By \eqref{derivatovesgremoveconstantgreat} and by \eqref{deralphaphd}, we have that $\dot u \to \frac{-k'\sqrt{b}}{k^4}\e^{-k'\sqrt{b}t_1} \cos(k'y)$ and $\ddot u \to \frac{(k')^2b}{k^4} \e^{-k'\sqrt{b}t_1} \cos(k'y)$ as $t\to t_1^+.$ Since on $[t_1-\frac{1}{100}, t_1],$ $u$ is equal to $u=\frac{1}{k^4}\e^{-k'\sqrt{b}t}\cos(k'y)$, the function $u$ is $C^2$ at the endpoint $t=t_1.$

Using that $\lim_{t \to (t_1+w)^{-}} \alpha^{(k)}(t)=0$ for $k \in \{1, 2\},$ a similar argument shows that $u$ converges in a $C^2$ way to $\cos(k'y)\e^{-k'\sqrt{b}t}$ as $t \to (t_1+w)^{-}$ (we already saw the continuity in \eqref{functionuremoveconstantforestimation}).
}

\pseudosection{2) The regularity of $U$ on $[t_1, t_1+w]$}
On $[t_1, t_1+w]$, we saw above that $U=\cos(k'y) e^{h(t)}. $ Observe that 
\begin{align}\label{derivatovesgremoveconstantgreat}
\frac{d}{dt}\left(\e^{h(t)}\right)(t) = \left(-k'\sqrt{b}-4\ln(k) \dot{\alpha}(t)\right) \e^{h(t)}, \hspace{0.5cm} \frac{d^2}{dt^2}\left(e^{h(t)}\right)(t) = \left(-k'\sqrt{b}-4\ln(k) \dot{\alpha}(t)\right)^2 \e^{h(t)} -4\ln(k) \ddot{\alpha}(t) \e^{h(t)}.    
\end{align}
Also, since $\e^{h(t)} = \e^{-k'\sqrt{b}t - 4\ln(k) \alpha(t)}$, we have that $|\e^{h(t)}| \leq \e^{-k'\sqrt{b}t}$ because $k \geq 1$ (by assumption) and $\alpha(t) \geq 0$ (see \eqref{propalphagreatgreat}). By \eqref{propalphagreatgreat} and since $1\leq k \leq k'$, it then comes, 
\begin{align}\label{estiatesfordotgremoveconstant}
 |\dot{U}(t)| \lesssim \left(k'+\sqrt{\ln(k)} (k')^{1/3}\right)\e^{-k'\sqrt{b}t} \lesssim k'\e^{-k'\sqrt{b}t}    
\end{align}
and 
\begin{align}\label{estimatedddotgforremoveconstant}
|\ddot U(t)| \lesssim \left((k')^2+(k')^{2/3}\right)\e^{-k'\sqrt{b}t} \lesssim (k')^2 \e^{-k'\sqrt{b}t}
\end{align}

\comment{

Finally, since $\left|-k'\sqrt{b}-4\ln(k) \dot{\alpha}(t)\right| \lesssim k'$ by \eqref{estiatesfordotgremoveconstant}, we get by \eqref{derivatovesgremoveconstantgreat} and \eqref{propalphagreatgreat}
\begin{align}\label{estimatedddotgforremoveconstant}
|\ddot g(t)| \lesssim \left((k')^2+(k')^{2/3}\right)\e^{-k'\sqrt{b}t} \lesssim (k')^2 \e^{-k'\sqrt{b}t}
\end{align}
since $k'\geq 1$.}

Hence, we showed that
\begin{align}\label{estimatesforuremoveconstants}
    U=g(t) \cos(k'y) \hspace{0.5cm} \mbox{ where $g \in C^2$ satisfies} \hspace{0.5cm} |g^{(\alpha)}(t)| \lesssim (k')^{\alpha} \e^{-k'\sqrt{b}t}
\end{align}
for all $0 \leq \alpha \leq 2.$

It remains to verify that the matrix $B$ defined in \eqref{defbiguaftersometime} is in the regularity class $R(20,10).$
\\

\pseudosection{The regularity of $B$}
By the definition of $B$ \eqref{defbiguaftersometime} and $A$ \eqref{uaftersometimebeginningphaselaststep}, we have 
\begin{align}\label{Bclosetotoneaftertime}
    B=
\left\{
\begin{array}{cl}
\begin{pmatrix}
    a & 0 \\
    0 &  b
\end{pmatrix} & \mbox{if } t \in [t_1-1/100,t_1], \\
\begin{pmatrix}
    a & 0 \\
    0 & \tilde b
\end{pmatrix} & \mbox{if } t \in [t_1, t_1+w]
\end{array}
\right. 
\end{align}
By definition \eqref{btildeforconneceqtslowdiwnnew}, $\tilde b= \frac{\ddot{h}}{k'^2} + \left(\frac{\dot{h}}{k'}\right)^2.$ Recall also that  $h(t)=-k'\sqrt{b}t + \ln(\frac{1}{k^4})\alpha(t)$ by definition \eqref{defgslowingdownremovec} and that $\alpha(t)=\theta\left(\frac{t-t_1}{w} \right)$ (see \eqref{defalphagslowingdownremovec}).

\comment{
Recall from \eqref{btildeforconneceqtslowdiwnnew} that 
\begin{align}\label{timdebrecalldefnew}
    \tilde b := \frac{\ddot{h}}{(k')^2} + \left(\frac{\dot{h}}{k'}\right)^2
\end{align}
where  $h(t)=-k'\sqrt{b}t +\ln\left(\frac{1}{k^4}\right) \alpha(t)$ (see \eqref{defgslowingdownremovec}) and that $\alpha(t)=\theta\left(\frac{t-t_1}{w} \right)$ (see \eqref{defalphagslowingdownremovec}).}

Clearly
\begin{align}\label{derivativegslowdown}
   \dot{h}(t) = -k'\sqrt{b} - \frac{4\ln(k)}{w} \dot{\theta}\left(\frac{t-t_1}{w}\right), \hspace{0.5cm} \ddot h(t) = \frac{-4\ln(k)}{w^2} \ddot\theta \left(\frac{t-t_1}{w}\right), \quad \dddot h(t)=\frac{-4\ln(k)}{w^3} \dddot\theta \left(\frac{t-t_1}{w}\right).
\end{align}
We recall that
\begin{align}\label{usefulrelationslownew}
    w=\frac{\sqrt{4\ln(k)}}{(k')^{1/3}}, \hspace{0.5cm} \sup_{\tau \in [0,1]}| \theta^{(n)} (\tau)| \leq C_n
\end{align}
 where $C_n$ is some universal constant and for any $n \geq 0$.

\pseudosection{1) The regularity of B at $t_1$}
By \eqref{Bclosetotoneaftertime}, \eqref{derivativegslowdown} and by \eqref{propalphaaftersometime}, it is straightforward to see that $B$ is $C^1$ at $t=t_1.$

\comment{
Recall that $\theta^{(k)}\left(\frac{t-t_1}{w}\right)\to 0$ as $t\to t_1^+$ for $k \in \{1,2\},$ by the definition of $\theta$.
Hence, by \eqref{timdebrecalldefnew} and \eqref{derivativegslowdown}, $\tilde b \to b$ as $t \to t_1^+$, and $B=\begin{pmatrix}
    a & 0 \\
    0 & \tilde b
\end{pmatrix} \to \begin{pmatrix}
    a & 0 \\
    0 & b
\end{pmatrix}$ as $t \to t_1^+$, which is exactly the matrix $A$ on $[t_1-\frac{1}{100}, t_1]$ (see above \eqref{newtonephd}). Hence, $B$ is continuous at $t_1.$ 

Since  $\theta^{(k)}\left(\frac{t-t_1}{w}\right)\to 0$ as $t\to (t_1+w)^-$ for $k \in \{1,2\}$, a similar argument shows that $B$ goes in a continuous way to $\begin{pmatrix}
    a & 0 \\
    0 & b
\end{pmatrix}$ at $t \to (t_1+w)^-.$ Hence  $B$ is continuous at the endpoints of $[t_1, t_1+w]$ and is clearly continuous on the interior as well.
}
\pseudosection{2) The uniform ellipticity of $B$ in $[t_1, t_1+w]$}
In $[t_1, t_1+w]$, $B=\begin{pmatrix}
    a & 0 \\
    0 & \tilde b
\end{pmatrix}$ by \eqref{Bclosetotoneaftertime} . By definition \eqref{btildeforconneceqtslowdiwnnew}, $\tilde b= \frac{\ddot{h}}{k'^2} + \left(\frac{\dot{h}}{k'}\right)^2.$ We estimate the two terms separately:
\begin{itemize}
    \item \underline{The first term.}
    Using \eqref{derivativegslowdown} and \eqref{usefulrelationslownew}, we have for $t \in [t_1, t_1+w]$,
    \begin{align}\label{estimatefirsttermuniformelliptcslownew}
    \left|\frac{\ddot h}{(k')^2}\right| \lesssim \frac{(k')^{2/3}}{(k')^2}  \ll 1        
    \end{align}
    
    since $k'\gg 1$ by assumption.

\item \underline{The second term.} Using \eqref{derivativegslowdown}, we have for $t \in [t_1, t_1+w],$
$$
\frac{\dot h}{k'} = -\sqrt{b} - \frac{4\ln(k)}{wk'} \dot{\theta}\left(\frac{t-t_1}{w}\right).
$$
Therefore,
\begin{align}\label{secondtermsquaredfed}
\left(\frac{\dot h}{k'} \right)^2 = b + \frac{8 \sqrt{b} \ln(k)}{wk'} \dot{\theta}\left(\frac{t-t_1}{w}\right) + \left[ \frac{4\ln(k)}{wk'} \dot{\theta}\left(\frac{t-t_1}{w}\right)\right]^2.    
\end{align}

By \eqref{usefulrelationslownew}, 
\begin{align}\label{firstsecondtermsow}
\left|  \frac{8 \sqrt{b} \ln(k)}{wk'} \dot{\theta}\left(\frac{t-t_1}{w}\right)\right| \lesssim \frac{\sqrt{\ln(k)}}{(k')^{2/3}} \ll 1    
\end{align}

since $k' > k \gg 1$ by assumption.
Similarly,
\begin{align}\label{secondsecondtermslow}
\left[ \frac{4\ln(k)}{wk'} \dot{\theta}\left(\frac{t-t_1}{w}\right)\right]^2 \lesssim \frac{\ln(k)}{(k')^{4/3}} \ll 1    .
\end{align}

Hence,  by \eqref{secondtermsquaredfed}, \eqref{firstsecondtermsow} and \eqref{secondsecondtermslow}, the second term $\left(\frac{\dot h}{k'} \right)^2$ is estimated by 
\begin{align}\label{estimesecondtermslowunifellip}
    \left(\frac{\dot h}{k'} \right)^2 = b + \mathcal{O}\left(\frac{1}{1000}\right),
\end{align}
where $\mathcal{O}\left(\frac{1}{1000}\right)$ is  an error term smaller than $1/1000$ in absolute value.
Recalling that $\tilde b =  \frac{\ddot h}{(k')^2}+\left(\frac{\dot h}{k'} \right)^2  $ by definition \eqref{btildeforconneceqtslowdiwnnew}, we get by combining \eqref{estimatefirsttermuniformelliptcslownew} and \eqref{estimesecondtermslowunifellip} that 
$$
\tilde b = b +\mathcal{O}\left(\frac{1}{1000}\right).
$$
Since by assumption $\frac{1}{10} \leq a, b \leq 10$, we get that the matrix $\begin{pmatrix}
    a & 0 \\
    0 & \tilde b
\end{pmatrix}$
is uniformly elliptic, i.e., in $[t_1, t_1+w],$
\begin{align}\label{bisunifomrelliptcslowdown}
\frac{1}{20} |\xi|^2 \leq (B \xi, \xi) \leq 20 |\xi|^2.
\end{align}

\end{itemize}

\pseudosection{3) The $C^1$ smoothness of $B$ in $[t_1, t_1+w]$} 

\comment{
Recall that $\tilde b = \frac{\ddot h}{(k')^2} + \left(\frac{\dot h}{k'}\right)^2$ by \eqref{timdebrecalldefnew}. Hence, 
\begin{align}\label{derivateivegdotnewblabla}
\dot{\tilde b} = \frac{\dddot h}{(k')^2} + 2\frac{\dot h \ddot h}{(k')^2}.    
\end{align}

We recall also from \eqref{derivativegslowdown} that  the first and second derivatives of $h$ are given by,
\begin{align}\label{dervgrecallblabla}
   \dot{h}(t) = -k'\sqrt{b} - \frac{4\ln(k)}{w} \dot{\theta}\left(\frac{t-t_1}{w}\right), \hspace{0.5cm} \ddot h(t) = \frac{-4\ln(k)}{w^2} \ddot\theta \left(\frac{t-t_1}{w}\right).
\end{align}
Hence, 
\begin{align}\label{thirdderivativegblabla}
    \dddot h(t) = \frac{-4\ln(k)}{w^3} \dddot\theta \left(\frac{t-t_1}{w}\right).
\end{align}

We also recall from \eqref{usefulrelationslownew} that
\begin{align}\label{recallwthetaboundedblablanew}
    w=\frac{\sqrt{4\ln(k)}}{(k')^{1/3}}, \hspace{0.5cm} \sup_{\tau \in [0,1]}| \theta^{(n)} (\tau)| \leq C_n
\end{align}
for some universal constants $C_n>0$ and for all $n \geq 0.$

\pseudosection{1) The $C^1$ regularity of $B$ at $t_1$}
We can argue in the same way as we did for the continuity of $B$ at $t_1^+$ and $ (t_1+w)^-$ to see that $\dot B$ converges continuously to the zero matrix at the endpoints $t_1^+$ and $ (t_1+w)^-$ and therefore $B$ is $C^1$ at the endpoints (since on $[t_1-\frac{1}{100}, t_1]$ and for $t \geq t_1+w$, we have the constant matrix $\begin{pmatrix}
    a & 0 \\ 0 & b
\end{pmatrix}$).

\pseudosection{2) The $C^1$ regularity of $B$ on $[t_1, t_1+w]$}
}
In $[t_1, t_1+w]$, $B=\begin{pmatrix}
    a & 0 \\
    0 & \tilde b
\end{pmatrix}$ by \eqref{Bclosetotoneaftertime} . By definition \eqref{btildeforconneceqtslowdiwnnew}, $\tilde b= \frac{\ddot{h}}{k'^2} + \left(\frac{\dot{h}}{k'}\right)^2.$ Clearly $\dot{\tilde b} = \frac{\dddot h}{k'^2} + \frac{2\dot{h}\ddot h}{k'^2}$. Before estimating these two terms separately, we recall that  $h(t)=-k'\sqrt{b}t + \ln(\frac{1}{k^4})\alpha(t)$ by definition \eqref{defgslowingdownremovec} and that $\alpha(t)=\theta\left(\frac{t-t_1}{w} \right)$ (see \eqref{defalphagslowingdownremovec}). Moreover, \begin{align}\label{derivativegslowdownagain}
   \dot{h}(t) = -k'\sqrt{b} - \frac{4\ln(k)}{w} \dot{\theta}\left(\frac{t-t_1}{w}\right), \hspace{0.5cm} \ddot h(t) = \frac{-4\ln(k)}{w^2} \ddot\theta \left(\frac{t-t_1}{w}\right), \quad \dddot h(t)=\frac{-4\ln(k)}{w^3} \dddot\theta \left(\frac{t-t_1}{w}\right).
\end{align}
We recall that
\begin{align}\label{usefulrelationslownewagain}
    w=\frac{\sqrt{4\ln(k)}}{(k')^{1/3}}, \hspace{0.5cm} \sup_{\tau \in [0,1]}| \theta^{(n)} (\tau)| \leq C_n
\end{align}
 where $C_n$ is some universal constant and for any $n \geq 0$.    

\begin{itemize}
    \item \underline{The first term:} By \eqref{derivativegslowdownagain} and \eqref{usefulrelationslownewagain},
\begin{align}\label{firsttermbtildedotblablaslow}
    \left|\frac{\dddot h}{(k')^2} \right| = \left| \frac{4\ln(k)}{w^3 (k')^2} \dddot\theta \left(\frac{t-t_1}{w}\right)\right| \lesssim \frac{1}{k' \sqrt{\ln(k)}} \ll 1
\end{align}
 $k'>k \gg 1$ by assumption.

\item \underline{The second term:} By \eqref{derivativegslowdownagain} and \eqref{usefulrelationslownewagain},
\begin{align}\label{estimatrsecondbhjbfbfaerfdf}
    \left| \frac{\dot h \ddot h}{(k')^2} \right| &\leq \left| \frac{4 k'\sqrt{b}\ln(k)}{(k')^2w^2} \ddot\theta \left(\frac{t-t_1}{w}\right)\right| +\left| \frac{\big(4\ln(k)\big)^2}{w^3(k')^{2}}  \dot{\theta}\left(\frac{t-t_1}{w}\right) \ddot\theta \left(\frac{t-t_1}{w}\right)\right| \nonumber\\
    & \lesssim \frac{1}{(k')^{1/3}} +\frac{\sqrt{\ln(k)}}{k'} \ll 1
\end{align}
since $k'>k \gg 1$ by assumption. Hence, by combining \eqref{firsttermbtildedotblablaslow} and \eqref{estimatrsecondbhjbfbfaerfdf}, we get that 
$$
|\dot{\tilde b}| \leq 1.
$$
\end{itemize}

Since in $[t_1, t_1+w]$, the matrix $B$ is of the form $B=\begin{pmatrix}
    a & 0 \\
    0 & \tilde b
\end{pmatrix}$, since $\tilde b$ depends on $t$ only and since $a$ is a constant, we get that $B$ is uniformly $C^1$ in $[t_1, t_1+w]$. In particular, by the previous ellipticity bound \eqref{bisunifomrelliptcslowdown}, we get that 
$$
B \in R(20,1) \quad \text{when } t \in [t_1, t_1+w].
$$

\pseudosection{When $t \geq t_1+w$}
Using the definition of $U$ and $B$ \eqref{defbiguaftersometime}, of $u_3$ \eqref{functionuremoveconstantforestimation} and of $\tilde b$ \eqref{btildeforconneceqtslowdiwnnew} together with the properties of $\alpha$ \eqref{propalphaaftersometime} and by arguing in the same fashion as the regularity questions at $t=t_1$, we leave it to the reader to verify that $U$ converge in the $C^2$ norm to $\cos(k'y)e^{-k'\sqrt{b}t}$ when $t \to t_1+w$ and that $B$ converges in the $C^1$ norm to $\begin{pmatrix}
    a & 0 \\ 0 & b
\end{pmatrix}$
when $t \to t_1+w.$

\pseudosection{Partial conclusion 2:}

In this last step, we have been able to  transform 
\begin{align}\label{laststepslowingdownsecondpartialccl}
    \frac{1}{k^{4}} \cos(k'y) \e^{-k'\sqrt{b}t} \hspace{0.5cm} \mbox{ into } \hspace{0.5cm}  \cos(k'y) \e^{-k'\sqrt{b}t}
\end{align}
in $\frac{\sqrt{4\ln(k)}}{(k')^{1/3}} $ amount of time  within the regularity class $R(20,1).$ 
\\

\textbf{Conclusion:} By the partial conclusion 1 \eqref{patialconclusionslowblablanew} and by this last step \eqref{laststepslowingdownsecondpartialccl}, we have been able to transform 
\begin{align*}
    \cos(k x) \e^{-k \sqrt{a} t} \hspace{0.5cm} \mbox{ into } \hspace{0.5cm}  \cos(k'y) \e^{-k'\sqrt{b}t}
\end{align*}
in $\frac{1}{k^{4/3}} + \frac{8 \ln(k)}{k \sqrt{a} - k' \sqrt{b}} + \frac{\sqrt{4\ln(k)}}{(k')^{1/3}}+\frac{1}{100}$ amount of time and we did it within the regularity class $R(20,10).$ We also showed that throughout this transformation, the solution $u$ is of the form $u(x,y,t)=f(t)\cos(kx) +g(t) \cos(k'y)$ where $f, g \in C^2$ satisfy $|f^{(\alpha)}(t)| \lesssim k^{\frac{7\alpha}{3}} \e^{-k\sqrt{a}t}$ and $|g^{(\alpha)}(t)| \lesssim (k')^{\alpha} \e^{-k'\sqrt{b}t}$ for all $0 \leq \alpha \leq 2$. This finishes the reduction of Proposition \ref{slownewzeroc} to Proposition \ref{perturbationnew}.

    
\end{proof}

We are now left with the proof of Proposition \ref{perturbationnew}. We recall it for the reader's convenience:
\begin{prop*}[\ref{perturbationnew}, adding a small perturbation]
 Consider  the equation $\ddot{u}+\div( A\nabla u) = 0,$ where $A := \begin{pmatrix}a&0\\0&b\end{pmatrix}$ is a matrix with constant coefficients $a,b$ in $(\frac{1}{10}, 10)$. Define two solutions $$u_1(x,t) := \cos(kx)e^{-k\sqrt{a} t} \textup{ and } u_2(y,t) := \cos(k'y)e^{-k'\sqrt{b}t}$$
    where  $ 1 \leq k \leq k' \leq 2k$   .    

\begin{itemize}
\item There exists a universal constant $D>0$ such that for any $0<\epsilon<D$, if $\epsilon^{-\frac{1}{4}} \leq k, k'\leq \epsilon^{-\frac{1}{3}}$, then we can transform $u_1$ into $u_1+\epsilon u_2$ via a solution $u$ to $\ddot u + \div(\tilde A\nabla u)=0$ in $\epsilon^{\frac{1}{3}}$ time
(within the set $\T^2 \times [0, \epsilon^{1/3}]$) with $\tilde A$ in the regularity class $R(20,10).$ 

\item For $t \in [0, \epsilon^{1/3}],$ the function $u$ is of the form 
\begin{align}\label{perturbationgneralfomrstatement}
    u(x,y,t)= f(t) \cos(kx)  +  g(t) \cos(k'y)
\end{align}
where $f, g \in C^2$ satisfy $|f^{(\alpha)}(t)| \lesssim k^{\alpha} \e^{-k\sqrt{a}t}$ and $|g^{(\alpha)}(t)|  \lesssim (k')^{\alpha}\epsilon^{\frac{3-\alpha}{3}} \e^{-k'\sqrt{b}t}$ for any $0 \leq \alpha\leq 2.$
\comment{
    \item  Then, there exists a universal constant $D>0$ such that for any $0<\epsilon < D $, we can transform $u_1$ into $u_1+\epsilon u_2$ in $\epsilon^{\frac{1}{3}}$ time
(within the set $\T^2 \times [0, \epsilon^{1/3}]$) if 
$$\epsilon^{-\frac{1}{4}} \leq k, k'\ll\epsilon^{-\frac{1}{3}}.$$

\underline{Regularity of the transformation:} There exists a matrix $A'$ and a function $u$ such that $u_1$ can be transformed into $u_1+\epsilon u_2$ via the function $u$ such that $\ddot{u}+\div\left[(A+A')\nabla u\right] = 0,$ where $A+A'$ is in the regularity class $R(20,10)$. Moreover, for $t \in [0, \epsilon^{1/3}],$ the function $u$ is of the form 
\begin{align}\label{perturbationgneralfomrstatement}
    u(x,y,t)= f(t) \cos(kx)  +  g(t) \cos(k'y)
\end{align}
where $f, g \in C^2$ satisfy $|f^{(\alpha)}(t)| \lesssim k^{\alpha} \e^{-k\sqrt{a}t}$ and $|g^{(\alpha)}(t)|  \lesssim (k')^{\alpha}\epsilon^{\frac{3-\alpha}{3}} \e^{-k'\sqrt{b}t}$ for any $0 \leq \alpha\leq 2.$
\\ 
}
\item Similarly, we can go from $\epsilon u_1+u_2$ to $u_2$ in the same amount of time (within the same set), under the same conditions on $\epsilon, k, k'$ and with the same regularity of the transformation $u$. The solution $u$ will be of the same form $u(x,y,t)= f(t) \cos(kx)  +  g(t) \cos(k'y)$ (but $f, g$ will be different from $f, g$ in \eqref{perturbationgneralfomrstatement}) and the following will hold: $f, g \in C^2$ and  $|f^{(\alpha)}(t)| \lesssim k^{\alpha} \epsilon^{\frac{3-\alpha}{3}} \e^{-k\sqrt{a}t}$ and $|g^{(\alpha)}(t)|  \lesssim (k')^{\alpha} \e^{-k'\sqrt{b}t}$ for any $0 \leq \alpha\leq 2.$
\end{itemize}
\end{prop*}

\begin{proof}[Proof of Proposition \ref{perturbationnew}]
   Let $\epsilon >0$ and denote $w:= \epsilon^{1/3}.$ Let also $\epsilon^{-\frac{1}{4}} \leq k, k'\leq \epsilon^{-\frac{1}{3}}.$ Define $\alpha(t):= \theta(t/w)$ where $\theta(\tau)$ is the smooth function going from 1 to 0 as $\tau$ goes from 0 to 1, which is presented in Picture \ref{fig:theta} (see also footnote \footnote{We choose $\theta(t):= 1-G(\tan(\pi(t-1/2)))$ where $G(x)=\frac{1}{\sqrt{\pi}}\int_{-\infty}^x \e^{-\eta^2} \ud{\eta}$. \label{footnotebuildingblockcoeffcblanewnew}}).   \comment{ Denote 
    $$
    u_1(x,t) = \cos(kx)e^{-k\sqrt{a} t} \textup{ and } u_2(y,t) = \cos(k'y)e^{-k'\sqrt{b}t}
    $$
    where  
    \begin{align*}
        1 \leq k \leq k' \leq 2k, \hspace{0.5cm} \epsilon^{-1/4} \leq k, k' \ll \epsilon^{-1/3}.
    \end{align*}
    Note that $u_1$ and $u_2$ are solutions to $\ddot u_i + \div\left[\begin{pmatrix}
        a & 0 \\
        0 & b
    \end{pmatrix} \nabla u_i\right] = 0$, $i\in \{1, 2\}$. }For $t \in [0, w],$ consider the function 
\begin{align}\label{defufastdecaylemma}
u(x,y,t):= u_1 +  (1-\alpha) \epsilon u_2. 
\end{align}

We look for a matrix $\tilde A:= A + A'$ so that $\ddot u + \div(\tilde A \nabla u) = 0.$ We have 
$$
\ddot u = \ddot u_1 + (1-\alpha)\epsilon \ddot u_2  -\epsilon\left( \ddot \alpha u_2 +2 \dot \alpha \dot u_2 \right).
$$

Since $\alpha$ depends only on the $t$ variable, we have
$$
\div(\tilde A \nabla u) = \div(A\nabla u_1)+(1-\alpha)\epsilon \div(A\nabla u_2) + \div(A'\nabla u).
$$

Since $u_1, u_2$ are $A$-harmonic, the equation $\ddot u + \div(\tilde A\nabla u)=0$ becomes
\begin{align}\label{almostpdeforaprimenew}
\div(A'\nabla u) =\epsilon\left( \ddot \alpha u_2 +2 \dot \alpha \dot u_2 \right).    
\end{align}

Since $u_2=\cos(k'y)\e^{-k'\sqrt{b}t},$ we see that $\dot u_2 = -k'\sqrt{b}u_2$. Hence, we can rewrite \eqref{almostpdeforaprimenew} as 
\begin{align}\label{eqtforaprimefastdecay}
    \div(A' \nabla u) = \epsilon \left(\ddot \alpha - 2 \dot \alpha k' \sqrt{b} \right) u_2=: \beta(t) u_2.
\end{align}
We stress that in \eqref{eqtforaprimefastdecay}, the unknown is the matrix $A'$. We also note that 
\begin{align*}
    u&= \cos(kx) \e^{-k\sqrt{a}t} + (1-\alpha) \epsilon \cos(k'y) \e^{-k'\sqrt{b}t} \\
    &=  \e^{-k \sqrt{a}t} \left(\cos(k x) + \epsilon (1-\alpha) \e^{(k \sqrt{a}-k'\sqrt{b})t} \cos(k'y) \right) \\
    & = \e^{-k \sqrt{a}t} \big(\cos(k x) + s \cos(k'y) \big)
\end{align*}
where we defined $s:=s(t):=\epsilon (1-\alpha) \e^{(k \sqrt{a}-k'\sqrt{b})t} .$ Since \eqref{eqtforaprimefastdecay} does not involve $t$ derivatives, we can consider the $t$ variable fixed. We will use the following Lemma: 
\begin{lemma}\label{claim2Dtwopointoneproof}
Let $1\leq k \leq k' \leq 2k.$ Given a positive number $s$ and a function $u = \cos(k x)+s\cos(k'y)$ on $\T^2$, there exists a smooth vector field $V$ with divergence $\div(V)=\cos(k'y)$ such that the matrix $A_s:= \begin{pmatrix}a_s(x,y)&b_s(x,y)\\b_s(x,y)&0\end{pmatrix}$ uniquely defined by $A_s\nabla u=V$ is smooth and satisfies 
\begin{align}\label{estimatestosatisfyforas}
        \|A_s\|_{\infty} \lesssim \frac{(1+|s|)}{k^2}, \hspace{0.5cm} \|\nabla A_s\|_{\infty} \lesssim \frac{(1+|s|)}{k}
    \end{align}
    where we recall that the gradient only involves spatial derivatives.
 Moreover, the function $s\mapsto A_s$ satisfies
    $$
    \bigg\|\frac{\partial A_s}{ \partial s}\bigg\|_{\infty} \lesssim \frac{1}{k^2}.
    $$
\end{lemma}
The proof of Lemma \ref{claim2Dtwopointoneproof} is presented in Appendix \ref{prooftwodlemmaappendix}.
\\

Now, Lemma \ref{claim2Dtwopointoneproof}  ensures that there exists a matrix $A'_s$ such that
\begin{align*}
    \div\left(A'_s \nabla (u \e^{k \sqrt{a}t})\right) = \cos(k'y).
\end{align*}
Since $u_2=\cos(k'y)\e^{-k'\sqrt{b}t},$ we see that by letting $A':= \beta(t) \e^{(k \sqrt{a} - k' \sqrt{b})t} A'_s$, we get \eqref{eqtforaprimefastdecay}:
\begin{align*}
    \div(A' \nabla u) = \beta(t) u_2
\end{align*}
and therefore $u$ solves 
\begin{align}
    \ddot u + \div\bigg((A+A') \nabla u\bigg) =0.
\end{align}

Before verifying the regularity of $u$ and $\tilde A=A+A'$ we state some useful facts about the function $\alpha(t)=\theta(t/w)$ that was introduced at the beginning of the proof (we recall the definition of  $\theta$ in footnote \footnote{We choose $\theta(t):= 1-G(\tan(\pi(t-1/2)))$ where $G(x)=\frac{1}{\sqrt{\pi}}\int_{-\infty}^x \e^{-\eta^2} \ud{\eta}$. \label{footnoteblablanewblanewbla}}).
\\

\textbf{Facts:} For $t \in [0, w],$
\begin{align}\label{propalphagreatgreatnewgreat}
0 \leq \alpha(t) \leq 1, \hspace{0.3cm} |\dot{\alpha}(t)| \lesssim w^{-1}, \hspace{0.3cm} |\ddot{\alpha}(t)| \lesssim w^{-2}, \hspace{0.3cm} w:=\epsilon^{1/3}.
\end{align}

The first fact follows directly from the definition of $\theta$. To see the second and third fact, we use that $\theta$ satisfies $\sup_{\tau \in [0, 1]} |\theta^{(n)}(\tau)| \leq C_n$ for some universal constants $C_n$ and for all $n \geq 0.$ Indeed, note that for any $\tau \in [0, 1]$ and for any $n \geq 1,$
$$
\theta^{(n)} (\tau) = \e^{-\tan^2(\pi(\tau-1/2))} P_n\big(\tan(\pi(\tau-1/2))\big)
$$
where $P_n$ is a polynomial. 

\pseudosection{The regularity estimates for $u$}
By definition \eqref{defufastdecaylemma}, $u$ is given by $u=u_1+(1-\alpha) \epsilon u_2$ where $u_1=\cos(kx)\e^{-k\sqrt{a}t}$ and $u_2=\cos(k'y)\e^{-k'\sqrt{b}t}.$ In particular, we can write 
\begin{align}\label{shapeofupertubartionverify}
u(x,y,t)=f(t) \cos(kx) + g(t) \cos(k'y)    
\end{align}
where $f(t):= \e^{-k\sqrt{a}t}$ and $g(t):=(1-\alpha(t)) \epsilon \e^{-k'\sqrt{b}t}.$ We verify the estimates for $f, g$ and their derivatives. By assumption, $a, b \in (\frac{1}{10}, 10)$, hence,
\begin{align}\label{estimatesforfertubartion}
  |f(t)|= \e^{-k\sqrt{a}t}, \hspace{0.5cm}  |\dot{f}(t)| \lesssim k \e^{-k\sqrt{a}t}, \hspace{0.5cm} |\ddot f(t)| \leq k^2 \e^{-k\sqrt{a}t}
\end{align}
as claimed in Proposition \ref{perturbationnew}. Moreover,
    $$
    \dot{g}(t) = -\left(\dot{\alpha}(t) +k'\sqrt{b} (1- \alpha(t)) \right) \epsilon \e^{-k'\sqrt{b}t}, \hspace{0.5cm} \ddot{g} = \left( -\ddot{\alpha}(t) + 2k'\sqrt{b}\dot{\alpha}(t) + (k'\sqrt{b})^2 (1- \alpha(t)) \right)\epsilon \e^{-k'\sqrt{b}t}.
    $$

Hence, by \eqref{propalphagreatgreatnewgreat} and by choosing $\epsilon \leq 1$, we get
    \begin{align}\label{estimatesforgperturbation}
        |g(t)| \leq \epsilon \e^{-k'\sqrt{b}t}, \hspace{0.5cm} |\dot{g}(t)| \lesssim k' \epsilon^{2/3}\e^{-k'\sqrt{b}t}, \hspace{0.5cm} |\ddot{g}(t)| \lesssim (k')^2 \epsilon^{1/3} \e^{-k'\sqrt{b}t}
    \end{align}
 since $a,b \in (\frac{1}{10}, 10)$ and since $k' \geq 1$ by assumption. Equations \eqref{shapeofupertubartionverify}, \eqref{estimatesforfertubartion} and \eqref{estimatesforgperturbation} proves \eqref{perturbationgneralfomrstatement} from Proposition \ref{perturbationnew}

Finally, remember that by definition, the function $\alpha$ satisfies 
$$
\lim_{t \to 0}\alpha(t)=1, \quad \lim_{t \to w}\alpha(t)=0, \quad \lim_{t \to 0} \alpha^{(k)}(t)=0, \quad \lim_{t \to w} \alpha^{(k)}(t)=0.
$$
for $k \in \{1, 2\}$.

It is then straightforward to see that 
$$
\lim_{t \to 0}u^{(k)}(t) = u_1^{(k)}(0), \quad \lim_{t \to w}u^{(k)}(t) = u_1^{(k)}(w)+\epsilon u_2^{(k)}(w)
$$
for $k \in \{0, 1, 2\}$.  This shows that the function given by $u_1$ for $t \leq 0,$ $u=u_1+(1-\alpha)\epsilon u_2$ for $t \in [0, w]$ and $u_1+\epsilon u_2$ for $t \geq w$ is $C^2$.

\pseudosection{The regularity of $\tilde A$}

We verify the regularity of $\tilde A:= A+A'.$  We first recall the objects involved in the estimation of $\tilde A$:
\begin{itemize}
    \item The matrix $\tilde A$ is given by 
    \begin{align}\label{fastdecaylemmarecallaprime}
        \tilde A = A+A', \hspace{0.5cm} A'= \beta(t) \e^{-(k'\sqrt{b}-k \sqrt{a})t} A'_s, \hspace{0.5cm} \beta(t)=\epsilon\left(\ddot \alpha - 2 \dot \alpha k' \sqrt{b} \right).
    \end{align}
    \item The matrix $A'_s$ was given by Lemma \ref{claim2Dtwopointoneproof} and satisfies the bounds
    \begin{align}\label{fastdecaylemmaaprimesrecall}
        \|A'_s\|_{\infty} \lesssim \frac{1+\|s(t)\|_{\infty}}{k^2}, \hspace{0.5cm} \|\nabla A'_s\|_{\infty} \lesssim \frac{1+\|s(t)\|_{\infty}}{k}, \hspace{0.5cm} \left|\partial_s(A'_s)\right| \lesssim \frac{1}{k^2},
    \end{align}
    where  $s=s(t)=\epsilon(1-\alpha) \e^{(k \sqrt{a}- k'\sqrt{b})t}.$
    \item For $t \in [0, w]$, the smooth function $\alpha$ was given by $\alpha(t)=\theta(t/w)$ and satisfies
    \begin{align}\label{fastdecaylemmarecallalpha}
        0 \leq \alpha(t)\leq 1, \hspace{0.5cm} \|\dot \alpha\|_{\infty} \lesssim w^{-1}, \hspace{0.5cm} \|\ddot \alpha\|_{\infty} \lesssim w^{-2}, \hspace{0.5cm} w=\epsilon^{1/3}.
    \end{align}
    \comment{since $\theta$ satisfies $\sup_{\tau \in [0,1]}|\theta^{(n)}(\tau)|\leq C_n$ for some universal constant $C_n$ and for all $n \geq 0$  (see footnote \footnote{We choose $\theta(t):= 1-G(\tan(\pi(t-1/2)))$ where $G(x)=\frac{1}{\sqrt{\pi}}\int_{-\infty}^x \e^{-\eta^2} \ud{\eta}$. \label{footnoteblablanewblabla}} for the definition of $\theta$).}
    \item Recall also that by hypothesis $k, k' \leq w^{-1}.$ Therefore, 
    \begin{align}\label{fastdecaylemmarecallexponent}
        \e^{-(k'\sqrt{b}-k \sqrt{a})t} \lesssim 1
    \end{align}
    for any $t \in [0, w].$ We finally recall that \begin{align}\label{fastdecaylemmarecallassumplambda}
         1 \leq k \leq k' \leq 2k, \hspace{0.5cm} \epsilon^{-1/4} \leq k, k' \leq \epsilon^{-1/3}.
    \end{align}
\end{itemize}

\pseudosection{The ellipticity bound for $\tilde A$}
Note that $A$ satisfies $\frac{1}{10} |\xi|^2 \leq (A\xi, \xi) \leq 10 |\xi|^2$ since $A$ is a diagonal matrix with entries $a, b$ where $\frac{1}{10} \leq a, b \leq 10$ by assumption.
Hence, it remains to estimate $\tilde A - A =A'.$ By \eqref{fastdecaylemmarecallaprime}, $$A'= \beta(t) \e^{-(k'\sqrt{b}-k \sqrt{a})t} A'_s$$ and $ \beta(t)=\epsilon\left(\ddot \alpha - 2 \dot \alpha k' \sqrt{b} \right).$ Hence, by \eqref{fastdecaylemmarecallalpha}, $|\beta| \lesssim \frac{\epsilon}{w^2} + \frac{\epsilon k'}{w}.$ Since $w=\epsilon^{1/3}$ by definition and since $k' \leq \epsilon^{-1/3}$ by assumption, this estimate reduces to 
\begin{align}\label{estimatesbetanewnewnbla}
|\beta| \lesssim \epsilon^{1/3}.    
\end{align}

By \eqref{fastdecaylemmaaprimesrecall}, \eqref{fastdecaylemmarecallalpha} and \eqref{fastdecaylemmarecallexponent}, $|A'_s| \lesssim \frac{1}{k^2}$. Hence,  \begin{align}\label{exponenttimesmatrixas}
\left|\e^{-(k'\sqrt{b}-k \sqrt{a})t} A'_s\right| \lesssim \frac{1}{k^2}.    
\end{align}
  Therefore, 
$$
|A'| \lesssim \frac{\epsilon^{1/3}}{k^2} \lesssim\epsilon^{1/3}
$$
since $k \geq 1$ by assumption. Therefore, by choosing $\epsilon$ small enough,  $\tilde A$ satisfies the ellipticity bound 
 \begin{align}\label{ellipticetycboudndfortildeanewgood}
 \frac{1}{20}|\xi|^2 \leq (\tilde A \xi, \xi) \leq 20 |\xi|^2     
 \end{align}
as claimed.

\pseudosection{The $C^1$ bound for $\tilde A $ in the $x, y$ variables} Recall that $\tilde A = A+A'$ and that $A$ has constant coefficients. Therefore, 
\begin{align*}
    |\nabla \tilde A|= |\nabla A'| = |\beta(t) \e^{-(k'\sqrt{b}-k \sqrt{a})t} \nabla A'_s|.
\end{align*}
By \eqref{estimatesbetanewnewnbla}, $|\beta| \lesssim \epsilon^{1/3}$ and by \eqref{fastdecaylemmarecallexponent}, $\e^{-(k'\sqrt{b}-k \sqrt{a})t} \lesssim 1$. By \eqref{fastdecaylemmaaprimesrecall}, $\|\nabla A'_s\|_{\infty} \lesssim \frac{1}{k} \lesssim 1$ since $k \geq 1$ by assumption. Hence, 
\begin{align}\label{estimategrdadientatildeblablanewgood}
\|\nabla \tilde A\|_{\infty} \lesssim \epsilon^{1/3}.    
\end{align}

\pseudosection{The $C^1$ bound  for $\tilde A$ in the $t$ variable} Since $\tilde A=A+A' $ and $A$ has constant coefficients, $\dot{\tilde A} = \dot{A'}$ where $A'= \beta(t) \e^{-(k'\sqrt{b}-k \sqrt{a})t} A'_s$ and $s=s(t)=\epsilon(1-\alpha) \e^{(k \sqrt{a}- k'\sqrt{b})t}.$ So, to estimate $\dot{A'}$, there are three terms to estimates:
\begin{itemize}
    \item \textbf{The first term in $\dot{A}'$:} $\dot{\beta}(t) \e^{-(k'\sqrt{b}-k \sqrt{a})t} A'_s.$  
    
    Since $\beta(t) = \epsilon\left(\ddot \alpha - 2 \dot \alpha k' \sqrt{b} \right) $ by \eqref{fastdecaylemmarecallaprime}, since $\alpha(t)=\theta(t/w)$ and since $\sup_{\tau \in [0,1]}|\theta^{(n)}(\tau)|\leq C_n$ for some universal constant $C_n$ and for all $n \geq 0$, we get 
    $$
    |\dot{\beta}(t)| \lesssim \frac{\epsilon}{w^3} + \frac{k' \epsilon}{w^2}
    $$
which reduces to 
$$
|\dot{\beta}(t)| \lesssim 1
$$
since $w=\epsilon^{1/3}$  by definition and $k'\leq \epsilon^{-1/3}$ by assumption. Therefore, since $\left|\e^{-(k'\sqrt{b}-k \sqrt{a})t} A'_s\right| \lesssim \frac{1}{k^2}$ by \eqref{exponenttimesmatrixas}, we get
\begin{align}\label{estimatefirsttermconebound}
    \left|\dot{\beta}(t) \e^{-(k'\sqrt{b}-k \sqrt{a})t} A'_s \right| \lesssim \frac{1}{k^2} \leq \epsilon^{1/2}
\end{align}
since $k\geq \epsilon^{-1/4}$ by assumption.

\item \textbf{The second term in $\dot{A}'$:} $-(k'\sqrt{b}-k\sqrt{a})\beta(t) \e^{-(k'\sqrt{b}-k\sqrt{a})t} A'_s.$

By \eqref{estimatesbetanewnewnbla}, $|\beta| \lesssim \epsilon^{1/3}$. By \eqref{exponenttimesmatrixas}, $\left|\e^{-(k'\sqrt{b}-k \sqrt{a})t} A'_s\right| \leq \frac{1}{k^2}.$ Therefore, since $k'\leq 2k$, 
\begin{align}\label{thesecondtermtderivativeblabla}
\left| (k'\sqrt{b}-k\sqrt{a})\beta(t) \e^{-(k'\sqrt{b}-k\sqrt{a})} A'_s\right| \lesssim \frac{\epsilon^{1/3}}{k} \lesssim \epsilon^{1/3}   
\end{align}
since $k \geq 1$ by assumption.

\item \textbf{The third term in $\dot{A}'$:} $\beta(t) \e^{-(k'\sqrt{b}-k\sqrt{a})t} \partial_t(A'_s)$.

Recall that $s=s(t)=\epsilon(1-\alpha) \e^{(k \sqrt{a}- k'\sqrt{b})t}$ is a function of $t$. Hence, 
$$
\partial_t(A'_s) = \partial_s(A'_s) \dot{s}(t).
$$
By \eqref{fastdecaylemmaaprimesrecall}, $\left|\partial_s(A'_s)\right| \lesssim \frac{1}{k^2}$. Observe that
$$
\dot{s}(t) = \epsilon \left( - \dot \alpha  +  (1-\alpha) (k \sqrt{a}- k'\sqrt{b}) \right)\e^{(k \sqrt{a}- k'\sqrt{b})t}.
$$
Since $\e^{-(k'\sqrt{b}-k\sqrt{a})t} \lesssim 1$ by \eqref{fastdecaylemmarecallexponent}, since $k' \leq 2k$, since $0\leq \alpha \leq 1$ and since $\alpha(t) = \theta(t/w)$ where $\dot{\theta}$ is universally bounded, we have $\left|\dot{s}(t)\right| \lesssim\frac{\epsilon}{w} +\epsilon k.$ Therefore, 
$$
\left|\partial_t(A'_s)\right| \lesssim \frac{\epsilon}{k^2 w} + \frac{\epsilon}{k}.
$$
By \eqref{estimatesbetanewnewnbla}, $|\beta| \lesssim \epsilon^{1/3}$. Therefore, 
\begin{align}\label{thirdtermbtildedotblablanew}
\left| \beta(t) \e^{-(k'\sqrt{b}-k\sqrt{a})t} \partial_t(A'_s)\right| \lesssim \frac{\epsilon^{4/3}}{k^2 w} + \frac{\epsilon^{4/3}}{k} \lesssim \epsilon    
\end{align}
since $w=\epsilon^{1/3}$ by definition and since $k \geq 1$ by assumption.
\end{itemize}
We combine \eqref{estimatefirsttermconebound}, \eqref{thesecondtermtderivativeblabla} and \eqref{thirdtermbtildedotblablanew} to conclude that
$$
|\dot{\tilde A}| =|\dot{A'}| \lesssim \epsilon^{1/3} \ll 1
$$
by choosing $\epsilon$ small enough. In particular, by combining $\|\nabla \tilde A\|_{\infty} \lesssim \epsilon^{1/3}$ (see \eqref{estimategrdadientatildeblablanewgood}) with the estimate for $\dot{\tilde A}$ we just obtained, we conclude that the derivatives of $\tilde A$ can be bounded by 10 by choosing $\epsilon$ small enough. Therefore, since we showed in \eqref{ellipticetycboudndfortildeanewgood}  that $ \frac{1}{20}|\xi|^2 \leq (\tilde A \xi, \xi) \leq 20 |\xi|^2 $, we conclude that $\tilde A \in R(20, 10).$ Moreover, the solution $u$ is given by $u(x,y,t)=f(t) \cos(kx) + g(t) \cos(k'y)  $ (see  \eqref{shapeofupertubartionverify}) where the functions $f$ and $g$ satisfy the estimates $|f^{(\alpha)}(t)| \lesssim k^{\alpha} \e^{-k\sqrt{a}t}$ (see \eqref{estimatesforfertubartion})and $|g^{(\alpha)}(t)| \lesssim (k')^{\alpha} \epsilon^{\frac{3-\alpha}{3}} \e^{-k\sqrt{b}t}$  (see \eqref{estimatesforgperturbation}) for all $0 \leq \alpha \leq 2$, as claimed in Proposition \ref{perturbationnew}.

Moreover, by using that 
$$
\lim_{t \to 0}\alpha(t)=1, \quad \lim_{t \to w}\alpha(t)=0, \quad \lim_{t \to 0} \alpha^{(n)}(t)=0, \quad \lim_{t \to w} \alpha^{(n)}(t)=0.
$$
for $n \in \{1, 2\}$, we leave it to the reader to verify that the matrix $A'$ and all its first derivatives converge to zero as $t \to 0$ and $t \to w.$ Hence, the matrix defined by $A$ for $t \leq 0$, $\tilde A=A+A'$ for $t \in [0, w]$ and $A$ for $t \geq w$ is uniformly $C^1.$

Finally, to transform $\epsilon u_1 + u_2$ into $u_2$, we now consider the function $u:=\epsilon \alpha u_1+u_2$ . We then proceed as above: we look for a matrix $A'$ such that $\ddot u + \div[(A+A')\nabla u]=0$. As above, plugging $u$ into this equation will give us a new equation where the unknown is $A'$. We then use Lemma \ref{lemmabistwopointnineinproofoftwopointne} (presented below) instead of Lemma \ref{claim2Dtwopointoneproof}. The rest of the proof is similar and we do not include it. This finishes the proof of the Proposition \ref{perturbationnew}.

\begin{lemma}\label{lemmabistwopointnineinproofoftwopointne}
Let $1\leq k \leq k' \leq 2k.$  Given a positive number $s$ and a function $v = s\cos(k x)+\cos(k'y)$ on $\T^2$, there exists a smooth vector field $V$ with divergence $\div(V)=\cos(kx)$ such that the matrix  $A_s = \begin{pmatrix}0&b_s(x,y)\\b_s(x,y)&c_s(x,y)\end{pmatrix}$ uniquely defined by $A_s \nabla v=V$ is smooth and satisfies \begin{align*}
        \|A_s\|_{\infty} \lesssim \frac{(1+|s|)}{k^2}, \hspace{0.5cm} \|\nabla A_s\|_{\infty} \lesssim \frac{(1+|s|)}{k}
    \end{align*}
    where we recall that the gradient only involves spatial derivatives.
     Moreover, the function $s\mapsto A_s$ satisfies
    $$
    \bigg\|\frac{\partial A_s}{ \partial s}\bigg\|_{\infty} \lesssim \frac{1}{k^2}.
    $$
\end{lemma}

\end{proof}

\subsection{Proof of Proposition \ref{accelnew}}
\label{proof35}
We recall it for the reader's convenience:
\begin{prop*}[\ref{accelnew}, acceleration]

Let $k \gg 1 $. Let also $a, b, b' \in (\frac{1}{10}, 10)$ with $b \leq b'$. Consider two functions $u_1: = \cos(ky) \e^{-k\sqrt{b}t}$ and $u_2:=\cos(ky) \e^{-k\sqrt{b'}t}$ which are solutions to the equations
$$\ddot{u}_1+\div \left[ \begin{pmatrix}a&0\\0&b\end{pmatrix}\nabla u_1 \right] = 0 \text{ \quad and \quad }\ddot{u}_2+\div \left [\begin{pmatrix}a&0\\0&b'\end{pmatrix}\nabla u_2 \right]= 0.$$
For any $t_1\geq 0$ and for any $c_1>0$, there exists $c_2>0$ such that we can transform $c_1 u_1$ into $c_2 u_2$ within the set $\T^2 \times [t_1, t_1+C]$ via a solution $u$ to $\ddot u + \div(A\nabla u)=0$ and where $A$ is in the regularity class $R(80, 10).$ The constants $c_1$ and $c_2$ are related by $$c_1 \e^{-k\sqrt{b}t_1}= c_2 \e^{-k\sqrt{b'}t_1}.$$

The time of the transformation is $C:=400.$ Moreover, for $t \in [t_1, t_1+C],$ the solution $u$ is of the form 
\begin{align}\label{littlegaccelpropstate}
    u(x,y,t)=g(t) \cos(ky)  
\end{align}
  where $g \in C^2$ satisfies  $|g^{(\alpha)}(t)| \lesssim c_1 k^{\alpha} \e^{-k\sqrt{b}t}$ for all $0 \leq \alpha \leq 2$.
\end{prop*}

We first reduce Proposition \ref{accelnew} to the case $t_1=0$ and $c_1=c_2=1.$

\begin{prop}\label{accelnewzeroc}
    Proposition \ref{accelnew} is true when $t_1=0$, $c_1=c_2=1.$
\end{prop}

The proof of this reduction is similar to the reduction of Proposition \ref{slownew} to Proposition \ref{slownewzeroc} and we skip it. We now prove the Proposition \ref{accelnewzeroc}.

\begin{proof}[Proof of Proposition \ref{accelnewzeroc}]
Let $\theta(t)$ be the smooth function, presented in Picture \ref{fig:theta} (see also footnote \footnote{We choose $\theta(t):= 1-G(\tan(\pi(t-1/2)))$ where $G(x)=\frac{1}{\sqrt{\pi}}\int_{-\infty}^x \e^{-\eta^2} \ud{\eta}$. \label{footnotebuildingblockcoeffcbla}}). It is equal to 1 for $t \leq 0$ and to zero for $t \geq 1$ and it monotonically decreases for $t \in [0,1]$. Denote by $\alpha$ the function 
 \begin{align}\label{defalphaaccel}
     \alpha(t):=\theta\left(\frac{t}{w} \right)
 \end{align}
where $w$ is a fixed universal number to be chosen later. Denote also by 
\begin{align}\label{defgaccellemma}
    g(t):=\left( \sqrt{b'} + (\sqrt{b}-\sqrt{b'})\alpha(t)\right) k t
\end{align}
and consider the function 
\begin{align}\label{uacceldef}
    u:= \cos(k y) \e^{-g(t)}
\end{align}
for $t \in [0, w]$. This function goes from 
\begin{align*}
    \cos(k y) \e^{-k \sqrt{b}t} \hspace{0.5cm} \mbox{ to } \hspace{0.5cm} \cos(k y) \e^{-k \sqrt{b'}t}
\end{align*}
as $t$ goes from $0$ to $w$. We will now look for a matrix $A:=\begin{pmatrix}
    a & 0 \\ 0 & \tilde b
\end{pmatrix}$ for some function $\tilde b$  depending on $t$ only and such that $\ddot{u} + \div(A \nabla u)=0.$ Plugging $u=\cos(ky)\e^{-g(t)}$ into the equation yields
\begin{align}\label{defbtildeaccellemma}
    \tilde b := \frac{- \ddot{g}}{k^2} + \left(\frac{\dot{g}}{k}\right)^2.
\end{align}

Since $g(t)=\left( \sqrt{b'} + (\sqrt{b}-\sqrt{b'})\alpha(t)\right) k t$ by definition \eqref{defgaccellemma}, 
\begin{align}\label{derivgaccellemma}
    &\dot{g} = (\sqrt{b} - \sqrt{b'})\dot{\alpha} k t + \left(\sqrt{b'} + (\sqrt{b}- \sqrt{b'} )\alpha \right) k, \\
   & \ddot{g} = (\sqrt{b} - \sqrt{b'})\ddot{\alpha} k t +2 (\sqrt{b}- \sqrt{b'} )\dot{\alpha} k \nonumber
\end{align}

Before verifying the regularity of $u$ and $A$ we state some useful facts about the function $\alpha(t)=\theta(t/w)$ that was introduced at the beginning of the proof.
\\

\textbf{Facts:} For $t \in [0, w],$
\begin{align}\label{propalphagreatgreatnewgreatnewfacts}
0 \leq \alpha(t) \leq 1, \hspace{0.3cm} |\dot{\alpha}(t)| \lesssim w^{-1}, \hspace{0.3cm} |\ddot{\alpha}(t)| \lesssim w^{-2}
\end{align}
and $w$ is a universal number to be chosen later.
The first fact follows directly from the definition of $\theta$. To see the second and third fact, we use that $\theta$ satisfies $\sup_{\tau \in [0, 1]} |\theta^{(n)}(\tau)| \leq C_n$ for some universal constants $C_n$ and for all $n \geq 0.$ Indeed, note that for any $\tau \in [0, 1]$ and for any $n \geq 1,$
$$
\theta^{(n)} (\tau) = \e^{-\tan^2(\pi(\tau-1/2))} P_n\big(\tan(\pi(\tau-1/2))\big)
$$
where $P_n$ is a polynomial.

\pseudosection{The regularity estimates for $u$}
By \eqref{defgaccellemma} and \eqref{uacceldef}, $u$ is of the form $u=\cos(ky) G(t)$ where 
\begin{align}\label{biggdefaccel}
G(t):= e^{-g(t)}=\e^{-\big( \sqrt{b'} + (\sqrt{b}-\sqrt{b'})\alpha(t)\big) k t} .   
\end{align}

By \eqref{biggdefaccel}, 
\begin{align}\label{derivativesbigggg}
\dot G(t) = -\dot g(t)G(t), \hspace{0.5cm} \ddot G(t) = \left(-\ddot g(t) +\dot g(t)^2\right) G(t).    
\end{align}

We first bound $G(t)$. Recall that $0 \leq \alpha \leq 1$  by \eqref{propalphagreatgreatnewgreatnewfacts} and  $b \leq b'$ by assumption. Therefore,  $0 \leq (\sqrt{b'}-\sqrt{b})\alpha \leq (\sqrt{b'}-\sqrt{b})$, which implies
\begin{align}\label{ineqforgdotfirstaccelestimatebigg}
    \sqrt{b} \leq \left(\sqrt{b'} - (\sqrt{b'}- \sqrt{b} )\alpha \right) \leq \sqrt{b'}.
\end{align}
Hence, by \eqref{biggdefaccel}, 
\begin{align}\label{firstestimatebigg}
    |G(t)| \leq \e^{-k\sqrt{b}t}.
\end{align}

We now bound $\dot G(t)$. By \eqref{derivgaccellemma}, for $t \in [0, w],$ $\dot{g} = (\sqrt{b} - \sqrt{b'})\dot{\alpha} k t + \left(\sqrt{b'} + (\sqrt{b}- \sqrt{b'} )\alpha \right) k$. So, by \eqref{propalphagreatgreatnewgreatnewfacts} and by \eqref{ineqforgdotfirstaccelestimatebigg}, we get for $t \in [0, w],$ $|\dot{g}(t)| \lesssim k $ since $b,b' \in (\frac{1}{10}, 10)$. Therefore, by \eqref{derivativesbigggg} and by \eqref{firstestimatebigg}, we get for $t \in [0, w],$
\begin{align}\label{secondestimatesbigggg}
    |\dot{G}(t)| \lesssim k \e^{-k\sqrt{b}t}.
\end{align}

Finally, we bound $\ddot G$. By \eqref{derivgaccellemma}, for $t \in [0, w],$ $\ddot{g} = (\sqrt{b} - \sqrt{b'})\ddot{\alpha} k t +2 (\sqrt{b}- \sqrt{b'} )\dot{\alpha} k.$ So by \eqref{propalphagreatgreatnewgreatnewfacts}, we get for $t \in [0, w]$, $|\ddot g (t)| \lesssim \frac{k}{w}$ since $b, b'\in (\frac{1}{10}, 10).$ Therefore, by \eqref{derivativesbigggg}, we get for $t\in [0, w],$
\begin{align}\label{thirdestimatesbigggg}
    |\ddot G(t)| \lesssim \left(\frac{k}{w} + k^2\right)\e^{-k\sqrt{b}t} \lesssim k^2 \e^{-k\sqrt{b}t}
\end{align}
since $k \geq 1$ by assumption and since $w$ is a universal constant to be chosen later. Hence, the estimates from Proposition \ref{accelnewzeroc} are satisfied (i.e., the estimates \eqref{littlegaccelpropstate} with $c_1=1$).  

Moreover, by using that 
$$
\lim_{t \to 0}\alpha(t)=1, \quad \lim_{t \to w}\alpha(t)=0, \quad \lim_{t \to 0} \alpha^{(n)}(t)=0, \quad \lim_{t \to w} \alpha^{(n)}(t)=0.
$$
for $n \in \{1, 2\}$, we leave it to the reader to verify that the function $u=\cos(ky)\e^{-g(t)}$ (respectively all its first and second derivatives) converge to the function $\cos(ky)\e^{-k\sqrt{b}t}$ (respectively all its first and second derivatives) as $t \to 0$ and to the function $\cos(ky)\e^{-k\sqrt{b'}t}$ (respectively all its first and second derivatives) as $t \to w.$ 

\pseudosection{Uniform ellipticity and boundedness of $A$}
We will now show that the  matrix $A=\begin{pmatrix}
    a & 0 \\ 0& \tilde b
\end{pmatrix}$ belongs to the regularity class $R(80, 10)$. Before starting the estimation, we recall that $\tilde b=\frac{-\ddot g}{k^2} + \left( \frac{\dot g}{k}\right)^2$ (see \eqref{defbtildeaccellemma}) where $g$ is given by $g(t)=\left( \sqrt{b'} + (\sqrt{b}-\sqrt{b'})\alpha(t)\right) k t$ (see \eqref{defgaccellemma}) and $\alpha(t)=\theta\left(\frac{t}{w}\right)$ where $w$ is a universal constant to be chosen later. We also recall that  $\theta$ is defined in footnote \footnote{We choose $\theta(t):= 1-G(\tan(\pi(t-1/2)))$ where $G(x)=\frac{1}{\sqrt{\pi}}\int_{-\infty}^x \e^{-\eta^2} \ud{\eta}$. \label{footnotebuildingblockcoeffcbla}}. By \eqref{propalphagreatgreatnewgreatnewfacts}, the following holds: for $t \in [0, w],$
\begin{align}\label{propalphaacelforbigg}
    0 \leq \alpha(t) \leq 1, \hspace{0.3cm} |\dot{\alpha}(t)|\lesssim w^{-1}, \hspace{0.3cm}, |\ddot{\alpha}(t)|\lesssim w^{-2}.
\end{align}

\comment{
\begin{align}\label{estimationthetarecallnew}
    \sup_{\tau \in [0, 1]} |\theta^{(n)} (\tau)| \leq C_n
\end{align}
for some universal constants $C_n$ and for all $n \geq 0$. This fact will be used many times in the estimation below. To see it, note that for any $\tau \in [0, 1]$ and for any $n \geq 1,$
$$
\theta^{(n)} (\tau) = \e^{-\tan^2(\pi(\tau-1/2))} P_n\big(\tan(\pi(\tau-1/2))\big)
$$
where $P_n$ is a polynomial. 
\\
In the following, we prove the uniform ellipticity of $\begin{pmatrix}
    a & 0 \\
    0 & \tilde b
\end{pmatrix}$
}

\underline{We start with the ellipticity.} We first estimate $\tilde b := \frac{- \ddot{g}}{k^2} + \left(\frac{\dot{g}}{k}\right)^2.$ We will treat the two terms of separately.

\begin{itemize}
    \item \textbf{The first term:} $\frac{ \ddot{g}}{k^2}.$ By \eqref{derivgaccellemma}, $\ddot{g} = (\sqrt{b} - \sqrt{b'})\ddot{\alpha} k t +2 (\sqrt{b}- \sqrt{b'} )\dot{\alpha} k .$
    
    By definition, $\alpha(t)=\theta\left(\frac{t}{w}\right)$ and $b, b' \in (\frac{1}{10}, 10)$ by assumption. So, for $t \in [0, w],$ we have,
\begin{align}\label{firsttermbtildeaccellemma}
\left|\frac{\ddot g}{k^2} \right| \lesssim \frac{k w}{w^2 k^2} + \frac{k}{wk^2} \lesssim \frac{1}{kw} \ll     1
\end{align}
    since $w$ is a universal constant to be chosen later and since $k \gg 1$ by assumption.    

\item \textbf{The second term:} $\left(\frac{\dot{g}}{k}\right)^2$. Using the derivatives of $g$ \eqref{derivgaccellemma}, we get that
\begin{align}\label{gdotdefoverlambdanaccellema}
    \frac{\dot{g}}{k} = -(\sqrt{b'}-\sqrt{b})\dot{\alpha}t + \left(\sqrt{b'} - (\sqrt{b'}-\sqrt{b})\alpha \right).
\end{align}
By \eqref{ineqforgdotfirstaccelestimatebigg}, \comment{Since by definition \eqref{defalphaaccel},  $0 \leq \alpha \leq 1$ and since by assumption $b \leq b'$, we have that $0 \leq (\sqrt{b'}-\sqrt{b})\alpha \leq (\sqrt{b'}-\sqrt{b})$ and therefore, }
\begin{align}\label{ineqforgdotfirstaccel}
    \sqrt{b} \leq \left(\sqrt{b'} - (\sqrt{b'}- \sqrt{b} )\alpha \right) \leq \sqrt{b'}.
\end{align}
Also, since $\alpha$ is monotonically decreasing (because $\theta$ is monotonically decreasing by footnote \ref{footnotebuildingblockcoeffcbla}), and since we assumed that $b \leq b'$, we have that
\begin{align}\label{ineqforgdotsecondaccellemma}
    -(\sqrt{b'}-\sqrt{b})\dot{\alpha}t \geq 0.
\end{align}
Combining \eqref{gdotdefoverlambdanaccellema},  \eqref{ineqforgdotfirstaccel} and \eqref{ineqforgdotsecondaccellemma} implies that $\frac{\dot{g}}{k} \geq \sqrt{b}$, i.e.,  
\begin{align}\label{finalestimatesecondtermbtildeaccellemma}
    \left(\frac{\dot{g}}{k}\right)^2 \geq b.
\end{align}

In conclusion, since $\tilde b=\frac{-\ddot g}{k^2} + \left(\frac{\dot g}{k} \right)^2$ and since we showed that $\left|\frac{\ddot g}{k^2} \right| \ll 1$ in \eqref{firsttermbtildeaccellemma} and $\left(\frac{\dot{g}}{k}\right)^2 \geq b$ by \eqref{finalestimatesecondtermbtildeaccellemma}, we get that 
\begin{align}\label{ellipticityboundaccellemmabtilde}
    \tilde b \geq \frac{b}{2} \geq \frac{1}{80}.
\end{align}
since we assumed $b \geq \frac{1}{10}.$

\underline{We now look at boundedness.} First recall that by \eqref{defbtildeaccellemma}, $ \tilde b := \frac{- \ddot{g}}{k^2} + \left(\frac{\dot{g}}{k}\right)^2$. Moreover, by \eqref{gdotdefoverlambdanaccellema}, $\frac{\dot{g}}{k} = -(\sqrt{b'}-\sqrt{b})\dot{\alpha}t + \left(\sqrt{b'} - (\sqrt{b'}-\sqrt{b})\alpha \right).$ By \eqref{ineqforgdotfirstaccel},  $\sqrt{b}\leq \left(\sqrt{b'} - (\sqrt{b'}-\sqrt{b})\alpha \right) \leq \sqrt{b'}$ and by \eqref{ineqforgdotsecondaccellemma}, $-(\sqrt{b'}-\sqrt{b})\dot{\alpha}t \geq 0.$ Hence,
  
\begin{align}\label{boundedoneaccelnew}
    0 \leq \frac{\dot{g}}{k} \leq -(\sqrt{b'}-\sqrt{b})\dot{\alpha}t + \sqrt{b'}.
\end{align}

Since $\alpha(t)=\theta\left(\frac{t}{w}\right)$ and since $\theta$ is monotonically decreasing, we get for $t \in [0, w],$ 

\begin{align}\label{boundedtwoaccelnew}
-\dot \alpha(t) t =|\dot \alpha(t) t | \leq \sup_{\tau \in[0,1]} |\dot{\theta}| \leq \sqrt{\pi}    
\end{align}

by the definition of $\theta$ (see footnote \footnote{We choose $\theta(t):= 1-G(\tan(\pi(t-1/2)))$ where $G(x)=\frac{1}{\sqrt{\pi}}\int_{-\infty}^x \e^{-\eta^2} \ud{\eta}$. \label{footnotebuildingblockcoeffcblabla}}). So by \eqref{boundedoneaccelnew} and \eqref{boundedtwoaccelnew}, 
\begin{align}\label{ineqforboundednessaccelgdotlambdan}
    0 \leq \frac{\dot{g}}{k} \leq  (\sqrt{b'}-\sqrt{b})\sqrt{\pi} + \sqrt{b'}.
\end{align}
Therefore, since we assumed that $b, b' \in (\frac{1}{10}, 10),$ \eqref{ineqforboundednessaccelgdotlambdan} implies that
\begin{align}\label{boundednaccellemmaexplicitseventy}
    \left(\frac{\dot{g}}{k} \right)^2 \leq 70.
\end{align}
Since $\tilde b = \frac{- \ddot{g}}{k^2} + \left(\frac{\dot{g}}{k}\right)^2$ and since  we saw in \eqref{firsttermbtildeaccellemma} that $\left|\frac{ \ddot{g}}{k^2} \right| \ll 1$, we can therefore conclude that 
$    \tilde b \leq 80.
$
Combining \eqref{ellipticityboundaccellemmabtilde} with this last estimate, we showed 
$
\frac{1}{80} \leq \tilde b \leq 80.
$
Since  $A=\begin{pmatrix}
    a & 0 \\
    0 & \tilde b
\end{pmatrix}$ and since $a\in (\frac{1}{10}, 10)$, we get that the matrix $A$ satisfies the bounds
 \begin{align}\label{finalellipticityboundaccelllemmaunifaswell}
     \frac{1}{80} |\xi|^2 \leq (A\xi, \xi) \leq 80 |\xi|^2.
 \end{align}
\end{itemize}

\phantom{Test}\\\phantom{Test}

\textbf{The $C^1$ bound:} We are left with proving the $C^1$ bound for the matrix $A=\begin{pmatrix}
    a & 0 \\
    0 & \tilde b
\end{pmatrix}$. Since $a$ is a constant, we only have to estimate the derivatives of $\tilde b.$ Recall that $\tilde b$ was given in \eqref{defbtildeaccellemma} by 
$$
\tilde b := \frac{- \ddot{g}}{k^2} + \left(\frac{\dot{g}}{k}\right)^2
$$
and note that $\tilde b$ only depends on the $t$ variable. We have,
\begin{align}\label{derivativebtildeaccellemma}
    \dot{\tilde b} = \frac{-\dddot{g}}{k^2} + \frac{2 \dot{g} \ddot{g}}{k^2}.
\end{align}
Since $g$ was defined in \eqref{defgaccellemma} by
\begin{align*}
    g(t):=\left( \sqrt{b'} + (\sqrt{b}-\sqrt{b'})\alpha(t)\right) k t, 
\end{align*}
we immediately see that
\begin{align}\label{threederivativesgaccellemma}
     &\dot{g} = (\sqrt{b} - \sqrt{b'})\dot{\alpha} k t + \left(\sqrt{b'} + (\sqrt{b}- \sqrt{b'} )\alpha \right) k, \nonumber \\
   & \ddot{g} = (\sqrt{b} - \sqrt{b'})\ddot{\alpha} k t +2 (\sqrt{b}- \sqrt{b'} )\dot{\alpha} k  \\
   & \dddot{g}= (\sqrt{b} - \sqrt{b'})\dddot{\alpha} k t +3 (\sqrt{b}- \sqrt{b'} )\ddot{\alpha} k \nonumber.
\end{align}
We recall that $\alpha(t)$ was defined by
$     \alpha(t)=\theta\left(\frac{t}{w} \right)
$ for $t \in [0,w]$, where $w$ is a universal constant to be chosen later. We also recall that $\sup_{\tau \in [0, 1]} |\theta^{(n)}(\tau)| \leq C_n$ for some universal constants $C_n$ and for all $n \geq 0$ (see \eqref{propalphagreatgreatnewgreatnewfacts}). Therefore, we immediately see that the first term in \eqref{derivativebtildeaccellemma} is estimated, for $t \in [0, w],$ by 
\begin{align}\label{estiamtefirsttermbtildederaccelemmaddot}
    \left|\frac{\dddot{g}}{k^2}\right| \lesssim \frac{kw}{k^2 w^3} + \frac{k}{w^2k^2}\lesssim\frac{1}{w^2 k}  \ll 1
\end{align}
since $w$ is a universal constant and since $k \gg 1$ by assumption.

To estimate the second term $\frac{2\dot g \ddot g}{k^2}$ from \eqref{derivativebtildeaccellemma}, we use the bound \eqref{ineqforboundednessaccelgdotlambdan}
\begin{align}\label{oneothertermforconeaccellemma}
     \left|\frac{\dot{g}}{k}\right| \leq (\sqrt{b'}-\sqrt{b})\sqrt{\pi} + \sqrt{b'}
\end{align}
and the estimate
\begin{align}\label{oneoftheestimateforconeaccel}
   \left| \frac{\ddot{g}}{k}\right| \leq \frac{\sqrt{b'}-\sqrt{b}}{w}\left(\|\ddot{\theta}\|_{\infty} + 2\|\dot{\theta}\|_{\infty} \right).
\end{align}
which follows from \eqref{threederivativesgaccellemma}. Therefore, since by assumption $\frac{1}{10} \leq b, b' \leq 10$ and since by our choice of $\theta$ (see footnote \footnote{We choose $\theta(t):= 1-G(\tan(\pi(t-1/2)))$ where $G(x)=\frac{1}{\sqrt{\pi}}\int_{-\infty}^x \e^{-\eta^2} \ud{\eta}$. \label{footnotebuildingblockcoeffcblabla}}), we have $|\dot{\theta}| \leq \sqrt{\pi} $ and $|\ddot{\theta}| \leq 35$, we see that
\begin{align}\label{secondtermconebound}
 \left|\frac{2 \dot g \ddot g}{k^2} \right| \leq \frac{2}{w}\left[  (\sqrt{b'}-\sqrt{b})\sqrt{\pi} + \sqrt{b'}\right] \left[(\sqrt{b'}-\sqrt{b})\left(\|\ddot{\theta}\|_{\infty} + 2\|\dot{\theta}\|_{\infty} \right) \right]\leq 5  
\end{align}
by choosing $w:=400$. Since $\dot{\tilde b} = \frac{-\dddot{g}}{k^2} + \frac{2 \dot{g} \ddot{g}}{k^2}$ by \eqref{derivativebtildeaccellemma} and since $\left|\frac{\dddot{g}}{k^2}\right| \ll 1$ by \eqref{estiamtefirsttermbtildederaccelemmaddot}, we get using \eqref{secondtermconebound} that 
$$
|\dot{\tilde b}| \leq 10.
$$

Finally, by using that 
$$
\lim_{t \to 0}\alpha(t)=1, \quad \lim_{t \to w}\alpha(t)=0, \quad \lim_{t \to 0} \alpha^{(k)}(t)=0, \quad \lim_{t \to w} \alpha^{(k)}(t)=0.
$$
for $k \in \{1, 2\}$, we leave it to the reader to verify that the matrix $A=\begin{pmatrix}
    a & 0 \\
    0& \tilde b
\end{pmatrix}$ (respectively all its first derivatives) converge to the matrix $\begin{pmatrix}
    a & 0 \\
    0&  b \end{pmatrix}$ (respectively all its first derivatives) as $t \to 0$ and to the matrix $\begin{pmatrix}
    a & 0 \\
    0&  b' \end{pmatrix}$ (respectively all its first derivatives) as $t \to w.$ 
\\

\textbf{Conclusion:} We have been able to transform  
\begin{align*}
    \cos(k y) \e^{-k \sqrt{b}t} \hspace{0.5cm} \mbox{ into } \hspace{0.5cm} \cos(k y) \e^{-k \sqrt{b'}t}
\end{align*}
as $t$ goes from $0$ to $w:=400$, via a function $u$ solution to $\ddot u + \div(A\nabla u)=0$ where $A$ is in the regularity class $R(80, 10).$ We saw in \eqref{biggdefaccel} that $u=\cos(ky) G(t)$ for a smooth function $G$ which satisfies for $t \in [0, w]$ (see \eqref{firstestimatebigg}, \eqref{secondestimatesbigggg}, \eqref{thirdestimatesbigggg}) 

$$
|G(t)|\leq \e^{-k\sqrt{b}t}, \hspace{0.3cm} |\dot G(t)|\lesssim k\e^{-k\sqrt{b}t}, \hspace{0.3cm} |\ddot G(t)| \lesssim k^2\e^{-k\sqrt{b}t}.
$$
 Hence, the estimates from Proposition \ref{accelnewzeroc} are satisfied. This finishes the proof of the Proposition.
\end{proof}

\section{The super-exponential decay of a solution to a parabolic equation}
\label{parab}

In this section, we prove our last main result Theorem \ref{maintheoremparabolicbeg}, that we recall below, about the construction of a complex-valued, super-exponentially decaying solution to a parabolic equation $\dot{u} = \Delta u +B \nabla u$ in the half-cylinder $\T^2 \times \R^+$ with $B \in \mathbb C^2$ continuous and bounded.

\begin{thm*}[\ref{maintheoremparabolicbeg}, The parabolic setting]
    In the 3-dimensional cylinder $\T^2 \times \R^+$,  there exists a complex vector field $B \in \mathbb C^2$ which is continuous and uniformly bounded and there exists a non-zero, uniformly $C^2$ complex-valued function $u$ such that 
    \begin{align}
        \dot{u} = \Delta u + B\nabla u
    \end{align}
    and $u$ has double exponential decay: for any $T\gg 1,$
    \begin{align}
       \sup_{\T^2 \times \{t \geq T\} }|u(x,y,t)| \leq e^{-ce^{cT}}
    \end{align}
    for some numerical $c>0.$
\end{thm*}

\begin{rem}
    The vector field $B$ in the Theorem above is not uniformly $C^1$-smooth.
\end{rem}

As in the elliptic case above, we start with the building block. 
\begin{lemma}\label{buildingblockparabolic}
    Let $1\leq k \leq k' \leq k+10.$ Denote $u_1:=\e^{ikx} \e^{-k^2t}$ and $u_2:=\e^{ik'y} \e^{-(k')^2t}$ and note that they are both solutions to the same equation $\dot{u}=\Delta u$. 

    We can transform $u_1$ into $u_2$ via a solution $u$ to $\dot{u}=\Delta u + B\nabla u$, within the set $\T^2 \times [0, \frac{7}{2k}]$ where the vector field $B$ is continuous and uniformly bounded. For $t \in [0, \frac{7}{2k}]$, the function $u$ is of the form 
    \begin{align}\label{shapeubuildingblockparab}
        u(x,y,t) = f(t) \e^{ikx} + g(t) \e^{ik'y}
    \end{align}
    where $f, g \in C^2$ satisfy for all $0 \leq \alpha \leq 2,$
    $$
    |f^{(\alpha)}(t)|\lesssim k^{2\alpha} \e^{-k^2t}, \hspace{0.5cm}  |g^{(\alpha)}(t)|\lesssim (k')^{2\alpha} \e^{-(k')^2t}.
    $$
    Moreover, for $t \in [0, \frac{1}{2k}],$
    $$
    u = u_1, \hspace{0.5cm} B=0,
    $$
    and for $t \in [\frac{3}{k}, \frac{7}{2k}],$
    $$
    u = u_2, \hspace{0.5cm} B=0. 
    $$
\end{lemma}

\begin{proof}[Proof of Lemma \ref{buildingblockparabolic}]
    
\pseudosection{Step 1: adding a function} In $\T^2 \times [0, \frac{2}{k}],$ we will construct a $C^2$ function $\tilde u$ and a uniformly bounded continuous vector field $\tilde B \in \mathbb{C}^2$ such that $\dot{\tilde u}=\Delta \tilde u+\tilde B\nabla \tilde u$. 

We start with $\tilde u=u_1$ and $B=0$ on $\T^2 \times [0, \frac{1}{2k}].$  Denote by $\theta(t)$ the smooth function  going from 1 to 0 as $t$ goes from zero to one (see footnote \footnote{We choose $\theta(t):= 1-G(\tan(\pi(t-1/2)))$ where $G(x)=\frac{1}{\sqrt{\pi}}\int_{-\infty}^x \e^{-\eta^2} \ud{\eta}$ \label{footonethetaparabolicstepone}.}).

For $t \in [\frac{1}{2k}, \frac{3}{2k}]$, define $\alpha(t):=\theta(kt-\frac{1}{2}).$ The function $\alpha$ goes from 1 to 0 as $t$ goes from $\frac{1}{2k}$ to $\frac{3}{2k}.$ Define also the function 
    \begin{align}\label{defusteponeparabzeroc}
        u(x,y,t) = u_1+(1-\alpha)u_2
    \end{align}
    which goes from $u_1$ to $u_1+u_2$ as $t$ moves in $[\frac{1}{2k}, \frac{3}{2k}].$ We will now find a uniformly bounded vector field $\tilde B$ such that $\dot u = \Delta u+\tilde B\nabla u$ in $\T^2 \times [\frac{1}{2k}, \frac{3}{2k}].$ Since $\dot u_1=\Delta u_1$ and $\dot u_2=\Delta u_2$, and since $\alpha$ depends on $t$ only, it comes in $\T^2 \times [\frac{1}{2k}, \frac{3}{2k}],$
    \begin{align}\label{eqtparabchoosebnewstepone}
       \dot{u} = \dot u_1+(1-\alpha) \dot u_2-\dot{\alpha} u_2 = \Delta u_1+(1-\alpha)\Delta u_2 -\dot{\alpha} u_2 =  \Delta u -  \dot{\alpha} u_2.  
    \end{align}
We want to choose $\tilde B \in \mathbb{C}^2$ uniformly bounded so that $\dot{u} = \Delta u + \tilde B \nabla u$ is satisfied. We take $\tilde B:= \binom{\tilde B_1}{0}$. Hence, $\tilde B\nabla u = \tilde B_1 \partial_x u_1$ since $u=u_1(x,t) + \big(1-\alpha(t)\big)u_2(y,t)$. By \eqref{eqtparabchoosebnewstepone}, we then see that $\dot u = \Delta u + \tilde B\nabla u$ if 
\begin{align}\label{defbonestepineparaboliczeroc}
\tilde B_1:= \frac{-\dot{\alpha}u_2}{\partial_x u_1}.  
\end{align}

Finally, in $\T^2 \times [\frac{3}{2k}, \frac{2}{k}],$ we take $\tilde u=u_1+u_2$ and $\tilde B=0.$ In conclusion, in $\T^2 \times [0, \frac{2}{k}],$ we choose $\tilde u$ and $\tilde B$ to be
\begin{align}\label{completedefubsteponeparazeroc}
    \tilde u:= \begin{cases}
        u_1 & \mbox{ in } \T^2 \times [0, \frac{1}{2k}] \\
        u_1+(1-\alpha)u_2 & \mbox{ in } \T^2 \times [\frac{1}{2k}, \frac{3}{2k}] \\
        u_1+u_2 & \mbox{ in } \T^2 \times [\frac{3}{2k}, \frac{2}{k}]
    \end{cases}, \hspace{0.5cm} \tilde B:=  \begin{cases}
        0 & \mbox{ in } \T^2 \times [0, \frac{1}{2k}] \\
        \binom{\tilde B_1}{0} & \mbox{ in } \T^2 \times [\frac{1}{2k}, \frac{3}{2k}] \\
        0 & \mbox{ in } \T^2 \times [\frac{3}{2k}, \frac{2}{k}]
    \end{cases}.
\end{align}

\textbf{The regularity of $\tilde u$ and $\tilde B$}

Before discussing the regularity of $\tilde u$ and $\tilde B$, we state some useful facts about the function $\alpha(t)=\theta(kt-\frac{1}{2})$ that was introduced above:

\textbf{Facts:} For $t \in [\frac{1}{2k}, \frac{3}{2k}],$
\begin{align}\label{factsalphasteponeparabolic}
\begin{cases}
    0 \leq \alpha(t) \leq 1, \\
    |\dot{\alpha}(t)| \lesssim k,\\
    |\ddot{\alpha}(t)| \lesssim k^2,
\end{cases} \hspace{0.3cm}
\begin{cases}
    \lim_{t \rightarrow \frac{1}{2k}} \alpha(t)=1, \\
    \lim_{t \rightarrow \frac{1}{2k}} \dot\alpha(t)=0, \\
     \lim_{t \rightarrow \frac{1}{2k}} \ddot\alpha(t)=0 ,
\end{cases} \hspace{0.3cm}
\begin{cases}
    \lim_{t \rightarrow \frac{3}{2k}} \alpha(t)=0, \\
    \lim_{t \rightarrow \frac{3}{2k}} \dot\alpha(t)=0, \\
     \lim_{t \rightarrow \frac{3}{2k}} \ddot\alpha(t)=0.
\end{cases}
\end{align}
The first line of facts follows directly from the definition of $\theta$. To see the second and third lines of facts, we recall that $\alpha(t)=\theta(kt-\frac{1}{2})$ and that
 for any $\tau \in [0, 1]$ and for any $n \geq 1,$
$$
\theta^{(n)} (\tau) = \e^{-\tan^2(\pi(\tau-1/2))} P_n\big(\tan(\pi(\tau-1/2))\big)
$$
where $P_n$ is a polynomial. This also shows that $\theta$ satisfies $\sup_{\tau \in [0, 1]} |\theta^{(n)}(\tau)| \leq C_n$ for some universal constants $C_n$ and for all $n \geq 1.$

\begin{itemize}
    \item \textbf{The regularity of $\tilde u:$} First note that by the property of $\alpha$ given in the second and third column of \eqref{factsalphasteponeparabolic}, it is clear that $\tilde u$ (defined in \eqref{completedefubsteponeparazeroc}) is $C^2$ on $\T^2 \times [0, \frac{2}{k}].$ Also, by the definition of $\tilde u, u_1, u_2$ and $\alpha$ it is clear that 
    \begin{align}\label{shapetildeustepone}
        \tilde u(x,y,t) = f(t) \e^{ikx} +g(t) \e^{ik'y}
    \end{align}
on  $\T^2 \times [0, \frac{2}{k}]$ where 
 $$
 f(t):=\e^{-k^2t}, \hspace{0.5cm} g(t):=\begin{cases}
     0 & \mbox{ for } t \in [0, \frac{1}{2k}] \\
     (1-\alpha(t)) \e^{-(k')^2t} & \mbox{ for } t \in [\frac{1}{2k}, \frac{3}{2k}] \\
     \e^{-(k')^2t} & \mbox{ for } t \in [\frac{3}{2k}, \frac{2}{k}].
 \end{cases}
 $$
 In particular, $f, g$ are $C^2$ on $[0, \frac{2}{k}]$ by the properties for $\alpha$ \eqref{factsalphasteponeparabolic}. Moreover, since $1\leq k \leq k'$ and since  $|\dot{\alpha}|\lesssim k$ and $|\ddot \alpha| \lesssim k^2$, it is clear that $f, g$ satisfy
for all $0 \leq n \leq 2$ and for all $t \in [0, \frac{2}{k}],$
    \begin{align}\label{propfgsteponeparabzeroc}
    |f^{(n)}(t)|\lesssim k^{2n} \e^{-k^2t}, \hspace{0.5cm}  |g^{(n)}(t)|\lesssim (k')^{2n} \e^{-(k')^2t}.    
    \end{align}

    \item \textbf{The regularity of $\tilde B$:} Recall that $\tilde B$ was defined in \eqref{completedefubsteponeparazeroc} by 
    \begin{align*}
    \tilde B:=  \begin{cases}
        0 & \mbox{ in } \T^2 \times [0, \frac{1}{2k}] \\
        \binom{\tilde B_1}{0} & \mbox{ in } \T^2 \times [\frac{1}{2k}, \frac{3}{2k}] \\
        0 & \mbox{ in } \T^2 \times [\frac{3}{2k}, \frac{2}{k}]
    \end{cases}    
    \end{align*}
    where $\tilde B_1=\frac{-\dot{\alpha}u_2}{\partial_x u_1}$ (see \eqref{defbonestepineparaboliczeroc}) where $u_1=\e^{ikx} \e^{-k^2t}$ and where $u_2=\e^{ik'y}\e^{-(k')^2t}.$ By the properties of $\alpha$ \eqref{factsalphasteponeparabolic} and since $\partial_x u_1 =iku_1$, it is clear that $\tilde B$ is continuous on $[0, \frac{2}{k}]$. To see that it is also uniformly bounded, we recall that $|\dot \alpha| \lesssim k$ (see \eqref{factsalphasteponeparabolic}) and we note that
    $$
    |\tilde B_1| = \left|\frac{\dot{\alpha}u_2}{\partial_x u_1}\right| \lesssim k \left|\frac{u_2}{k u_1}\right|=\e^{-\left((k')^2-k^2\right)t} \lesssim 1
    $$
    since $t\geq 0$ and $k \leq k'$ by assumption. 
\end{itemize}
\textbf{Conclusion of Step 1:}
    In conclusion, in $\T^2 \times [0, \frac{2}{k}],$ we defined a $C^2$ functions $\tilde u$ and a uniformly bounded continuous vector field $\tilde B \in \mathbb C^2$ such that $\dot{\tilde u} = \Delta \tilde u + \tilde B\nabla \tilde u.$

    \pseudosection{Step 2: Removing a function}
In this step, we will construct a $C^2$ function $\tilde v$ and a uniformly bounded continuous vector field $\tilde C \in \mathbb C^2$ such that $\dot{\tilde v} = \Delta \tilde v + \tilde C \nabla \tilde v$ in $\T^2 \times [\frac{2}{k}, \frac{7}{2k}].$

 Denote by $\theta(t)$ the smooth function  going from 1 to 0 as $t$ goes from zero to one (see footnote \footnote{We choose $\theta(t):= 1-G(\tan(\pi(t-1/2)))$ where $G(x)=\frac{1}{\sqrt{\pi}}\int_{-\infty}^x \e^{-\eta^2} \ud{\eta}$ \label{footonethetaparabolicstepone}.}). 
    For $t \in [\frac{2}{k}, \frac{3}{k}]$, define $\alpha(t):=\theta(kt-2).$ The function $\alpha$ goes from 1 to 0 as $t$ goes from $\frac{2}{k}$ to $\frac{3}{k}.$ Define also the function 
    \begin{align}\label{defusteponeparabzeroc}
        v(x,y,t) = \alpha u_1+u_2
    \end{align}
    which goes from $u_1+u_2$ to $u_2$ as $t$ moves in $[\frac{2}{k}, \frac{3}{k}].$ We will now find a uniformly bounded vector field $\tilde C$ such that $\dot v = \Delta v+\tilde C\nabla v$ in $\T^2 \times [\frac{2}{k}, \frac{3}{k}].$ Since $\dot u_1=\Delta u_1$ and $\dot u_2=\Delta u_2$, and since $\alpha$ depends on $t$ only, it comes in $\T^2 \times [\frac{2}{k}, \frac{3}{k}],$
    \begin{align}\label{eqtparabchoosebnewsteponesteptwo}
       \dot{v} = \alpha \dot u_1+ \dot u_2+\dot{\alpha} u_1 = \alpha \Delta u_1+\Delta u_2 +\dot{\alpha} u_1 =  \Delta v +  \dot{\alpha} u_1.  
    \end{align}
We want to choose $\tilde C \in \mathbb{C}^2$ uniformly bounded so that $\dot{v} = \Delta v + \tilde C \nabla v$ is satisfied. We take $\tilde C:= \binom{0}{\tilde C_2}$. Hence, $\tilde C\nabla v = \tilde C_2 \partial_y u_2$ since $v=\alpha(t) u_1(x,t) + u_2(y,t)$. By \eqref{eqtparabchoosebnewsteponesteptwo}, we then see that $\dot v = \Delta v + \tilde C\nabla v$ if 
\begin{align}\label{defbonestepineparaboliczerocsteptwo}
\tilde C_2:= \frac{\dot{\alpha}u_1}{\partial_y u_2}.  
\end{align}

Finally, in $\T^2 \times [\frac{3}{k}, \frac{7}{2k}],$ we take $\tilde v=u_2$ and $\tilde C=0.$ In conclusion, in $\T^2 \times [\frac{2}{k}, \frac{7}{2k}],$ we choose $\tilde v$ and $\tilde C$ to be
\begin{align}\label{completedefubsteponeparazerocsteptwo}
    \tilde v:= \begin{cases}
        \alpha u_1+u_2 & \mbox{ in } \T^2 \times [\frac{2}{k}, \frac{3}{k}] \\
        u_2 & \mbox{ in } \T^2 \times [\frac{3}{k}, \frac{7}{2k}] 
    \end{cases}, \hspace{0.5cm} \tilde C:=  \begin{cases}
        \binom{0}{\tilde C_2} & \mbox{ in } \T^2 \times [\frac{2}{k}, \frac{3}{k}] \\
        0 & \mbox{ in } \T^2 \times [\frac{3}{k}, \frac{7}{2k}]
    \end{cases}.
\end{align}

\textbf{The regularity of $\tilde v$ and $\tilde C$}

Before discussing the regularity of $\tilde v$ and $\tilde C$, we recall some useful facts about the function $\alpha(t)=\theta(kt-2)$ that was introduced earlier (see footnote \ref{footonethetaparabolicstepone} for the definition of $\theta$):

\textbf{Facts:} For $t \in [\frac{2}{k}, \frac{3}{k}],$
\begin{align}\label{factsalphasteponeparabolicsteptwo}
\begin{cases}
    0 \leq \alpha(t) \leq 1, \\
    |\dot{\alpha}(t)| \lesssim k,\\
    |\ddot{\alpha}(t)| \lesssim k^2,
\end{cases} \hspace{0.3cm}
\begin{cases}
    \lim_{t \rightarrow \frac{2}{k}} \alpha(t)=1, \\
    \lim_{t \rightarrow \frac{2}{k}} \dot\alpha(t)=0, \\
     \lim_{t \rightarrow \frac{2}{k}} \ddot\alpha(t)=0 ,
\end{cases} \hspace{0.3cm}
\begin{cases}
    \lim_{t \rightarrow \frac{3}{k}} \alpha(t)=0, \\
    \lim_{t \rightarrow \frac{3}{k}} \dot\alpha(t)=0, \\
     \lim_{t \rightarrow \frac{3}{k}} \ddot\alpha(t)=0.
\end{cases}
\end{align}
The first line of facts follows directly from the definition of $\theta$ (see footnote \ref{footonethetaparabolicstepone}). To see the second and third lines of facts, we recall that $\alpha(t)=\theta(kt-2)$ and that
 for any $\tau \in [0, 1]$ and for any $n \geq 1,$
$$
\theta^{(n)} (\tau) = \e^{-\tan^2(\pi(\tau-1/2))} P_n\big(\tan(\pi(\tau-1/2))\big)
$$
where $P_n$ is a polynomial. This also shows that $\theta$ satisfies $\sup_{\tau \in [0, 1]} |\theta^{(n)}(\tau)| \leq C_n$ for some universal constants $C_n$ and for all $n \geq 1.$

\begin{itemize}
    \item \textbf{The regularity of $\tilde v:$} First note that by the property of $\alpha$ given in the second and third column of \eqref{factsalphasteponeparabolicsteptwo}, it is clear that $\tilde v$ (defined in \eqref{completedefubsteponeparazerocsteptwo}) is $C^2$ on $\T^2 \times [\frac{2}{k}, \frac{7}{2k}].$ Also, by the definition of $\tilde v, u_1, u_2$ and $\alpha$ it is clear that 
    \begin{align}\label{shapetildevparab}
    \tilde v(x,y,t) = f(t) \e^{ikx} +g(t) \e^{ik'y}    
    \end{align}
on  $\T^2 \times [\frac{2}{k}, \frac{7}{2k}]$ where 
 $$
 f(t):=\begin{cases}
     \alpha(t) \e^{-k^2t} & \mbox{ for } t \in [\frac{2}{k}, \frac{3}{k}] \\
     0 & \mbox{ for } t \in [\frac{3}{k}, \frac{7}{2k}] 
 \end{cases}, \hspace{0.5cm} g(t):= \e^{-(k')^2t}
 $$
 In particular, $f, g$ are $C^2$ on $[\frac{2}{k}, \frac{7}{2k}].$ Moreover, since $1\leq k$ and since  $|\dot{\alpha}|\lesssim k$ and $|\ddot \alpha| \lesssim k^2$, it is clear that $f, g$ satisfy
for all $0 \leq n \leq 2$ and for all $t \in [\frac{2}{k}, \frac{7}{2k}],$
    \begin{align}\label{propfgsteponeparabzerocsteptwo}
    |f^{(n)}(t)|\lesssim k^{2n} \e^{-k^2t}, \hspace{0.5cm}  |g^{(n)}(t)|\lesssim (k')^{2n} \e^{-(k')^2t}.    
    \end{align}

    \item \textbf{The regularity of $\tilde C$:} Recall that $\tilde C$ was defined in \eqref{completedefubsteponeparazerocsteptwo} by 
    \begin{align*}
     \tilde C:=  \begin{cases}
        \binom{0}{\tilde C_2} & \mbox{ in } \T^2 \times [\frac{2}{k}, \frac{3}{k}] \\
        0 & \mbox{ in } \T^2 \times [\frac{3}{k}, \frac{7}{2k}]
    \end{cases}.  
    \end{align*}
    where $\tilde C_2=\frac{\dot{\alpha}u_1}{\partial_y u_2}$ (see \eqref{defbonestepineparaboliczerocsteptwo}) where $u_1=\e^{ikx} \e^{-k^2t}$ and where $u_2=\e^{ik'y}\e^{-(k')^2t}.$ By the properties of $\alpha$ \eqref{factsalphasteponeparabolicsteptwo} and since $\partial_y u_2 =ik'u_2$, it is clear that $\tilde C$ is continuous on $[\frac{2}{k}, \frac{7}{2k}]$. To see that it is also uniformly bounded, we recall that $|\dot \alpha| \lesssim k$ (see \eqref{factsalphasteponeparabolicsteptwo}) and we note that for $t \in [\frac{2}{k}, \frac{3}{k}],$
    \begin{align}\label{uniformboundedctwosteptwoparab}
       |\tilde C_2| = \left|\frac{\dot{\alpha}u_1}{\partial_y u_2}\right| \lesssim k \left|\frac{u_1}{k' u_2}\right|\leq\e^{\left((k')^2-k^2\right)t} \leq \e^{3\left((k')^2-k^2\right)/k}
    \end{align}
    
    since $1\leq k \leq k'$ by assumption. By the mean value Theorem, $(k')^2-k^2 \leq 2 k'(k'-k)$. Therefore, 
    $$
    \frac{\left((k')^2-k^2\right)}{k} \leq 2\frac{k'}{k}(k'-k).
    $$
By assumption, $k'-k\leq 10.$ In particular, this implies $\frac{k'}{k} \leq 11$ since $k \geq 1$ by assumption. Hence, 
\begin{align}\label{mvtparabsteptwo}
  \frac{\left((k')^2-k^2\right)}{k} \lesssim 1
\end{align}
 and therefore by \eqref{uniformboundedctwosteptwoparab}, $\tilde C_2$ is uniformly bounded. Note that by \eqref{uniformboundedctwosteptwoparab} and \eqref{mvtparabsteptwo}, the function $\tilde C_2$ is not uniformly $C^1$ in $k, k'$ (for instance differentiating in $x$ yields a factor $k$ that cannot be absorbed).

\end{itemize}
\textbf{Conclusion of Step 2:}
    In conclusion, in $\T^2 \times [\frac{2}{k}, \frac{7}{2k}],$ we defined a $C^2$ functions $\tilde v$ and a uniformly bounded continuous vector field $\tilde C \in \mathbb C^2$ such that $\dot{\tilde v} = \Delta \tilde v + \tilde C\nabla \tilde v.$
  
\pseudosection{Step 3: The gluing}
We define a function $u$ and a vector field $B \in \mathbb C^2$ on $\T^2 \times [0, \frac{7}{2k}]$ by 
\begin{align}\label{defudefbfinalparabzroc}
   U= \begin{cases}
        \tilde u & \mbox{ in } \T^2 \times [0, \frac{2}{k}] \\
        \tilde v & \mbox{ in } \T^2 \times [\frac{2}{k}, \frac{7}{2k}]
    \end{cases}, \hspace{0.5cm} B= \begin{cases}
        \tilde B & \mbox{ in } \T^2 \times [0, \frac{2}{k}] \\
        \tilde C & \mbox{ in } \T^2 \times [\frac{2}{k}, \frac{7}{2k}]
    \end{cases}
\end{align}
where $\tilde u, \tilde B$ were introduced in step 1 \eqref{completedefubsteponeparazeroc} and $\tilde v, \tilde C$ were introduced in step 2 \eqref{completedefubsteponeparazerocsteptwo}. 

\begin{itemize}
    \item By step 1 and step 2, to see that $U \in C^2$, we only need to verify the regularity of $U$ at $t =\frac{2}{k}$. By \eqref{completedefubsteponeparazeroc}, $U=\tilde u = u_1+u_2$ in $\T^2 \times [\frac{3}{2k}, \frac{2}{k}]$. By \eqref{completedefubsteponeparazerocsteptwo}, $U=\tilde v= \alpha u_1+u_2$ in $\T^2 \times [\frac{2}{k}, \frac{3}{k}].$ By \eqref{factsalphasteponeparabolicsteptwo},
\begin{align}\label{recallpropalphastepthreezerocparab}
      \lim_{t \rightarrow \frac{2}{k}} \alpha(t)=1, \hspace{0.3cm}
    \lim_{t \rightarrow \frac{2}{k}} \dot\alpha(t)=0,\hspace{0.3cm}
     \lim_{t \rightarrow \frac{2}{k}} \ddot\alpha(t)=0.  
\end{align}

Therefore, $U$ is $C^2$ at $t=\frac{2}{k}$ and therefore is $C^2$ in $\T^2 \times [0, \frac{7}{2k}].$

By \eqref{shapetildeustepone} and \eqref{propfgsteponeparabzeroc} in step 1, $U=\tilde u = f(t)\e^{ikx} + g(t) \e^{ik'y}$ with $f, g \in C^2$ satisfying $|f^{(n)}(t)|\lesssim k^{2n} \e^{-kt}$ and $|g^{(n)}(t)|\lesssim (k')^{2n} \e^{-k't}$ for every $ t \in [0, \frac{2}{k}].$ By \eqref{shapetildevparab} and \eqref{propfgsteponeparabzerocsteptwo}, the same is true for $U=\tilde v$ in $\T^2 \times [\frac{2}{k}, \frac{7}{2k}].$ Hence, we can conclude that the function $U$ defined in \eqref{defudefbfinalparabzroc} is $C^2$ and satisfies $U=f(t)\e^{ikx} + g(t) \e^{ik'y}$ with $f, g \in C^2$ such that $|f^{(n)}(t)|\lesssim k^{2n} \e^{-kt}$ and $|g^{(n)}(t)|\lesssim (k')^{2n} \e^{-k't}$ for every $ t \in [0, \frac{7}{2k}].$

\item By step 1 and step 2, to see that $B$ is continuous, we only need to verify that $B$ is continuous at $t=\frac{2}{k}.$ By \eqref{completedefubsteponeparazeroc} in step 1, $B=\tilde B= 0$ in $\T^2 \times [\frac{3}{2k}, \frac{2}{k}]$. By \eqref{completedefubsteponeparazerocsteptwo} in step 2, $B=\tilde C=\binom{0}{\tilde C_2}$ in $\T^2 \times [\frac{2}{k}, \frac{3}{k}]$ where $\tilde C_2$ was defined in \eqref{defbonestepineparaboliczerocsteptwo} by $\tilde C_2 = \frac{\dot{\alpha}u_1}{\partial_y u_2}.$ By \eqref{recallpropalphastepthreezerocparab}, it is clear that $B$ is continuous at $t=\frac{2}{k}.$ Hence, we get that $B$ defined in \eqref{defudefbfinalparabzroc} is a continuous vector field in $\T^2 \times [0, \frac{7}{2k}].$ It is also uniformly bounded by the conclusion of steps 1 and 2. We also note that the equation $\dot u = \Delta u + B \nabla u$ is satisfied in $\T^2 \times [0, \frac{7}{2k}].$ And finally, for $t \in [0, \frac{1}{2k}],$ $U=u_1, B=0$ by \eqref{completedefubsteponeparazeroc} and for $t \in [\frac{3}{k}, \frac{7}{2k}],$ $U=u_2, B=0$ by \eqref{completedefubsteponeparazerocsteptwo}. This finishes the proof of Lemma \ref{buildingblockparabolic}.

\end{itemize}
 \end{proof}

As in the elliptic case, the next step is to prove a shifted version of Lemma \ref{buildingblockparabolic}.

\begin{lemma}\label{shiftedblockparab}
    Let $1\leq k \leq k' \leq k+10.$ Denote $u_1:=\e^{ikx} \e^{-k^2t}$ and $u_2:=\e^{ik'y} \e^{-(k')^2t}$ and note that they are both solutions to the same equation $\dot{u}=\Delta u$. 

    For any $c_1>0$ and for any $t_1\geq 0$, there exists $c_2>0$ such that we can transform $c_1u_1$ into $c_2u_2$ via a solution $u$ to $\dot{u}=\Delta u + B\nabla u$, within the set $\T^2 \times [t_1, t_1+\frac{7}{2k}]$ where the vector field $B$ is continuous and uniformly bounded. The constants $c_1, c_2$ are related by 
    $$
    c_1\e^{-k^2t_1} = c_2\e^{-(k')^2t_1}.
    $$    
    
    For $t \in [t_1, t_1+\frac{7}{2k}]$, the function $u$ is of the form 
    \begin{align}\label{shapeubuildingblockparab}
        u(x,y,t) = f(t) \e^{ikx} + g(t) \e^{ik'y}
    \end{align}
    where $f, g \in C^2$ satisfy for all $0 \leq \alpha \leq 2,$
    $$
    |f^{(\alpha)}(t)|\lesssim c_1k^{2\alpha} \e^{-k^2t}, \hspace{0.5cm}  |g^{(\alpha)}(t)|\lesssim c_2(k')^{2\alpha} \e^{-(k')^2t}.
    $$
    Moreover, for $t \in [t_1, t_1+\frac{1}{2k}],$
    $$
    u = c_1u_1, \hspace{0.5cm} B=0,
    $$
    and for $t \in [t_1+\frac{3}{k}, t_1+\frac{7}{2k}],$
    $$
    u = c_2u_2, \hspace{0.5cm} B=0. 
    $$
\end{lemma}

\begin{proof}[Proof of Lemma \ref{shiftedblockparab}]
    Consider the function $u$ and the vector field $B$ from Lemma \ref{buildingblockparabolic}. They satisfy $\dot u = \Delta u+B \nabla u$ where $B$ is continuous and uniformly bounded and $u$ transforms $\e^{ikx}\e^{-k^2t}$ into $\e^{ik'y} \e^{-(k')^2t}$ within the set $\T^2 \times [0, \frac{7}{2k}]$. 

    Let $t_1\geq 0.$ Denote 
    \begin{align}
        \tilde u(t):= u(t-t_1), \hspace{0.5cm} \tilde B(t):=B(t-t_1).
    \end{align}
Then $\tilde u$ transforms $\e^{ikx} \e^{-k^2(t-t_1)}$ into $\e^{ik'y} \e^{-(k')^2(t-t_1)}$ within the set $\T^2 \times [t_1, t_1+\frac{7}{2k}]$ and $\tilde B$ is continuous and uniformly bounded. Let $c_1>0$ and define $c_2$ by
\begin{align}\label{defconectwoparab}
    c_1\e^{-k^2t_1} = c_2\e^{-(k')^2t_1}.
\end{align}

Denote $U(x,y,t):=c_1 \e^{-k^2t_1} \tilde u(x,y,t).$ Then, using \eqref{defconectwoparab}, we see that $U$ transforms $c_1\e^{ikx} \e^{-k^2t}$ into $c_2 \e^{ik'y} \e^{-(k')^2t}$ within the set $\T^2 \times [t_1, t_1+\frac{7}{2k}]$ and $\tilde B$ is continuous and uniformly bounded. Moreover, $U$ solves $\dot U = \Delta U + \tilde B \nabla U$ in $\T^2 \times [t_1, t_1+\frac{7}{2k}]$ .

By Lemma \ref{buildingblockparabolic}, for $t \in [0, \frac{7}{2k}]$, the function $u$ is of the form 
    \begin{align}\label{shapeubuildingblockparab}
        u(x,y,t) = f(t) \e^{ikx} + g(t) \e^{ik'y}
    \end{align}
    where $f, g \in C^2$ satisfy for all $0 \leq \alpha \leq 2,$
    $$
    |f^{(\alpha)}(t)|\lesssim k^{2\alpha} \e^{-k^2t}, \hspace{0.5cm}  |g^{(\alpha)}(t)|\lesssim (k')^{2\alpha} \e^{-(k')^2t}.
    $$
Hence, the function $U=c_1 \e^{-k^2t_1} u(x,y,t-t_1)$, $t \in [t_1, t_1+\frac{7}{2k}]$ is of the form 
$$
U(x,y,t)=F(t) \e^{ikx} + G(t) \e^{ik'y}
$$
where $F(t)=c_1 \e^{-k^2t_1} f(t-t_1)$ and $G(t)=c_2 \e^{-(k')^2t_1} g(t-t_1)$ (using \eqref{defconectwoparab}). Clearly, $F, G \in C^2$ and they satisfy for all $0 \leq \alpha \leq 2,$
$$
|F^{(\alpha)}(t)|\lesssim c_1k^{2\alpha} \e^{-k^2t}, \hspace{0.5cm}  |G^{(\alpha)}(t)|\lesssim c_2(k')^{2\alpha} \e^{-(k')^2t}.
$$

Finally, since for $t \in [0, \frac{1}{2k}],$ $u(x,y,t)=\e^{ikx}\e^{-k^2t}, B=0$ and since for $t \in [\frac{3}{k}, \frac{7}{2k}]$ $u=\e^{ik'y} \e^{-(k')^2t}, B=0$, it is clear that $U$ and $\tilde B$ satisfy for $t \in [t_1, t_1+\frac{1}{2k}],$ $U(x,y,t)=c_1\e^{ikx}\e^{-k^2t}, \tilde B=0$ and  for $t \in [t_1+\frac{3}{k}, t_1+\frac{7}{2k}]$ $U(x,y,t)=c_2\e^{ik'y} \e^{-(k')^2t}, \tilde B=0$. This finishes the proof of Lemma \ref{shiftedblockparab}.
\end{proof}

We can now prove the main Theorem \ref{maintheoremparabolicbeg}. We recall it:
\begin{thm*}[\ref{maintheoremparabolicbeg}]
    In the 3-dimensional cylinder $\T^2 \times \R^+$, there exists a complex vector field $B \in \mathbb C^2$ which is continuous and uniformly bounded and there exists  a non-zero, uniformly $C^2$ complex-valued function $u$ such that 
    \begin{align}
        \dot{u} = \Delta u + B\nabla u
    \end{align}
    and $u$ has double exponential decay: for any $T\gg 1,$
    \begin{align}
        \sup_{\T^2 \times \{t \geq T\}}|u(x,y,t)| \leq e^{-ce^{cT}}
    \end{align}
    for some numerical $c>0.$
\end{thm*}

\begin{proof}[Proof of Theorem \ref{maintheoremparabolicbeg}]
    Let 
    \begin{align}\label{defknparabolicfull}
       k_n := n, \hspace{0.5cm} n\geq 1 .
    \end{align}
    Let $t_1:=0$ and for $n \geq 2$, consider the sequence of times 
    \begin{align}\label{seqeuncetimeparab}
        t_n:=\sum_{l=1}^{n-1} \frac{7}{2k_l}.
    \end{align}
    We decompose $\R^+$ as the almost disjoint union of intervals $[t_n, t_{n+1}].$ Let $c_1:=1$ and define the sequence $\{c_n\}$ by 
    \begin{align}\label{defcnfullparab}
        c_n \e^{-k_n^2t_n} = c_{n+1} \e^{-k_{n+1}^2 t_n}.
    \end{align}
By Lemma \ref{shiftedblockparab}, we can transform $c_n \e^{ik_nx} \e^{-k_n^2t}$ into $c_{n+1}\e^{ik_{n+1}y} \e^{-k_{n+1}^2t}$ within the set $\T^2 \times [t_n, t_{n+1}]$, via a solution $\tilde u_n(x,y,t)$ to $\dot{\tilde{u}}_n = \Delta \tilde u_n + \tilde B_n \nabla \tilde u_n$ where $\tilde B_n$ is continuous and uniformly bounded. In particular, by Lemma \ref{shiftedblockparab}, for  $t \in [t_n, t_n+\frac{1}{2k_n}],$ 
\begin{align}\label{tildeunbegfullparab}
    \tilde u_n(x,y,t) = c_n \e^{ik_nx} \e^{-k_n^2t}, \hspace{0.5cm} \tilde B_n=0
\end{align}
and for $t \in [t_{n+1}-\frac{1}{2k_n}, t_{n+1}],$
\begin{align}\label{tildeunendfullparab}
    \tilde u_n(x,y,t) = c_{n+1}\e^{ik_{n+1}y} \e^{-k_{n+1}^2t}, \hspace{0.5cm} \tilde B_n=0.
\end{align}
Moreover, by Lemma \ref{shiftedblockparab}, for $t \in [t_n, t_{n+1}],$ the function $\tilde u_n$ is of the form 
\begin{align}\label{shapetideuneparabfullproof}
    \tilde u_n(x,y,t) = f_n(t) \e^{ik_n x} + g_{n+1}(t) \e^{ik_{n+1}y}
\end{align}
where $f_n, g_{n+1}\in C^2$ satisfy for all $0\leq \alpha \leq 2$ and for all $t \in [t_n, t_{n+1}],$
\begin{align}\label{estimatefngnplusonefullparab}
    |f_n^{(\alpha)}(t)| \lesssim c_n (k_n)^{2\alpha} \e^{-k_n^2t}, \hspace{0.5cm} |g_{n+1}^{(\alpha)}(t)| \lesssim c_{n+1} (k_{n+1})^{2\alpha} \e^{-k_{n+1}^2t}.
\end{align}
For $t \in [t_n, t_{n+1}],$ define the functions 
\begin{align}\label{defunbnfullpproofprab}
    u_n(x,y,t):=\begin{cases}
        \tilde u_n(x,y,t), & n \mbox{ odd },\\
        \tilde u_n(y,x,t), & n \mbox{ even },
    \end{cases}, \hspace{0.5cm} B_n(x,y,t):= \begin{cases}
        \tilde B_n(x,y,t), & n \mbox{ odd }, \\
        \tilde B_n^t(y,x,t), & n \mbox{ even },
    \end{cases}
\end{align}
where if a vector $C=\binom{C_1}{C_2}$, we define  the vector $C^t:=\binom{C_2}{C_1}.$ We will now glue all these functions $u_n, B_n$ to get a function $u$ and a vector field $B$ on $\T^2 \times \R^+.$ For all $n \geq 1$ and for all $t \in [t_n, t_{n+1}],$ we define 
\begin{align}\label{defuandbparab}
    u(x,y,t):=u_n(x,y,t), \hspace{0.5cm} B(x,y,t):=B_n(x,y,t).
\end{align}

\pseudosection{The regularity of $B$}
We verify that the vector field $B$ defined in \eqref{defuandbparab} is continuous and uniformly bounded in $\T^2 \times \R^+$. By the definition of $B$ \eqref{defuandbparab}, by the definition of $B_n$ \eqref{defunbnfullpproofprab} and since $\tilde B_n$ is the vector field from Lemma \ref{shiftedblockparab}, it is clear that $B$ is continuous in $\T^2 \times (t_n, t_{n+1})$ and is uniformly bounded in $\T^2 \times \R^+.$ We only need to show that $B$ is continuous across the endpoints $t_n$, $n \geq 2.$ By the definition of $B$ \eqref{defuandbparab} (see also \eqref{defunbnfullpproofprab}) and by \eqref{tildeunbegfullparab}, \eqref{tildeunendfullparab}, for $n \geq 1,$
\begin{align}\label{biscontinuousunifbdedparab}
  \begin{cases}
   B= 0, & t \in [t_{n+1}-\frac{1}{2k_n}, t_{n+1}], \\
   B=0, & t \in [t_{n+1}, t_{n+1}+\frac{1}{2k_{n+1}}].
\end{cases}  
\end{align}

Therefore, $B$ is continuous across the endpoints and we conclude that $B$ is continuous and uniformly bounded in $\T^2 \times \R^+.$

\pseudosection{The regularity of $u$}
We will show that $u$ is uniformly $C^2$ in $\T^2 \times \R^+.$ The proof will be split into several steps.
\\

\textbf{Step 1: The regularity across $t_n$, $n \geq 2.$}
Without loss of generality, we assume $n$ is odd. By the definition of $u$ \eqref{defuandbparab} (see also \eqref{defunbnfullpproofprab}) and by \eqref{tildeunbegfullparab}, \eqref{tildeunendfullparab}, for $n \geq 1,$
$$
\begin{cases}
   u= c_{n+1} \e^{ik_{n+1}y} \e^{-k_{n+1}^2t}, & t \in [t_{n+1}-\frac{1}{2k_n}, t_{n+1}], \\
   u=c_{n+1} \e^{ik_{n+1}y} \e^{-k_{n+1}^2t}, & t \in [t_{n+1}, t_{n+1}+\frac{1}{2k_{n+1}}].
\end{cases}
$$
Therefore, $u$ is $C^2$ across the endpoints.
\\

\textbf{Step 2: The regularity of $u$ on $\T^2 \times [t_n, t_{n+1}]$.}
We assume again that $n$ is odd. By the definition of $u$ \eqref{defuandbparab} (see also \eqref{defunbnfullpproofprab}) and by \eqref{shapetideuneparabfullproof}, for $t \in [t_n, t_{n+1}],$
   $ u(x,y,t) = f_n(t) \e^{ik_n x} + g_{n+1}(t) \e^{ik_{n+1}y}$ with $f_n, g_{n+1} \in C^2.$ Therefore, $u$ is $C^2$ in $\T^2 \times [t_n, t_{n+1}].$

\textbf{Partial conclusion:} By step 1 and step 2 above, we can conclude that $u$ is $C^2$ in $\T^2 \times \R^+$. Moreover, by construction, $u$ solves $\dot u = \Delta u + B \nabla u$ on $\T^2 \times \R^+.$ Note that we do not claim yet that $u$ is \textit{uniformly} $C^2$.

\pseudosection{Step 3: The super-exponential decay of $u$ and the uniform $C^2$ boundedness.}
In this step, we show the following result: for any multi-index $\alpha \in \mathbb N^3$ of order $|\alpha| \leq 2,$ the following holds: for any $T \gg 1,$
\begin{align}\label{superexpforallderparab}
    \sup_{\T^2 \times \{t \geq T\}} |\partial^{\alpha}u(x,y,t)| \leq \e^{-c\e^{cT}}
\end{align}
for some numerical constant $c>0.$ This result shows that $u$ as well as all its derivatives of order 2 or less, decay super-exponentially at infinity. In particular, \eqref{superexpforallderparab} shows that $u$ has super-exponential decay at infinity and that $u$ is uniformly $C^2$ in $\T^2 \times \R^+.$

\begin{proof}[Proof of \eqref{superexpforallderparab}]
    Let $t \in [t_n, t_{n+1}]$ for some $n \geq 1.$ Suppose without loss of generality that $n$ is odd. By the definition of $u$ \eqref{defuandbparab} (see also \eqref{defunbnfullpproofprab}) and by \eqref{shapetideuneparabfullproof}, for $t \in [t_n, t_{n+1}],$
   \begin{align}\label{shapeuparabfinalgoodblabla}
   u(x,y,t) = f_n(t) \e^{ik_n x} + g_{n+1}(t) \e^{ik_{n+1}y}    
   \end{align}
    with $f_n, g_{n+1} \in C^2$ satisfying  for all $0\leq \beta \leq 2$ and for all $t \in [t_n, t_{n+1}],$
\begin{align*}
    |f_n^{(\beta)}(t)| \lesssim c_n (k_n)^{2\beta} \e^{-k_n^2t}, \hspace{0.5cm} |g_{n+1}^{(\beta)}(t)| \lesssim c_{n+1} (k_{n+1})^{2\beta} \e^{-k_{n+1}^2t}.
\end{align*}
    \\
    
    \textbf{Claim 1:} For all $t \gg 1$, i.e., for $t \in [t_n, t_{n+1}]$ with $n \gg 1,$ and for all $0\leq \beta \leq 2,$ we have 
\begin{align}\label{estimatesfordoubledecayparabolicfull}
    |f_n^{(\beta)}(t)| \lesssim   \e^{-\frac{7}{8}n(n-1)}, \hspace{0.5cm} |g_{n+1}^{(\beta)}(t)| \lesssim   \e^{-\frac{7}{8}n(n-1)}.
\end{align}
We postpone the proof of Claim 1 and we continue with the proof of \eqref{superexpforallderparab}. Let $T \gg 1,$ that is, $T \in [t_{n_1}, t_{n_1+1}]$ for some $n_1\gg 1.$ For any $t \geq T$, $t$ will be in some interval $[t_{n_2}, t_{n_2+1}]$, $n_2\geq n_1.$ By \eqref{shapeuparabfinalgoodblabla}, $
   u(x,y,t) = f_{n_2}(t) \e^{ik_{n_2} x} + g_{n_2+1}(t) \e^{ik_{n_2+1}y} $ (assuming $n_2$ is odd). Therefore, for any $0\leq \beta \leq 2,$ and for $t \in [t_{n_2}, t_{n_2+1}]$,
   \begin{align}\label{almostdoneprbfull}
       |\partial_t^{\beta}u(x,y,t)|\leq |f_{n_2}^{(\beta)}(t)| + |g_{n_2+1}^{(\beta)}(t)| \lesssim \e^{-\frac{7}{8}n_2(n_2-1)} \leq \e^{-\frac{7}{8}n_1(n_1-1)}
   \end{align}
since $n_2\geq n_1.$

Now, recall the definition $t_{n_1}=\sum_{l=1}^{n_1-1}\frac{7}{2k_l}=\sum_{l=1}^{n_1-1}\frac{7}{2l}$ since $k_l=l$ by definition (see \eqref{defknparabolicfull}). Therefore, since $T \leq t_{n_1+1},$ we get 
$$
T \leq \sum_{l=1}^{n_1}\frac{7}{2l} =\frac{7}{2}\left( 1+  \sum_{l=2}^{n_1}\frac{1}{l}\right)
\leq\frac{7}{2}\left(1+\int_{1}^{n_1}\frac{1}{x} \,dx\right) 
\leq 7 \ln(n_1).
$$
Therefore, $\e^{\frac{T}{7}}\leq n_1$, hence, by \eqref{almostdoneprbfull}, we get for any $0\leq \beta \leq 2$ and for any $t \geq T,$
$$
|\partial_t^{\beta}u(x,y,t)|\leq \e^{-\frac{7}{8}e^{\frac{T}{7}}}.$$
Since this holds for any $t\geq T$, we proved \eqref{superexpforallderparab} in the case where the multi-index $\alpha=(0, 0, \beta).$ By combining \eqref{shapeuparabfinalgoodblabla} with Claim 1 \eqref{estimatesfordoubledecayparabolicfull}, a similar argument ensures that \eqref{superexpforallderparab} holds for any multi-index $\alpha$ of order $|\alpha|\leq 2.$ This finishes the proof of \eqref{superexpforallderparab} assuming Claim 1.
\end{proof}

    We now prove Claim 1.
    \begin{proof}[Proof of Claim 1] By \eqref{shapeuparabfinalgoodblabla}, for $t \in [t_n, t_{n+1}]$ (assuming $n$ is odd), we have
    $$
    u(x,y,t) = f_n(t) \e^{ik_n x} + g_{n+1}(t) \e^{ik_{n+1}y}    $$
    
    with $f_n, g_{n+1} \in C^2$  satisfying  for all $0\leq \beta \leq 2$ and for all $t \in [t_n, t_{n+1}],$
\begin{align*}
    |f_n^{(\beta)}(t)| \lesssim c_n (k_n)^{2\beta} \e^{-k_n^2t}, \hspace{0.5cm} |g_{n+1}^{(\beta)}(t)| \lesssim c_{n+1} (k_{n+1})^{2\beta} \e^{-k_{n+1}^2t}.
\end{align*}
In particular, since $t \geq t_n$ we get 
\begin{align*}
    |f_n^{(\beta)}(t)| \lesssim c_n (k_n)^{2\beta} \e^{-k_n^2t_n}, \hspace{0.5cm} |g_{n+1}^{(\beta)}(t)| \lesssim c_{n+1} (k_{n+1})^{2\beta} \e^{-k_{n+1}^2t_n}.
\end{align*}
Recalling that $c_{n+1} = c_n \e^{(k_{n+1}^2-k_n^2)t_n}$ by definition \eqref{defcnfullparab}, we get
\begin{align}\label{estimatefngnplusonefullparabfullblabl}
    |f_n^{(\beta)}(t)| \lesssim c_n (k_n)^{2\beta} \e^{-k_n^2t_n}, \hspace{0.5cm} |g_{n+1}^{(\beta)}(t)| \lesssim c_n (k_{n+1})^{2\beta} \e^{-k_n^2t_n}.
\end{align}

\textbf{Claim 2:} For all $n \geq 1$, define $C_n:=c_n \e^{-k_n^2t_n}.$ Then, $C_n\leq \e^{\frac{-7}{4}(n-1)n}.$
\\

We postpone the proof of Claim 2 and continue with the proof of Claim 1 \eqref{estimatesfordoubledecayparabolicfull}. By \eqref{estimatefngnplusonefullparabfullblabl} and by Claim 2, we get
$$
|f_n^{(\beta)}(t)| \lesssim  (n)^{2\beta} \e^{-\frac{7}{4}n(n-1)}, \hspace{0.5cm} |g_{n+1}^{(\beta)}(t)| \lesssim  (n+1)^{2\beta} \e^{-\frac{7}{4}n(n-1)}
$$
where we used $k_n=n$ by definition (see \eqref{defknparabolicfull}). Therefore, for all $t \gg 1$, i.e., for $t \in [t_n, t_{n+1}]$ with $n \gg 1,$ and for all $0\leq \beta \leq 2,$ we get 
\begin{align}\label{estimatesfordoubledecayparabolicfullb}
    |f_n^{(\beta)}(t)| \lesssim   \e^{-\frac{7}{8}n(n-1)}, \hspace{0.5cm} |g_{n+1}^{(\beta)}(t)| \lesssim   \e^{-\frac{7}{8}n(n-1)}.
\end{align}
This proves Claim 1 assuming Claim 2.
\end{proof}

We now prove Claim 2.

\begin{proof}[Proof of Claim 2]
    We have $C_{n+1} = c_{n+1}\e^{-k_{n+1}^2t_{n+1}} =c_n \e^{(k_{n+1}^2-k_n^2)t_n} e^{-k_{n+1}^2t_{n+1}}$
where we used $c_{n+1} = c_n \e^{(k_{n+1}^2-k_n^2)t_n}$ by definition \eqref{defcnfullparab}. By definition \eqref{seqeuncetimeparab}, $t_{n+1}=t_n+\frac{7}{2k_n}.$ Therefore, $$C_{n+1} = c_n\e^{(k_{n+1}^2-k_n^2)t_n} e^{-k_{n+1}^2(t_n+\frac{7}{2k_n})} =c_n \e^{-k_n^2t_n} \e^{\frac{-7k_{n+1}^2}{2k_n}}=C_n\e^{\frac{-7k_{n+1}^2}{2k_n}}.
$$
By definition \eqref{defknparabolicfull}, $k_{n+1}=k_n+1\geq k_n$. Therefore, 
$$
C_{n+1}\leq C_n \e^{\frac{-7k_n}{2}}.
$$
Since $k_n=n$, we get for all $n \geq 1,$
$$
C_n \leq C_1\e^{\frac{-7}{2} \sum_{l=1}^{n-1}l} = \e^{\frac{-7n(n-1)}{4}}
$$
since $C_1=c_1 \e^{-k_1^2t_1}$ and since $c_1=1$ and $t_1=0$ by definition (see \eqref{defcnfullparab} and \eqref{seqeuncetimeparab}). This finishes the proof of Claim 2. 
\end{proof}

\pseudosection{What we proved:}
\begin{itemize}
    \item We constructed a function $u$ and a vector field $B$ (see \eqref{defuandbparab}).
    \item In \eqref{biscontinuousunifbdedparab}, we showed that $B$ is continuous and uniformly bounded in $\T^2 \times \R^+.$
    \item In Step 3 (see \eqref{superexpforallderparab}) we showed that $u$ is uniformly $C^2$ in $\T^2 \times \R^+$ and that it has super-exponential decay.
    \item In the partial conclusion of Step 2, we showed that $u$ solves $\dot u=\Delta u+B \nabla u$ in $\T^2 \times \R^+$.
\end{itemize}
This finishes the proof of Theorem \ref{maintheoremparabolicbeg}.
\end{proof}

\section{Appendix}

\subsection{Proof of Lemma \ref{PML1}}\label{theboundsforaanducounterexampleappendix}
We recall Lemma \ref{PML1} for the reader's convenience.
Let $\theta(t)$ be a smooth function which is equal to 1 for $t\leq 0$ and to zero for $t\geq 1$ and which monotonically decreases for $t\in(0, 1)$ (see footnote \footnote{We choose $\theta(t):= 1-G(\tan(\pi(t-1/2)))$ where $G(x)=\frac{1}{\sqrt{\pi}}\int_{-\infty}^x \e^{-\eta^2} \ud{\eta}$. \label{footnotebuildingblockcoeffcblanew}} and   Figure \ref{fig:theta}).

\begin{lemma*}[\ref{PML1}]

Consider two harmonic functions in $\T^2\times \mathbb{R}$ $$f:=\cos(k x)e^{-k t} \textup{\quad and \quad} g:=\cos(k'y)e^{-k' t}$$
and assume that the positive numbers $k, k',w >0$ satisfy 
\begin{equation} \label{assumptionsproof}
0<k'-k\lesssim w^{-1}\lesssim k.
\end{equation}
Put $\alpha(t) := \theta(t/w)$ and
\begin{equation}
    \label{defn_uproof}
    u := \begin{cases} f+\big(1-\alpha(t)\big)g, & t\leq w\\
 \alpha(t-w)f+g,&t \geq w.
 \end{cases}
\end{equation}

Then the following holds.
\begin{enumerate}
    \item[a)] There exists a $C^1$ smooth matrix-valued function $A$ such that $u$ solves  $$\ddot u+ \div(A\nabla u)=0 \hspace{0.5cm} \textup{ on } \quad  \T^2 \times \R$$
    \item[b)] and
\begin{equation}
\label{propsproof}
    \|A-Id\|\lesssim \frac{1}{wk}, \hspace{0.5cm} \|\nabla A\|\lesssim \frac{1}{w}, \hspace{0.5cm} \| \dot A \| \lesssim \frac{1}{w} \quad \textup{ on } \quad \T^2 \times \R. 
\end{equation}

\item[c)] The solution $u$ satisfies the following estimates on $\T^2\times [0, 2w]$: For any multi-index $\alpha \in \mathbb N^3$ of order $|\alpha|\leq 2$,
$$
|\partial^{\alpha} u(x,y,t)| \lesssim (k')^{|\alpha|} \sup_{\T^2 \times [0, 2w]} |u|.
$$
\end{enumerate}
\end{lemma*}

\begin{proof}[Proof of Lemma \ref{PML1}]

    \pseudosection{Step 1: The construction of $A$}
 \begin{itemize}
     \item \textbf{When $t \in [0, w]$:}  
We want a matrix $A_1$ such that $\ddot u+\div(A_1\nabla u)=0$. We look for $A_1$  in the form $A_1=Id+A'_1$ for some $A'_1$ to be chosen. With this choice, the equation $\ddot u+\div(A_1\nabla u)=0$ becomes  
\begin{align}\label{eqtlemmatwoonestepone}
\ddot u + \Delta u + \div(A'_1\nabla u)=0.    
\end{align}

Since $t\in [0, w],$ $u=f+\big(1-\alpha(t)\big)g$ and since $\alpha$ depends only on $t$, we get
$$
\Delta u= \Delta f + (1-\alpha)\Delta g, \hspace{0.5cm} \ddot u = \ddot f +(1-\alpha)\ddot g -2\dot \alpha \dot g - \ddot \alpha g.
$$
Since $f, g$ are harmonic, $\ddot f+\Delta f=\ddot g+\Delta g=0,$ \eqref{eqtlemmatwoonestepone} becomes
\begin{align*}
    \div(A'_1\nabla u)=2\dot \alpha \dot g + \ddot \alpha g.
\end{align*}
Since $g=\cos(k'y)\e^{-k't},$ we get $2\dot \alpha \dot g + \ddot \alpha g =(-2k'\dot \alpha  + \ddot \alpha)g :=\beta_1(t)g $ and
\begin{align}\label{eqtforaprimesteponetwoone}
    \div(A'_1\nabla u)= \beta_1(t)g , \hspace{0.5cm} \beta_1(t):=-2k'\dot \alpha  + \ddot \alpha. 
\end{align}
    Note that in \eqref{eqtforaprimesteponetwoone}, the $2\times 2$ matrix $A'_1$ is the unknown and notice that since $\div, \nabla$ only involve space derivatives, we can consider the variable $t$ as fixed. The next Lemma ensures the existence of the matrix $A'_1$ satisfying \eqref{eqtforaprimesteponetwoone} and such that the bounds \eqref{propsproof} for the matrix $A_1=Id+A'_1$ will be satisfied.

    \begin{lemma*}[\ref{claim2Dtwopointoneproof}]
Let $1\leq k \leq k' \leq 2k.$ Given a positive number $s$ and a function $u = \cos(k x)+s\cos(k'y)$ on $\T^2$, there exists a smooth vector field $V$ with divergence $\div(V)=\cos(k'y)$ such that the matrix $A_s:= \begin{pmatrix}a_s(x,y)&b_s(x,y)\\b_s(x,y)&0\end{pmatrix}$ uniquely defined by $A_s\nabla u=V$ is smooth and satisfies 
\begin{align}\label{estimatestosatisfyforas}
        \|A_s\|_{\infty} \lesssim \frac{(1+|s|)}{k^2}, \hspace{0.5cm} \|\nabla A_s\|_{\infty} \lesssim \frac{(1+|s|)}{k}
    \end{align}
    where the gradient only involves spatial derivatives.
 Moreover, the function $s\mapsto A_s$ satisfies
    $$
    \bigg\|\frac{\partial A_s}{ \partial s}\bigg\|_{\infty} \lesssim \frac{1}{k^2}.
    $$
\end{lemma*}

We postpone the proof of Lemma \ref{claim2Dtwopointoneproof} to Appendix \ref{prooftwodlemmaappendix} and we continue with the proof of Lemma \ref{PML1}. Recall that $u=\cos(kx)\e^{-kt}+\big(1-\alpha\big) \cos(k'y)\e^{-k't},$ which we can rewrite as
$$
u=\e^{-kt}\left(\cos(kx)+(1-\alpha) \e^{(k-k')t}\cos(k'y)\right).
$$
Denoting $s_1(t):=(1-\alpha) \e^{(k-k')t}$, we get
$$
u=\e^{-kt}\left(\cos(kx)+s_1(t)\cos(k'y)\right).
$$
Considering the variable $t$ fixed (and denoting $s_1:=s_1(t)$), we are in the setting of Lemma \ref{claim2Dtwopointoneproof}: there exists a matrix-valued function $A_{s_1}$ such that 
$$
\div(A_{s_1}\nabla(u\e^{kt})) = \cos(k'y).
$$
Since we want a matrix $A'_1$ such that $\div(A'_1\nabla u)=\beta_1(t)\cos(k'y)\e^{-k't},$ (by \eqref{eqtforaprimesteponetwoone}),
we choose 
\begin{align}
    A'_1:=\beta_1(t)\e^{(k-k')t}A_{s_1}.
\end{align}

In conclusion, for $t \in [0, w]$ the matrix $A_1$ such that $\ddot u+\div(A_1\nabla u)=0$ in $\T^2\times [0, w]$ is given by 
\begin{align}\label{defastepone}
A_1:=Id+\beta_1(t)\e^{(k-k')t}A_{s_1}.
\end{align}

\item \textbf{When $t \in [w, 2w]$:} We will be brief since the construction is very close to the case $t\in [0, w].$ We want a matrix $A_2$ such that $\ddot u+\div(A_2\nabla u)=0.$ We again look for $A_2$ in the form $A_2=Id+A'_2.$ Since $t \in [w, 2w],$ $u=\alpha(t-w)f+g.$ Using again that $f, g$ are harmonic, the equation that $A'_2$ should satisfy is
\begin{align}\label{eqtforaprimesteponetwoonetwo}
    \div(A'_2\nabla u)=\beta_2(t)f, \hspace{0.5cm} \beta_2(t):=2k \dot \alpha - \ddot \alpha.
\end{align}
This should be compared to \eqref{eqtforaprimesteponetwoone} when $t \in [0, w].$ We again use a two-dimensional Lemma to ensure the existence of the matrix $A'_2$ satisfying \eqref{eqtforaprimesteponetwoonetwo} and such that the bounds \eqref{propsproof} for $A_2=Id+A'_2$ will be satisfied.

\begin{lemma*}[\ref{lemmabistwopointnineinproofoftwopointne}]
Let $1\leq k \leq k' \leq 2k.$  Given a positive number $s$ and a function $v = s\cos(k x)+\cos(k'y)$ on $\T^2$, there exists a smooth vector field $V$ with divergence $\div(V)=\cos(kx)$ such that the matrix  $A_s = \begin{pmatrix}0&b_s(x,y)\\b_s(x,y)&c_s(x,y)\end{pmatrix}$ uniquely defined by $A_s \nabla v=V$ is smooth and satisfies \begin{align*}
        \|A_s\|_{\infty} \lesssim \frac{(1+|s|)}{k^2}, \hspace{0.5cm} \|\nabla A_s\|_{\infty} \lesssim \frac{(1+|s|)}{k}
    \end{align*}
    where we recall that the gradient only involves spatial derivatives.
     Moreover, the function $s\mapsto A_s$ satisfies
    $$
    \bigg\|\frac{\partial A_s}{ \partial s}\bigg\|_{\infty} \lesssim \frac{1}{k^2}.
    $$
\end{lemma*}
We postpone the proof of Lemma \ref{lemmabistwopointnineinproofoftwopointne} to Appendix \ref{prooftwodlemmaappendix} and we continue with the proof of Lemma \ref{PML1}. Since $u=\alpha \cos(kx)\e^{-kt}+\cos(k'y)\e^{-k't},$ we can rewrite it as
$$
u=\e^{-k't}\left(\alpha \e^{(k'-k)t}\cos(kx)+\cos(k'y) \right).
$$
Denoting $s_2(t):=\alpha \e^{(k'-k)t}$, we get
$$
u=\e^{-k't}\left(s_2(t)\cos(kx)+\cos(k'y) \right).
$$
Considering the $t$ variable fixed (and denoting $s_2:=s_2(t)$), we are in the setting of Lemma \ref{lemmabistwopointnineinproofoftwopointne}: there exists a matrix-valued function $A_{s_2}$ such that 
$$
\div(A_{s_2}\nabla (u\e^{k't}))=\cos(kx).
$$
Since we want a matrix $A'_2$ such that $\div(A'_2\nabla u)=\beta_2(t)\cos(kx)\e^{-kt}$ (by \eqref{eqtforaprimesteponetwoonetwo}), we choose
\begin{align}\label{defaprimetwopointone}
    A'_2:=\beta_2(t)\e^{(k'-k)t}A_{s_2}.
\end{align}

In conclusion, for $t \in [w, 2w],$ the matrix $A_2$ such that $\ddot u+\div(A_2\nabla u)=0$ in $\T^2\times [w, 2w]$ is given by 
\begin{align}\label{defatwotwopointone}
    A_2:=Id+\beta_2(t)\e^{(k'-k)t}A_{s_2}.
\end{align}

\textbf{The matrix $A$ in $\T^2 \times \R$:}
Recall the definition of $u$ \eqref{defn_uproof}:
\begin{align}\label{defufiulltwoone}
    u := \begin{cases} f+\big(1-\alpha(t)\big)g, & t\leq w\\
 \alpha(t-w)f+g,&t \geq w.
 \end{cases}  
\end{align}
In particular, for $t \leq 0$, $u=f$ and for $t\geq 2w,$ $u=g$, both being harmonic. Hence we define the matrix $A$ in $\T^2\times \R$ by 
\begin{align}\label{defafulllcylindertwoone}
    A:=\begin{cases}
        Id & \mbox{ in } \T^2\times (-\infty, 0], \\
        A_1 & \mbox{ in } \T^2\times [0, w], \\
        A_2 & \mbox{ in } \T^2\times [w, 2w], \\
        Id & \mbox{ in } \T^2\times [2w, +\infty).
    \end{cases}
\end{align}
By the construction above, we know that $\ddot u+\div(A\nabla u)=0$ is satisfied in each of these four cases. We need to make sure that $A$ is $C^1$ in $\T^2 \times \R$ (and in particular at $t=0, t=w, t=2w$). We also need to make sure that the bounds from \eqref{propsproof} and that we recall now are satisfied:
\begin{equation}\label{estimateforatwooneproof}
    \|A-Id\|\lesssim \frac{1}{wk}, \hspace{0.5cm} \|\nabla A\|\lesssim \frac{1}{w}, \hspace{0.5cm} \| \dot A \| \lesssim \frac{1}{w} \quad \textup{ on } \quad \T^2 \times \R. 
\end{equation}
Finally, we also need to make sure that $u$ is $C^2$ and that the claimed bound for $u$ that we recall below is satisfied:
$$
|\partial^{\alpha} u(x,y,t)| \lesssim (k')^{|\alpha|} \sup_{\T^2 \times [0, 2w]} |u|.
$$ 
This is the content of the next steps.
 \end{itemize}

Before discussing the regularity of $A$ and $u$, we state some useful facts about the function $\alpha(t):=\theta(t/w)$ for $t \in [0, w]$ (see footnote \footnote{We choose $\theta(t):= 1-G(\tan(\pi(t-1/2)))$ where $G(x)=\frac{1}{\sqrt{\pi}}\int_{-\infty}^x \e^{-\eta^2} \ud{\eta}$. \label{footnotebuildingblockcoeffcblanew}} and   Figure \ref{fig:theta} for the definition of $\theta$). The same facts will be true for the function $ \alpha(t-w)$ and $t \in [w, 2w]$.
\\

\textbf{Facts:} For $t \in [0, w],$
\begin{align}\label{factsalphasteponeparabolictwopointone}
\begin{cases}
    0 \leq \alpha(t) \leq 1, \\
    |\dot{\alpha}(t)| \lesssim w^{-1},\\
    |\ddot{\alpha}(t)| \lesssim w^{-2},
\end{cases} \hspace{0.3cm}
\begin{cases}
    \lim_{t \rightarrow 0} \alpha(t)=1, \\
    \lim_{t \rightarrow 0} \dot\alpha(t)=0, \\
     \lim_{t \rightarrow 0} \ddot\alpha(t)=0 ,
\end{cases} \hspace{0.3cm}
\begin{cases}
    \lim_{t \rightarrow w} \alpha(t)=0, \\
    \lim_{t \rightarrow w} \dot\alpha(t)=0, \\
     \lim_{t \rightarrow w} \ddot\alpha(t)=0.
\end{cases}
\end{align}

The first line of facts follows directly from the definition of $\theta$. To see the second and third lines of facts, we recall that $\alpha(t)=\theta(t/w)$ and that
 for any $\tau \in [0, 1]$ and for any $n \geq 1,$
\begin{align}\label{thetanallderivativezerowblablablabla}
\theta^{(n)} (\tau) = \e^{-\tan^2(\pi(\tau-1/2))} P_n\big(\tan(\pi(\tau-1/2))\big)    
\end{align}

where $P_n$ is a polynomial. This also shows that $\theta$ satisfies $\sup_{\tau \in [0, 1]} |\theta^{(n)}(\tau)| \leq C_n$ for some universal constants $C_n$ and for all $n \geq 1.$

\pseudosection{Step 2: The estimates \eqref{estimateforatwooneproof} and the regularity for $A$}
 From \eqref{defafulllcylindertwoone}, when $t \leq w,$
\begin{align}\label{recalaozerowregtwoone}
    A:=\begin{cases}
        Id & \mbox{ in } \T^2\times (-\infty, 0], \\
        A_1 & \mbox{ in } \T^2\times [0, w]
        \end{cases}, \hspace{0.5cm} A_1=Id+\beta_1(t)\e^{(k-k')t}A_{s_1}, \hspace{0.5cm} \beta_1=-2k'\dot\alpha + \ddot\alpha 
\end{align}
where, by Lemma \ref{claim2Dtwopointoneproof}, $A_{s_1}$ satisfies for $t\in [0, w],$
\begin{align}\label{boundsasonezerowtwoone}
 \|A_{s_1}\|_{\infty} \lesssim \frac{(1+|s_1|)}{k^2}, \hspace{0.5cm} \|\nabla A_{s_1}\|_{\infty} \lesssim \frac{(1+|s_1|)}{k}, \hspace{0.5cm} s_1(t)=(1-\alpha) \e^{(k-k')t}.    
\end{align}

\begin{itemize}
    \item \textbf{ We show that $A$ is continuous :} We first show that $A$ given in \eqref{recalaozerowregtwoone} is continuous at $t=0$. For $t \in [0,w]$, $s_1(t)=(1-\alpha) \e^{(k-k')t}$. By assumption, $k'-k>0$. Hence, for $t\in [0, w],$ $\e^{(k-k')t}\leq 1$. Therefore, using $0\leq \alpha \leq 1$ by \eqref{factsalphasteponeparabolictwopointone}, we get for $t \in [0,w],$
\begin{align}\label{zerowsonebounded}
|s_1(t)|\lesssim 1    
\end{align}
Hence $\|A_{s_1}\|_{\infty}\lesssim \frac{1}{k^2}$ by \eqref{boundsasonezerowtwoone}. Using again $k<k'$, we get for $t\in [0,w],$
\begin{align}\label{aoneisidplusbonetimebounded}
    \|\e^{(k-k')t}A_{s_1}\|_{\infty}\lesssim \frac{1}{k^2}.
\end{align}

Since $\beta_1=-2k'\dot \alpha+\ddot \alpha$ (see \eqref{recalaozerowregtwoone}) and since the derivatives of $\alpha$ converge to zero when $t \to 0$ by \eqref{factsalphasteponeparabolictwopointone}, we get  $\lim_{t \to 0}\beta_1(t) =0.$ Hence, since $A=Id+\beta_1(t) \e^{(k-k')t}A_{s_1}$, we get 
$\lim_{t \to 0^+}A=Id$
and $A$ is continuous at $t=0.$ The same argument shows that 
$\lim_{t \rightarrow w^-} A=Id.$

When $t \in [w, 2w],$ $A:=Id+\beta_2(t)\e^{(k'-k)t}A_{s_2}$ (see \eqref{defafulllcylindertwoone} and \eqref{defatwotwopointone}) where $A_{s_2}$ satisfies estimates similar to \eqref{boundsasonezerowtwoone} with $s_2$ in place of $s_1$ (by Lemma \ref{lemmabistwopointnineinproofoftwopointne}) and where $s_2(t)=\alpha(t-w) \e^{(k'-k)t}$ . We recall also $\beta_2(t):=2k \dot \alpha(t-w) - \ddot \alpha(t-w) $ by \eqref{eqtforaprimesteponetwoonetwo}. To see that 
$$
\lim_{t \rightarrow w+} A=Id, \hspace{0.5cm} \lim_{t \rightarrow 2w^-} A=Id,
$$
we can proceed as above. The only difference is in the estimation of $\e^{(k'-k)t}.$ Since $k'-k>0,$ the argument used above does not work anymore but we recall that by assumption $k'-k\lesssim w^{-1}.$ Hence, for $t\in [w, 2w],$ we again have $\e^{(k'-k)t}\lesssim 1.$

Finally, Lemma \ref{claim2Dtwopointoneproof} ensures that $A$ is continuous in $(0, w)$ and Lemma \ref{lemmabistwopointnineinproofoftwopointne} ensures that $A$ is continuous in $(w, 2w).$ \underline{In conclusion, $A$ is continuous in $\T^2\times \R$.}

\item \textbf{We show that  $\|A-Id\|_{\infty} \lesssim \frac{1}{wk}$}: \quad For $t \in [0, w]$,  
\begin{align}\label{amoinsidzerowtwoone}
  \|A-Id\|_{\infty}=\|A_1-Id\|_{\infty} =\|\beta_1(t)e^{-(k'-k) t}A_{s_1} \|_{\infty} 
\end{align}
  by \eqref{recalaozerowregtwoone}. Since $\beta_1=-2k'\dot \alpha + \ddot \alpha$ by \eqref{recalaozerowregtwoone}, and since  $ |\dot{\alpha}(t)| \lesssim w^{-1}, 
    |\ddot{\alpha}(t)| \lesssim w^{-2}$ by \eqref{factsalphasteponeparabolictwopointone}, 
\begin{align}\label{estimatesbetaonezerowtwoone}
    |\beta_1| \lesssim\frac{k'}{w} + \frac{1}{w^2}.
\end{align}

By \eqref{zerowsonebounded}, $|s_1|\lesssim 1,$ so by \eqref{boundsasonezerowtwoone} $\|A_{s_1}\|_{\infty} \lesssim \frac{1}{k^2}$. Since $\e^{-(k'-k)t}\leq 1$ (as $k'>k$ by assumption), we get using \eqref{amoinsidzerowtwoone}, that,
$$
\|A-Id\|_{\infty} \lesssim\frac{k'}{k^2w} + \frac{1}{k^2w^2}.
$$
By assumption $k'-k\lesssim w^{-1},$ hence, $\frac{k'}{k} \lesssim 1+\frac{1}{kw}.$ Therefore, for $t \in [0, w],$
$$
\|A-Id\|_{\infty} \lesssim \frac{1}{kw} + \frac{1}{k^2w^2} \lesssim \frac{1}{kw}
$$
since  $kw\gtrsim 1$ by assumption. 

The case $t \in [w, 2w]$ is similar. Indeed, $\|A-Id\|_{\infty}=\|A_2-Id\|_{\infty} =\|\beta_2(t)e^{(k'-k) t}A_{s_2} \|_{\infty}$ and $\beta_2(t):=2k \dot \alpha(t-w) - \ddot \alpha(t-w) $. So $|\beta_2|\lesssim\frac{k}{w}+\frac{1}{w^2}.$  The estimates for $A_{s_2}$ are similar to the one for $A_{s_1}$ with $s_2$ in place of $s_1$ (by Lemma \ref{lemmabistwopointnineinproofoftwopointne}) where $s_2(t)=\alpha(t-w) \e^{(k'-k)t}$. Again $\e^{(k'-k)t} \lesssim 1$ since $t\leq 2w$ and $k'-k\lesssim w^{-1}$ by assumption. So, the estimate $\|A-Id\|\lesssim \frac{1}{kw}$ holds when $t \in [w, 2w].$ Since $A=Id$ when $t\leq 0$ and $t\geq 2w$, we conclude that 
$$
\|A-Id\|_{\infty} \lesssim\frac{1}{kw} \quad \mbox{ in } \T^2\times \R.
$$

\item \textbf{We show that $\|\nabla A\|_{\infty}\lesssim \frac{1}{w}$}: For $t \in [0, w],$ by \eqref{recalaozerowregtwoone}, $$A=A_1=Id+\beta_1(t)\e^{(k-k')t}A_{s_1}.$$  Since $\nabla$ only involve $x, y$ derivatives, 
\begin{align}\label{nablaatwooneow}
  \nabla A = \beta_1(t)\e^{(k-k')t} \nabla A_{s_1}.  
\end{align}

By \eqref{estimatesbetaonezerowtwoone}, $|\beta_1| \lesssim\frac{k'}{w} + \frac{1}{w^2}.$ By assumption $k<k'$ so $\e^{(k-k')t}\leq 1$ since $t\geq 0.$ Therefore, 
\begin{align}\label{betaoneetwooneow}
|\beta_1(t)\e^{(k-k')t}| \lesssim\frac{k'}{w} + \frac{1}{w^2}.    
\end{align}

By \eqref{boundsasonezerowtwoone}, 
$
\|\nabla A_{s_1}\|_{\infty} \lesssim \frac{(1+|s_1|)}{k}    
$
and by \eqref{zerowsonebounded}, $|s_1(t)|\lesssim 1$. Therefore, 
\begin{align}\label{nablasonezerowtwoone}
  \|\nabla A_{s_1}\|_{\infty} \lesssim \frac{1}{k}.  
\end{align}

Combining \eqref{nablaatwooneow}, \eqref{betaoneetwooneow} and \eqref{nablasonezerowtwoone}, we get
$$
\|\nabla A\|_{\infty} \lesssim \frac{k'}{kw} + \frac{1}{kw^2}.
$$

Since $k'-k\lesssim w^{-1}$ by assumption, we get $\frac{k'}{k}\lesssim 1+\frac{1}{kw}\lesssim 1$ since $kw\geq 1$ by assumption. We conclude that for $t\in [0, w],$
$$
\|\nabla A\|_{\infty} \lesssim  \frac{1}{w}.
$$

It is again straightforward to see that the same argument will work for $t \in [w, 2w]$. Indeed, when $t \in [w, 2w],$ we saw when proving $\|A-Id\| \lesssim \frac{1}{kw}$ above that $|\beta_2(t)\e^{-(k-k')t}| \lesssim\frac{k}{w} + \frac{1}{w^2}.$  By Lemma \ref{lemmabistwopointnineinproofoftwopointne}, the estimates for $\nabla A_{s_2}$ are similar to the one for $ \nabla A_{s_1}$ with $s_2$ in place of $s_1$  where $s_2(t)=\alpha(t-w) \e^{(k'-k)t}$. Again $\e^{(k'-k)t} \lesssim 1$ since $t\leq 2w$ and $k'-k\lesssim w^{-1}$ by assumption. So $|s_2(t)|\lesssim 1$ and $\|\nabla A_{s_2}\|_{\infty} \lesssim \frac{1}{k}.$ Since $A=Id+\beta_2(t)\e^{(k'-k)t}A_{s_2}$,  the estimate $\|\nabla A\|_{\infty} \lesssim \frac{1}{w}$ holds for $t \in [w, 2w]$ (as $kw\gtrsim 1$ by assumption). Finally, since $A=Id$ when $t\leq 0$ and $t\geq 2w$, we conclude that 
$$
\|\nabla A\|_{\infty} \lesssim \frac{1}{w} \quad \mbox{ in } \T^2 \times \R.
$$

\item \textbf{We show $\nabla A$ is continuous at $t=0, t=w$ and $t=2w$:} For $t \in [0, w]$, by \eqref{nablaatwooneow}
  $$\nabla A = \beta_1(t)\e^{(k-k')t} \nabla A_{s_1}. $$ By \eqref{nablasonezerowtwoone}, $\|\nabla A_{s_1}\|_{\infty} \lesssim \frac{1}{k}$. Since $k<k',$ $\e^{(k-k')t}\leq 1.$ By \eqref{recalaozerowregtwoone},  $\beta_1=-2k'\dot\alpha + \ddot\alpha $ and goes to zero as $t$ goes to 0 and $w$ since the derivatives of $\alpha$ do (recall \eqref{factsalphasteponeparabolictwopointone}). Hence, 
  $$
  \lim_{t \to 0^+} \nabla A=  \lim_{t \to w^-} \nabla A=0.
  $$
Since $A=Id$ when $t\leq 0$, we showed that $\nabla A$ is continuous at $t=0.$ We skip the justification at $t=w^+$ and $t=2w^-$ since it is very similar.

\item \textbf{We show $\|\dot A\|\lesssim \frac{1}{w}$:} For $t \in [0, w]$, by \eqref{recalaozerowregtwoone} and by \eqref{boundsasonezerowtwoone},
\begin{align}\label{recalabetaonezeroonetwoone}
 A=A_1=Id+\beta_1(t)\e^{(k-k')t}A_{s_1}, \hspace{0.5cm} \beta_1=-2k'\dot\alpha + \ddot\alpha , \hspace{0.5cm}  s_1(t)=(1-\alpha) \e^{(k-k')t}.      
\end{align}

In particular, $s_1$ depends on $t.$ Therefore,
\begin{align}\label{computederivativeofAwrttimenew}
\dot{A} = \partial_t\left[\beta_1(t)\e^{-(k'-k)t}\right]A_{s_1}+ \beta_1(t)\e^{-(k'-k)t}\dot{s_1}\frac{\pl A_{s_1}}{\pl s_1}.    
\end{align}

\pseudosection{Estimation of the first term of in \eqref{computederivativeofAwrttimenew}}
 We have
\begin{align}\label{firsttermwhhendifferetiatingbetanew}
   \left\|\partial_t\left[\beta_1(t)\e^{-(k'-k)t}\right]A_{s_1}\right\|_{\infty} \leq \|\dot\beta_1(t) \e^{-(k'-k)t}A_{s_1}\|_{\infty} + \|\beta_1(t)(k'-k)\e^{-(k'-k)t}A_{s_1}\|_{\infty} .
\end{align}

\textbf{The first term in \eqref{firsttermwhhendifferetiatingbetanew}:}
By \eqref{boundsasonezerowtwoone}, $\|A_{s_1}\|_{\infty} \lesssim \frac{1+|s_1|}{k^2}.$ Since $0\leq \alpha \leq 1$ by \eqref{factsalphasteponeparabolictwopointone}, since $\e^{(k-k')t}\leq 1$ (as $k'>k$ by assumption and $t\geq 0$), and since $s_1=(1-\alpha)\e^{(k-k')t}$ by \eqref{recalabetaonezeroonetwoone}, we have $|s_1|\leq 1$ and 
\begin{align}\label{asonerecllzerowtwoone}
   \|A_{s_1}\|_{\infty} \lesssim \frac{1}{k^2}. 
\end{align}

Since $\beta_1=-2k'\dot\alpha + \ddot \alpha$ (see\eqref{recalabetaonezeroonetwoone}), and since $\partial_t^{(n)}\alpha(t)\lesssim w^{-n}$ for all $n \geq 1$ (see \eqref{factsalphasteponeparabolictwopointone}), we get 
\begin{align}\label{dotbetaonefirstfirst}
|\dot \beta_1| \lesssim \frac{k'}{w^2}+\frac{1}{w^3}.    
\end{align}

Therefore, the first term in \eqref{firsttermwhhendifferetiatingbetanew} is estimated by 
\begin{align}\label{firsttermfirsttermdota}
    \|\dot\beta_1(t) \e^{-(k'-k)t}A_{s_1}\|_{\infty} \lesssim \frac{k'}{k^2w^2}+\frac{1}{k^2w^3} \lesssim \left(1+\frac{1}{kw}\right) \left( \frac{1}{(kw)w}\right)+\frac{1}{(kw)^2w} \lesssim \frac{1}{w}
\end{align}
 where we used $\frac{k'}{k}\lesssim 1+\frac{1}{kw}$ (since $k'-k\lesssim w^{-1}$ by assumption) in the second inequality and $kw \geq 1$ (by assumption) in the last inequality.
\\

\textbf{The second term in \eqref{firsttermwhhendifferetiatingbetanew}:} By \eqref{recalabetaonezeroonetwoone} and since $\partial_t^{(n)}\alpha(t)\lesssim w^{-n}$ for all $n \geq 1$ (see \eqref{factsalphasteponeparabolictwopointone}), we get
\begin{align}\label{estimatebetaonerecallzerowtwoone}
|\beta_1| \lesssim\frac{k'}{w}+\frac{1}{w^2}.   
\end{align}

We saw above that $\e^{(k-k')t}\leq 1$ (as $k'>k$ by assumption and $t\geq 0$) and $$
   \|A_{s_1}\|_{\infty} \lesssim \frac{1}{k^2}.
$$
by \eqref{asonerecllzerowtwoone}. Therefore, the second term in \eqref{firsttermwhhendifferetiatingbetanew} is estimated by 
\begin{align}\label{firsttermsecondtermzerow}
  \|\beta_1(t)(k'-k)\e^{-(k'-k)t}A_{s_1}\|_{\infty} \lesssim (k'-k) \left(\frac{k'}{w}+\frac{1}{w^2}\right)\frac{1}{k^2}   \lesssim \frac{k'}{k^2w^2}+\frac{1}{k^2w^3} \lesssim \frac{1}{w}
\end{align}
where the second inequality follows from $k'-k\lesssim w^{-1}$ by assumption and where the last inequality follows from the same argument as in \eqref{firsttermfirsttermdota}.

Combining \eqref{firsttermfirsttermdota} and \eqref{firsttermsecondtermzerow} with \eqref{firsttermwhhendifferetiatingbetanew}, we get
\begin{align}\label{finalestimatefirstterm}
     \left\|\partial_t\left[\beta_1(t)\e^{-(k'-k)t}\right]A_{s_1}\right\|_{\infty}\lesssim\frac{1}{w}.
\end{align}

\pseudosection{Estimation of the second term in \eqref{computederivativeofAwrttimenew}:} We recall the term we want to estimate
$$
\|\beta_1(t)\e^{-(k'-k)t}\dot{s_1}\frac{\pl A_{s_1}}{\pl s_1}\|_{\infty}, \hspace{0.5cm} \mbox{ where }  s_1(t)=(1-\alpha) \e^{(k-k')t}.
$$

By \eqref{estimatebetaonerecallzerowtwoone}
$|\beta_1| \lesssim\frac{k'}{w}+\frac{1}{w^2},
$
hence, since $\e^{-(k'-k)t} \leq 1$ (as $k<k'$ by assumption and $t\geq 0$), we get
 
\begin{align}\label{boundexpagainagainagin}
|\beta_1 \e^{-(k'-k)t} | \lesssim\frac{k'}{w}+\frac{1}{w^2}.   
\end{align}

Since $s_1=(1-\alpha) \e^{(k-k')t}$, we get
$$
\dot s_1 = -\dot \alpha \e^{(k-k')t} -(1-\alpha)(k'-k)  \e^{(k-k')t}.
$$
We remember that $|\dot \alpha|\lesssim w^{-1} $  and $1-\alpha \leq 1$ since $0\leq \alpha \leq 1$ (see \eqref{factsalphasteponeparabolictwopointone}). Since $k'-k \lesssim w^{-1}$ by assumption, we get 
\begin{align}\label{boundsonezerowagain}
|\dot s_1| \lesssim \frac{1}{w}.    
\end{align}

By Lemma \ref{claim2Dtwopointoneproof},
\begin{align}\label{boundderivativeasonezerowtwoone}
\left\|\frac{\pl A_{s_1}}{\pl s_1}\right\|_{\infty} \lesssim \frac{1}{k^2}.    
\end{align}

Therefore, combining  \eqref{boundexpagainagainagin}, \eqref{boundsonezerowagain} and \eqref{boundderivativeasonezerowtwoone}, we get 
\begin{align}\label{secondtermzerowonewbla}
    \|\beta_1(t)\e^{-(k'-k)t}\dot{s_1}\frac{\pl A_{s_1}}{\pl s_1}\|_{\infty}\lesssim \frac{k'}{k^2w^2} + \frac{1}{k^2w^3} \lesssim \frac{1}{w}
\end{align}
by the last inequality in \eqref{firsttermsecondtermzerow}.

\textbf{Summing up:} We wanted to estimate $\dot A$. By \eqref{computederivativeofAwrttimenew}, we saw
$$
\dot{A} = \partial_t\left[\beta_1(t)\e^{-(k'-k)t}\right]A_{s_1}+ \beta_1(t)\e^{-(k'-k)t}\dot{s_1}\frac{\pl A_{s_1}}{\pl s_1}. 
$$
By \eqref{finalestimatefirstterm}
$$
\left\|\partial_t\left[\beta_1(t)\e^{-(k'-k)t}\right]A_{s_1}\right\|_{\infty}\lesssim\frac{1}{w}.
$$
     By \eqref{secondtermzerowonewbla} $$
    \|\beta_1(t)\e^{-(k'-k)t}\dot{s_1}\frac{\pl A_{s_1}}{\pl s_1}\|_{\infty}\lesssim \frac{k'}{k^2w^2} + \frac{1}{k^2w^3} \lesssim \frac{1}{w}.   
     $$
  
Therefore, when $t\in [0,w]$
$$
\|\dot A\|_{\infty} \lesssim \frac{1}{w}
$$
as claimed. We leave the case $t\in [w, 2w]$ to the reader. A similar argument will yield the bound $\frac{1}{w}$ in that case as well. Finally, since $A=Id$ for $t\leq 0$ and $t\geq 2w$, we conclude that 
$$
\|\dot A\|_{\infty} \lesssim \frac{1}{w} \quad \mbox{ in } \T^2 \times \R.
$$

\item \textbf{We show $\dot A$ is continuous at $t=0, t=w$ and $t=2w$ :}
In \eqref{computederivativeofAwrttimenew}, we saw for $t\in [0, w]$
$$
\dot{A} = \partial_t\left[\beta_1(t)\e^{-(k'-k)t}\right]A_{s_1}+ \beta_1(t)\e^{-(k'-k)t}\dot{s_1}\frac{\pl A_{s_1}}{\pl s_1}. 
$$
We recall that $\beta_1(t)=-2k'\dot \alpha(t) + \ddot\alpha(t)$ (see \eqref{recalabetaonezeroonetwoone}). All the terms in $\dot A$ above will involve  $\beta_1$ or its derivative, i.e., they involve the derivative of $\theta(t/w)$ (since $\alpha(t)=\theta(t/w)$ by definition). Recall that all the derivatives of $\theta(t/w)$ go to zero as $t$ go to $0^+$ or to $w^-$ (see \eqref{thetanallderivativezerowblablablabla}). An inspection of the bound  $|\dot A| \lesssim w^{-1}$ that we just proved will show that every term in $\dot A$ other than $\beta_1$ and its derivative can be bounded independently of $t$. Hence, $\dot A$ is estimated by a combination of constants (with respect to $t$, that is, terms involving $k$ and $w$ which are fixed), each being multiplied by a derivative of $\theta (t/w).$ Since all the derivatives of $\theta(t/w)$ go to zero as $t$ go to $0^+$ or to $w^-$, we get that $\dot A$ goes to zero as $t$ goes to $0^+$ and $w^-$. The same argument will show that $\dot A$ goes to zero as $t$
 goes to $w^+$ and $2w^-$. Since $A=Id$ for $t\leq 0$ and $t\geq 2w$, we get that $\dot A$ is continuous in $\T^2 \times \R.$ We leave the details to the interested reader.
\end{itemize}

\pseudosection{The estimate for $u$}

We recall what we want to prove: For any multi-index $\alpha \in \mathbb N^3$ of order $|\alpha|\leq 2$,
\begin{align}\label{boundforutoshow}
|\partial^{\alpha} u(x,y,t)| \lesssim (k')^{|\alpha|} \sup_{\T^2 \times [0, 2w]} |u|  
\end{align}

where $u$ was defined by
\begin{equation}
    u = \begin{cases} f+(1-\alpha(t))g, & t\leq w,\\
 \alpha(t-w)f+g,&t \geq w,
 \end{cases}
\end{equation}

where $$f=\cos(k x)e^{-k t} \textup{\quad and \quad} g=\cos(k'y)e^{-k' t}.$$

\begin{itemize}
    \item \textbf{The function $u$ is $C^2$:} Recall that the derivatives of $\alpha(t)$ go to zero as $t$ goes to 0 and $w^-$ and the derivative of $\alpha(t-w)$ go to zero as $t$ goes to $w^+$ and $2w^-.$ Therefore, $u$ is clearly $C^2$ at $t=w$ and therefore in $\T^2\times \R.$
    \item \textbf{The bounds \eqref{boundforutoshow} for $u$:} We will only treat the case $t \in [0, w]$ since the other case is analogous. First note that $u$ solves $ \ddot u+ \div(A \nabla u)=0$ and therefore the maximum principle applies: 
\begin{align}\label{maxprincipleforuforlemmazerooneproof}
    \sup_{t\in[0, 2w]} |u| = \sup_{t \in \{0, 2w\}} |u| = \sup_{t=0} |u|=1
\end{align}
by the definition of $u$. By a direct computation,  
$$|\dot{u}| \leq |\dot{f}|+(1-\alpha)|\dot{g}|+|\dot{\alpha}g| \lesssim k+k'+w^{-1}.$$ 
By using the hypothesis $w^{-1} \lesssim k$ and $k\leq k'$, we get the bound
\begin{align}\label{boundforudotproof}
|\dot{u}| \lesssim k'.  
\end{align}

Analogously,
\begin{align}\label{boundnablauproof}
|\nabla u|\lesssim |\nabla f|+(1-\alpha)|\nabla g| \lesssim k+k'\lesssim k'.    
\end{align}

We leave the estimation of the other derivatives to the reader (keeping in mind the assumption $0 \leq k'-k\leq w^{-1} \leq k$) .
This finishes the proof of Lemma \ref{PML1}. It only remains to prove Lemma \ref{claim2Dtwopointoneproof} and Lemma \ref{lemmabistwopointnineinproofoftwopointne}. This is included in Appendix \ref{prooftwodlemmaappendix}.
\end{itemize}
\end{proof}

\subsection{Proof of Lemma \ref{claim2Dtwopointoneproof}\label{prooftwodlemmaappendix}.}

In this section, we present the proof of the two-dimensional Lemma \ref{claim2Dtwopointoneproof}. We recall it for the reader's convenience.

\begin{lemma*}[\ref{claim2Dtwopointoneproof}]
Let $1\leq k \leq k' \leq 2k.$ Given a positive number $s$ and a function $u = \cos(k x)+s\cos(k'y)$ on $\T^2$, there exists a smooth vector field $V$ with divergence $\div(V)=\cos(k'y)$ such that the matrix $A_s:= \begin{pmatrix}a_s(x,y)&b_s(x,y)\\b_s(x,y)&0\end{pmatrix}$ uniquely defined by $A_s\nabla u=V$ is smooth and satisfies 
\begin{align}\label{estimatestosatisfyforas}
        \|A_s\|_{\infty} \lesssim \frac{(1+|s|)}{k^2}, \hspace{0.5cm} \|\nabla A_s\|_{\infty} \lesssim \frac{(1+|s|)}{k}
    \end{align}
    where we recall that the gradient only involves spatial derivatives.
 Moreover, the function $s\mapsto A_s$ satisfies
    $$
    \bigg\|\frac{\partial A_s}{ \partial s}\bigg\|_{\infty} \lesssim \frac{1}{k^2}.
    $$
\end{lemma*}

We also present a version of Lemma \ref{claim2Dtwopointoneproof} for a slightly different function $v.$
\begin{lemma*}[\ref{lemmabistwopointnineinproofoftwopointne}]
Let $1\leq k \leq k' \leq 2k.$  Given a positive number $s$ and a function $v = s\cos(k x)+\cos(k'y)$ on $\T^2$, there exists a smooth vector field $V$ with divergence $\div(V)=\cos(kx)$ such that the matrix  $A_s = \begin{pmatrix}0&b_s(x,y)\\b_s(x,y)&c_s(x,y)\end{pmatrix}$ uniquely defined by $A_s \nabla v=V$ is smooth and sataisfies \begin{align*}
        \|A_s\|_{\infty} \lesssim \frac{(1+|s|)}{k^2}, \hspace{0.5cm} \|\nabla A_s\|_{\infty} \lesssim \frac{(1+|s|)}{k}
    \end{align*}
    where we recall that the gradient only involves spatial derivatives.
     Moreover, the function $s\mapsto A_s$ satisfies
    $$
    \bigg\|\frac{\partial A_s}{ \partial s}\bigg\|_{\infty} \lesssim \frac{1}{k^2}.
    $$
\end{lemma*}

\begin{rem}
    The proof of Lemma \ref{lemmabistwopointnineinproofoftwopointne} is identical to the proof of Lemma \ref{claim2Dtwopointoneproof}  and we do not include it.
\end{rem}

\pseudosection{Heuristic idea for the proof of Lemma \ref{claim2Dtwopointoneproof}}
\textbf{Observation(first and failed attempt)}. The vector field $W:=\begin{pmatrix}
    0 \\ \frac{\sin(k'y)}{k'}
\end{pmatrix}$
is such that $\div(W)=\cos(k'y).$ By defining the matrix $A_s:= \begin{pmatrix}a_s(x,y)&b_s(x,y)\\b_s(x,y)&0\end{pmatrix}$ with $A_s \nabla u= V$, we get 
\begin{align}\label{asexpressed}
  \begin{cases}
    b_s=\frac{v_2}{\pl u/\pl x},\\a_s = \frac{v_1 - v_2\frac{\pl u/\pl y}{\pl u/\pl x}}{\pl u/\pl x},
\end{cases} \iff  \begin{cases}
    b_s=\frac{-v_2}{k\sin(kx)},\\a_s = \frac{-v_1}{k\sin(kx)}  + s \frac{v_2 k'\sin(k'y)}{k^2\sin(kx)^2}
\end{cases} 
\end{align}
since $u=\cos(kx)+s\cos(k'y)$ by definition. Hence, if we choose $V\equiv W,$ we can see that the coefficient $b_s$ becomes
$$
b_s=\frac{-\sin(k'y)}{k k' \sin(kx)}
$$
which blows up on $\T^2$ and therefore does not satisfy the boundedness estimates \eqref{estimatestosatisfyforas} from Lemma \ref{claim2Dtwopointoneproof}. We therefore have to choose $V$ more carefully. 
\\

We want to choose $V$ such that $A_s$ given in \eqref{asexpressed} satisfies the estimates \eqref{estimatestosatisfyforas} and such that $\div(V)=\cos(k'y)$. To achieve this goal, we can try to look for $v_2$ in the form $\frac{v_2 k'}{\sin(kx)^2}=g(x,y)$ where $|g|\lesssim 1$ (this comes from the relation for $a_s$ in \eqref{asexpressed}). So, we are led to consider
$$
v_2(x,y) = \frac{\sin(kx)^2 }{k'} g(x,y), \hspace{0.5cm} |g(x,y)|\lesssim 1.
$$
Recalling $\sin(x)^2=\frac{1}{2}(1-\cos(2x))$, we can rewrite $v_2$ as 
\begin{align}\label{guesstheformulavtwo}
v_2(x,y) = \frac{g(x,y)}{2k'} - \frac{g(x,y)\cos(2kx)}{2k'}, \hspace{0.5cm} |g(x,y)|\lesssim 1.    
\end{align}

We also want $\div(V)=\cos(k'y)$ and as we saw in the observation above, the vector field $\begin{pmatrix}
    0 \\ \frac{\sin(k'y)}{k'}
\end{pmatrix}$ satisfies this requirement. Hence, \eqref{guesstheformulavtwo} suggests the representation $V=W+X$ where $W:=\begin{pmatrix}
    0 \\ \frac{\sin(k'y)}{k'}
\end{pmatrix}$ and $X$ is divergence free. 
This is achieved by taking the function $g$ to be $g(x,y)=2\sin(k'y)$ (which satisfies $|g|\lesssim 1$). Such a $g$ yields  
$$
v_2(x,y) = \frac{\sin(k'y)}{k'} - \frac{\sin(k'y)\cos(2kx)}{k'}
$$
and  $X:=\begin{pmatrix}
    X_1 \\ - \frac{\sin(k'y)\cos(2kx)}{k'}
\end{pmatrix}$ which should be divergence-free (since $\div(W)=\cos(k'y)$ and since we want $\div(V)=\cos(k'y)$). The divergence-free condition on $X$ gives
$$
\frac{\pl X_1}{\pl x} = -\frac{\pl X_2}{\pl y} = \cos(k' y) \cos(2 k x)
$$
which in turn leads to $X_1=\frac{\cos(k'y)\sin(2kx)}{2k}$. Hence, so far we have 
$$
V=\begin{pmatrix}
    \frac{\sin(2kx)\cos(k'y)}{2k} \\
    \frac{2\sin(k'y) \sin(kx)^2}{k'}
\end{pmatrix}, \hspace{0.5cm}
 \div(V)=\cos(k'y).
 $$
It remains to make sure that such a $V$ yields a matrix $A_s$ defined by \eqref{asexpressed}, that satisfies the bounds \eqref{estimatestosatisfyforas} from Lemma \ref{claim2Dtwopointoneproof}: $ \|A_s\|_{\infty} \lesssim \frac{(1+|s|)}{k^2}, \hspace{0.5cm} \|\nabla A_s\|_{\infty} \lesssim \frac{(1+|s|)}{k}.$ Recalling that $1\leq k \leq k' \leq 2k$, this  is now straightforward to see. This finishes the heuristic argument for choosing $V.$ We present the proof of Lemma \ref{claim2Dtwopointoneproof} below.

\begin{proof}[Proof of Lemma \ref{claim2Dtwopointoneproof}]

\pseudosection{Definition of $V$}
We choose 
\begin{align}\label{defvoffcial}
V:= \begin{pmatrix}
    \frac{\sin (2k x) \cos(k'y)}{2k} \\  \frac{2\sin(k' y) \sin(k x)^2}{k'}
\end{pmatrix}.    
\end{align}

A direct computation shows that $\div(V) = \cos(k'y)$.

\pseudosection{Definition of $A_s$}
Recall that $u=\cos(kx)+s \cos(k'y)$ and that that the matrix $A_s$ is defined by  $$A_s:= \begin{pmatrix}a_s(x,y)&b_s(x,y)\\b_s(x,y)&0\end{pmatrix}, \quad A_s \nabla u= V.$$ Therefore, we get 
\begin{align}\label{asexpressedbis}
   \begin{cases}
    b_s=\frac{v_2}{\pl u/\pl x},\\a_s = \frac{v_1 - v_2\frac{\pl u/\pl y}{\pl u/\pl x}}{\pl u/\pl x},
\end{cases} \iff  \begin{cases}
    b_s=\frac{-v_2}{k\sin(kx)},\\a_s = \frac{-v_1}{k\sin(kx)}  + s \frac{v_2 k'\sin(k'y)}{k^2\sin(kx)^2},
\end{cases} 
\end{align}

The choice of $V$ \eqref{defvoffcial} leads to

\begin{align}\label{defofbs}
  b_s= \frac{-2\sin(k'y)\sin(k x)}{kk'},  
\end{align}

and
\begin{align}\label{defofas}
    a_s = -\frac{\cos(k'y)\cos(k x)}{k^2}+\frac{2 s\sin(k'y)^2}{k^2}.
\end{align}

The bounds on $A_s$ and $\nabla A_s$ from Lemma \ref{claim2Dtwopointoneproof} are immediate. The above considerations are true for all $s \in \R^+$. Hence, we can consider $s\rightarrow A_s$. It comes
    $$
    \bigg\|\frac{\partial A_s}{ \partial s}\bigg\|_{\infty} \lesssim \frac{1}{k^2}.
    $$
    This finishes the proof of Lemma \ref{claim2Dtwopointoneproof}.
\end{proof}

\subsection{Estimation of the duration in Lemma \ref{altabstractlemmaevaluenewnew}}\label{appendixdurationsymmetry}
As mentioned in the beginning of the proof of Lemma \ref{altabstractlemmaevaluenewnew}, to be more quantitative, we make an explicit choice for $t_2-t_1$ such that the smooth function $a(t)$ has its first derivative bounded by 10.

We first choose explicitly the function $a(t).$ Recall that we want $a(t)=\frac{\mu}{2k^2}$ for $t\leq t_1$ and $a(t)=1+\frac{\mu}{k^2}$ for $t\geq t_2$. We choose 
$$
a(t):=1+\frac{\mu}{k^2}-\left(\frac{\mu}{2k^2}+1 \right)\theta\left(\frac{t-t_1}{t_2-t_1}\right)$$

where $\theta(\tau)$ is a smooth function going from 1 to 0 as $\tau$ goes from 0 to 1 (see \footnote{We choose $\theta(t):= 1-G(\tan(\pi(t-1/2)))$ and where $G(x)=\frac{1}{\sqrt{\pi}}\int_{-\infty}^x \e^{-\eta^2} \ud{\eta}$.}). We have 
\begin{align}\label{dotaduration}
    \dot{a}(t) = -\left(\frac{\mu}{2k^2}+1 \right)(t_2-t_1)^{-1}\dot{\theta}\left(\frac{t-t_1}{t_2-t_1}\right).
\end{align}

By our choice of $\theta$, we get that 
$$
\left|\dot{\theta}\left( \frac{t-t_1}{t_2-t_1}\right)\right| \leq \sqrt{\pi}.
$$
Therefore, 
\begin{align}\label{adotbisduration}
|\dot{a}(t)|\leq \sqrt{\pi}\left(\frac{\mu}{2k^2}+1 \right)(t_2-t_1)^{-1}.    
\end{align}

We choose $t_2-t_1$ such that the right-hand side in \eqref{adotbisduration} equals 10. This gives us
$$
t_2-t_1 = \frac{1}{10}\sqrt{\pi} \left(1+\frac{\mu}{2k^2}\right) \leq \frac{2}{5}
$$
since $k \gg \sqrt{\mu}$ by assumption. This finishes the proof.

\end{document}